\documentclass[10pt]{amsart}
\usepackage{amsmath}
\usepackage{amssymb}
\usepackage{stmaryrd}
\usepackage{amsthm}
\usepackage{amscd}
\usepackage{amsfonts}
\usepackage{graphicx}
\usepackage{fancyhdr}
\usepackage[utf8]{inputenc}
\usepackage{mathrsfs}
\usepackage{tikz-cd}
\usepackage{mathtools}
\usepackage[all,cmtip]{xy}
\usepackage{hyperref}
\numberwithin{equation}{section}
\newtheorem{theorem}[equation]{Theorem}

\newtheorem{assumption}[equation]{Assumption}
\newtheorem{lemma}[equation]{Lemma}
\newtheorem{proposition}[equation]{Proposition}
\newtheorem{corollary}[equation]{Corollary}

\newtheorem*{corollary*}{Corollary}

\theoremstyle{definition}
\newtheorem{remark}[equation]{Remark}
\newtheorem{construction}[equation]{Construction}
\newtheorem{definition}[equation]{Definition}

\newenvironment{myproof}[2] {\emph{Proof of {#1} {#2}.}}{\hfill$\square$}

\def\Gal{\mathrm{Gal}}
\def\GL{\mathrm{GL}}
\def\GSp{\mathrm{GSp}}
\def\GSpin{\mathrm{GSpin}}
\def\Sp{\mathrm{Sp}}

\def\GU{\mathrm{GU}}
\def\SO{\mathrm{SO}}

\def\Nilp{\mathrm{Nilp}}

\def\Lie{\mathrm{Lie}}

\def\Ker{\mathrm{Ker}}
\def\Coker{\mathrm{Coker}}
\def\min{\mathrm{min}}
\def\lr{\mathrm{lr}}
\def\mix{\mathrm{mix}}
\def\ram{\mathrm{ram}}

\def\Pa{\mathrm{Pa}}
\def\Kl{\mathrm{Kl}}
\def\Sie{\mathrm{Sie}}
\def\sing{\mathrm{sing}}

\def\Frob{\mathrm{Frob}}
\def\Ta{\mathrm{Ta}}
\def\IH{\mathrm{IH}}

\DeclareMathOperator{\rank}{rank}

\DeclareMathOperator{\Hom}{Hom}

\DeclareMathOperator{\End}{End}

\DeclareMathOperator{\Gr}{Gr}
\DeclareMathOperator{\Spec}{Spec}

\def\calD{\mathcal{D}}
\def\calE{\mathcal{E}}
\def\calF{\mathcal{F}}
\def\calH{\mathcal{H}}
\def\calL{\mathcal{L}}
\def\calN{\mathcal{N}}
\def\calM{\mathcal{M}}
\def\calO{\mathcal{O}}

\def\calS{\mathcal{S}}
\def\calT{\mathcal{T}}

\def\calQ{\mathcal{Q}}

\def\fracm{\mathfrak{m}}
\def\fracn{\mathfrak{n}}

\def\AAA{\mathbb{A}}

\def\bfB{\mathbf{B}}
\def\CC{\mathbb{C}}
\def\DD{\mathbb{D}}
\def\FF{\mathbb{F}}
\def\GG{\mathbb{G}}
\def\bfG{\mathbf{G}}

\def\sfH{\mathsf{H}}

\def\PP{\mathbb{P}}
\def\QQ{\mathbb{Q}}
\def\RR{\mathbb{R}}
\def\TT{\mathbb{T}}
\def\bfT{\mathbf{T}}
\def\bfx{\mathbf{x}}
\def\XX{\mathbb{X}}
\def\sfX{\mathsf{X}}
\def\ZZ{\mathbb{Z}}
\def\bfU{\mathbf{U}}

\def\rma{\mathrm{a}}
\def\rmb{\mathrm{b}}
\def\rmc{\mathrm{c}}

\def\rmv{\mathrm{v}}

\def\rmx{\mathrm{x}}

\def\rms{\mathrm{s}}

\def\sfA{\mathsf{A}}
\def\rmB{\mathrm{B}}
\def\rmC{\mathrm{C}}
\def\rmD{\mathrm{D}}
\def\rmE{\mathrm{E}}
\def\rmF{\mathrm{F}}
\def\rmG{\mathrm{G}}
\def\rmH{\mathrm{H}}
\def\rmI{\mathrm{I}}
\def\rmJ{\mathrm{J}}
\def\rmK{\mathrm{K}}
\def\rmL{\mathrm{L}}
\def\rmM{\mathrm{M}}
\def\rmN{\mathrm{N}}
\def\rmO{\mathrm{O}}
\def\rmP{\mathrm{P}}
\def\rmQ{\mathrm{Q}}
\def\rmU{\mathrm{U}}
\def\rmV{\mathrm{V}}
\def\rmW{\mathrm{W}}
\def\rmR{\mathrm{R}}
\def\rmS{\mathrm{S}}
\def\rmX{\mathrm{X}}
\def\rmY{\mathrm{Y}}
\def\rmT{\mathrm{T}}
\def\rmZ{\mathrm{Z}}

\def\Iw{\mathrm{Iw}}
\def\Spf{\mathrm{Spf}}

\def\unr{\mathrm{unr}}

\def\Sh{\mathrm{Sh}}

\usepackage[top=1.1in, bottom=1.1in, left=1.0in, right=1.0in]{geometry}
\begin{document}
\title[level raising on quaternionic unitary Shimura variety]
{Arithmetic level raising on certain quaternionic unitary Shimura variety}


\begin{abstract}
In this article we prove an arithmetic level raising theorem for the symplectic group of degree four in the ramified case. This result is a key step towards the Beilinson--Bloch--Kato conjecture for certain  Rankin-Selberg motives associated to orthogonal groups within the framework of the Gan--Gross--Prasad conjecture. The theorem itself can be also viewed as an analogue of the Ihara's lemma or the Tate conjectures for special fibers of Shimura varieties at ramified characteristics. The proof relies heavily on the description of the supersingular locus of certain quaternionic unitary Shimura variety which is closely related to the classical Siegel threefold.
\end{abstract}

\author{Haining Wang }
\address{\parbox{\linewidth}{Haining Wang\\Shanghai Center for Mathematical Sciences,\\ Fudan University,\\No.2005 Songhu Road,\\Shanghai, 200438, China.~ }}
\email{wanghaining@fudan.edu.cn}
\date{\today}

\maketitle
\tableofcontents
\section{Introduction}
Let $f$ be an elliptic modular newform of weight $2$, which is normalized of level $\Gamma_{0}(\rmN)$ for some positive integer $\rmN$. Let $\ell\geq 3$ be a prime and we fix an isomorphism $\iota_{\ell}: \CC\cong\overline{\QQ}_{\ell}$.  Let $E=\QQ(f)$ be the Hecke field of $f$ and suppose $\lambda$ is the place of $E$ over $\ell$ determined by $\iota_{\ell}$. We write $E_{\lambda}$ for the localization of $E$ at $\lambda$ with valuation ring $\calO_{\lambda}$ and whose residual field is $k$. Then we can attach a Galois representation 
\begin{equation*}
\rho_{f,\lambda}: \rmG_{\QQ}\rightarrow \GL_{2}(E_{\lambda})
\end{equation*}
 to $f$ such that $\rho_{f,\lambda}$ is unramfied outside $\rmN \ell$ and is characterized by the property that the trace of Frobenius at a prime $p$ not dividing $\rmN l$ equals to the $p$-th Fourier coefficient of $f$.  Let $\overline{\rho}_{f,\lambda}: \rmG_{\QQ}\rightarrow \GL_{2}(k)
$ be the residual representation which we assume is absolutely irreducible. Then Ribet \cite{ribet} proves the following level raising theorem.
\begin{theorem}[Ribet]
Suppose that $\overline{\rho}_{f,\lambda}$ is absolutely irreducible and $p$ is a prime away from $\rmN l$. If we have
\begin{equation*}
a^{2}_{p}\equiv (p+1)^{2}\mod \lambda,
\end{equation*}
then there exists a weight $2$ normalized newform $g$ of level $\Gamma_{0}(\rmN p)$ such that  $\overline{\rho}_{f,\lambda}\cong \overline{\rho}_{g,\lambda}$.
\end{theorem}
We will refer to the condition $a^{2}_{p}\equiv (p+1)^{2}\mod \lambda$ as the level raising condition. Let $\alpha_{p}, \beta_{p}$ be the eigenvalues of $\overline{\rho}_{f,\lambda}(\Frob_{p})$. Then the level raising condition is equivalent to $\{\alpha_{p}, \beta_{p}\}=\{1, p\}$ or $\{\alpha_{p}, \beta_{p}\}=\{-1, -p\}$ which is further equivalent to the condition that $\alpha_{p}/\beta_{p}=p^{\pm1}$. Let $\AAA$ be adele ring over $\QQ$. Let $\pi$ be the automorphic representation of $\GL_{2}(\mathbb{A})$ associated to $f$. Since  the eigenvalues of $\overline{\rho}_{f,\lambda}(\Frob_{p})$ correspond to the Satake parameters of $\pi_{p}$ modulo $\lambda$, the level raising condition can be interpreted as the condition for $\pi_{p}$ to degenerates to the twisted Steinberg representation under the reduction modulo $\lambda$. This representation theoretic interpretation sheds some lights on the question of generalizing the above theorem to automorphic forms on higher rank reductive groups. One of the themes of this article is to establish a level raising result for cuspidal automorphic representations of $\GSp_{4}(\mathbb{A}_{\QQ})$ of general type in the sense of Arthur \cite{Arthur-GSp}. 

Due to the importance of level raising theorems, there are quite a few investigations in this direction using various different methods. We give a survey of some these methods. The original method of Ribet exploits the natural integral structure of the cohomology of the modular curves $\rmX_{0}(\rmN)$ and $\rmX_{0}(\rmN p)$. In particular, it relies on the so-called Ihara's lemma, which roughly speaking, says that the fundamental group of the special fiber of $\rmX_{0}(\rmN)$ modulo $p$ is generated by the Frobenius elements supported on the supersingular points of the special fiber of $\rmX_{0}(\rmN)$. It is quite difficult to prove the generalization of the Ihara's lemma to other reductive groups but fortunately there is a method of proving the level raising theorem using modularity lifting theorems. In particular, it relies on the Ihara's avoidance technique of Taylor. For an implementation of this type of method, see for example \cite{Gee11}.  The third method comes from part of Bertolini--Darmon's first reciprocity law \cite[Theorem 5.15, Corollary 5.18]{BD-Main}. Nowadays, the third method is known as the arithmetic level raising in the terminology of \cite{LTXZZ}. We will go into some details on this third method later as it directly inspires the present article. Finally, there is a more representation theoretic method due to Clozel \cite{Clo} which relies heavily on the theory of universal modules in mod $\ell$ representation theory of reductive groups. Note in particular, following this method Sorensen \cite{Sor-level-raising} proves a level raising theorem for $\GSp_{4}$. We will make some comments on the relation of his results and the result in this article in the last subsection of this introduction. For the moment, we only remark that these representation theoretic methods seem to work well even when the residual Galois representation has small image, for example, in the endoscopic or the CAP cases. 

The first step of proving the arithmetic level raising theorem is to transfer $f$ to an automorphic form on a definite quaternion algebra $\overline{\rmB}$ unramified at $p$ using the Jacquet-Langlands correspondence. Then $f$ is viewed as an element of $\calO_{\lambda}[\Sh(\overline{\rmB}, \rmK_{\rmH})]$
where $\Sh(\overline{\rmB}, \rmK_{\rmH})$ is the Shimura set of $\overline{\rmB}$ and $\rmK_{\rmH}$ is a suitable level structure which is unramified at $p$. As we know that if $f$ satisfies the level raising condition, then its associated automorphic representation $\pi$  has its local component $\pi_{p}$ degenerates to a twisted Steinberg representation which transfers to a representation of the quaternion division algebra over $\QQ_{p}$, it makes sense to consider the nearby quaternion algbera $\rmB$ which is obtained by switching the local invariant at $\infty$ and $p$ of $\overline{\rmB}$. From $\rmB$, we can construct a Shimura curve $\Sh(\rmB, \rmK_{\Iw})$ where $\rmK_{\Iw}$ is a suitable level structure such that 
its local model is the same as the classical modular curve with Iwahori level at $p$. Let $\TT^{\Sigma}$ be the Hecke algebra unramified away from a finite set $\Sigma$ of non-archimedean places of $\QQ$ containing all the places where either $\pi$ or $\rmB$ is ramified. Then the Hecke eigenvalues of $f$ furnishes a homomorphism $\phi_{\lambda}: \TT^{\Sigma}\rightarrow\calO_{\lambda}$. Let $\fracm=\ker(\TT^{\Sigma}\xrightarrow{\phi}\calO_{\lambda}\rightarrow\calO_{\lambda}/\lambda)$, we say $\fracm$ is the maximal ideal of $\TT^{\Sigma}$ corresponding to $\overline{\rho}_{f,\lambda}$. Let $m$ be the $\lambda$-adic valuation of $a^{2}_{p}-(p+1)^{2}$. The arithmetic level raising theorem in this setting says the following.
\begin{theorem}[Bertolini--Darmon]
Let $f$ be an elliptic modular newform of level $\Gamma_{0}(\rmN)$ and weight $2$ as above and suppose it can be transferred to an automorphic form on $\overline{\rmB}$. Let $p$ be a level raising prime for $f$. Assume that the residual representation $\overline{\rho}_{f,\lambda}$
\begin{enumerate}
\item  is absolutely irreducible and has image containing $\GL_{2}(\FF_{p})$;
\item  is $\ell$-isolated in the sense of \cite[Definition 1.2]{BD-Main}.
\end{enumerate} 
Then we have a canonical isomorphism
\begin{equation*}
\calO_{\lambda}/\lambda^{m}[\Sh(\overline{\rmB}, \rmK_{\rmH})]_{\fracm}\xrightarrow{\sim} \rmH^{1}_{\mathrm{sing}}(\QQ_{p^{2}}, \rmH^{1}(\Sh(\rmB, \rmK_{\Iw}), \calO_{\lambda}/\lambda^{m}(1))_{\fracm}).
\end{equation*}
\end{theorem}
The righthand side of the above isomorphism is the singular quotient of the Galois cohomology group 
\begin{equation*}
\rmH^{1}(\QQ_{p^{2}}, \rmH^{1}(\Sh(\rmB, \rmK_{\Iw}), \calO_{\lambda}/\lambda^{m}(1))_{\fracm}). 
\end{equation*}
The theorem implies that the maximal ideal $\fracm$ appears in the support of the cohomology $\rmH^{1}(\Sh(\rmB, \rmK_{\Iw}), \calO_{\lambda}(1))$ and since $\rmB$ is ramified at $p$, the automorphic representations carried by $\rmH^{1}(\Sh(\rmB, \rmK_{\Iw}), \calO_{\lambda}(1))$ are necessarily ramified and hence gives rise to a level raising form $g$ by the Jacquet--Langlands correspondence. Thus the above arithmetic level raising theorem implies the usual level raising theorem.

The isomorphism 
\begin{equation*}
\calO_{\lambda}/\lambda^{m}[\Sh(\overline{\rmB}, \rmK_{\rmH})]_{\fracm}\xrightarrow{\sim} \rmH^{1}_{\mathrm{sing}}(\QQ_{p^{2}}, \rmH^{1}(\Sh(\rmB, \rmK_{\Iw}), \calO_{\lambda}/\lambda^{m}(1))_{\fracm})
\end{equation*}
is really an analogue of the Ihara's lemma for the Shimura curve ramified at $p$. Here the set 
\begin{equation*}
\calO_{\lambda}/\lambda^{m}[\Sh(\overline{\rmB}, \rmK_{\rmH})]_{\fracm}
\end{equation*}
comes from the parametrizing set of the irreducible components of the supersingular locus of the special fiber of $\Sh(\rmB, \rmK_{\Iw})$ which turns out to be the whole special fiber. This isomorphism is also an analogue of the Tate conjecture proved by in Xiao and Zhu in \cite{XZ} which, roughly speaking, says that the middle degree cohomology of the special fiber of a Shimura variety of good reduction should be generated by Tate classes coming from their supersingular locus when we localize the cohomology at a sufficiently generic maximal ideal.

It is natural to ask if one can generalize the above arithmetic level raising result to higher dimensional Shimura varieties and deduce level raising theorems for automorphic forms on other reductive groups. For this direction, we only mention the following results.
\begin{itemize}
\item In \cite{Liu-cubic},Yifeng Liu proves the arithmetic level raising theorem  for certain quaternionic Hilbert modular threefold at a ramified prime. Parallel to this, the author proves the the arithmetic level raising theorem for triple product of Shimura curves at a ramified prime in \cite{Wangd}. These results are tailored to have applications to construct Euler system in the frame work of the arithmetic Gan--Gross--Prasad conjecture for $\SO_{3}\times\SO_{4}$. 
\item In a recent breakthrough, the authors of \cite{LTXZZ} establish the arithmetic level raising theorem for unitary Shimura varieties at a ramified prime. This is closely related to the arithmetic Gan--Gross--Prasad conjecture for  $\rmU_{n}\times\rmU_{n+1}$ and its associated Beillinson--Bloch--Kato conjecture. Indeed, the main result of \cite{LTXZZ} establishes one direction in the rank $0$ and rank $1$ case of the Beillinsion--Bloch--Kato conjecture attached to the Rankin--Selberg motive of $\rmU_{n}\times\rmU_{n+1}$.
\end{itemize}
In this article, we will establish the arithmetic level raising theorem for $\GSp_{4}$ over $\QQ$ and discuss its applications to level raising results for cuspidal automorphic representations of  $\GSp_{4}(\mathbb{A}_{\QQ})$. Our main results will be described in more details in the next subsection.

\subsection{Main results} Let $\pi$ be a cuspidal automorphic representation of $\GSp_{4}(\mathbb{A})$ of general type in the sense of Arthur \cite{Arthur-GSp}. This means that $\pi$ is non-endoscopic and non-CAP. Alternatively, we can define $\pi$ to be of general type if $\pi$ transfers to a cuspidal automorphic representation $\Pi$ of $\GL_{4}(\mathbb{A}_{\QQ})$. We assume that $\pi$ is of weight $(3, 3)$ with trivial central character. Suppose $E$ is a strong coefficient field for $\pi$ in the sense of Definition \ref{strong-field}. We fix as before an isomorphism $\iota_{\ell}:\CC\cong\overline{\QQ}_{\ell}$ which induces the place $\lambda$ of $E$ over $\ell$. We can attach to such $\pi$ a continuous homomorphism
\begin{equation*}
\rho_{\pi,\lambda}: \rmG_{\QQ}\rightarrow\GSp_{4}(E_{\lambda})
\end{equation*}
characterized by the property that for a prime $p$ such that $\pi_{p}$ is unramfied, we have
\begin{equation*}
\det(\rmX-\rho_{\pi,\lambda}(\Frob_{p}))=\calQ_{p}(\rmX)
\end{equation*}
where $\calQ_{p}(\rmX)$ is the Hecke polynomial given by
\begin{equation*}
\rmX^{4}-a_{p,2}\rmX^{3}+(pa_{p,1}+(p^{3}+p)a_{p, 0})\rmX^{2}-p^{3}a_{p, 0}a_{p, 2}\rmX+p^{6}a^{2}_{p, 0}
\end{equation*}
whose coefficients $a_{p, 0}, a_{p, 1}, a_{p, 2}$ are the eigenvalues of the spherical Hecke operators $\rmT_{p,0}, \rmT_{p,2}, \rmT_{p,1}$, defined as in \ref{Hecke-operators}, 
which generate the spherical Hecke algebra for $\GSp_{4}(\QQ_{p})$ acting on the space $\pi^{\GSp_{4}(\ZZ_{p})}_{p}$ of spherical vectors of $\pi_{p}$.
The four roots $[\alpha_{p}, \beta_{p}, \gamma_{p}, \delta_{p}]$ of the Hecke polynomial $\calQ_{p}(\rmX)$ are called the \emph{Hecke parameters} of $\pi$ at $p$ which we can arrange to be in the form $[\alpha_{p}, \beta_{p}, p^{3}\beta^{-1}_{3}, p^{3}\alpha^{-1}_{p}]$ under our assumptions imposed on $\pi$.  
We will further assume that $\pi_{p}$ is an unramified type $\rmI$ representation in the classification of Sally--Tadic \cite{ST-classification} and Schmidt \cite{Sch-Iwahori}. See the beginning of section $3$ for a review of this notion.
\begin{definition}
Suppose $p$ is a prime such that $\pi_{p}$ is an unramified type $\rmI$ representation and $m$ is a positive integer. We say $p$ is level raising special for $\pi$ of depth $m$ if 
\begin{enumerate}
\item $\ell\nmid (p^{2}-1)p$;
\item the Hecke parameters $[\alpha_{p}, \beta_{p}, p^{3}\beta^{-1}, p^{3}\alpha^{-1}_{p}]$ of $\pi_{p}$ satisfy simultaneously the following three condiitons, 
\begin{enumerate}
\item $\beta_{p}+p^{3}\beta^{-1}_{p}\equiv u (p+ p^{2})\mod \lambda$ for some $u\in\{\pm1\}$;
\item $\alpha_{p}+p^{3}\alpha^{-1}_{p}\not\equiv \pm (p+ p^{2})\mod \lambda$.
\end{enumerate} 
and such that the $\lambda$-adic valuation of $\beta_{p}+p^{3}\beta^{-1}_{p}-u (p+ p^{2})$ is exactly $m$.
\end{enumerate}
\end{definition}
In terms of the classification of non-supercuspidal irreducible admissible representations of $\GSp_{4}(\QQ_{p})$ of Sally--Tadic \cite{ST-classification} and Schmidt \cite{Sch-Iwahori},  this condition means $\pi_{p}$, which is an unramified type $\rmI$ representation, degenerates into a ramified representation of type $\rmI\rmI{\rma}$ under the reduction modulo $\lambda$. In the theory of local newforms for $\GSp_{4}$ of \cite{BS-newform}, representations of type $\rmI\rmI{\rma}$ can be viewed as the local newforms of paramodular level. Moreover, representations of type $\rmI\rmI{\rma}$ can be transferred to a representation of the inner form $\GU_{2}(\rmD)$ of $\GSp_{4}(\QQ_{p})$ for the quaternion division algebra $\rmD$ over $\QQ_{p}$ under the Jacquet--Langlands correspondence. This puts us in a similar setting as in the $\GL_{2}$ case discussed in the beginning of the introduction. We therefore will proceed analogously to prove the arithmetic level raising theorem.

Let $\pi$ be as above and $p$ be a level raising special prime for $\pi$ of depth $m$. Let $\overline{\rmB}$ be the definite quaternion algebra ramified at $q\infty$ where $q$ is a prime which will play a completely auxiliary role in this article. This restriction on the discriminant of the  quaternion algebra $\overline{\rmB}$ is completely unnecessary and is only assumed for simplicity. 
The quaternionic unitary group $\bfG(\overline{\rmB})=\GU_{2}(\overline{\rmB})$ of degree $2$ over $\QQ$ is an inner form of $\bfG=\GSp_{4}$ over $\QQ$. It defines a Shimura set $\Sh(\overline{\rmB}, \rmK_{\rmH})$ where $\rmK_{\rmH}$ is a suitable level structure which in particular is hyperspecial at $p$. Similar to the $\GL_{2}$-case, we will also switch the local invariant of $\overline{\rmB}$ at $p$ and $\infty$ to obtain the indefinite quaternion algebra $\rmB$ ramified at $pq$ and in this case we have the quaternionic unitary group $\bfG(\rmB)=\GU_{2}(\rmB)$ of degree $2$ over $\QQ$ which is another inner form of $\bfG=\GSp_{4}$. There is an associated PEL-type Shimura variety $\Sh(\rmB, \rmK_{\Pa})$ which classifies abelian fourfolds with quaternionic multiplication and principal polarizations. This Shimura variety is the central geometric object in this article, its level structure $\rmK_{\Pa}$ is chosen so that its integral model over $\ZZ_{p}$ has the same local model as the classical Siegel modular threefold with paramodular level at $p$. In particular, its integral model $\rmX_{\Pa}(\rmB)$ has isolated quadratic singularities at its special fiber. This allows us to compute the monodromy action on the cohomology of this Shimura variety using the Picard--Lefschetz formula. We will write the special fiber of $\rmX_{\Pa}(\rmB)$ as $\overline{\rmX}_{\Pa}(\rmB)$ and we will use the same notation $\overline{\rmX}_{\Pa}(\rmB)$ to denote its base change to the algebraic closure $\FF$ of $\FF_{p}$.

Now we can state our main result on arithmetic level raising for $\GSp_{4}$ in the ramified case. Let $\Sigma$ be a finite set of places of $\QQ$ containing all the places at which $\pi$ is ramified. Then the Hecke parameters of $\pi$ defines a morphism $\phi_{\pi}: \TT^{\Sigma}\rightarrow E$ and its $\lambda$-adic avatar $\phi_{\pi, \lambda}: \TT^{\Sigma}\rightarrow \calO_{\lambda}$. Using the same terminology as in the $\GL_{2}$ case, we say
\begin{equation*}
\fracm_{\pi}=\ker(\TT^{\Sigma}\xrightarrow{\phi_{\pi,\lambda}}\calO_{\lambda}\rightarrow\calO_{\lambda}/\lambda)
\end{equation*}
is the maximal ideal corresponding to $\overline{\rho}_{\pi,\lambda}$ and we define $\fracm$ by
\begin{equation*}
\fracm=\TT^{\Sigma\cup\{p\}}\cap \fracm_{\pi}.
\end{equation*}

\begin{theorem}[Arithmetic level raising for $\GSp_{4}$]\label{intro-main}
Suppose that $\pi$ is a cuspidal automorphic representation of $\GSp_{4}(\mathbb{A}_{\QQ})$ of general type with weight $(3,3)$ and trivial central character. 
Let $p$ be a level raising special prime for $\pi$ of depth $m$. 

We assume further that
\begin{enumerate}
\item the prime $\ell\geq 3$;
\item $\overline{\rho}_{\pi,\lambda}$ is rigid for $(\Sigma_{\min}, \Sigma_{\lr})$ as in Definition \ref{rigid};
\item $\overline{\rho}_{\pi,\lambda}$ is absolutely irreducible and the image of $\overline{\rho}_{\pi,\lambda}(\rmG_{\QQ})$ contains $\GSp_{4}(\FF_{\ell})$;
\item $\fracm$ is in the support of $\calO_{\lambda}[\Sh(\overline{\rmB}, \rmK_{\rmH})]$ and is generic in the sense of Definition \ref{generic}.
\end{enumerate}
Then we have an isomorphism
\begin{equation*}
\rmH^{1}_{\sing}(\QQ_{p^{2}}, \rmH^{3}_{\rmc}(\Sh(\rmB, \rmK_{\Pa}), \calO_{\lambda}/\lambda^{m}(2))_{\fracm})\xrightarrow{\sim}\calO_{\lambda}/\lambda^{m}[\Sh(\overline{\rmB}, \rmK_{\rmH})]_{\fracm}.
\end{equation*}
In particular $\fracm$ occurs in the support of $\rmH^{3}_{\rmc}(\Sh(\rmB, \rmK_{\Pa}), \calO_{\lambda}(2))$. 
\end{theorem}
Here the cohomology $\rmH^{d}_{\rmc}(\Sh(\rmB, \rmK_{\Pa}), \calO_{\lambda})_{\fracm}$ in the above theorem are understood as the cohomology of the Shimura variety $\Sh(\rmB, \rmK_{\Pa})$ base changed to $\overline{\QQ}$.  To prove the theorem, we identify $\rmH^{3}_{\rmc}(\Sh(\rmB, \rmK_{\Pa}), \calO_{\lambda})$ with the nearby cycle cohomology $\rmH^{3}_{\rmc}(\overline{\rmX}_{\Pa}(\rmB), \rmR\Psi(\calO_{\lambda}))$  and apply the Picard--Lefschetz formula to calculate the monodromy action on the cohomology $\rmH^{3}_{\rmc}(\overline{\rmX}_{\Pa}(\rmB), \rmR\Psi(\calO_{\lambda}))$ in terms of the space of vanishing cycles on $\overline{\rmX}_{\Pa}(\rmB)$. In the process of proving this theorem, we will give some refinements of the Lefschetz formula. In particular we will identify the graded pieces of the monodromy filtration in this setting with certain intersection cohomologies and prove the weight-monodromy conjecture for some interior cohomologies, see Corollary \ref{WM}.

The most difficult step in the proof of the main theorem is to prove that the Tate cycles coming from the supersingular locus of the Shimura variety $\overline{\rmX}_{\Pa}(\rmB)$ generate the degree $2$  and degree $4$ intersection cohomology of the special fiber $\overline{\rmX}_{\Pa}(\rmB)$ localized at a generic maximal ideal. More precisely, we prove the following theorem.

\begin{theorem}\label{intro-Tate}
Let $\fracm$ be a generic maximal ideal for $\pi$ in the sense of Definition \ref{generic}.
\begin{enumerate}
\item There is an isomorphism
\begin{equation*}
\calO_{\lambda}[\Sh(\overline{\rmB}, \rmK_{\rmH})]^{\oplus 2}_{\fracm} \xrightarrow{\sim}  \mathrm{IH}^{2}(\overline{\rmX}_{\Pa}(\rmB), \calO_{\lambda}(1))_{\fracm}.
\end{equation*}

\item There is an isomorphism
\begin{equation*}
 \mathrm{IH}^{4}_{\rmc}(\overline{\rmX}_{\Pa}(\rmB), \calO_{\lambda}(2))_{\fracm}\xrightarrow{\sim} \calO_{\lambda}[\Sh(\overline{\rmB}, \rmK_{\rmH})]^{\oplus 2}_{\fracm}.
\end{equation*}
\end{enumerate}
\end{theorem}

The appearance of the two copies of $\calO_{\lambda}[\Sh(\overline{\rmB}, \rmK_{\rmH})]_{\fracm}$ is related to the fact that the irreducible components of the supersingular locus of $\overline{\rmX}_{\Pa}(\rmB)$ are naturally parametrized by two copies of $\Sh(\overline{\rmB}, \rmK_{\rmH})$. This fact along with other finer properties of the supersingular locus are obtained in author's work \cite{Wanga, Wangb} and Oki's work \cite{Oki}.  

Coming back to the above theorem, the following interesting picture emerges: it seems that the low degree cohomology of the special fiber $\overline{\rmX}_{\Pa}(\rmB)$ will only be related to oldforms via the Jacquet-Langlands correspondence while the middle degree cohomology will contribute to the Galois representations of the newforms. This should be compared to a conjecture of Arthur predicting that the low degree cohomology of the Shimura variety (not the special fiber) should be generated by automorphic forms coming from theta correspondences.  Note that a similar picture is confirmed for the special fiber of certain unitary Shimura variety at a ramified characteristic, see \cite[Proposition 6.3.1]{LTXZZ}. Another related results to Theorem \ref{intro-Tate} is that the cohomology of the special fiber $\overline{\rmX}_{\Pa}(\rmB)$  with coefficient in $\calO_{\lambda}$ are all torsion free after localization at a generic maximal ideal $\fracm$. This seems to be a quite difficult question in general, the author can not prove it directly using only geometric techniques. Fortunately, this can be proved using techniques from the theory of Galois deformations developed in \cite{LTXZZa} for unitary Shimura varieties.  In the companion article \cite{Wangd}, we adapt their methods to our setting and prove a freeness result of the cohomology of the quaternionic unitary Shimura variety and the quaternionic unitary Shimura set over suitable universal deformation rings. 

\subsection{Applications and prospects} As mentioned in the discussion for the $\GL_{2}$ case, the arithmetic level raising theorem is not only an interesting theorem in its own right but also has applications in many other related problems. We finish this introduction by discussing some of these applications. 

\subsubsection{Level raising for $\GSp_{4}$} The most direct application of our arithmetic level raising theorem is to deduce a level raising theorem for 
cuspidal automorphic representations of $\GSp_{4}(\mathbb{A})$ of general type. More precisely, we can deduce the following theorem from the arithemtic level raising theorem.
\begin{theorem}\label{intro-level-raise}
Let $\pi$ be an automorphic representation of $\GSp_{4}(\mathbb{A})$ of general type with trivial central character and weight $(3,3)$ satisfying all the assumptions in Theorem \ref{intro-main}. Let $p$ be a level raising special prime for $\pi$. 

Then there exists an automorphic representation $\mathbf{\Pi}$ of $\GSp_{4}(\mathbb{A})$ of general type with weight $(3,3)$ and trivial central character such that
\begin{enumerate}
\item the component $\mathbf{\Pi}_{p}$ at $p$ is of type $\rmI\rmI\rma$;
\item we have an isomorphism of Galois representation as in the Construction 
\begin{equation*}
\overline{\rho}_{\pi, \lambda}\cong \overline{\rho}_{\mathbf{\Pi}, \lambda}.
\end{equation*}
\end{enumerate}
In this case we will say $\mathbf{\Pi}$ is a level raising of $\pi$. 
\end{theorem}
Sorensen proves a related level raising result for $\GSp_{4}$ in \cite{Sor-level-raising}. We comment that the level raising condition in his theorem is different from our's, in fact he assumes that $\pi_{p}$ is congruent to the trivial representation modulo $\lambda$. Also the focus of his article is different, while we assume $\pi$ is of general type throughout this article, $\pi$ is mainly of Saito--Kurokawa type in  \cite{Sor-level-raising} which makes it possible to have applications in the Bloch--Kato conjecture in the $\GL_{2}$ setting. 

\subsubsection{Beillinson--Bloch--Kato conjecture} The arithmetic level raising theorem is the first step towards the Beillinson--Bloch--Kato conjecture of certain Rankin-Selberg product motive for $\SO_{4}\times\SO_{5}$ within the frame work of the Gan--Gross--Prasad conjecture. In a sequel of this work, we will show that if certain period integral of on the definite Shimura set $\Sh(\overline{\rmB}, \rmK_{\rmH})$ is non-zero modulo $\lambda$, then certain Bloch--Kato Selmer group of certain Rankin--Selberg product motive for $\SO_{4}\times\SO_{5}$ is zero. This period integral should have direct relations to the Rankin-Selberg  $L$-function for automorphic forms on $\SO_{4}\times\SO_{5}$ once the the Global Gan--Gross--Prasad conjecture and the Ichino--Ikeda conjecture \cite{Ichino-Ikeda} is confirmed. However, as far as the author's knowledge, this is only known in certain endoscopic or CAP cases. Recently, based on ideas and constructions of this article, Sweeting \cite{Sweeting} finds some interesting applications on the Bloch--Kato conjecture for some four dimensional Galois representations of symplectic type.

\subsubsection{Level lowering theorem for $\GSp_{4}$} In the companion article \cite{Wange}, we will prove certain level lowering result for cuspidal automorphic representations $\pi$ of $\GSp_{4}(\mathbb{A})$ whose component at the level lowering prime $p$ admits paramodular fixed vectors.  The approach in that article follows the famous approach of Ribet in his proof of Serre's epsilon conjecture for modular forms. In particular, we need to apply the so called $p, q$-switch trick in order to realize the representation $\pi$ in the cohomology of the quaternionic unitary Shimura variety $\Sh(\rmB, \rmK_{\Pa})$ studied in this article. To achieve this, one has to first take an auxiliary prime $q$ and raise the level of $\pi$ at $q$ and our Theorem \ref{intro-level-raise} makes this possible. 

It would be interesting to generalize the results of this article to cuspidal automorphic representations $\pi$ of $\GSp_{4}(\mathbb{A}_{\rmF})$ where $\rmF$ is a totally real field. The main reason we make this restriction in this article is that the results in \cite{Wanga, Wangb} and \cite{Oki} are obtained only for the quaternionic unitary Shimura variety over $\QQ$. As far as the author's knowledge, the current Rapoport--Zink uniformization theorem does not apply to those Shimura varieties studied in this article if one considers totally real field. However, we still believe the description of the supersingular locus should be similar if one is able to extend our results to the totally real field case. It is also reasonable to ask if our result works for those $\pi$ of general weights. This is only a matter of formalisim if we restrict the Hodge--Tate weights of $\rho_{\pi, \lambda}$ in the regular Fontaine--Laffaille range. Finally, our results and computations should shed some light on the general problem of realizing arithmetic level raising on general orthogonal type Shimura varieties at ramified characteristic.

\subsection{Notations and conventions}  We will use common notations and conventions in algebraic number theory and algebraic geometry. The cohomology  of schemes appearing in this article will be understood as computed over the \'{e}tale sites. 
\begin{itemize}
\item For a field $\rmK$, we denote by $\overline{\rmK}$ a separable closure of $\rmK$ and put $\rmG_{\rmK}:=\Gal(\overline{\rmK}/\rmK)$ the Galois group of $\rmK$.  Suppose $\rmK$ is a number field with ring of integers $\calO_{\rmK}$. Let $v$ be a place of $\rmK$, then $\rmK_{v}$ is the completion of $\rmK$ at $v$ with valuation ring $\calO_{v}$ whose maximal ideal is also written as $v$ if no danger of confusion could arise. Let $\mathrm{Art}:\rmK^{\times}_{v}\rightarrow \mathrm{W}^{\mathrm{ab}}_{\rmK}$ be the local Artin map sending the uniformizers to the geometric Frobenius element $\Frob_{p}$. We let $\phi_{v}$ be the arithmetic Frobenius element at $v$. 
\item We let $\mathbb{A}_{\rmK}$ be the ring of ad\`{e}les over $\rmK$ and $\mathbb{A}^{(\infty)}_{\rmK}$ be the subring of finite ad\`{e}les. More generally, let $\rmS$ be a finite set of places of $\rmK$, then $\mathbb{A}^{(\rmS)}_{\rmK}$ will denote the ring of ad\`{e}les over $\rmK$ away from the places in $\rmS$. If $\rmK=\QQ$, then we omit $\QQ$ from the notations. Similarly, we let $\widehat{\calO}^{(\rmS)}_{\rmK}=\prod_{v\not\in\rmS}\calO_{v}$, in particular $\widehat{\ZZ}^{(\rmS)}=\prod_{v\not\in\rmS}\ZZ_{v}$.

\item Let $\rmK$ be a local field with valuation ring $\calO_{\rmK}$ and residue field $k$. We let $\rmI_{\rmK}$ be the inertia subgroup of $\rmG_{K}$. Suppose $\rmM$ is a $\rmG_{\rmK}$-module. Then the finite part $\rmH^{1}_{\mathrm{fin}}(\rmK, \rmM)$ of $\rmH^{1}(\rmK, \rmM)$ is defined to be $\rmH^{1}(k, \rmM^{\rmI_{\rmK}})$ and the singular part $\rmH^{1}_{\mathrm{sing}}(\rmK, \rmM)$ of $\rmH^{1}(\rmK, \rmM)$ is defined to be the quotient of $\rmH^{1}(\rmK, \rmM)$ by the image of $\rmH^{1}_{\mathrm{fin}}(\rmK, \rmM)$. 

\item Let $\ell\geq 3$ be a prime and we fix an isomorphism $\iota_{\ell}: \CC\cong \overline{\QQ}_{\ell}$. We will denote by $\epsilon_{\ell}$ the $\ell$-adic cyclotomic character. The prime $\ell$ will be the residue characteristic of the coefficient of the \'etale cohomologies in this article.

\item Let $p$ be a prime and let $\FF$ be an algebraically closed field containing $\FF_{p}$. We denote by $\sigma$ be the Frobenius element on $\FF$. Let $\rmW_{0}=\rmW(\FF)$ be the ring of Witt vectors of $\FF$ and $\rmK_{0}=\rmW(\FF)_{\QQ}$ its fraction field, then $\sigma$ extends to an action on $\rmW_{0}$ which we will denote by the same symbol. Let $\FF_{p^{d}}$ be the finite field of $p^{d}$ elements. Then $\ZZ_{p^{d}}$ is defined to be $\rmW(\FF_{p^{d}})$. Let $\rmM_{1}\subset \rmM_{2}$ be two $\rmW_{0}$-modules we wrire $\rmM_{1}\subset^{d} \rmM_{2}$ if the $\rmW_{0}$ colength of the inclusion is $d$. 
\item Let $\rmR$ be a ring. Let $\rmM_{1}\subset \rmM_{2}$ be two $\rmR$-modules, then we write $\rmM_{1}\subset_{\Gr\rmM}\rmM_{2}$ if $\rmM_{2}/\rmM_{1}=\Gr\rmM$. If $\rmR$ is ring and $\rmL$ is an $\rmR$-module and $\rmR^{\prime}$ is an $\rmR$-algebra, we define $\rmL_{\rmR^{\prime}}=\rmL\otimes_{\rmR} \rmR^{\prime}$.

\item Let $\rmX$ be a smooth variety, then we write $\calT_{\rmX}$ for the tangent bundle of $\rmX$. Suppose $\rmY\subset \rmX$ be a smooth  subvariety, then we denote by $\calN_{\rmY}(\rmX)$ the normal bundle of $\rmY$ in $\rmX$. 
\item If $\rmS=\Spec\rmR$, we write $\mathbb{A}^{n}_{\rmS}$ for the affine scheme over $\rmR$ of relative dimension $n$. Let $\rmT$ be a set of polynomials over $\rmR$ with $n$-variables, then $\rmV(\rmT)\subset \mathbb{A}^{n}_{\rmS}$ is defined to be the vanishing locus of the polynomials contained in $\rmT$. 


\item For our convention, $\GSp_{4}$ be the reductive group over $\ZZ$ defined by 
\begin{equation*}
\GSp_{4}=\{g\in\GL_{4}: g\rmJ g^{t}=\rmc(g)\rmJ\}
\end{equation*}
where $\rmJ$ is the antisymmetric matrix given by $\begin{pmatrix}0&s\\-s&0\\ \end{pmatrix}$
for $s=\begin{pmatrix}0&1\\ 1&0 \\ \end{pmatrix}$ and where $\rmc$ is the similitude character of $\GSp_{4}$. By definition, we have an exact sequence 
\begin{equation*}
1\rightarrow \Sp_{4}\rightarrow \GSp_{4}\xrightarrow{\rmc}\GG_{m}\rightarrow 1.
\end{equation*}
For the split reductive group $\bfG=\GSp_{4}$ over $\ZZ$, we will fix a datum $\bfT\subset \bfB =\bfT\bfU$ where $\bfT$ is the split torus of the diagonal matrices and $\bfB$ is the Borel subgroup of the upper triangular matrices. 


\item Let $\rmG$ be a reductive group over a $p$-adic field $\rmF$, we write $C^{\infty}_{\rmc}(G(\rmF))$ for the space of locally constant functions on $\rmG(\rmF)$ with compact support. 

\item We use covariant Dieudonn\'e theory in this article: a Dieudonn\'e module $\rmM$ over $\FF$ is a left $\rmW_{0}$-module equipped with a Frobenius linear $\rmF$ and a Frobenius anti-linear operator $\rmV$ such that $\rmM$ is finitely generated and free over $\rmW_{0}$. A polarization of $\rmM$ is an alternating form
\begin{equation*}
\langle\cdot,\cdot\rangle:\rmM\times\rmM\rightarrow \rmW_{0}
\end{equation*}
such that $\langle \rmF x, y\rangle=\langle x, \rmV y\rangle^{\sigma}$. Let $\rmN$ be an isocyrstal over $\rmK_{0}$, then we will refer to a Dieudonn\'e module $\rmM\subset \rmN$ as a Dieudonn\'e  lattice of $\rmN$.
\item Let $\sfA$ be an abelian scheme over a scheme $\rmS$. Then we have the Hodge exact sequence
\begin{equation*}
0\rightarrow\omega_{\sfA^{\vee}/\rmS}\rightarrow \rmH^{\mathrm{dR}}_{1}(\sfA/\rmS)\rightarrow \Lie(\sfA/\rmS)\rightarrow 0.
\end{equation*}
Let $\sfA$ be an abelian variety over $\FF$. We denote by $\rmD(\sfA)$ the covariant Dieudonn\'e module of the $p$-divisible group $\sfA[p^{\infty}]$.  Let $\Lie(\sfA)$ be the Lie algebra of $\sfA$. Then we have an exact sequence relating $\rmD(\sfA)$ and $\Lie(\sfA)$
\begin{equation*}
0\rightarrow \rmV\rmD(\sfA)\rightarrow \rmD(\sfA)\rightarrow \Lie(\sfA)\rightarrow 0.
\end{equation*}
If we identify $\rmD(\sfA)/p\rmD(\sfA)$ with $\rmH^{\mathrm{dR}}_{1}(\sfA)$, then $\rmV\rmD(\sfA)$ is identified with $\omega_{\sfA^{\vee}}\otimes_{\FF,\sigma}\FF\cong \omega_{\sfA^{\vee}}$.
Let $v$ be a rational prime, then $\mathrm{Ta}_{v}(\sfA)$ will denote the $v$-adic Tate module of $\sfA$.

\item For an $\FF$-scheme $\rmS$ and a $p$-divisible group $\sfX$ over $\rmS$, let $\DD(\sfX/\rmS)$ be the Lie algebra of the universal vector extension of $\sfX$ with rank equal to the height of $\sfX$. Then we have an exact sequence of coherent $\calO_{\rmS}$-modules
\begin{equation*}
0\rightarrow\omega_{\sfX^{\vee}/\rmS}\rightarrow \DD(\sfX/\rmS)\rightarrow \Lie(\sfX/\rmS)\rightarrow 0.
\end{equation*}
Suppose $\sfX$ is a $p$-divisible group over $\FF$ and let $\rmM=\rmM(\sfX)$ be the the covariant Dieudonn\'e module of $\sfX$. If we identify $\rmM/p\rmM$ with $\DD(\sfX)$, then $\rmV\rmM/p\rmM$ is identified with $\omega_{\sfX^{\vee}}\otimes_{\FF,\sigma}\FF$.

\end{itemize}

\subsection*{Acknowledgements}  The author would like to thank Liang Xiao for all his help and encouragement throughout the course of writing this article. The author is also grateful to Yifeng Liu and Yichao Tian for very helpful conversations. Finally we would like to express our gratitude to the referee for all the corrections and suggestions.

\subsection*{Statements and Declarations}
The author declares no competing interests.

\section{Quaternionic unitary Shimura variety}

\subsection{Quaternionic unitary groups} We will recall the definition of the quaternionic unitary groups of degree $2$ and some of their parahoric subgroups. Then we define integral models of certain quaternionic unitary Shimura varity with paramodular level structure and some related quaternionic unitary Shimura set. This  quaternionic unitary Shimura variety is closely related to the classical Siegel threefold with paramodular level structure. Their relation is similar to that of a classical modular curve with a Shimura curve for an indefinite quaternion algebra.

Let $\rmB=\rmB_{pq}$ be an indefinite quaternion algebra over $\QQ$ which is ramified exactly at two distinct primes $p$ and $q$ which are distinct from the prime $\ell$. Let $\ast$ be a nebentype involution on $\rmB$ and $\calO_{\rmB}$ be a maximal order fixed by $\ast$, see \cite[A.4, A.5]{KR-Siegel}. We have the quaternionic unitary group $\bfG(\rmB)=\mathrm{GU}_{2}(\rmB)$ defined by an alternating form $(\cdot, \cdot)$ on $\rmV=\rmB\oplus\rmB$ with the property that $(bx, y)=(x, b^{\ast}y)$ for $x, y\in \rmV$ and $b\in \rmB$. To fix a definite choice of such alternating form, we use the same recipe as in \cite[\S 7.1]{Oki}. Then we define the quaternionic unitary group of degree $2$ over $\rmB$ by
\begin{equation*}
\bfG(\rmB)(\QQ)=\{g\in \GL_{\rmB}(\rmV): (gx, gy)=\rmc(g)(x, y),\phantom{.}\rmc(g)\in \QQ^{\times}\}. 
\end{equation*}
In the local setting, let $\rmD$ be the quaternion division algebra over $\QQ_{p}$. Then we can define the group $\GU_{2}(\rmD)$ over $\QQ_{p}$ in the same way as above using an alternating form defined as in \cite[\S 2.1]{Oki}. We have $\bfG(\rmB)(\QQ_{p})=\GU_{2}(\rmD)(\QQ_{p})$. Note that the group $\GU_{2}(\rmD)$ is non-quasi-split but $\GU_{2}(\rmD)/\QQ_{p^{2}}\cong \GSp_{4}/\QQ_{p^{2}}$. 

Let $\overline{\rmB}=\rmB_{\infty q}$ be the definite quaternion algebra ramified at the achimedean place and the prime $q$.  Let $\ast$ be a {main type involution} on $\overline{\rmB}$ and $\calO_{\overline{\rmB}}$ be a maximal order fixed by $\ast$. We have the quaternionic unitary group $\bfG(\overline{\rmB})=\mathrm{GU}_{2}(\overline{\rmB})$ defined by an alternating form $(\cdot, \cdot)$ on $\rmV=\overline{\rmB}\oplus \overline{\rmB}$ with the property that $(bx, y)=(x, b^{\ast}y)$ for $x, y\in \rmV$ and $b\in \overline{\rmB}$. To fix a definite choice of such alternating form, we use the one defined in \cite[\S 7.1]{Oki}. Then we define the quaternionic unitary group of degree $2$ over $\overline{\rmB}$ by
\begin{equation*}
\bfG(\overline{\rmB})(\QQ)=\{g\in \GL_{\overline{\rmB}}(\rmV): (gx, gy)=\rmc(g)(x, y), \rmc(g)\in \QQ^{\times}\}. 
\end{equation*}
Note that $\bfG(\overline{\rmB})(\QQ)$ is anisotropic modulo the center and $\bfG(\overline{\rmB})(\QQ_{p})=\GSp_{4}(\QQ_{p})$ as $p$ is unramfied in $\overline{\rmB}$.

\subsection{Parahoric subgroups}
The affine Dynkin diagram of type $\widetilde{\rmC}_{2}$ is given by
\begin{displaymath}
  \xymatrix{\underset{0}\bullet \ar@2{->}[r] &\underset{1}\bullet &\underset{2}\bullet \ar@2{->}[l]}
\end{displaymath}
where the nodes $0$ and $2$ are special and $1$ is not special in the sense that the parahoric subgroup corresponding to $0$ and $2$ are special parahoric subgroups of $\bfG(\QQ_{p})$. These nodes correspond to the so-called vertex lattices in a symplectic space $\rmV$ over $\QQ_{p}$. 

\begin{definition}
We denote by $\calL_{\{i\}}$ the set of vertex lattices of type $i\in\{0, 1, 2\}$. They are by definition given by
\begin{equation*}\label{vertex-lattice}
\begin{split}
&\calL_{\{0\}}=\{\text{vertex lattice $\Lambda_{0}$ of type $0$ in $\rmV$: } p\Lambda^{\vee}_{0}\subset^{4} \Lambda_{0}\subset^{0} \Lambda^{\vee}_{0}\};\\
&\calL_{\{2\}}=\{\text{vertex lattice $\Lambda_{2}$ of type $2$ in $\rmV$: }  p\Lambda^{\vee}_{2}\subset^{0} \Lambda_{2}\subset^{4} \Lambda^{\vee}_{2}\};\\
&\calL_{\{1\}}=\{\text{vertex lattice $\Lambda_{1}$ of type $1$ in $\rmV$: }  p\Lambda^{\vee}_{1}\subset^{2} \Lambda_{1}\subset^{2} \Lambda^{\vee}_{1}\}.\\
\end{split}
\end{equation*}
We will usually write a vertex lattice $\Lambda_{1}$ of type $1$ as $\Lambda_{\Pa}$ as it is related to the paramodular subgroup. For a vertex lattice $\Lambda_{i}$ with $i\in\{0,1,2\}$, we denote its stabilizer in $\GSp(\rmV)\cong\GSp_{4}(\QQ_{p})$ by $\rmK_{\{i\}}$. 
\end{definition}

These vertex lattices and their stabilizer groups give rise to the following parahoric subrgoups of $\bfG(\QQ_{p})$. We introduce some notations and terminologies for these parahoric subgroups that will be used later. 
\begin{itemize}
\item The groups $\rmK_{\{0\}}$ and $\rmK_{\{2\}}$ are the {hyperspecial subgroups} of $\bfG(\QQ_{p})$ and are conjugate to $\bfG(\ZZ_{p})$. These are the maximal special parahoric subgroups of $\bfG(\QQ_{p})$. The group $\rmK_{\{1\}}$ is called the {paramodular subgroup} of $\bfG(\QQ_{p})$. This is the unique maximal non-special parahoric subgroup of $\bfG(\QQ_{p})$. We will also sometimes use the notation $\Pa$ for $\rmK_{\{1\}}$.

\item  Let $\rmK_{\{01\}}$ be the stabilizer of a pair of vertex lattices $(\Lambda_{0}, \Lambda_{1})$ of type $0$ and $1$ with the relation
\begin{equation*}
p\Lambda^{\vee}_{0}\subset p\Lambda^{\vee}_{1}\subset\Lambda_{1}\subset \Lambda_{0}. 
\end{equation*}
The group $\rmK_{\{01\}}$ is called the {Klingen parahoric} of $\bfG(\QQ_{p})$. We will also sometimes use the notation $\Kl$ for the {\em Klingen parahoric} $\rmK_{\{01\}}$.

\item Let $\rmK_{\{02\}}$  is the stabilizer of a pair of vertex lattices $(\Lambda_{0}, \Lambda_{2})$ of type $0$ and $2$ with the relation 
\begin{equation*}
p\Lambda^{\vee}_{0}\subset p\Lambda^{\vee}_{2}=\Lambda_{2}\subset \Lambda_{0}. 
\end{equation*}
The group $\rmK_{\{02\}}$ is called the {Siegel parahoric} of $\bfG(\QQ_{p})$. We will also sometimes use the notation $\Sie$ for the {\em Siegel parahoric} $\rmK_{\{02\}}$.

\item Let $\rmK_{\{012\}}$ be the stabilizer of a pair of vertex lattices $(\Lambda_{0}, \Lambda_{1}, \Lambda_{2})$ of type $0$, $1$ and $2$ with the relation 
\begin{equation*}
p\Lambda^{\vee}_{0}\subset p\Lambda^{\vee}_{1}\subset p\Lambda^{\vee}_{2}=\Lambda_{2}\subset\Lambda_{1}\subset\Lambda_{0}. 
\end{equation*}
The group $\rmK_{\{012\}}$ is called the {Iwahori subgroup} of $\bfG(\QQ_{p})$. Then $\rmK_{\{012\}}$ is conjugate to usual Iwahori subgroup defined as the preimage of the Borel subgroup $\bfB(\FF_{p})$ in $\bfG(\ZZ_{p})$. We will also sometimes use the notation $\Iw$ for the Iwahori subgroup $\rmK_{\{012\}}$. We define the { pro-$p$ Iwahori subgroup} $\Iw_{1}$ to be the preimage of $\bfU(\FF_{p})$ in $\bfG(\ZZ_{p})$.
\end{itemize}

The quaternionic unitary group $\GU_{2}(\rmD)$ is the unique inner form of $\GSp_{4}$ over $\QQ_{p}$, it has the same affine Dynkin diagram of  type $\widetilde{\rmC}_{2}$  given by
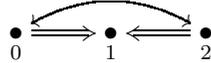
\begin{displaymath}
 \xymatrix{\underset{0}\bullet \ar@/^1pc/[rr]\ar@2{->}[r] &\underset{1}\bullet &\underset{2}\bullet \ar@2{->}[l] \ar@/_1pc/[ll]}
\end{displaymath}
but with the Frobenius acting on the diagram by permuting the node ${0}$ and the node $2$. Therefore the node $0$ and the node $2$ should be identified when we consider vertex lattices for $\GU_{2}(\rmD)$. We define a vertex lattice of type $\{0,2\}$ as a pair of lattices $(\Lambda_{0}, \Lambda_{2})$ where $\Lambda_{0}$ is a vertex lattice of type $0$ and $\Lambda_{2}$ is a vertex lattice of type $2$ with the relation $p\Lambda^{\vee}_{0}\subset p\Lambda^{\vee}_{2}=\Lambda_{2}\subset \Lambda_{0}$. 
\begin{itemize}

\item The stabilizer of a vertex lattice $(\Lambda_{0}, \Lambda_{2})$ of type $\{0, 2\}$ in $\GU_{2}(\rmD)$ is denoted by $\rmK^{\rmD}_{\{0,2\}}$. The parahoric subgroup $\rmK^{\rmD}_{\{0,2\}}$ splits to the Siegel parahoric and therefore we refer to it as the Siegel parahoric of $\GU_{2}(\rmD)$. Later on, we will use the notation $\Sie^{\rmD}$ for the group $\rmK^{\rmD}_{\{0,2\}}$.

\item The stabilizer of a vertex lattice $\Lambda_{1}$ of type $1$ in $\GU_{2}(\rmD)$ is denoted by $\rmK^{\rmD}_{\{1\}}$. The parahoric subgroup $\rmK^{\rmD}_{\{1\}}$ splits to the paramodular subgroup over $\ZZ_{p^{2}}$. Therefore we will sometimes refer to it as  the paramodular subgroup in $\GU_{2}(\rmD)$. Later on, we will use the notation $\mathrm{Pa}^{\rmD}$ for the group  $\rmK^{\rmD}_{\{1\}}$. 
\end{itemize}

\subsection{The local quaternionic unitary Shimura variety}
Let $\rmN$ be an isocrystal of height $8$ over $\FF$ which is isotypic of slope $\frac{1}{2}$. Let $\iota: \rmD\rightarrow \End(\rmN)$ be an action of $\rmD$ on $\rmN$. We can write $\rmD$ as $\QQ_{p^{2}}+\QQ_{p^{2}}\Pi$ with $\Pi^{2}=p$. Suppose $\rmN$ is equipped with an alternating form $(\cdot,\cdot): \rmN\times \rmN \rightarrow \rmK_{0}$ with the property that $(\rmF x, y)= (x, \rmV y)^{\sigma}$ and $(bx, y)= (x, b^{*}y)$ where $b^{*}$ is the image of $b$ under the neben type involution given by 
\begin{equation*}
\begin{split}
&x^{*}=\sigma(x),\phantom{.}x\in \QQ_{p^{2}}\\
&\Pi^{*}=\Pi.\\
\end{split}
\end{equation*}
Since $\ZZ_{p^{2}}\otimes\rmW_{0}\cong \rmW_{0}\times\rmW_{0}$, the action of $\mathcal{O}_{\rmD}\otimes\rmW_{0}$ decomposes $\rmN$ into $\rmN=\rmN_{0}\oplus \rmN_{1}$. Restricting the alternating form $(x, y)_{0}:= ( x, \Pi y)$ to $\rmN_{0}$ gives an identification between the group $\GU_{\rmD}(\rmN)$ and $\GSp(\rmN_{0})$
over $\rmW_{0}$ cf. \cite[1.42]{RZ-space}. We will use covariant Dieudonn\'{e} theory throughout this article. A Dieudonn\'{e} lattice $\rmM$ is a lattice in $\rmN$  with the property that $p\rmM \subset \rmF\rmM\subset \rmM$.  A Dieudonn\'{e} lattice is superspecial if $\rmF^{2}\rmM=p\rmM$ and in this case $\rmV\rmM=\rmF\rmM$. We are concerned with a Dieudonn\'{e} lattice $\rmM$ with an additional endomorphism $\iota: \mathcal{O}_{\rmD}\rightarrow \End(\rmM)$. We will represent $\mathcal{O}_{\rmD}$ by $\ZZ_{p^{2}}+\ZZ_{p^{2}}\Pi$ and hence we can decompose M as $\rmM_{0}\oplus \rmM_{1}$ with an additional operator $\Pi$ that swaps the two components. The alternating form $(\cdot,\cdot)$, restricted to $\rmM$, induces a pairing $(\cdot,\cdot): \rmM_{0}\times \rmM_{1}\rightarrow \rmW_{0}$. We will fix a $p$-divisible group $\XX$ over $\FF$ of dimension $4$ and height $8$ with quaternionic multiplication $\iota_{\XX}: \mathcal{O}_{\rmD}\rightarrow \End(\XX)$ and with polarization $\lambda: \XX\rightarrow \XX^{\vee}$ whose isocrystal agrees with $\rmN$ with its additional structures discussed above. Let $\Nilp$ be the category of $\rmW_{0}$-schemes over which $p$ is locally nilpotent.

\begin{construction}
We consider the set valued functor $\mathcal{N}_{\mathrm{Pa}}$ that sends $\rmS\in\Nilp$ to the isomorphism classes of the collection $(\sfX, \iota_{\sfX},  \lambda_{\sfX}, \rho_{\sfX})$ where:
\begin{itemize} 
\item $\mathsf{X}$ is a $p$-divisible group of dimension $4$ and height $8$ over $\rmS$;
\item $\lambda_{\mathsf{X}}: \sfX\rightarrow \sfX^{\vee}$ is a principal polarization;
\item $\iota_{\sfX}: \calO_{\rmD}\rightarrow \End_{\rmS}(\sfX)$ is an action of $\calO_{\rmD}$ on $\sfX$ defined over $\rmS$;
\item $\rho_{\sfX}: \sfX\times_{\rmS} \overline{\rmS}\rightarrow \XX\times_{\FF}\overline{\rmS}$ is an $\calO_{\rmD}$-linear quasi-isogeny over $\overline{\rmS}$ which is the special fiber of $\rmS$ at $p$.
\end{itemize}

We require that $\iota_{\sfX}$ satisfies the Kottwitz condition 
\begin{equation*}
\det(\rmT-\iota_{\sfX}(c)\vert \Lie(\sfX))=(\rmT^{2}-\mathrm{Trd}^{0}(c)\rmT+ \mathrm{Nrd}^{0}(c))^{2}
\end{equation*}
for $c\in\mathcal{O}_{\rmD}$, where $\mathrm{Trd}^{0}(c)$ is the reduced trace of $c$ and $\mathrm{Nrd}^{0}(c)$ is the reduced norm of $c$. For $\rho_{\sfX}: \sfX\times_{\rmS} \overline{\rmS}\rightarrow \XX\times_{\FF}\overline{\rmS}$, we require that $\rho_{\sfX}^{\vee}\circ\lambda_{\XX}\circ\rho_{\sfX}=c(\rho_{\sfX})\lambda_{\sfX}$
for a $\QQ_{p}$-multiple $c(\rho_{\sfX})$. 
\end{construction}

This moduli problem is representable by a formal scheme $\calN_{\Pa}$, locally formally of finite type over $\Spf\phantom{.}\mathrm{W}_{0}$. The formal scheme $\calN_{\Pa}$ can be decomposed into open and closed formal subschemes: 
\begin{equation*}
\calN_{\Pa}=\bigsqcup_{i\in\ZZ}\calN_{\Pa}(i) 
\end{equation*}
where each $\calN_{\Pa}(i)$ is isomorphic to $\calN_{\Pa}(0)$. Let $b$ be the $\sigma$-conjugacy class corresponding to the isocrystal $\rmN$, the space $\calN_{\Pa}$ admits an action of the group $\rmJ_{b}(\QQ_{p})\cong \GSp_{4}(\QQ_{p})$. 

Denote by $\calM_{\Pa}$ the reduced scheme underlying the formal scheme $\calN_{\Pa}(0)$. The main result of \cite{Wanga}, \cite{Oki} concerns the structure of $\calM_{\Pa}$ which we starts to recall. More precisely, the scheme $\calM_{\Pa}$ admits the {\em Bruhat-Tits stratification} 
\begin{equation*}
\calM_{\Pa}= \rmM_{\Pa}\{0\}\sqcup  \rmM_{\Pa}\{2\} \sqcup  \rmM_{\Pa}\{02\} \sqcup \rmM_{\Pa}\{1\}
\end{equation*}
into locally closed subvarieties. We introduce the following notations related to the strata.
\begin{itemize}
\item The stratum $\rmM_{\Pa}\{0\}$ has its irreducible components indexed by the set $\calL_{\{0\}}$ of vertex lattices of type $\{0\}$. 
For each vertex lattice $\Lambda_{0}\in \calL_{\{0\}}$, the corresponding irreducible component is denoted by $\rmM_{\Pa}(\Lambda_{0})$. 
We write $\calM_{\Pa}\{0\}$ for the closure of $\rmM_{\Pa}\{0\}$ and  $\calM_{\Pa}(\Lambda_{0})$ for the closure of $\rmM_{\Pa}(\Lambda_{0})$.

\item The stratum $\rmM_{\Pa}\{2\}$ has its irreducible components indexed by the set $\calL_{\{2\}}$ of vertex lattices of type $\{2\}$. For each vertex lattice $\Lambda_{2}\in \calL_{\{2\}}$, the corresponding irreducible component is denoted by $\rmM_{\Pa}(\Lambda_{2})$. We write $\calM_{\Pa}\{2\}$ for the closure of $\rmM_{\Pa}\{2\}$ and  $\calM_{\Pa}(\Lambda_{2})$ for the closure of $\rmM_{\Pa}(\Lambda_{2})$.

\item The stratum $\rmM_{\Pa}\{02\}$ has its irreducible components indexed by a pair of vertex lattices $(\Lambda_{0}, \Lambda_{2})$ with the relation
$p\Lambda^{\vee}_{0}\subset p\Lambda^{\vee}_{2}=\Lambda_{2}\subset \Lambda_{0}$. For such a pair $(\Lambda_{0}, \Lambda_{2})$, the corresponding irreducible component is denoted by $\rmM_{\Pa}(\Lambda_{0}, \Lambda_{2})$ and its irreducible component is denoted by $\calM_{\Pa}(\Lambda_{0}, \Lambda_{2})$.

\item The stratum $\rmM_{\Pa}\{1\}=\calM_{\Pa}\{1\}$ is closed and is a discrete set of points indexed by the set  $\calL_{\{1\}}$ of vertex lattices of type $1$. For a lattice $\Lambda_{1}$ of type $1$, we denote the corresponding point by $\calM_{\Pa}(\Lambda_{1})$. For each vertex lattice $\Lambda_{i}$ with $i\in\{0,2\}$, we define $\calM_{\Pa}(\Lambda_{i})\{1\}$ to be $\calM_{\Pa}(\Lambda_{i})\cap \calM_{\Pa}\{1\}$.

\end{itemize}

\begin{theorem}\label{RZ-description}
We have the following descriptions of the Bruhat--Tits stratification of $\calM_{\Pa}$.
\begin{enumerate}
\item For each vertex lattice $\Lambda_{i}\in \calL_{\{i\}}$ with $i\in\{0,2\}$, the corresponding irreducible component $\rmM_{\Pa}(\Lambda_{i})$ is isomorphic to a certain Deligne--Lusztig variety  whose closure $\calM_{\Pa}(\Lambda_{i})$ is given by the projective surface  
\begin{equation*}
\rmZ_{3}^{p}\rmZ_{0}-\rmZ^{p}_{0}\rmZ_{3}+\rmZ^{p}_{2}\rmZ_{1}-\rmZ^{p}_{1}\rmZ_{2}=0
\end{equation*}
in a suitable coordinate system $[\rmZ_{0}: \rmZ_{1}:  \rmZ_{2}: \rmZ_{3}]$ of $\PP^{3}$. 

\item The stratum $\rmM_{\Pa}\{1\}=\calM_{\Pa}\{1\}$ consists of superspecial points of $\calM_{\Pa}$.

\item For each pair of vertex lattices $(\Lambda_{0}, \Lambda_{2})$ such that
$p\Lambda^{\vee}_{0}\subset p\Lambda^{\vee}_{2}=\Lambda_{2}\subset \Lambda_{0}$, the corresponding irreducible component $\rmM_{\Pa}(\Lambda_{0}, \Lambda_{2})$ is the intersection of $\rmM_{\Pa}(\Lambda_{0})$ and $\rmM_{\Pa}(\Lambda_{2})$, and its closure $\calM_{\Pa}(\Lambda_{0}, \Lambda_{2})$ is isomorphic to a projective line $\PP^{1}$.

\item For each pair $(\Lambda_{0}, \Lambda^{\prime}_{0})$ of vertex lattices of type $0$ such that $\Lambda_{1}=\Lambda_{0}\cap\Lambda^{\prime}_{0}$ is a vertex lattice of type $1$, the irreducible components $\calM_{\Pa}(\Lambda_{0})$ and $\calM_{\Pa}(\Lambda^{\prime}_{0})$ intersect at the point $\calM_{\Pa}(\Lambda_{1})$. 

\item For each pair $(\Lambda_{2}, \Lambda^{\prime}_{2})$ of vertex lattices of type $2$ such that $\Lambda_{1}=\Lambda_{2}+\Lambda^{\prime}_{2}$ is a vertex lattice of type $1$, the irreducible components $\calM_{\Pa}(\Lambda_{2})$ and $\calM_{\Pa}(\Lambda^{\prime}_{2})$ intersect at the point $\calM_{\Pa}(\Lambda_{1})$.
\end{enumerate}
\end{theorem}
\begin{proof}
This follows from  \cite[Theorem 5.1]{Wanga}.
\end{proof}

We will need to recall some ingredients going into the proof of the above theorem. Given an $\FF$-point $x=(\sfX, \lambda_{\sfX}, \iota_{\sfX}, \rho_{\sfX} )\in\calM_{\Pa}(\FF)$, we denote by $\mathrm{M}=\mathrm{M}_{0}\oplus\mathrm{M}_{1}$ the Dieudonn\'{e} lattice attached to $x$. Then the set $\calM_{\Pa}(\FF)$ can be identified with the set  
\begin{equation*}
\{\rmM\subset \rmN: \rmM^{\perp}=\rmM,\phantom{.}p\rmM\subset^{4} \rmV\rmM\subset^{4} \rmM,\phantom{.}p\rmM\subset^{4} \Pi \rmM\subset^{4} \rmM\}
\end{equation*}
where $\rmM^{\perp}$ is the integral dual of $\rmM$ in $\rmN$ with respect to $(\cdot, \cdot)$.  Then projecting $\rmM$ to $\mathbb{D}=\rmM_{0}$ identifies this set with 
\begin{equation}\label{RZ-F-point}
\{\DD \subset \rmN_{0}:p\DD^{\vee} \subset^{2} \DD\subset^{2} \DD^{\vee},\phantom{.}p\tau\DD^{\vee}\subset^{2} \DD\subset^{2} \tau\DD^{\vee}\}
\end{equation}
where $\DD^{\vee}$ is the integral dual of $\DD$ with respect to the pairing $(\cdot, \cdot)_{0}$ on $\rmN_{0}$ defined by $(x, y)_{0}=(x, \Pi y)$ for $x\in \rmM_{0}$ and $y\in \rmM_{0}$ and $\tau$ is given by $\Pi\rmV^{-1}$. Under this identification, the superspecial points of $\calM_{\Pa}$ correspond to those lattices $\DD\subset\rmN_{0}$ such that $\tau\DD=\DD$.

\subsubsection{The scheme $\calM_{\Pa}(\Lambda_{0})$} Let $\calM^{\square}_{\Pa}(\Lambda_{0})=\calM_{\Pa}(\Lambda_{0})-\calM_{\Pa}\{1\}$. In terms of the descriptions of \ref{RZ-F-point}, an $\FF$-point of $\calM^{\square}_{\Pa}(\Lambda_{0})$ corresponds to a lattice $\DD$ such that $\DD+\tau(\DD)$ is $\tau$-stable and gives rise to a vertex lattice $\Lambda_{0}$ of type $0$. In this case $\DD$ and  $\tau\DD$ inserts in the chains
\begin{equation*}
\begin{aligned}
& p\Lambda_{0}\subset p\DD^{\vee}\subset \DD\cap \tau{\DD} \subset \DD \subset \Lambda_{0}\\
&p\Lambda_{0}\subset p\tau\DD^{\vee}\subset  \DD\cap \tau{\DD} \subset \tau\DD \subset \Lambda_{0}.\\
\end{aligned}
\end{equation*}
Taking the quotient by $p\Lambda_{0}$, we obtain two complete isotropic flags
\begin{equation*}
\begin{aligned}
p\DD^{\vee}/p\Lambda_{0}&\subset  \DD\cap \tau{\DD}/p\Lambda_{0} \subset \DD/p\Lambda_{0} \subset \Lambda_{0}/p\Lambda_{0} \\
p\tau\DD^{\vee}/p\Lambda_{0}&\subset  \DD\cap \tau{\DD}/p\Lambda_{0} \subset \tau\DD/p\Lambda_{0} \subset \Lambda_{0}/p\Lambda_{0}\\
\end{aligned}
\end{equation*}
in the symplectic space $\overline{\rmV}_{\Lambda_{0}}=\Lambda_{0}/p\Lambda_{0}$. Now it is clear that the latter two isotropic flags agree with the Deligne--Lusztig variety defined in \cite[2.4]{DL} and its equation is worked out to be
\begin{equation}\label{DL-eq}
\rmZ_{3}^{p}\rmZ_{0}-\rmZ^{p}_{0}\rmZ_{3}+\rmZ^{p}_{2}\rmZ_{1}-\rmZ^{p}_{1}\rmZ_{2}=0
\end{equation}
by \cite[Theorem 3.6]{Wanga}. We denote this Deligne--Lusztig variety by $\mathrm{DL}(\Lambda_{0})$ which is given by the set of lines $l\subset\overline{\rmV}_{\Lambda_{0}}$ such that
\begin{equation*}
l\subset \sigma(l)^{\perp}.
\end{equation*}
Here $ \sigma(l)^{\perp}$ is the orthogonal complement of the line $\sigma(l)$ with respect to the symplectic form on $\overline{\rmV}_{\Lambda_{0}}$. The set of lines fixed by $\sigma$ in $\mathrm{DL}(\Lambda_{0})$ will be called the set of superspecial points on $\mathrm{DL}(\Lambda_{0})$ and the complement of the set of superspecial points will be  denoted by $\mathrm{DL}^{\square}(\Lambda_{0})$. The map sending $\DD$ to $l=p\DD^{\vee}/\Lambda_{0}$ gives a bijection from $\calM_{\Pa}(\Lambda_{0})(\FF)$ to $\mathrm{DL}(\Lambda_{0})(\FF)$ which induces a bijection from $\calM^{\square}_{\Pa}(\Lambda_{0})(\FF)$ to $\mathrm{DL}^{\square}(\Lambda_{0})(\FF)$.

For the next lemma, we need to recall the map from $\calM_{\Pa}(\Lambda_{0})$ to $\mathrm{DL}(\Lambda_{0})$ following \cite[\S 4]{Wanga}. Let $\XX_{\Lambda_{0}}$ be the $p$-divisible group over $\FF$ whose Dieudonn\'{e}  module is given by $\Lambda_{0}\oplus \Pi^{-1}\Lambda_{0}$ \cite[Lemma 4.1]{Wanga}. There is a universal isogeny
\begin{equation}
\rho^{+}_{\Lambda_{0}}:\sfX\rightarrow\XX_{\Lambda_{0}}
\end{equation}
over $\calM_{\Pa}(\Lambda_{0})$ for the universal $p$-divisible group $\sfX$. The $\ZZ_{p^{2}}$-action on $\DD(\sfX)$ induces decompositions $\DD(\sfX)=\DD(\sfX)_{0}\oplus \DD(\sfX)_{1}$, $\omega_{\sfX^{\vee}}=\omega_{\sfX^{\vee},0}\oplus \omega_{\sfX^{\vee},1}$ and $\mathrm{Lie}(\sfX)=\mathrm{Lie}(\sfX)_{0}\oplus \mathrm{Lie}(\sfX)_{1}$. For an $\FF$-algebra $\rmR$ and any point $(\sfX, \lambda_{\sfX}, \iota_{\sfX}, \rho_{\sfX} )\in\calM_{\Pa}(\Lambda_{0})(\rmR)$, sending this point to $\rho^{+}_{\Lambda_{0}}(\omega_{\sfX^{\vee},0})\subset \overline{\rmV}_{\Lambda_{0}}\otimes\rmR$ gives the desired map from $\calM_{\Pa}(\Lambda_{0})$ to $\mathrm{DL}(\Lambda_{0})$ which is in fact an isomorphism by \cite[Proposition 4.12]{Wanga} and induces an isomorphism from $\calM^{\square}_{\Pa}(\Lambda_{0})$ to $\mathrm{DL}^{\square}(\Lambda_{0})$.
\begin{lemma}\label{tan-0}
The tangent bundle $\calT_{\calM^{\square}_{\Pa}(\Lambda_{0})}$ of $\calM^{\square}_{\Pa}(\Lambda_{0})$ is given by
\begin{equation*}
\calT_{\calM^{\square}_{\Pa}(\Lambda_{0})}\cong \mathrm{Hom}(\omega_{\sfX^{\vee},0}/\Pi\omega_{\sfX^{\vee},1}, \Lie(\sfX)_{0})
\end{equation*}
for the universal $p$-divisible group $\sfX$ over $\calM^{\square}_{\Pa}(\Lambda_{0})$. Moreover, there is an isomorphism of line bundles from $\Lie(\sfX)_{1}/\Pi\Lie(\sfX)_{0}$ to $\calO_{\mathrm{DL}^{\square}(\Lambda_{0})}(-p)$.
\end{lemma}
\begin{proof}
This is a standard computation using  Grothendieck--Messing deformation theory which proceeds similarly as in \cite[Theorem 4.3.5]{LTXZZ} and \cite[Theorem 6.37]{Sweeting}. We only remark that the kernel of $\rho^{+}_{\Lambda_{0}}$ on $\omega_{\sfX^{\vee},0}$ is given by $\Pi \omega_{\sfX^{\vee},1}$ which explains the term $\omega_{\sfX^{\vee},0}/\Pi\omega_{\sfX^{\vee},1}$ in the above formula.

For the moreover part, note that $\rho^{+}_{\Lambda_{0}}(\omega_{\sfX^{\vee},0})$ is identified with the line bundle $\calO_{\mathrm{DL}^{\square}(\Lambda_{0})}(-1)$ under the identification of $\calM^{\square}_{\Pa}(\Lambda_{0})$ with $\mathrm{DL}^{\square}(\Lambda_{0})$. Then $\rho^{+}_{\Lambda_{0}}(\rmV\DD(\sfX)_{1})\cong \rho^{+}_{\Lambda_{0}}(\omega^{(p)}_{\sfX^{\vee},0})$ is identified with $\calO_{\mathrm{DL}^{\square}(\Lambda_{0})}(-p)$. Now the kernel of $\rho^{+}_{\Lambda_{0}}$ on $\DD(\sfX)_{1}$ is given by $\omega_{\sfX^{\vee}, 1}+\Pi\DD(\sfX)_{0}$ and $\rmV$ is an isomorphism from $\DD(\sfX)_{1}$ to $\DD(\sfX)_{0}$, hence the map $\rmV\rho^{+}_{\Lambda_{0}}$ indues an identification of $\Lie(\sfX)_{1}/\Pi\Lie(\sfX)_{0}$ with $\calO_{\mathrm{DL}^{\square}(\Lambda_{0})}(-p)$.
\end{proof}


\subsubsection{The scheme $\calM_{\Pa}(\Lambda_{2})$} The discussion for $\calM_{\Pa}(\Lambda_{2})$ runs completely parallel to that of $\calM_{\Pa}(\Lambda_{0})$. Let $\calM^{\square}_{\Pa}(\Lambda_{2})=\calM_{\Pa}(\Lambda_{2})-\calM_{\Pa}\{1\}$. In terms of the descriptions of \ref{RZ-F-point}, an $\FF$-point of $\calM^{\square}_{\Pa}(\Lambda_{2})$ corresponds to a lattice $\DD$ such that $\DD\cap\tau\DD$ is $\tau$-stable and gives rise to the vertex lattice $\Lambda_{2}$ of type $2$. Moreover the lattices $\DD$ and  $\tau\DD$ inserts into the chain 
\begin{equation*}
\begin{aligned}
&\Lambda_{2}\subset \DD\subset \DD+\tau{\DD} \subset \DD^{\vee} \subset \Lambda^{\vee}_{2}\\
&\Lambda_{2}\subset \tau\DD\subset  \DD+\tau{\DD} \subset \tau\DD^{\vee} \subset \Lambda^{\vee}_{2}\\
\end{aligned}
\end{equation*}
respectively. Taking the quotient by $\Lambda_{2}$, we obtain two complete isotropic flags
\begin{equation*}
\begin{aligned}
&\DD/\Lambda_{2}\subset  \DD+\tau{\DD}/\Lambda_{2} \subset \DD^{\vee}/\Lambda_{2} \subset \Lambda^{\vee}_{2}/\Lambda_{2}\\
&\tau\DD/\Lambda_{2}\subset  \DD+\tau{\DD}/\Lambda_{2} \subset \tau\DD^{\vee}/\Lambda_{2} \subset \Lambda^{\vee}_{2}/\Lambda_{2}\\
\end{aligned}
\end{equation*}
in the symplectic space $\overline{\rmV}_{\Lambda_{2}}=\Lambda^{\vee}_{2}/\Lambda_{2}$. In this case, the closure $\calM_{\Pa}(\Lambda_{2})$ of $\calM^{\square}_{\Pa}(\Lambda_{2})$ is again isomorphic to the Deligne--Lusztig variety as in \ref{DL-eq} given by the set of lines $l\in\overline{\rmV}_{\Lambda_{2}}$ such that
\begin{equation*}
l\subset \sigma(l)^{\perp}.
\end{equation*}
We denote by $\mathrm{DL}(\Lambda_{2})$ this Deligne--Lusztig variety and the complement of the set of superspecial points in $\mathrm{DL}(\Lambda_{2})$ is denoted by $\mathrm{DL}^{\square}(\Lambda_{2})$. By sending $\DD/\Lambda_{2}$ to $l$, we have an isomorphism from $\calM_{\Pa}(\Lambda_{2})(\FF)$ to $\mathrm{DL}(\Lambda_{2})(\FF)$ which induces an isomorphism from $\calM^{\square}_{\Pa}(\Lambda_{2})(\FF)$ to $\mathrm{DL}^{\square}(\Lambda_{2})(\FF)$.

For the next lemma, we need to recall the map from $\calM_{\Pa}(\Lambda_{2})$ to $\mathrm{DL}(\Lambda_{2})$ following \cite[\S 4]{Wanga}. Let $\XX_{\Lambda_{2}}$ be the $p$-divisible group over $\FF$ whose Dieudonn\'{e}  module is given by $\Lambda^{\vee}_{2}\oplus \Pi\Lambda^{\vee}_{2}$ \cite[Lemma 4.7]{Wanga}. There is a universal isogeny
\begin{equation}
\rho^{-}_{\Lambda_{2}}:\sfX\rightarrow\XX_{\Lambda_{2}}
\end{equation}
for any $\sfX$ parametrized by $\calM_{\Pa}(\Lambda_{2})$. The $\ZZ_{p^{2}}$-action on $\DD(\sfX)$ induces decompositions $\DD(\sfX)=\DD(\sfX)_{0}\oplus \DD(\sfX)_{1}$, $\omega_{\sfX^{\vee}}=\omega_{\sfX^{\vee},0}\oplus \omega_{\sfX^{\vee},1}$ and $\mathrm{Lie}(\sfX)=\mathrm{Lie}(\sfX)_{0}\oplus \mathrm{Lie}(\sfX)_{1}$. For an $\FF$-algebra $\rmR$ and any point $(\sfX, \lambda_{\sfX}, \iota_{\sfX}, \rho_{\sfX} )\in\calM_{\Pa}(\Lambda_{2})(\rmR)$, sending this point to $\rho^{-}_{\Lambda_{2}}(\DD(\sfX)_{0})\subset \overline{\rmV}_{\Lambda_{2}}\otimes\rmR$ gives the desired map from $\calM_{\Pa}(\Lambda_{2})$ to $\mathrm{DL}(\Lambda_{2})$ which is in fact an isomorphism by \cite[Proposition 4.12]{Wanga} and induces an isomorphism from $\calM^{\square}_{\Pa}(\Lambda_{2})$ to $\mathrm{DL}^{\square}(\Lambda_{2})$. 

 \begin{lemma}\label{tan-2}
Over $\calM^{\square}_{\Pa}(\Lambda_{2})$, the tangent bundle $\calT_{\calM^{\square}_{\Pa}(\Lambda_{2})}$ is given by
\begin{equation*}
\calT_{\calM^{\square}_{\Pa}(\Lambda_{2})}\cong \mathrm{Hom}(\omega_{\sfX^{\vee},0}, \Pi\Lie(\sfX)_{1})
\end{equation*}
where $\sfX$ is the universal $p$-divisible group. Moreover, there is an isomorphism $\Lie(\sfX)_{0}/\Pi\Lie(\sfX)_{1}$ with $\calO_{\mathrm{DL}^{\square}(\Lambda_{2})}(-p)$.
\end{lemma}
\begin{proof}
This is a standard computation using  Grothendieck--Messing deformation theory which proceeds similarly as in \cite[Theorem 4.3.5]{LTXZZ} and \cite[Theorem 6.37]{Sweeting}. We only remark that the kernel of $\rho^{-}_{\Lambda_{2}}$ on $\DD(\sfX)_{0}$ is given by $\Pi\DD(\sfX)_{1}+\omega_{\sfX^{\vee}, 0}$ which in particular contains $\omega_{\sfX^{\vee}, 0}$ and we have
\begin{equation*}
\Pi\DD(\sfX)_{1}+\omega_{\sfX^{\vee}, 0}/\omega_{\sfX^{\vee}, 0}\cong \Pi\DD(\sfX)_{1}/\omega_{\sfX^{\vee}, 0}\cap\Pi\DD(\sfX)_{1}\cong \Pi\Lie(\sfX)_{1}
\end{equation*}
which explains the term $\Pi\Lie(\sfX)_{1}$ in  the above formula.

For the moreover part, note that $\rho^{-}_{\Lambda_{2}}(\DD(\sfX)_{0})$ is identified with the line bundle $\calO_{\mathrm{DL}^{\square}(\Lambda_{2})}(-1)$ under the identification of $\calM^{\square}_{\Pa}(\Lambda_{2})$ with $\mathrm{DL}^{\square}(\Lambda_{2})$. Then $\rmV\rho^{-}_{\Lambda_{2}}(\DD(\sfX)_{0})$ is identified with $\calO_{\mathrm{DL}^{\square}(\Lambda_{2})}(-p)$. Now the kernel of $\rho^{-}_{\Lambda_{2}}$ on $\DD(\sfX)_{0}$ is given by $\omega_{\sfX^{\vee}, 0}+\Pi\DD(\sfX)_{1}$ and $\rmV$ is an isomorphism from $\DD(\sfX)_{0}$ to $\DD(\sfX)_{1}$, hence the map $\rmV\rho^{-}_{\Lambda_{2}}$ indues an identification of $\Lie(\sfX)_{0}/\Pi\Lie(\sfX)_{1}$ with $\calO_{\mathrm{DL}^{\square}(\Lambda_{2})}(-p)$.
\end{proof}


\subsection{The quaternionic unitary Shimura variety}
We define the integral model of the quaternionic unitary Shimura variety with paramodular level $\rmX_{\Pa}(\rmB)$ over $\ZZ_{(p)}$ which represents certain PEL-type moduli problem. Recall that $\rmV=\rmB\oplus \rmB$ considered as a vector space of dimension $8$ over $\QQ$ and we have
the quaternionic unitary group
\begin{equation*}
\bfG(\rmB)(\QQ)=\{g\in \GL_{\rmB}(\rmV): (gx, gy)=\rmc(g)(x, y),\phantom{.}\rmc(g)\in\QQ^{\times}\}
\end{equation*}
defined by an alternating form $(\cdot,\cdot)$ on $\rmV$ fixed in $\S 2.1$. 

Since $\rmB$ is split at $\infty$, $\bfG(\rmB)(\RR)=\GSp(4)(\RR)$. Let $h: \GG_{m}\rightarrow \bfG(\rmB)(\CC)$ be the cocharacter given by
\begin{equation*}
z\mapsto\begin{pmatrix} z&&&\\&z&&\\&&1&\\&&&1\\ \end{pmatrix}\text{ for }z\in\CC^{\times}.
\end{equation*}
Then $h$ defines a decomposition $\rmV_{\CC}=\rmV_{1}\oplus \rmV_{2}$ where $h(z)$ acts on $\rmV_{1}$ by $z$ and on $\rmV_{2}$ by $\bar{z}$. We fix an open compact subgroup $\rmK=\rmK^{p}\rmK_{p}$ of $\bfG(\rmB)(\mathbb{A}^{(\infty)})$ and we assume $\rmK^{p}$ is sufficiently small and $\rmK_{p}=\rmK^{\rmD}_{\{1\}}$ is the paramodular subgroup of $\GU_{2}(\rmD)$.
Let $\Lambda=\calO_{\rmB}\oplus\calO_{\rmB}$, then $\rmK_{p}$ is also the stabilizer of $\Lambda\otimes \ZZ_{p}=\calO_{\rmD}\oplus\calO_{\rmD}$. 
This datum defines a Shimura variety $\Sh(\rmB, \rmK_{\Pa})$ over $\QQ$ such that its $\CC$ points are given by
\begin{equation*}
\Sh(\rmB, \rmK_{\Pa})(\CC)\cong \bfG(\rmB)(\QQ)\backslash \calH^{\pm }\times \bfG(\rmB)(\mathbb{A}^{(\infty)})/\rmK^{(p)}\rmK_{p} 
\end{equation*}
where $\calH^{\pm}$ is the lower and upper Siegel half-spaces of dimension $3$. 

\begin{construction}
To the datum $(\rmB,\ast,  \rmV, (\cdot,\cdot), h ,\Lambda)$, we associate the following moduli problem $\rmX_{\Pa}(\rmB)$ over $\ZZ_{(p)}$. To a scheme $\rmS$ over $\ZZ_{(p)}$, $\rmX_{\Pa}(\rmB)(\rmS)$ classifies the set of isomorphism classes of the quadruple $(\sfA, \iota , \lambda, \eta^{p})$ where:
\begin{itemize}
\item $\mathsf{A}$ is an abelian scheme of relative dimension $4$ over $\rmS$;
\item  $\lambda: \mathsf{A}\rightarrow \sfA^{\vee}$ is a prime-to-$p$ polarization which is $\calO_{\rmB}\otimes\ZZ_{(p)}$-linear;
\item  $\iota: \calO_{\rmB} \rightarrow \End_{\rmS}(\sfA)$ is a morphism such that $\lambda\circ \iota(a^{*})= \iota(a)^{\vee}\circ\lambda$ and satisfies the Kottwitz condition 
\begin{equation*}
\det(\rmT-\iota(a)\vert \Lie(\sfA))=(\rmT^{2}-\mathrm{Trd}^{0}(a)\rmT+\mathrm{Nrd}^{0}(a))^{2}
\end{equation*}
for all $a\in \calO_{\rmB}$, where $\mathrm{Trd}^{0}$ is the reduced trace on $\rmB$ and $\mathrm{Nrd}^{0}$ is the reduced norm on $\rmB$;
\item $\eta^{p}: \Lambda\otimes_{\ZZ} {\widehat{\ZZ}}^{(p)} \rightarrow\widehat{\rmT}^{(p)}(\sfA_{\overline{s}})$ is a $\pi_{1}(\rmS,\overline{s})$-invariant $\rmK^{p}$-orbit of $\calO_{\rmB}\otimes_{\ZZ}\widehat{\ZZ}^{(p)}$-linear isomorphisms where we put
\begin{equation*}
\widehat{\rmT}^{(p)}(\sfA_{\overline{s}})=\prod_{v \neq p}\Ta_{v}(\sfA_{\overline{s}})
\end{equation*}
for a geometric point $\overline{s}\in\rmS$.
We require that $\eta^{p}$ is compatible with the Weil-pairing on the righthand side and the alternating form $(\cdot,\cdot)$ on $\rmV$ restricted to $\Lambda$ on the lefthand side.
\end{itemize}

This moduli problem is representable by a quasi-projective scheme $\rmX_{\Pa}(\rmB)$ over $\ZZ_{(p)}$ whose generic fiber is given by $\Sh(\rmB, \rmK_{\Pa})$. We will mainly use the integral model $\rmX_{\Pa}(\rmB)$ base changed to $\rmW_{0}$ in this article. Therefore we will abuse the notation and write $\rmX_{\Pa}(\rmB)$ for this base change from here on. The special fiber will be denoted by $\overline{\rmX}_{\Pa}(\rmB)$. 
\end{construction}

The singularity of the integral model of this Shimura variety is well-understood and the local model is discussed in \cite[\S 2.3]{Wanga}. In particular, we have the following result.
\begin{lemma}
The scheme $\rmX_{\Pa}(\rmB)$ has the following geometric properties.
\begin{enumerate}
\item The scheme $\rmX_{\Pa}(\rmB)$ over $\mathrm{W}_{0}$ is regular whose singular locus is concentrated on its special fiber $\overline{\rmX}_{\Pa}(\rmB)$ with isolated ordinary quadratic singularities. 

\item Suppose that the singular locus is denoted by $\overline{\rmX}^{\mathrm{sing}}_{\Pa}(\rmB)$. Then $\rmX_{\Pa}(\rmB)$ is \'etale locally around each singular point in $\overline{\rmX}^{\mathrm{sing}}_{\Pa}(\rmB)$ isomorphic to 
\begin{equation*}
\rmW_{0}[ \rmZ_{1}, \rmZ_{2}, \rmZ_{3}, \rmZ_{4}]/(\rmZ_{1}\rmZ_{2}-\rmZ_{3}\rmZ_{4}-p).
\end{equation*} 
\item The singular locus $\overline{\rmX}^{\mathrm{sing}}_{\Pa}(\rmB)$ consists of precisely the discrete set of points whose underlying $p$-divisible groups are superspecial and therefore it is contained in the supersingular locus $\overline{\rmX}^{\mathrm{ss}}_{\Pa}(\rmB)$.
\end{enumerate}
\end{lemma}
\begin{proof}
This follows immediately from \cite[Theorem 2.2]{Wanga}.
\end{proof}

For later use, we introduce the following two notations:
\begin{itemize}
\item we denote by $\overline{\rmX}^{\square}_{\Pa}(\rmB)=\overline{\rmX}_{\Pa}(\rmB)-\overline{\rmX}^{\mathrm{sing}}_{\Pa}(\rmB)$ the regular locus of $\overline{\rmX}_{\Pa}(\rmB)$;
\item we denote by $\overline{\rmX}^{\square\mathrm{ss}}_{\Pa}(\rmB)$ the complement of $\overline{\rmX}^{\mathrm{sing}}_{\Pa}(\rmB)$ in $\overline{\rmX}^{\mathrm{ss}}_{\Pa}(\rmB)$.
\end{itemize}

There is an arithmetic minimal compatification of the scheme $\rmX_{\Pa}(\rmB)$ whose boundary components are the zero-dimensional cusps. There is also a good theory of toroidal compatification of  $\rmX_{\Pa}(\rmB)$ compatible with its semi-stable model $\rmX^{\mathrm{Bl}}_{\Pa}(\rmB)$. Since we will not use them explicitly, we will not make these theories precise in this article.

We are interested in the supersingular locus  $\overline{\rmX}^{\mathrm{ss}}_{\Pa}(\rmB)$ of $\overline{\rmX}_{\Pa}(\rmB)$ considered as a reduced closed subscheme. The uniformization theorem of Rapoport--Zink translates this problem to the problem of  describing the corresponding Rapoport--Zink space. More precisely, we have the following isomorphism.

\begin{proposition}\label{uniformization}
There is an isomorphism of $\FF$-schemes
\begin{equation*}
\Theta_{\mathrm{RZ}}:\rmI(\QQ)\backslash {\calN}_{\Pa, \mathrm{red}}\times  \bfG({\rmB})(\AAA^{(\infty p)})/\rmK^{p} \cong \overline{\rmX}^{\mathrm{ss}}_{\Pa}(\rmB).
\end{equation*}
equivariant with respect to the $\bfG({\rmB})(\AAA^{(\infty p)})$-action
\end{proposition} 
\begin{proof}
This follows from the Rapoport--Zink uniformization theorem \cite[Theorem 6.1]{RZ-space}.
\end{proof}

Here $\rmI$ is an inner form of $\mathrm{GU}_{2}(\rmB)$ with $\rmI(\QQ_{p})=\rmJ_{b}(\QQ_{p})$ in our case. More precisely, it is defined in the following way. Let  $(\bf{A}, \boldsymbol{\iota}, \boldsymbol{\lambda}, \boldsymbol{\eta})$ be a supersingular point in $\overline{\rmX}^{\mathrm{ss}}_{\Pa}(\rmB)$ whose associated $p$-divisible group is given by $\XX$ which we use as a base point to define the Rapoport--Zink space. Then $\rmI$ is defined to be the algebraic group over $\QQ$ whose $\QQ$-points $\rmI(\QQ)$ are given by the group of quasi-isogenies in $\End(\bf{A})_{\QQ}$ which preserve the polarization $\boldsymbol{\lambda}$. In fact, it is not difficult to see that this group $\rmI$ is exactly the group $\bfG(\overline{\rmB})=\mathrm{GU}_{2}(\overline{\rmB})$. We will refer to $\Theta_{\mathrm{RZ}}$ as the Rapoport--Zink uniformization map.

\begin{construction}
Let $\rmK$ be an open compact subgroup of $\bfG(\overline{\rmB})(\mathbb{A}^{(\infty)})$. We define the Shimura set for the group $\bfG(\overline{\rmB})$ with level $\rmK$ by
\begin{equation*}
\Sh(\overline{\rmB}, \rmK)=\bfG(\overline{\rmB})(\QQ)\backslash \bfG(\overline{\rmB})(\mathbb{A}^{(\infty)})/\rmK.
\end{equation*}
When $\rmK_{p}$ is a hyperspecial subgroup of $\bfG(\QQ_{p})$, we will write the Shimura set as $\Sh(\overline{\rmB}, \rmK_{\rmH})$ to emphasize its level structure at $p$ is hyperspecial. 
\end{construction}

\begin{remark}
When we study the supersingular locus of the Shimura variety $\overline{\rmX}_{\Pa}(\rmB)$, these Shimura sets appear naturally, and in this setting we will prefer to use another set of notations. 
\begin{itemize}

\item For $i\in\{0,1,2\}$, we will write the Shimura set $\Sh(\overline{\rmB}, \rmK^{p}\rmK_{\{i\}})$ as $\rmZ_{\{i\}}(\overline{\rmB})$
and will refer to the Shimura set $\rmZ_{\{i\}}(\overline{\rmB})$ as the Shimura set for $\bfG(\overline{\rmB})$ with $\rmK_{\{i\}}$-level structure. 

\item Since $\rmK_{\{0\}}\cong \rmK_{\{2\}}\cong \bfG(\ZZ_{p})$, sometimes,  we identify the Shimura sets $\rmZ_{\{0\}}(\overline{\rmB})$ and $\rmZ_{\{2\}}(\overline{\rmB})$ with $\Sh(\overline{\rmB}, \rmK_{\rmH})$, in which case we will write them commonly as  $\rmZ_{\rmH}(\overline{\rmB})$. 

\item Similarly, we sometimes prefer to use the notation $\rmZ_{\mathrm{Pa}}(\overline{\rmB})$ to denote the Shimura set  $\rmZ_{\{1\}}(\overline{\rmB})$.
\end{itemize}
\end{remark}

The Rapoport--Zink uniformization theorem immediately transfers the results in Theorem \ref{RZ-description} to a description of the supersingular locus $\overline{\rmX}^{\mathrm{ss}}_{\Pa}(\rmB)$. For each vertex lattice $\Lambda_{i}$ of type $i$ with $i\in\{0,2\}$, we also denote by $\mathrm{DL}(\Lambda_{i})$ the image of $\calM_{\Pa}(\Lambda_{i})$ under the Rapoport--Zink uniformization map $\Theta_{\mathrm{RZ}}$ by a slight abuse of notations. 

\begin{corollary}
The supersingular locus $\overline{\rmX}^{\mathrm{ss}}_{\Pa}(\rmB)$ of $\overline{\rmX}_{\Pa}(\rmB)$ is pure of dimension $2$ and has the following properties.
\begin{enumerate} 
\item The irreducible components of  $\overline{\rmX}^{\mathrm{ss}}_{\Pa}(\rmB)$  are parametrized by the Shimura sets $\rmZ_{\{i\}}(\overline{\rmB})$ with $i\in\{0, 2\}$. Each irreducible component  is isomorphic to some $\mathrm{DL}(\Lambda_{i})$. 
\item The variety $\mathrm{DL}(\Lambda_{i})$ is isomorphic to a smooth projective surface of the form 
\begin{equation*}
\rmZ_{3}^{p}\rmZ_{0}-\rmZ^{p}_{0}\rmZ_{3}+\rmZ^{p}_{2}\rmZ_{1}-\rmZ^{p}_{1}\rmZ_{2}=0
\end{equation*}
in a suitable coordinate system $[\rmZ_{0}:\rmZ_{1}: \rmZ_{2}: \rmZ_{3}]$ of $\PP^{3}$.

\item For each pair of vertex lattices $(\Lambda_{0}, \Lambda_{2})$ such that
$p\Lambda^{\vee}_{0}\subset p\Lambda^{\vee}_{2}=\Lambda_{2}\subset \Lambda_{0}$, the corresponding irreducible component $\mathrm{DL}(\Lambda_{0})$ and $\mathrm{DL}(\Lambda_{2})$ intersects at a projective line $\PP^{1}$.

\item For each pair $(\Lambda_{0}, \Lambda^{\prime}_{0})$ of vertex lattices of type $0$ such that $\Lambda_{1}=\Lambda_{0}\cap\Lambda^{\prime}_{0}$ is a vertex lattice of type $1$, the irreducible components $\mathrm{DL}(\Lambda_{0})$ and $\mathrm{DL}(\Lambda^{\prime}_{0})$ intersect at a superspecial point. 

\item For each pair $(\Lambda_{2}, \Lambda^{\prime}_{2})$ of vertex lattices of type $2$ such that $\Lambda_{1}=\Lambda_{2}+\Lambda^{\prime}_{2}$ is a vertex lattice of type $1$, the irreducible components $\mathrm{DL}(\Lambda_{2})$ and $\mathrm{DL}(\Lambda^{\prime}_{2})$ intersect at a superspecial point.

\end{enumerate}
\end{corollary}

\begin{proof}
Notice that we can write the Shimura set $\rmZ_{\{i\}}(\overline{\rmB})$ in the form 
\begin{equation*}
\rmZ_{\{i\}}(\overline{\rmB})=\bfG(\overline{\rmB})(\QQ)\backslash \rmJ_{b}(\QQ_{p})/\rmK_{\{i\}}\times\bfG(\overline{\rmB})(\mathbb{A}^{(\infty p)})/\rmK^{p}. 
\end{equation*}
This follows from the fact that $\rmJ_{b}(\QQ_{p})\cong \GSp_{4}(\QQ_{p})$.  Then statement $(1)$ follows from the immediately from the Rapoport--Zink uniformization in Theorem \ref{uniformization} and Theorem \ref{RZ-description} $(1)$ along with the fact that $\rmJ_{b}(\QQ_{p})$ acts transitively on the set of vertex lattices of type $\{i\}$ for $i\in\{0,1,2\}$.  The rest of the statements follows from the Rapoport--Zink uniformization theorem and the descriptions of $\calM_{\Pa}$ in Theorem \ref{RZ-description}.
\end{proof}

\subsection{The blow-up  model}
Let $\mathbb{W}_{0}=\rmW_{0}(\sqrt{p})$ be a totally ramified quadratic extension of $\rmW_{0}$. There is a natural semi-stable model  $\rmX^{\mathrm{Bl}}_{\Pa}(\rmB)$ of $\rmX_{\Pa}(\rmB)$ over $\mathbb{W}_{0}$ by blowing up $\rmX_{\Pa}(\rmB)$ at  the singular locus $\Sigma_{\Pa}=\overline{\rmX}^{\mathrm{sing}}_{\Pa}(\rmB)$ of its special fiber. For simplicity, we denote by $\Sigma_{\Pa}$ the singular locus $\overline{\rmX}^{\mathrm{sing}}_{\Pa}(\rmB)$ of $\overline{\rmX}_{\Pa}(\rmB)$. The description of this semi-stable model is well-known, see \cite[Proposition 2.4]{Illusie-van}:
\begin{itemize}
\item Let $\rmX^{\mathrm{Bl}}_{\Pa}(\rmB)$ be the blow-up of  $\rmX_{\Pa}(\rmB)$ at the locus of singular points $\overline{\rmX}^{\mathrm{sing}}_{\Pa}(\rmB)$ over $\mathbb{W}_{0}$. The scheme  $\rmX^{\mathrm{Bl}}_{\Pa}(\rmB)$ is semi-stable and its special fiber has the form  $\overline{\rmX}^{\mathrm{Bl}}_{\Pa}(\rmB)=\overline{\rmX}^{\bullet}_{\Pa}(\rmB)+ \sum\limits_{\sigma\in\Sigma_{\Pa}}\rmD_{\sigma}$ as a divisor in $\rmX^{\mathrm{Bl}}_{\Pa}(\rmB)$. 
\item The scheme $\overline{\rmX}^{\bullet}_{\Pa}(\rmB)$ is the strict transform of $\overline{\rmX}_{\Pa}(\rmB)$ under this blow-up and $\rmD=\sum\limits_{\sigma\in\Sigma_{\Pa}}\rmD_{\sigma}$ is the exceptional divisor. Note the smooth scheme $\overline{\rmX}^{\bullet}_{\Pa}(\rmB)$ is also the blow-up of $\overline{\rmX}_{\Pa}(\rmB)$ at its singular locus $\Sigma_{\Pa}$.
\item Each $\rmD_{\sigma}$ is the quadric hypersurface of $\PP^{4}=\mathrm{Proj}\phantom{.}\FF[\rmX_{1},\cdots, \rmX_{4}, \rmT]$ defined by the equation $\overline{\rmQ}_{\sigma}-\overline{u}_{\sigma}\rmT^{2}=0$ for the reduction $\overline{\rmQ}_{\sigma}$ of a ordinary quadratic form $\rmQ_{\sigma}$ and a unit $\overline{u}_{\sigma}$ of $\rmW_{0}$. Let $\rmC_{\sigma}$ be the hyperplane of $\rmD_{\sigma}$ defined by $\rmT=0$ which is also a  smooth quadric. We then clearly have $\overline{\rmX}_{\Pa}(\rmB )\cap \rmD_{\sigma}=\rmC_{\sigma}$. 
\item The exceptional divisor of the blow-up of $\overline{\rmX}_{\Pa}(\rmB)$ at $\Sigma_{\Pa}$ is precisely given by $\sum\limits_{\sigma\in\Sigma_{\Pa}}\rmC_{\sigma}$ and each copy of $\rmC_{\sigma}$ is isomorphic to $\PP^{1}\times\PP^{1}$. Each $\PP^{1}$ on the surface $\rmC_{\sigma}\cong\PP^{1}\times\PP^{1}$ gives one of the two distinct gentracies on the quadratic surface of $\rmC_{\sigma}$. See \cite[XII \S 2]{SGA7} for details.
\end{itemize}

We would like to give a description of the supersingular locus $\overline{\rmX}^{\bullet\mathrm{ss}}_{\Pa}(\rmB)$ of $\overline{\rmX}^{\bullet}_{\Pa}(\rmB)$ which is defined as the strict transform of the supersingular locus $\overline{\rmX}^{\mathrm{ss}}_{\Pa}(\rmB)$ under the blow-up. To do so, we introduce an auxiliary moduli problem of abelian schemes and its corresponding Raopoport--Zink space.

\begin{construction}
We define the moduli problem $\overline{\rmX}_{\Pa, \calF}(\rmB)$ over $\FF_{p}$. For a scheme $\rmS$ over $\FF_{p}$, $\overline{\rmX}_{\Pa,\calF}(\rmB)(\rmS)$ classifies the set of isomorphism classes of the quadruple $(\sfA, \iota , \lambda, \eta^{p}, \calF)$ where:
\begin{itemize}
\item $(\sfA, \iota , \lambda, \eta^{p})$ is an element of $\rmX_{\Pa}(\rmB)(\rmS)$;
\item $\calF$ is a local direct summand of $\mathrm{Lie}(\mathsf{A})$ of rank $2$ stable under the $\calO_{\rmB}\otimes\ZZ_{p}$-action satisfies the conditions below.
\end{itemize}
Since $\calF$ is stable under $\calO_{\rmB}\otimes\ZZ_{p}=\ZZ_{p^{2}}[\Pi]$ induced by $\iota$, we can require that $\calF$ satisfies the following conditions: 
\begin{itemize}
\item $\calF$ is annihilated by $\Pi$ under the action of $\iota$;
\item $\calF$ decomposes into $\calF=\calF_{0}\oplus \calF_{1}$ under the action of $\iota\vert_{\FF_{p^{2}}}$ such that $\calF_{0}$ and $\calF_{1}$ are both of rank one.
\end{itemize}
\end{construction}

This moduli problem is representable by a quasi-projective scheme $\overline{\rmX}_{\Pa,\calF}(\rmB)$ over $\FF_{p}$. There is a natural map $\mathrm{pr}:\overline{\rmX}_{\Pa,\calF}(\rmB)\rightarrow \overline{\rmX}_{\Pa}(\rmB)$ by forgetting the datum $\calF$. We note that $\overline{\rmX}_{\Pa,\calF}(\rmB)$ exists only in characteristic $p$. 

\begin{lemma}\label{moduli-bullet}
There is an isomorphism between the scheme $\overline{\rmX}_{\Pa,\calF}(\rmB)$ and $\overline{\rmX}^{\bullet}_{\Pa}(\rmB)$ over $\FF$.
\end{lemma}
\begin{proof}
Let $(\sfA, \iota , \lambda, \eta^{p}, \calF)$ be an $\FF$-point of $\overline{\rmX}_{\Pa,\calF}(\rmB)$. Then $\Lie(\sfA)\cong\rmD(\sfA)/\rmV \rmD(\sfA)$ and the action of $\ZZ_{p^{2}}$ induces a decomposition $\Lie(\sfA)=\Lie(\sfA)_{0}\oplus\Lie(\sfA)_{1}\cong\rmD(\sfA)_{0}/\rmV \rmD(\sfA)_{1}\oplus \rmD(\sfA)_{1}/\rmV \rmD(\sfA)_{0}$. Moreover the action of $\Pi$ induces maps
\begin{equation*}
\begin{aligned}
\Pi_{0}: \Lie(\sfA)_{0}\rightarrow \Lie(\sfA)_{1}, \phantom{aa}\Pi_{1}: \Lie(\sfA)_{1}\rightarrow \Lie(\sfA)_{0}.
\end{aligned}
\end{equation*}
Since $\Pi^{2}=p$, $\ker(\Pi_{0})\subset\Lie(\sfA)_{0}$ and $\ker(\Pi_{1})\subset\Lie(\sfA)_{1}$ are one dimensional spaces unless $\Pi\rmD(\sfA)$ agrees with $\rmV\rmD(\sfA)$. Therefore away from $\Sigma_{\Pa}$ the natural map $\mathrm{pr}:\overline{\rmX}_{\Pa,\calF}(\rmB)\rightarrow \overline{\rmX}_{\Pa}(\rmB)$ induces an isomorphism over $\FF$. Over $\Sigma_{\Pa}$, we have to choose  lines $\calF_{0}$ and $\calF_{1}$ in the two dimensional spaces $\Lie(\sfA)_{0}$ and  $\Lie(\sfA)_{1}$. It follows that $\overline{\rmX}_{\Pa, \calF}(\rmB)\vert_{\Sigma_{\Pa}}$ is a $\PP^{1}\times\PP^{1}$-bundle over $\Sigma_{\Pa}$. Then the lemma follows from the universal property of blow-up which characterizes $\overline{\rmX}^{\bullet}_{\Pa}(\rmB)$.
\end{proof}

Let $\overline{\rmX}^{\bullet\mathrm{ss}}_{\Pa}(\rmB)$ be the total transform of the supersingular locus $\overline{\rmX}^{\mathrm{ss}}_{\Pa}(\rmB)$ under the natural blow-up map $\mathrm{pr}: \overline{\rmX}^{\bullet}_{\Pa}(\rmB)\rightarrow \overline{\rmX}_{\Pa}(\rmB)$. Let $\overline{\rmX}^{\bullet\mathrm{exp}}_{\Pa}(\rmB)$ be the exceptional locus of the blow-up. We will refer to $\overline{\rmX}^{\bullet\mathrm{ss}}_{\Pa}(\rmB)$ as the supersingular locus of $\overline{\rmX}^{\bullet}_{\Pa}(\rmB)$ and to $\overline{\rmX}^{\bullet\mathrm{exp}}_{\Pa}(\rmB)$ as the superspecial locus of $\overline{\rmX}^{\bullet}_{\Pa}(\rmB)$, we will define an auxiliary Rapoport--Zink space which uniformizes $\overline{\rmX}^{\bullet\mathrm{ss}}_{\Pa}(\rmB)$.
\begin{construction}
We consider the set valued functor $\mathcal{N}^{\bullet}_{\mathrm{Pa}}$ that sends an $\FF$-scheme $\rmS$ to the isomorphism classes of the collection $(\sfX, \iota_{\sfX},  \lambda_{\sfX}, \rho_{\sfX}, \calF)$ where:
\begin{itemize} 
\item $(\mathsf{X}, \iota_{\sfX},  \lambda_{\sfX}, \rho_{\sfX})$ is an element of $\mathcal{N}_{\mathrm{Pa}}(\rmS)$;
\item $\calF$ is a local direct summand of $\mathrm{Lie}(\mathsf{X})$ of rank $2$ stable under the $\calO_{\rmD}$-action.
\end{itemize}

We require that $\calF$ satisfies the following conditions: 
\begin{itemize}
\item $\calF$ is annihilated by $\Pi$;
\item $\calF$ decomposes into $\calF=\calF_{0}\oplus \calF_{1}$ under the action of $\iota_{\sfX}\vert_{\ZZ_{p^{2}}}$ such that $\calF_{0}$ and $\calF_{1}$ are both of rank one.
\end{itemize}
\end{construction}

This moduli problem is representable by a formal scheme $\calN^{\bullet}_{\Pa}$ over $\FF$. The formal scheme $\calN^{\bullet}_{\Pa}$ can be decomposed into open and closed subschemes: 
\begin{equation*}
\calN^{\bullet}_{\Pa}=\bigsqcup\limits_{i\in\ZZ}\calN^{\bullet}_{\Pa}(i) 
\end{equation*}
where each $\calN^{\bullet}_{\Pa}(i)$ is isomorphic to $\calN^{\bullet}_{\Pa}(0)$. We will denote  by $\calM^{\bullet}_{\Pa}$ the reduced scheme underlying the formal $\calN^{\bullet}_{\Pa}(0)$. We have a natural map $\mathrm{pr}:\calM^{\bullet}_{\Pa}\rightarrow\calM_{\Pa}$ by forgetting the datum $\calF$.  

\begin{lemma}\label{bl-ss}
The scheme $\mathcal{M}^{\bullet}_{\mathrm{Pa}}$ maps isomorphically to $\mathcal{M}_{\mathrm{Pa}}$ away from the stratum $\mathcal{M}_{\Pa}\{1\}$ and is a contraction of a $\PP^{1}\times\PP^{1}$-bundle over $\mathcal{M}_{\Pa}\{1\}$.
\end{lemma}
\begin{proof}
Recall that the set $\calM_{\Pa}(\FF)$ can be identified with the set  
\begin{equation*}
\{\rmM\subset \rmN: \rmM^{\perp}=\rmM,\phantom{.}p\rmM\subset^{4} \rmV\rmM\subset^{4} \rmM,\phantom{.}p\rmM\subset^{4} \Pi \rmM\subset^{4} \rmM\}
\end{equation*}
in terms of Dieudonn\'e modules. The local direct factor $\calF$ of $\mathrm{Lie}(\mathsf{X})=\rmM/\rmV\rmM$ gives rise to a sub-lattice $\mathrm{M}_{\calF}$ of $\rmM$ containing $\rmV\rmM$. Since $\calF$ is annihilated by $\Pi$, we know $\Pi\mathrm{M}_{\calF}$ is contained in $\rmV\rmM$. Thus $\calM^{\bullet}_{\Pa}(\FF)$ is given by the set
\begin{equation*}
\{\rmM\subset \rmN: \rmM^{\perp}=\rmM,\phantom{.}p\rmM\subset^{4} \rmV\rmM\subset^{2}\mathrm{M}_{\calF}\subset^{2} \rmM,\phantom{.}p\rmM\subset^{4} \Pi \rmM\subset^{2}\rmV^{-1}\Pi\mathrm{M}_{\calF}\subset^{2} \rmM\}.
\end{equation*}
Then projecting $\rmM$ to $\DD=\rmM_{0}$ as we did in \ref{RZ-F-point} identifies this set with 
\begin{equation*}
\{\DD \subset \rmN_{0}:p\DD^{\vee} \subset^{1}\DD_{\calF} \subset^{1} \DD\subset^{1}\DD^{\dagger}_{\calF}\subset^{1} \DD^{\vee},\phantom{.}p\tau\DD^{\vee}\subset^{1}\tau\DD_{\calF} \subset^{1}  \DD\subset^{1}\tau\DD^{\dagger}_{\calF} \subset^{1} \tau\DD^{\vee}\}.
\end{equation*}
To see this, recall that $\DD_{\calF}$ comes from $\calF_{0}$ and $\DD^{\dagger}_{\calF}$ comes from $\calF_{1}$ by considering the  orthogonal complement $\calF^{\perp}\subset \omega_{\mathsf{X}^{\vee}}$ of $\calF$ under the natural pairing 
$\mathrm{Lie}(\mathsf{X})\times \omega_{\mathsf{X}^{\vee}} \rightarrow \FF$. Note that away from $\mathcal{M}_{\Pa}\{1\}$, $\DD_{\calF}$ and $\DD^{\dagger}_{\calF}$ are redundant since $\DD_{\calF}=\DD\cap\tau^{-1}\DD$ and $\DD^{\dagger}_{\calF}=\DD\cup \tau^{-1}\DD $. On the other hand, over $\mathcal{M}_{\Pa}\{1\}$, $\DD_{\calF}$ resp. $\DD^{\dagger}_{\calF}$ is given by choosing a line in $\DD/p\DD^{\vee}$ resp. in $\DD^{\vee}/\tau^{-1}\DD$. Hence the fiber over the locus $\calM_{\Pa}\{1\}$ is isomorphic to $\PP^{1}\times\PP^{1}$. The lemma follows from the above discussions.
\end{proof}

We define the Bruhat--Tits stratification 
\begin{equation*}
\calM^{\bullet}_{\Pa}= {\rmM}^{\bullet}_{\Pa}\{0\}\sqcup  \rmM^{\bullet}_{\Pa}\{2\} \sqcup  \rmM^{\bullet}_{\Pa}\{02\} \sqcup \rmM^{\bullet}_{\Pa}\{1\}
\end{equation*}
of $\calM^{\bullet}_{\Pa}$ by taking the inverse image of the  Bruhat--Tits stratification 
\begin{equation*}
\calM_{\Pa}= \rmM_{\Pa}\{0\}\sqcup  \rmM_{\Pa}\{2\} \sqcup  \rmM_{\Pa}\{02\} \sqcup \rmM_{\Pa}\{1\}
\end{equation*}
of $\calM_{\Pa}$ under the natural map $\mathrm{pr}:\calM^{\bullet}_{\Pa}\rightarrow \calM_{\Pa}$. 
\begin{itemize}
\item The stratum $\rmM^{\bullet}_{\Pa}\{0\}$ has its irreducible components indexed by the set $\calL_{\{0\}}$ of vertex lattices of type $\{0\}$. For each vertex lattice $\Lambda_{0}\in \calL_{\{0\}}$, the corresponding irreducible component $\rmM^{\bullet}_{\Pa}(\Lambda_{0})$ maps isomorphically to $\rmM_{\Pa}(\Lambda_{0})$. We write $\calM^{\bullet}_{\Pa}\{0\}$ for the closure of $\rmM^{\bullet}_{\Pa}\{0\}$ and  $\calM^{\bullet}_{\Pa}(\Lambda_{0})$ for the closure of $\rmM^{\bullet}_{\Pa}(\Lambda_{0})$.

\item The stratum $\rmM^{\bullet}_{\Pa}\{2\}$ has its irreducible components indexed by the set $\calL_{\{2\}}$ of vertex lattices of type $\{2\}$. For each vertex lattice $\Lambda_{2}\in \calL_{\{2\}}$, the corresponding irreducible component $\rmM^{\bullet}_{\Pa}(\Lambda_{2})$ maps isomorphically to $\rmM_{\Pa}(\Lambda_{2})$. We write $\calM^{\bullet}_{\Pa}\{2\}$ for the closure of $\rmM^{\bullet}_{\Pa}\{2\}$ and  $\calM^{\bullet}_{\Pa}(\Lambda_{2})$ for the closure of $\rmM^{\bullet}_{\Pa}(\Lambda_{2})$.

\item The stratum $\rmM^{\bullet}_{\Pa}\{02\}$ has its irreducible components indexed by a pair of vertex lattices $(\Lambda_{0}, \Lambda_{2})$ with the relation
\begin{equation*}
p\Lambda^{\vee}_{0}\subset p\Lambda^{\vee}_{2}=\Lambda_{2}\subset \Lambda_{0}. 
\end{equation*}
For such a pair $(\Lambda_{0}, \Lambda_{2})$, the correpsonding irreducible component $\rmM^{\bullet}_{\Pa}(\Lambda_{0}, \Lambda_{2})$ is the intersection of $\rmM^{\bullet}_{\Pa}(\Lambda_{0})$ and $\rmM^{\bullet}_{\Pa}(\Lambda_{2})$, and its closure is denoted by $\calM^{\bullet}_{\Pa}(\Lambda_{0}, \Lambda_{2})$.

\item The stratum $\rmM^{\bullet}_{\Pa}\{1\}=\calM^{\bullet}_{\Pa}\{1\}$ is a set of $\PP^{1}\times\PP^{1}$ indexed by the set  $\calL_{\{1\}}$ of vertex lattices of type $1$.  For each vertex lattice $\Lambda_{1}$ of type $1$, we write the corresponding $\PP^{1}\times\PP^{1}$ by $\calM^{\bullet}_{\Pa}(\Lambda_{1})$. 
\end{itemize}

\begin{proposition}\label{RZ-description-bullet}
The Bruhat--Tits stratification of $\calM^{\bullet}_{\Pa}$ has the following properties.
\begin{enumerate}
\item For each vertex lattice $\Lambda_{i}\in \calL_{\{i\}}$ of type $i\in \{0,2\}$, the corresponding irreducible component $\rmM^{\bullet}_{\Pa}(\Lambda_{i})$ maps isomorphically to $\rmM_{\Pa}(\Lambda_{i})$. Its closure  $\calM^{\bullet}_{\Pa}(\Lambda_{i})$ is the blow-up of $\calM_{\Pa}(\Lambda_{i})$ at the set of its superspecial points $\calM_{\Pa}(\Lambda_{i})\{1\}$.

\item For each pair of vertex lattices $(\Lambda_{0}, \Lambda_{2})$ such that
\begin{equation*}
p\Lambda^{\vee}_{0}\subset p\Lambda^{\vee}_{2}=\Lambda_{2}\subset \Lambda_{0},
\end{equation*}
the correpsonding irreducible component $\rmM^{\bullet}_{\Pa}(\Lambda_{0}, \Lambda_{2})$ is the intersection of $\rmM^{\bullet}_{\Pa}(\Lambda_{0})$ and $\rmM^{\bullet}_{\Pa}(\Lambda_{2})$. Its closure $\calM^{\bullet}_{\Pa}(\Lambda_{0}, \Lambda_{2})$ is isomorphic to a projective line $\PP^{1}$.

\item For each vertex lattice $\Lambda_{1}$ of type $1$ such that $\Lambda_{1}\subset \Lambda_{0}$ for a vertex lattice $\Lambda_{0}$ of type $0$, $\calM^{\bullet}_{\Pa}(\Lambda_{0})$ and  $\calM^{\bullet}_{\Pa}(\Lambda_{1})$ intersects at a projective line $\PP^{1}$.

\item For each vertex lattice $\Lambda_{1}$ of type $1$ such that $\Lambda_{2}\subset \Lambda_{1}$ for a vertex lattice $\Lambda_{2}$ of type $2$, $\calM^{\bullet}_{\Pa}(\Lambda_{2})$ and $\calM^{\bullet}_{\Pa}(\Lambda_{1})$ intersects at a projective line $\PP^{1}$.

\item For a pair $(\Lambda_{0}, \Lambda^{\prime}_{0})$ of vertex lattices of type $0$ such that $\Lambda_{1}=\Lambda_{0}\cap\Lambda^{\prime}_{0}$ is a vertex lattice of type $1$, the irreducible components $\calM^{\bullet}_{\Pa}(\Lambda_{0})$ and $\calM^{\bullet}_{\Pa}(\Lambda^{\prime}_{0})$ has trivial intersection unless $\Lambda_{1}=\Lambda^{\prime}_{1}$. 

\item For a pair $(\Lambda_{2}, \Lambda^{\prime}_{2})$ of vertex lattices of type $2$ such that $\Lambda_{1}=\Lambda_{2}+\Lambda^{\prime}_{2}$ is a vertex lattice of type $1$, the irreducible components $\calM^{\bullet}_{\Pa}(\Lambda_{2})$ and $\calM^{\bullet}_{\Pa}(\Lambda^{\prime}_{2})$ has trivial intersection unless $\Lambda_{2}=\Lambda^{\prime}_{2}$.
\end{enumerate}
\end{proposition}
\begin{proof}
The statements in $(1), (2)$ are clear from the construction of $\calM_{\Pa}$ and Lemma \ref{bl-ss}. For $(3)$, the intersection of $\calM^{\bullet}_{\Pa}(\Lambda_{0})$ and $\calM^{\bullet}_{\Pa}(\Lambda_{1})$ is determined by the datum
$\DD=\Lambda_{1}$ along with a lattice $\DD_{\calF}$ inserting in $p\DD^{\vee}\subset^{1}\DD_{\calF}\subset^{1} \DD$ which clearly corresponds to a $\PP^{1}$. For $(4)$, the intersection of $\calM^{\bullet}_{\Pa}(\Lambda_{2})$ and $\calM^{\bullet}_{\Pa}(\Lambda_{1})$ is determined by the datum
$\DD=\Lambda_{1}$ along with a lattice $\DD^{\dagger}_{\calF}$ inserting in $\tau^{-1}\DD\subset^{1}\DD^{\dagger}_{\calF}\subset^{1} \DD^{\vee}$ which clearly corresponds to a $\PP^{1}$. For $(5)$, it follows from $(3)$ that for a vertex lattice $\Lambda_{0}$ of type $0$ the intersection of $\calM^{\bullet}_{\Pa}(\Lambda_{0})$ and $\calM^{\bullet}_{\Pa}(\Lambda_{1})$ is isomorphic to $\PP^{1}$ which represents the class $(1, 0)$ under the isomorphism $\calM^{\bullet}_{\Pa}(\Lambda_{1})\simeq \PP^{1}\times \PP^{1}$. The class $(1,0)\in \mathrm{CH}^{1}(\PP^{1}\times \PP^{1})$ has trivial self-intersection number and thus $(5)$ follows. The statement in $(6)$ follows from $(4)$ using the same argument as in $(5)$.
\end{proof}

\begin{corollary}\label{unifor-bullet}
There is an isomorphism of $\FF$-schemes
\begin{equation*}
\Theta^{\bullet}_{\mathrm{RZ}}:  \rmI(\QQ)\backslash {\calN}^{\bullet}_{\Pa, \mathrm{red}}\times  \bfG({\rmB})(\AAA^{(\infty p)})/\rmK^{p}\cong \overline{\rmX}^{\bullet\mathrm{ss}}_{\Pa}(\rmB)
\end{equation*}
equivariant with respect to the $\bfG({\rmB})(\AAA^{(\infty p)})$-action.
\end{corollary} 
\begin{proof}
This follows from the Rapoport--Zink uniformization in Theorem \ref{uniformization} and Lemma \ref{moduli-bullet} characterizing the blow-up $\overline{\rmX}^{\bullet}_{\Pa}(\rmB)$. 
\end{proof}

For each vertex lattice $\Lambda_{i}$ of type $i\in\{0,2\}$, let $\mathrm{DL}^{\bullet}(\Lambda_{i})$ be the image of $\calM^{\bullet}_{\Pa}(\Lambda_{i})$ under the Rapoport--Zink uniformization map $\Theta^{\bullet}_{\mathrm{RZ}}$.  For each vertex lattice $\Lambda_{1}$ of type $1$, we denote by $\rmC^{\bullet}(\Lambda_{1})$ the image of $\calM^{\bullet}_{\Pa}(\Lambda_{1})$ under $\Theta^{\bullet}_{\mathrm{RZ}}$.

\begin{corollary}\label{ss-bullet}
The scheme $\overline{\rmX}^{\bullet\mathrm{ss}}_{\Pa}(\rmB)$ is pure of dimension $2$ whose irreducible components are parametrized by 
the Shimura sets $\rmZ_{\{0\}}(\overline{\rmB})$ and $\rmZ_{\{2\}}(\overline{\rmB})$. 
\begin{enumerate}

\item An irreducible component of $\overline{\rmX}^{\bullet\mathrm{ss}}_{\Pa}(\rmB)$ is isomorphic to $\mathrm{DL}^{\bullet}(\Lambda_{i})$ for some vertex lattice $\Lambda_{i}$ of type $i\in\{0,2\}$ which is isomorphic to the blow-up of $\mathrm{DL}(\Lambda_{i})$ at all the superspecial points contained in $\mathrm{DL}(\Lambda_{i})$.

\item For each pair of vertex lattices $(\Lambda_{0}, \Lambda_{2})$ such that
\begin{equation*}
p\Lambda^{\vee}_{0}\subset p\Lambda^{\vee}_{2}=\Lambda_{2}\subset \Lambda_{0}, 
\end{equation*}
the corresponding irreducible component $\mathrm{DL}^{\bullet}(\Lambda_{0})$ and $\mathrm{DL}^{\bullet}(\Lambda_{2})$ intersects at a projective line $\PP^{1}$.

\item For each vertex lattice $\Lambda_{1}$ of type $1$ such that $\Lambda_{1}\subset \Lambda_{0}$ for a vertex lattice $\Lambda_{0}$ of type $0$, $\mathrm{DL}^{\bullet}(\Lambda_{0})$ and  $\rmC^{\bullet}(\Lambda_{1})$ intersects at a projective line $\PP^{1}$.

\item For each vertex lattice $\Lambda_{1}$ of type $1$ such that $\Lambda_{2}\subset \Lambda_{1}$ for a vertex lattice $\Lambda_{2}$ of type $2$, $\mathrm{DL}^{\bullet}(\Lambda_{2})$ and $\rmC^{\bullet}(\Lambda_{1})$ intersects at a projective line $\PP^{1}$.

\item For each pair $(\Lambda_{0}, \Lambda^{\prime}_{0})$ of vertex lattices of type $0$ such that $\Lambda_{1}=\Lambda_{0}\cap\Lambda^{\prime}_{0}$ is a vertex lattice of type $1$, the irreducible components $\mathrm{DL}^{\bullet}(\Lambda_{0})$ and $\mathrm{DL}^{\bullet}(\Lambda^{\prime}_{0})$ has trivial intersection unless $\Lambda_{0}=\Lambda^{\prime}_{0}$.

\item For each pair $(\Lambda_{2}, \Lambda^{\prime}_{2})$ of vertex lattices of type $2$ such that $\Lambda_{1}=\Lambda_{2}+\Lambda^{\prime}_{2}$ is a vertex lattice of type $1$, the irreducible components $\mathrm{DL}^{\bullet}(\Lambda_{2})$ and $\mathrm{DL}^{\bullet}(\Lambda^{\prime}_{2})$ has trivial intersection unless $\Lambda_{2}=\Lambda^{\prime}_{2}$.
\end{enumerate}
\end{corollary}

\begin{proof}
These statements follow immediately from Proposition \ref{RZ-description-bullet} and the isomorphism in Corollary \ref{unifor-bullet}.
\end{proof}

\subsection{Normal bundle computation}
The final task in this section is to compute the normal bundle for an irreducible component of the supersingular locus $\overline{\rmX}^{\bullet\mathrm{ss}}_{\Pa}(\rmB)$ in the special fiber $\overline{\rmX}^{\bullet}_{\Pa}(\rmB)$. Suppose $\mathrm{DL}^{\bullet}(\Lambda_{i})$ is  such an irreducible component corresponding to a vertex lattice $\Lambda_{i}$ of type $i\in\{0, 2\}$. We would like to compute the normal bundle $\calN_{\mathrm{DL}^{\bullet}(\Lambda_{i})}({\overline{\rmX}^{\bullet}_{\Pa}(\rmB)})$ of the embedding $\mathrm{DL}^{\bullet}(\Lambda_{i})$ in $\overline{\rmX}^{\bullet}_{\Pa}(\rmB)$. To do this, we first calculate the normal bundle $\calN_{\mathrm{DL}^{\square}(\Lambda_{i})}({\overline{\rmX}^{\square}_{\Pa}(\rmB)})$ using the exact sequence 
\begin{equation*}
0\rightarrow \calT_{\mathrm{DL}^{\square}(\Lambda_{i})}\rightarrow \calT_{\overline{\rmX}^{\square}_{\Pa}(\rmB)}\vert_{\mathrm{DL}^{\square}(\Lambda_{i})}\rightarrow \calN_{\mathrm{DL}^{\square}(\Lambda_{i})}({\overline{\rmX}^{\square}_{\Pa}(\rmB)})\rightarrow 0
\end{equation*}
and then reduce the calculation of the normal bundle $ \calN_{\mathrm{DL}^{\bullet}(\Lambda_{i})}({\overline{\rmX}^{\bullet}_{\Pa}(\rmB)})$ to a problem of calculating certain intersection numbers. This strategy follows \cite[\S6.4]{Sweeting} which simplifies our original approach to calculate $\calN_{\mathrm{DL}^{\bullet}(\Lambda_{i})}({\overline{\rmX}^{\bullet}_{\Pa}(\rmB)})$.

\begin{proposition}\label{square-Tangent}
For each vertex lattice $\Lambda_{i}$ of type $i$ with $i\in\{0,2\}$, the tangent bundle $\calT_{i}=\calT_{\overline{\rmX}^{\square}_{\Pa}(\rmB)}\vert_{\mathrm{DL}^{\square}(\Lambda_{i})}$ fits in the following exact sequence
\begin{equation*}
0\rightarrow \calT_{i}\rightarrow \mathrm{Hom}(\omega_{\sfX^{\vee},0}, \Lie(\sfX)_{0})\rightarrow \mathrm{Hom}(\Pi\omega_{\sfX^{\vee},1}, \Lie(\sfX)_{0}/\Pi\Lie(\sfX)_{1})\rightarrow 0
\end{equation*}
where $\sfX$ is the $p$-divisible group of the universal abelian scheme over $\mathrm{DL}^{\square}(\Lambda_{i})$.
\end{proposition}
\begin{proof}
This follows from \cite[Theorem 6.3.7]{Sweeting}.
\end{proof}

\begin{proposition}\label{normal-square}
The normal bundle $\calN_{i}=\calN_{\mathrm{DL}^{\square}(\Lambda_{i})}({\overline{\rmX}^{\square}_{\Pa}(\rmB)})$ is given by $\calO_{\mathrm{DL}^{\square}(\Lambda_{i})}(-2p)$.
\end{proposition}
\begin{proof}
We first treat the case when $i=0$. By Lemma \ref{tan-0}, we have 
\begin{equation*}
\calT_{\mathrm{DL}^{\square}(\Lambda_{0})}\cong \mathrm{Hom}(\omega_{\sfX^{\vee},0}/\Pi\omega_{\sfX^{\vee},1}, \Lie(\sfX)_{0})
\end{equation*}
which can be inserted in the exact sequence
\begin{equation*}
0\rightarrow \mathrm{Hom}(\omega_{\sfX^{\vee},0}/\Pi\omega_{\sfX^{\vee},1}, \Lie(\sfX)_{0})\rightarrow \mathrm{Hom}(\omega_{\sfX^{\vee},0}, \Lie(\sfX)_{0})\rightarrow \mathrm{Hom}(\Pi\omega_{\sfX^{\vee},1}, \Lie(\sfX)_{0})\rightarrow 0.
\end{equation*}
Now we consider the following commutative diagram 
\begin{equation*}
\begin{tikzcd}
            &                       &                       & \mathrm{Hom}(\Pi\omega_{\sfX^{\vee},1}, \Pi\Lie(\sfX)_{1}) \arrow[d]           &   \\
0 \arrow[r] &\mathrm{Hom}(\frac{\omega_{\sfX^{\vee},0}}{\Pi \omega_{\sfX^{\vee},1}}, \Lie(\sfX)_{0})  \arrow[r] \arrow[d] &  \mathrm{Hom}(\omega_{\sfX^{\vee},0}, \Lie(\sfX)_{0})  \arrow[r] \arrow[d, Rightarrow, no head] & \mathrm{Hom}(\Pi\omega_{\sfX^{\vee},1}, \Lie(\sfX)_{0}) \arrow[r] \arrow[d] & 0 \\
0 \arrow[r] & \calT_{0}  \arrow[r] \arrow[d] & \mathrm{Hom}(\omega_{\sfX^{\vee},0}, \Lie(\sfX)_{0}) \arrow[r]           & \mathrm{Hom}(\Pi\omega_{\sfX^{\vee},1}, \frac{\Lie(\sfX)_{0}}{\Pi\Lie(\sfX)_{1}}) \arrow[r]           & 0 \\
            & \calN_{0}.                   &                       &                       &  
\end{tikzcd}
\end{equation*}
It follows from the snake's lemma that $\calN_{0}\cong \mathrm{Hom}(\Pi\omega_{\sfX^{\vee},1}, \Pi\Lie(\sfX)_{1})$. Note that over $\mathrm{DL}^{\square}(\Lambda_{0})$, we have an isomorphism $\Lie(\sfX)_{1}/\Pi\Lie(\sfX)_{0}\cong \Pi\Lie(\sfX)_{1}$ induced by multiplication by $\Pi$ as they are both line bundles by \cite[2.5]{Wanga}. Then our claim follows from Lemma \ref{tan-0} since $\Pi \omega_{\sfX^{\vee},1}$ is dual to $\Lie(\sfX)_{1}/\Pi\Lie(\sfX)_{0}$ and hence corresponds to $\calO_{\mathrm{DL}^{\square}(\Lambda_{0})}(p)$.

We now move on to the case when $i=2$. By Lemma \ref{tan-2}, we have 
\begin{equation*}
\calT_{\mathrm{DL}^{\square}(\Lambda_{2})}\cong \mathrm{Hom}(\omega_{\sfX^{\vee},0}, \Pi\Lie(\sfX)_{1})
\end{equation*}
which can be inserted in the exact sequence
\begin{equation*}
0\rightarrow \mathrm{Hom}(\omega_{\sfX^{\vee},0}, \Pi\Lie(\sfX)_{1})\rightarrow \mathrm{Hom}(\omega_{\sfX^{\vee},0}, \Lie(\sfX)_{0})\rightarrow \mathrm{Hom}(\omega_{\sfX^{\vee},0}, \Lie(\sfX)_{0}/\Pi\Lie(\sfX)_{1})\rightarrow 0.
\end{equation*}
Now we consider the following commutative diagram 
\begin{equation*}
\begin{tikzcd}
            &                       &                       & \mathrm{Hom}(\frac{\omega_{\sfX^{\vee},0}}{\Pi\omega_{\sfX^{\vee},1}}, \frac{\Lie(\sfX)_{0}}{\Pi\Lie(\sfX)_{1}}) \arrow[d]           &   \\
0 \arrow[r] &\mathrm{Hom}(\omega_{\sfX^{\vee},0}, \Pi\Lie(\sfX)_{1})  \arrow[r] \arrow[d] &  \mathrm{Hom}(\omega_{\sfX^{\vee},0}, \Lie(\sfX)_{0})  \arrow[r] \arrow[d, Rightarrow, no head] & \mathrm{Hom}(\omega_{\sfX^{\vee},0}, \frac{\Lie(\sfX)_{0}}{\Pi\Lie(\sfX)_{1}}) \arrow[r] \arrow[d] & 0 \\
0 \arrow[r] & \calT_{2}  \arrow[r] \arrow[d] & \mathrm{Hom}(\omega_{\sfX^{\vee},0}, \Lie(\sfX)_{0}) \arrow[r]           & \mathrm{Hom}(\Pi\omega_{\sfX^{\vee},1}, \frac{\Lie(\sfX)_{0}}{\Pi\Lie(\sfX)_{1}}) \arrow[r]           & 0 \\
            & \calN_{2}.                   &                       &                       &  
\end{tikzcd}
\end{equation*}
It follows from the snake's lemma that $\calN_{2}\cong \mathrm{Hom}(\omega_{\sfX^{\vee},0}/\Pi\omega_{\sfX^{\vee},1}, \Lie(\sfX)_{0}/\Pi\Lie(\sfX)_{1})$. Note that over $\mathrm{DL}^{\square}(\Lambda_{2})$, we have an isomorphism $\omega_{\sfX^{\vee},0}/\Pi\omega_{\sfX^{\vee},1}\cong \Pi\omega_{\sfX^{\vee},0}$ induced by the multiplication by $\Pi$ as they are both line bundles by \cite[2.5]{Wanga}. Then our claim follows from Lemma \ref{tan-2} since $\Pi \omega_{\sfX^{\vee},0}$ is dual to $\Lie(\sfX)_{0}/\Pi\Lie(\sfX)_{1}$ and hence corresponds to $\calO_{\mathrm{DL}^{\square}(\Lambda_{2})}(p)$.
\end{proof}

Now we can compute the normal bundle $ \calN_{\mathrm{DL}^{\bullet}(\Lambda_{i})}({\overline{\rmX}^{\bullet}_{\Pa}(\rmB)})$ using the previous two propositions and Corollary \ref{ss-bullet}.

\begin{corollary}\label{intersection-number}
The normal bundle $\calN^{\bullet}_{i}=\calN_{\mathrm{DL}^{\bullet}(\Lambda_{i})}({\overline{\rmX}^{\bullet}_{\Pa}(\rmB)})$ is given by $\calO_{\mathrm{DL}^{\bullet}(\Lambda_{i})}(-2p)$.
\end{corollary}
\begin{proof}
Let $\Lambda$ be a vertex lattice of type $0$ and $\{\Lambda_{k}\}_{k=1,\cdots, n}$ be the set of vertex lattices of type $1$ contained in $\Lambda$.  Then $\mathrm{DL}^{\bullet}(\Lambda)$ is the blow-up of $\mathrm{DL}(\Lambda)$ along the superspecial points indexed by $\{\Lambda_{k}\}_{k=1,\cdots, n}$ with exceptional divisors denoted by $\rmE_{k}=\rmE(\Lambda_{k})$ which is isomorphic to $\PP^{1}$. Then we have $\calN^{\bullet}_{0}=\calO_{\mathrm{DL}^{\bullet}(\Lambda)}(-2p+\sum^{n}_{k=1}a_{k}\rmE_{k})$ for some $a_{k}\in\ZZ$ by Proposition \ref{normal-square}.

We will determine the numbers $a_{k}$ by calculating the triple intersection number 
\begin{equation*}
{m}_{k}(\Lambda)=\mathrm{DL}^{\bullet}(\Lambda)\cdot\mathrm{DL}^{\bullet}(\Lambda)\cdot\rmC^{\bullet}(\Lambda_{k})
\end{equation*}
over $\overline{\rmX}^{\bullet}_{\Pa}(\rmB)$. On one hand, we have 
\begin{equation*}
\begin{aligned}
{m}_{k}(\Lambda)&=(\mathrm{DL}^{\bullet}(\Lambda)\cdot\rmC^{\bullet}(\Lambda_{k}))\cdot_{\rmC^{\bullet}(\Lambda_{k})}(\mathrm{DL}^{\bullet}(\Lambda)\cdot\rmC^{\bullet}(\Lambda_{k}))\\
&=\theta_{k}\cdot_{\rmC^{\bullet}(\Lambda_{k})}\theta_{k}\\
&=0\\
\end{aligned}
\end{equation*}
where $\theta_{k}$ is one of the two generatrices of $\rmC^{\bullet}(\Lambda_{k})$ in \cite[Expose XII Theorem 3.3]{SGA7}. The second equality follows from Corollary \ref{ss-bullet} $(5)$  and the last equality follows from \cite[Expose XII Theorem 3.3 iii (b)]{SGA7}. On the other hand, we have
\begin{equation*}
\begin{aligned}
{m}_{k}(\Lambda)&=(\mathrm{DL}^{\bullet}(\Lambda)\cdot\mathrm{DL}^{\bullet}(\Lambda))\cdot_{\mathrm{DL}^{\bullet}(\Lambda)}(\mathrm{DL}^{\bullet}(\Lambda)\cdot\rmC^{\bullet}(\Lambda_{k}))\\
&=[\calO_{\mathrm{DL}^{\bullet}(\Lambda)}(-2p+\sum^{n}_{k=1}a_{k}\rmE_{k})]\cdot_{\mathrm{DL}^{\bullet}(\Lambda)}[\calO_{\mathrm{DL}^{\bullet}(\Lambda)}(\rmE_{k})]\\
&=-a_{k}\\
\end{aligned}
\end{equation*}
since $\rmE_{k}\cdot_{\mathrm{DL}^{\bullet}(\Lambda)}\rmE_{k}=-1$. Therefore $a_{k}=0$ for each $k$ and we are done. The case of $i=2$ is proved in the same way.
\end{proof}

\section{Level raising conditions}
\subsection{Level raising representations}
Let $\chi_{1}, \chi_{2}, \sigma$ be characters of $\QQ^{\times}_{p}$, we consider the principal series representation of $\bfG(\QQ_{p})=\GSp_{4}(\QQ_{p})$ given by
\begin{equation*}
\chi_{1}\times\chi_{2}\rtimes \sigma:= \mathrm{Ind}^{\bfG(\QQ_{p})}_{\mathrm{\bfB}(\QQ_{p})}\chi_{1}\otimes\chi_{2}\otimes \sigma.
\end{equation*}
This is defined by the normalized induction from the Borel subgroup for the character given by
\begin{equation*}
\chi_{1}\otimes\chi_{2}\otimes\sigma: \begin{pmatrix}a&\ast&\ast&\ast\\ &b&\ast&\ast\\  &&cb^{-1}&\ast\\  &&&ca^{-1}\\ \end{pmatrix}\mapsto \chi_{1}(a)\chi_{2}(b)\sigma(c).
\end{equation*}
This is irreducible if and only if none of the characters $\chi_{1}, \chi_{2}, \chi_{1}\chi^{\pm1}_{2}$ is equal to $\nu^{\pm1}$ where $\nu$ is the normalized valuation such that $\nu(p)=\vert p\vert=p^{-1}$. In this case, we say $\tau=\chi_{1}\times\chi_{2}\rtimes \sigma$ is a {\em type $\rmI$ representation}. If the characters $\chi_{1}, \chi_{2}, \sigma$ are unramified, then we say $\tau$ is an unramified type $\rmI$ representation. This terminology comes from the work of \cite{ST-classification} and \cite{Sch-Iwahori} classifying the non-cuspidal representations of $\bfG(\QQ_{p})$. 
The local Langlands correspondence associates $\tau=\chi_{1}\times\chi_{2}\rtimes \sigma$ with the $L$-parameter given by
\begin{equation*}
\mathrm{rec}_{\mathrm{GT}}(\tau)= \begin{pmatrix}\chi_{1}\chi_{2}\sigma&&&\\&\chi_{1}\sigma&&&\\&&\chi_{2}\sigma&&\\&&&&\sigma\end{pmatrix}
\end{equation*}
if $\tau$ is a type I representation. 

We will be interested in the case when $\tau$ is unramified. Then the representation $\mathrm{rec}_{\mathrm{GT}}(\tau\otimes\vert\rmc\vert^{-3/2})$ has the characteristic polynomial at the geometric Frobenius $\mathrm{Frob}_{p}$ given by
\begin{equation*}
\mathcal{Q}_{p}(\rmX)=\rmX^{4}- {a}_{p,2}\rmX^{3}+ (p{a}_{p,1}+(p^{3}+p)){a}_{p,0}\rmX^{2}-p^{3}{a}_{p,2}{a}_{p,0}\rmX+p^{6}{a}^{2}_{p,0}
\end{equation*}
where ${a}_{p, 0}, {a}_{p, 1},{a}_{p, 2}$ are the eigenvalues of the spherical operators $\mathrm{T}_{p,0}, \mathrm{T}_{p,1}, \mathrm{T}_{p,2}$ on the space $\tau^{\mathbf{G}(\ZZ_{p})}$ of spherical vectors in $\tau$, where 
\begin{equation}\label{Hecke-operators}
\begin{split}
&\rmT_{p,0}=\mathbf{char}(\bfG(\ZZ_{p})\begin{pmatrix}p&&&\\&p&&\\&&p&\\&&&p\\\end{pmatrix}\bfG(\ZZ_{p}));\\
&\rmT_{p,2}=\mathbf{char}(\bfG(\ZZ_{p})\begin{pmatrix}p&&&\\& p&&\\&&1&\\&&&1\\\end{pmatrix}\bfG(\ZZ_{p}));\\
&\rmT_{p,1}=\mathbf{char}(\bfG(\ZZ_{p})\begin{pmatrix}p^{2} &&&\\&p&&\\&&p&\\&&& 1\\\end{pmatrix}\bfG(\ZZ_{p}))\\
\end{split}
\end{equation}
and $\mathbf{char}(\ast)$ is the characteristic function of the double coset $\ast$. We call the roots of the polynomial $\calQ_{p}(\rmX)$ the {\em Hecke parameters} of $\tau$ which can be represented by a quadruple
\begin{equation*}
[\alpha_{p}, \beta_{p}, \gamma_{p}, \delta_{p}]=[\alpha_{p}, \beta_{p}, \xi_{p}\beta^{-1}_{p}, \xi_{p}\alpha^{-1}_{p}] 
\end{equation*}
where $\xi_{p}$ is the similitude character of  $\mathrm{rec}_{\mathrm{GT}}(\tau\otimes\vert \rmc\vert^{-3/2})$ evaluated at $\mathrm{Frob}_{p}$. We are interested in the case $\xi_{p}=p^{3}$.  Moreover,  we can arrange that
\begin{equation*}
\begin{aligned}
&p^{3}\alpha^{-1}_{p}=p^{3/2}\chi_{1}\chi_{2}\sigma\circ\mathrm{Art}^{-1}(\Frob_{p});\\
&p^{3}\beta^{-1}_{p}=p^{3/2}\chi_{1}\sigma\circ\mathrm{Art}^{-1}(\Frob_{p});\\
&\beta_{p}=p^{3/2}\chi_{2}\sigma\circ\mathrm{Art}^{-1}(\Frob_{p});\\
&\alpha_{p}=p^{3/2}\sigma\circ\mathrm{Art}^{-1}(\Frob_{p}).\\
\end{aligned}
\end{equation*}

The level raising representation can be thought of as the representation where the $\chi_{1}\sigma$, $\chi_{2}\sigma$ part of the type $\rmI$ representation degenerates into a Steinberg representation twisted by a quadratic character. More formally, we would like to have a congruence of the form
\begin{equation*}
\begin{aligned}
&\beta=p^{3/2}\chi_{2}\sigma\circ\mathrm{Art}^{-1}(\Frob_{p})\equiv p^{3/2}\chi\nu^{1/2}\circ\mathrm{Art}^{-1}(\Frob_{p})=u p \mod \lambda;\\
&p^{3}\beta^{-1}=p^{3/2}\chi_{1}\sigma\circ\mathrm{Art}^{-1}(\Frob_{p})\equiv p^{3/2}\chi\nu^{-1/2}\circ\mathrm{Art}^{-1}(\Frob_{p})=u p^{2} \mod \lambda.\\
\end{aligned}
\end{equation*}
Here $\chi$ is some quadratic character and $u=\chi(p)\in \{\pm1\}$.

\begin{definition}\label{level-raising-prime-local}
Let $p$ be prime and $\tau$ be the unramified type $\rmI$ representation as above. Then we say $p$ is level raising special for $\tau$ if the following two conditions are satisfied.
\begin{enumerate}
\item $\ell\nmid p(p^{2}-1)$;
\item the Hecke parameters $[\alpha_{p}, \beta_{p}, p^{3}\beta^{-1}, p^{3}\alpha^{-1}_{p}]$ of $\tau$ satisfy simultaneously the following two conditions
\begin{enumerate}
\item $\beta_{p}+p^{3}\beta^{-1}_{p}\equiv u (p+ p^{2})\mod \lambda$ for some $u\in\{\pm 1\}$;
\item $\alpha_{p}+p^{3}\alpha^{-1}_{p}\not\equiv \pm (p+ p^{2})\mod \lambda$.
\end{enumerate}
\end{enumerate}
We say $p$ is level raising special for $\pi$ of depth $m$ if it is level raising special for $\tau$ and the $\lambda$-adic valuation of
\begin{equation*}
\beta_{p}+p^{3}\beta^{-1}_{p}-u(p+ p^{2})
\end{equation*}
is exactly $m$.
\end{definition}

Examine the classification result of \cite{ST-classification} and \cite{Sch-Iwahori}. This can be understood as the congruence between the unramified type $\mathrm{I}$ representation and the ramified type $\mathrm{II}$ representation. Recall a representation $\tau$ is of type $\mathrm{II}$ if $\tau$ is an irreducible constitute of the representation $\nu^{-1/2}\chi\times\nu^{1/2}\chi\rtimes \sigma$ for the normalized valuation $\nu(p)=p^{-1}$. More precisely, it is called of type $\mathrm{IIa}$ if it is of the form $\chi\mathrm{St}_{\GL_{2}}\rtimes \sigma$ and it is called of type $\mathrm{II}\rmb$ if it is of the form $\chi\mathbf{1}_{\GL_{2}}\rtimes \sigma$. Both of these two representations have their common semi-simple part of the $L$-parameter given by
\begin{equation*}
\mathrm{rec}_{\mathrm{GT}}(\tau)= \begin{pmatrix}\chi^{2}\sigma&&&\\&\nu^{1/2}\chi\sigma&&&\\&&\nu^{-1/2}\chi\sigma&&\\&&&&\sigma\end{pmatrix}.
\end{equation*}
While the $L$-parameter of $\chi\mathbf{1}_{\GL_{2}}\rtimes \sigma$ has trivial monodromy $\mathrm{N}=0$,  the $L$-parameter of the type $\mathrm{II}\rma$ representation $\chi\mathrm{St}_{\GL_{2}}\rtimes \sigma$ has monodromy $\mathrm{N}$ conjugate to 
\begin{equation}\label{N1}
\rmN_{\Pa}=\begin{pmatrix}0&&&\\&0&1&\\&&0&\\&&&0\\ \end{pmatrix}. 
\end{equation}
Note that the type $\mathrm{IIa}$ representation is tempered and the type $\mathrm{IIb}$ representation is non-tempered. Our level raising condition can be understood as a condition for congruence between type $\mathrm{I}$ representation and  a special type of the type $\mathrm{II}\rma$ representation. In the work of Sorensen \cite{Sor-level-raising}, he further assumes that  $[\alpha_{p}, \beta_{p}, p^{3}\beta^{-1}, p^{3}\alpha^{-1}_{p}]$ is congruent to $[1, p, p^{2}, p^{3}]$. In fact, his level raising condition corresponds to the case when $\tau$ is congruent to the trivial representation. 

\subsection{Automorphic representation and Hecke algebra}
Let $\pi$ be a cuspidal automorphic representation of $\GSp_{4}(\mathbb{A})$ which is of general type. We assume that $\pi$ has trivial central character and  weight $(k_{1}, k_{2})=(3, 3)$. 

\begin{definition}\label{level-raise-prime}
Suppose $\pi$ is a cuspidal automorphic representation of $\GSp_{4}(\mathbb{A})$ as above and is an unramified principal series at $p$. Then we say $p$ is level raising special for $\pi$ of depth $m$ if $p$ is level raising special of depth $m$ for $\pi_{p}$ in the sense of Definition \ref{level-raising-prime-local}.
\end{definition}

 Let $\Sigma_{\pi}$ be the minimal set of non-archimedean places of $\QQ$ such that $\pi$ is unramified away from the places in $\Sigma_{\pi}$.
\begin{construction}\label{Hecke-Algebra}
Let $\pi$ be as above and $\Sigma$ be a finite set of non-archimedean places of $\QQ$ containing the set $\Sigma_{\pi}$.  We denote by
\begin{equation*}
\TT^{\Sigma}=\bigotimes_{v\not\in \Sigma}\ZZ[\bfG(\ZZ_{v})\backslash\bfG(\QQ_{v})/\bfG(\ZZ_{v})]
\end{equation*}
the abstract Hecke algebra unramified away from $\Sigma$. 
\begin{enumerate}
\item To such $\pi$, we can attach a homomorphism
\begin{equation*}
\phi_{\pi}: \TT^{\Sigma}\rightarrow \CC
\end{equation*}
determined by the Hecke parameters of $\pi$ at each place $v$ outside of $\Sigma$.  More precisely, $\phi_{\pi}$ will send 
$\rmT_{i, v}$ to $a_{v, i}$ for $i\in\{0, 1, 2\}$, where $a_{v, 0}=1$ as $\pi$ has trivial central character and $a_{v,1}$ and $a_{v,2}$ are the eigenvalues of $\rmT_{v,1}$ and $\rmT_{v, 2}$ on the space of spherical vectors of $\tau$ for each $v$ outside of $\Sigma$.

\item In fact, the image of $\phi_{\pi}$ is contained in a number field $E$ called the {\em coefficient field} of $\pi$.  We further assume that the image of $\phi_{\pi}$ is contained in the ring of integers $\calO_{E}$, that is we have a homomorphism
$\phi_{\pi}: \TT^{\Sigma}\rightarrow \calO_{E}$. Let $\lambda$ be a prime of $E$ induced by $\iota_{\ell}$ over $\ell$ and $\calO_{\lambda}$ be the valuation ring of $E_{\lambda}$. Then $\phi_{\pi}$ induces a morphism
\begin{equation*}
\phi_{\pi,\lambda}: \TT^{\Sigma}\rightarrow \calO_{\lambda}.
\end{equation*} 
which we will call the $\lambda$-adic avatar of $\phi_{\pi}$. 

\item We introduce the maximal ideal $\fracm$ of  the Hecke algebra $\TT^{\Sigma\cup\{p\}}$ by 
given by
\begin{equation*}
\fracm=\TT^{\Sigma\cup\{p\}}\cap \ker(\TT^{\Sigma}\xrightarrow{\phi_{\pi,\lambda}}\calO_{\lambda}\rightarrow\calO_{\lambda}/\lambda)
\end{equation*}
and the prime ideal of the Hecke algebra $\TT^{\Sigma\cup\{p\}}$ for each $m\geq 1$ by
\begin{equation*}
\fracn_{m}=\TT^{\Sigma\cup\{p\}}\cap \ker(\TT^{\Sigma}\xrightarrow{\phi_{\pi,\lambda}}\calO_{\lambda}\rightarrow\calO_{\lambda}/\lambda^{m}).
\end{equation*}
When $m$ is clear from the context, then we will write $\fracn_{m}$ simply by $\fracn$.
\end{enumerate}
\end{construction}

\section{Jacquet-Langlands correspondence for $\GSp_{4}$}
\subsection{Strong transfer and strong multiplicity one}
In this subsection, we state three results proved in \cite[\S10, \S11]{RW} concerning the analogue of Jacquet--Langlands transfer between automorphic representations of $\GSp_{4}(\mathbb{A})$ and automorphic representations of quaternionic unitary groups of general type. We recall the notion of a weak equivalence class of an irreducible cuspidal automorphic representation of a reductive group $\rmG$. Suppose  $\pi=\otimes_{v}\pi_{v}$ and $\pi^{\prime}=\otimes_{v}\pi^{\prime}_{v}$ are two such representations, then we say $\pi$ and $\pi^{\prime}$ are weakly equivalent if $\pi_{v}\cong \pi^{\prime}_{v}$ for all most all $v$. Let $\calE(\pi)$ be the weak equivalence packet of $\pi$ that is the set of isomorphism classes of automorphic representations that are weakly equivalent to $\pi$.

\begin{theorem}\label{mult1}
Let $\pi$ be a cuspidal automorphic representation of $\GSp_{4}(\mathbb{A})$ of general type with weight $(3, 3)$ and trivial central character.  Then 
\begin{enumerate}
\item each representation in $\calE(\pi)$ occurs with multiplicity one;
\item The set $\calE(\pi)$ is a cartesian product of local L-packets. Moreover the four dimensional Galois representation attached to $\pi$ as in Theorem \ref{Galois} uniquely determines $\calE(\pi)$.
\end{enumerate}
\end{theorem}
\begin{proof}
This is proved in \cite[Proposition 10.1]{RW}.
\end{proof}
\begin{remark}
This proposition also follows from Arthurs's multiplicity formula proved in \cite{Arthur-GSp} and \cite{GT19}.
\end{remark}

Let $\tau$ be an automorphic representation of $\bfG(\rmB)$ or $\bfG(\overline{\rmB})$ of general type. Here general type means $\tau^{(pq\infty)}$ is not CAP or weakly endoscopic, see \cite[\S 11]{RW}. The following two theorems provide simple instances of the global Jacquet--Langlands correspondence between $\GSp_{4}$ and $\bfG(\overline{\rmB})$ or $\bfG({\rmB})$.
\begin{theorem}\label{mult1-definite}
Let $\tau$ be a cuspidal automorphic representation of $\bfG(\overline{\rmB})$ of general type with weight $(3, 3)$ and trivial central character. Then
\begin{enumerate}
\item $\tau$ lifts uniquely to a cuspidal automorphic representation $\pi$ of $\GSp_{4}(\mathbb{A})$ of general type with weight $(3, 3)$ and trivial central character;
\item A cuspidal automorphic representation $\pi$ of $\GSp_{4}(\mathbb{A})$ of general type with weight $(3, 3)$ does not come from such a lift if and only if the local L-packets of $\pi_{q}$ is Klingen induced. In particular, if $\pi_{q}$ is ramified of type $\mathrm{IIa}$, then $\pi$ come from a lift of such $\tau$;
\item Each irreducible automorphic representation in $\calE(\tau)$ occurs with multiplicity one.
\end{enumerate}
\end{theorem}
\begin{proof}
This is a special case of  \cite[Theorem 11.4]{RW}.
\end{proof}

\begin{theorem}\label{mult1-indefinite}
Let $\tau$ be a cuspidal automorphic representation of $\bfG({\rmB})$ of general type with weight $(3, 3)$ and trivial central character. Then we have 
\begin{enumerate}
\item $\tau$ lifts uniquely to a cuspidal automorphic representation $\pi$ of $\GSp_{4}(\mathbb{A})$ of general type with weight $(3, 3)$ with trivial central character. Moreover $\tau$ is weakly equivalent to $\pi$;
\item A cuspidal automorphic representation $\pi$ of $\GSp_{4}(\mathbb{A})$ of general type with weight $(3, 3)$ does not come from such a lift if and only if the local L-packets of $\pi_{p}$ and $\pi_{q}$ are Klingen induced. In particular, if both $\pi_{p}$ and $\pi_{q}$ are ramified of type $\mathrm{IIa}$, then $\pi$ come from a lift of such $\tau$;
\item Each irreducible automorphic representation in $\calE(\tau)$ occurs with multiplicity one.
\end{enumerate}
\end{theorem}
\begin{proof}
This is a special case of  \cite[Theorem 11.5]{RW}.
\end{proof}

\subsection{Matching of orbital integrals}
In this subsection we will work locally at a prime $p$. Let $\rmG$ be the quaternionic unitary group $\GU_{2}(\rmD)$ and $\bfG$ be the group $\GSp_{4}$ over $\QQ_{p}$ throughout this subsection. We fix an inner twisting $\psi: \rmG\rightarrow \bfG$, that is $\psi$ is an isomorphism when base changed to $\overline{\QQ}_{p}$ such that $\sigma(\psi)\psi^{-1}$ is an inner automorphism of $\bfG$. This $\psi$ is defined over $\QQ_{p^{2}}$ and defines an injection from the semisimple stable conjugacy classes in $\rmG(\QQ_{p})$ to the semisimple stable conjugacy classes in $\bfG(\QQ_{p})$. Let $\gamma\in \rmG(\QQ_{p})$ be a semisimple element and $f\in{C}^{\infty}_{\rmc}(\rmG(\QQ_{p}))$. We recall the orbital integral of $f$ at $\gamma$ is defined by 
\begin{equation*}
{\rmO}^{\rmG}_{\gamma}({f})=\int_{\rmG_{\gamma}(\QQ_{p})\backslash\rmG(\QQ_{p})}f(g^{-1}\gamma g)\mathrm{d}g
\end{equation*}
where $\rmG_{\gamma}$ is the centralizer for $\gamma$ in $\rmG(\QQ_{p})$. Let $\{\tilde{\gamma}\}$ be the set of representatives of the stable conjugacy class of $\gamma$ modulo the center. The the stable orbital integral is 
\begin{equation*}
{\rmS\rmO}^{\rmG}_{\gamma}(f)=\sum_{\tilde{\gamma}} {e}(\rmG_{\tilde{\gamma}})\rmO^{\rmG}_{\tilde{\gamma}}(f)
\end{equation*}
where $e(\cdot)$ is the Kottwitz sign given in \cite[\S 5]{Ko83}. Let $f^{\rmG}_{p}\in{C}^{\infty}_{\rmc}(\rmG(\QQ_{p}))$ and $f^{\bfG}_{p}\in{C}^{\infty}_{\rmc}(\bfG(\QQ_{p}))$. Then we say $f^{\rmG}_{p}$ and $f^{\bfG}_{p}$ have matching orbital integrals if
\begin{equation*}
 \mathrm{SO}^{\bfG}_{\gamma^{\prime}}(f^{\bfG}_{p})=\mathrm{SO}^{\rmG}_{\gamma}(f^{\rmG}_{p})
\end{equation*} 
if $\gamma^{\prime}$ is in the semisimple conjugacy class given by $\psi(\gamma)$ and $\mathrm{SO}^{\bfG}_{\gamma^{\prime}}(f^{\bfG}_{p})=0$ for $\gamma$ not coming from $\rmG$. It follows from \cite[\S 1.5 Th\`eor\'em ]{Wal-trans} one can always find such a function with matching orbital integrals.
For the global case, one defines a global stable orbital integral as a product of local ones. 

Let $e_{\rmK}$ be the characteristic function of $\rmK$ and $\eta$ be the Atkkin-Lehner element given by
\begin{equation*}
\eta=\begin{pmatrix}&&1&\\ &&&1\\ p&&&\\ &p&&\\ \end{pmatrix}
\end{equation*}
Let $\Pa$ be the paramodular subgroup of $\bfG(\QQ_{p})$ and $\Pa^{\rmD}$ be the paramodular subgroup of $\GU_{2}(\rmD)$. 
\begin{theorem}[Sorensen]\label{matching}
The functions $e_{\eta\Pa}$ and $e_{\Pa^{\rmD}}$ have matching orbital integrals.
\end{theorem}
\begin{proof}
This matching theorem is due to Sorensen \cite[Theorem 3.2]{Sor-level-raising}.
\end{proof}

\subsection{Refinement of strong transfer} In this subsection, we will consider the global situation and refine the strong transfer results in $\S 4.1$. Therefore we will let $\rmG$ be the quaternionic unitary group $\bfG(\overline{\rmB})$ or $\bfG(\rmB)$ over $\QQ$ and $\bfG$ be the group $\GSp_{4}$ over $\QQ$. Up to equivalence, $\bfG$ and $\rmG$ admits a unique non-trivial elliptic endoscopic triple $(\rmH, s, \xi)$ consisting of the following data. 
\begin{itemize}
\item The endoscopic group $\rmH$ is 
\begin{equation*}
\rmH=\GL_{2}\times\GL_{2}/\mathbb{G}_{m}
\end{equation*}
where $\mathbb{G}_{m}$ acts by identifying $x$ with $(x, x^{-1})$. The dual group of $\rmH$ is given by $\hat{\rmH}=\{(g, g^{\prime})\in\GL_{2}(\CC)\times\GL_{2}(\CC): \det(g)=\det(g^{\prime})\}$. 

\item  The element $s\in \hat{\rmG}$ is given by
\begin{equation*}
s=\begin{pmatrix}1&&&\\ &-1&&\\ &&-1&\\  &&&1\\ \end{pmatrix}.
\end{equation*}

\item The homomorphism $\xi: \hat{\rmH}\rightarrow \hat{\rmG}$ whose image is given by $\rmZ_{\hat{\rmG}}(s)^{\circ}$ is defined by
\begin{equation*}
\xi: \begin{pmatrix}a&b\\ c&d\\ \end{pmatrix}\times \begin{pmatrix}a^{\prime}&b^{\prime}\\ c^{\prime}&d^{\prime}\\ \end{pmatrix}\rightarrow 
\begin{pmatrix}a^{\prime}&&&b^{\prime}\\ &a&b&\\ &c&d&\\ c^{\prime}&&&d^{\prime}\\ \end{pmatrix}.
\end{equation*}
\end{itemize}
Recall that a semisimple element $\delta\in\rmH(\QQ_{p})$ is said to be a $(\rmG, \rmH)$-regular if $\alpha(\delta)\neq 1$ for every root $\alpha$ of $\rmG$ that does not come from $\rmH$.
\begin{theorem}
For every test function $f^{\rmG}\in C^{\infty}_{\rmc}(\rmG(\QQ_{p}))$, there exists a matching function $f^{\rmH}\in C^{\infty}_{\rmc}(\rmH(\QQ_{p}))$ which means
\begin{equation*}
\mathrm{SO}^{\rmH}_{\delta}(f^{\rmH})=\sum\limits_{\gamma}\Delta_{\rmG, \rmH}(\delta,\gamma)e(\rmG_{\gamma}){\rmO}^{\rmG}_{\gamma}(f^{\rmG})
\end{equation*}
for all $(\rmG, \rmH)$-regular semisimple $\delta\in\rmH(\QQ_{p})$. Here the sum on $\gamma$ runs trough a set of representatives for conjugacy classes in $\rmG(\QQ_{p})$ in the stable conjugacy class associated to $\delta$ and $\Delta_{\rmG, \rmH}(\delta, \gamma)$ is the Langlands--Shelstad transfer factors.
\end{theorem}
\begin{proof}
This follows from the general theorem of Waldapurger \cite[\S1.5 Th\`eor\'em]{Wal-trans} and the standard fundamental lemma due to Hales in \cite[Proposition]{Hal97} in our setting.
\end{proof}
We consider the invariant trace formula for $\rmG$ and $\bfG$ but only for the discrete 
terms. We have spectral expansions for the invariant distributions $\rmI^{\rmG}_{\mathrm{disc}}$ and $\rmI^{\bfG}_{\mathrm{disc}}$: for smooth functions $f^{\rmG}$ and $f^{\bfG}$ for $\rmG(\mathbb{A})$ and  $\bfG(\mathbb{A})$
\begin{equation}\label{distri1}
\begin{aligned}
&\rmI^{\rmG}_{\mathrm{disc}}(f^{\rmG})=\sum_{\vec{\pi}} m^{\rmG}_{\mathrm{disc}}(\vec{\pi})\mathrm{tr}\phantom{.}\vec{\pi}(f^{\rmG})\\
&\rmI^{\bfG}_{\mathrm{disc}}(f^{\bfG})=\sum_{\pi} m^{\bfG}_{\mathrm{disc}}(\pi)\mathrm{tr}\phantom{.}{\pi}(f^{\bfG})\\
\end{aligned}
\end{equation}
where the sums run through the discrete automorphic representations of $\rmG(\mathbb{A})$ and $\bfG(\mathbb{A})$. We can and will assume that all these representations have trivial central characters. For the endoscopic group $\rmH$, we can define the invariant distribution $\rmI^{\rmH}_{\mathrm{disc}}$ and similarly we have a spectral expansion
\begin{equation}\label{distri2}
\rmI^{\rmH}_{\mathrm{disc}}(f^{\rmH})=\sum_{\rho} m^{\rmH}_{\mathrm{disc}}(\rho)\mathrm{tr}\phantom{.}{\rho}(f^{\rmH})\\
\end{equation}
where $\rho$ runs through the discrete automorphic representations of $\rmH(\mathbb{A}_{\QQ})$ and $f^{\rmH}$ is a smooth function for $\rmH(\mathbb{A})$.
When $\pi$ resp. $\vec{\pi}$ is cuspidal  but not a CAP representation, then $m^{\rmG}_{\mathrm{dic}}(\vec{\pi})$ resp. $m^{\bfG}_{\mathrm{dic}}(\pi)$ is the automorphic multiplicity of $\vec{\pi}$ resp. $\pi$. The distribution $\rmI^{\bfG}_{\mathrm{disc}}$ is not stable but can be stabilized by substracting suitable endoscopic terms. In fact, we have the following theorem of Arthur developed in \cite{Art-1}, \cite{Art-2} and \cite{Art-Main}.
\begin{theorem}\label{Arthur}
Let $f^{\rmG}$ and $f^{\bfG}$ be smooth functions as above.
\begin{enumerate}
\item The distribution $\mathrm{ST}^{\bfG}_{\mathrm{disc}}$ defined by
\begin{equation*}
\mathrm{ST}^{\bfG}_{\mathrm{disc}}(f^{\bfG})=\rmI^{\bfG}_{\mathrm{disc}}(f^{\bfG})-\frac{1}{4}\rmI^{\rmH}_{\mathrm{disc}}(f^{\bfG}_{\rmH})
\end{equation*}
is a stable distribution. Here $f^{\bfG}_{\rmH}$ is a function that has matching orbital integrals with $f^{\bfG}$.

\item The distribution $\mathrm{I}^{\rmG}_{\mathrm{disc}}$ can be expressed by
\begin{equation*}
\mathrm{I}^{\rmG}_{\mathrm{disc}}(f^{\rmG})=\mathrm{ST}^{\bfG}_{\mathrm{disc}}(f^{\bfG})+\frac{1}{4}\rmI^{\rmH}_{\mathrm{disc}}(f^{\rmG}_{\rmH}).
\end{equation*}
Here $f^{\rmG}_{\rmH}$  and $f^{\bfG}$ are functions that have matching orbital integrals with $f^{\rmG}$.
\end{enumerate}
\end{theorem}
\begin{proof}
The first part is given in \cite[Theorem 5.7]{Sor09}. The second part is given in \cite[Theorem 5.8]{Sor09}.
\end{proof}
There is a global character identity
\begin{equation}\label{char1}
\mathrm{tr}\phantom{.}\rho (f^{\rmH})=\sum\limits_{\vec{\pi}}\Delta_{\rmG, \rmH}(\rho, \vec{\pi})\mathrm{tr}\phantom{.}\vec{\pi}(f^{\rmG})
\end{equation}
Here $f^{\rmH}$ and $f^{\rmG}$ have matching orbital integrals and $\vec{\pi}$ runs through all irreducible admissible representations of $\rmG(\mathbb{A})$. The term $\Delta_{\rmG, \rmH}(\rho, \vec{\pi})$ is the global transfer factor which is a product of local transfer factors. We have arrived the following lemma.

\begin{lemma}\label{first-trace-id}
Suppose that $f^{\rmG}$ and $f^{\bfG}$ are functions as above with matching orbital integrals. Then following trace identity holds:
\begin{equation*}
\begin{aligned}
&\sum\limits_{\vec{\pi}} (m^{\rmG}_{\mathrm{disc}}(\vec{\pi})-\frac{1}{4}\sum\limits_{\rho}m^{\rmH}_{\mathrm{disc}}(\rho)\Delta_{\rmG, \rmH}(\rho, \vec{\pi}))\mathrm{tr}\phantom{.}\vec{\pi}(f^{\rmG})\\
&=\sum\limits_{\pi} (m^{\bfG}_{\mathrm{disc}}(\pi)-\frac{1}{4}\sum\limits_{\rho}m^{\rmH}_{\mathrm{disc}}(\rho)\Delta_{\bfG, \rmH}(\rho,\pi))\mathrm{tr}\phantom{.}{\pi}(f^{\bfG}).\\
\end{aligned}
\end{equation*}
\end{lemma}
\begin{proof}
This follows by inserting \eqref{distri1}, \eqref{distri2} and \eqref{char1} in the equations in $(1)$ and $(2)$ of  Theorem \ref{Arthur}.
\end{proof}

The following theorem is the main result of this subsection which refines the transfer result in Lemma \ref{mult1-definite} and Lemma \ref{mult1-indefinite}.
\begin{proposition}\label{JL}
Let $\pi$ be a cuspidal automorphic representation of $\bfG(\mathbb{A})$ of general type with weight $(3,3)$ and trivial central character. 
\begin{enumerate}
\item Suppose that $\overline{\QQ}_{\ell}[\Sh(\overline{\rmB}, \rmK)][\iota_{\ell}\pi^{\infty q}]$ is nontrivial for some open compact $\rmK$ of $\bfG(\overline{\rmB})$ which is paramodular at $q$. Then we can complete $\pi^{\infty q}$ to a cuspidal automorphic representation $\Pi$ of $\bfG(\mathbb{A})$ in a unique way. Moreover  for such $\Pi$, its local component $\Pi_{q}$ at $q$ is of type $\rmI\rmI\rma$. 
\item Suppose that $\rmH^{3}_{\rmc}(\Sh(\rmB, \rmK_{\Pa}), \overline{\QQ}_{\ell})[\iota_{\ell}\pi^{pq\infty}]$ is non-trivial for some open compact $\rmK_{\Pa}$ which is paramodular at $p$. Then we can complete $\pi^{\infty pq}$ to a cuspidal automorphic representation $\Pi$ of $\bfG(\mathbb{A})$. Moreover  for all such $\Pi$, both $\Pi_{p}$ and $\Pi_{q}$ are of type $\rmI\rmI\rma$. 
\end{enumerate}
\end{proposition}

\begin{proof}
For the first statement, note that $\pi^{\infty q}$ can be completed to an automorphic representation $\vec{\pi}$ of $\bfG(\overline{\rmB})$ whose component $\vec{\pi}_{q}$ is paraspherical. This means that $\vec{\pi}_{q}$ admits fixed vectors by the paramodular subgroup $\Pa^{\mathsf{D}}$ for the quaternion division algebra $\mathsf{D}$ over $\QQ_{q}$. We apply the trace identity in Lemma \ref{first-trace-id} and fix $\pi^{\infty q}$. Then we have
\begin{equation*}
\begin{aligned}
&\sum\limits_{\vec{\pi}_{q\infty}} m^{\bfG(\overline{\rmB})}_{\mathrm{disc}}(\vec{\pi}_{\infty}\otimes\vec{\pi}_{q}\otimes\pi^{\infty q})\mathrm{tr}\phantom{.}\vec{\pi}_{q\infty}(f^{\bfG(\overline{\rmB})}_{q\infty})\\
=&\sum\limits_{\Pi_{q\infty}} m^{\bfG}_{\mathrm{disc}}(\Pi_{\infty}\otimes\Pi_{q}\otimes\pi^{\infty q})\mathrm{tr}\phantom{.}{\Pi}_{q\infty}(f^{\bfG}_{q\infty}).\\
\end{aligned}
\end{equation*}
For any pair of functions $f^{\bfG(\overline{\rmB})}_{q\infty}$ and $f^{\bfG}_{q\infty}$ with matching orbital integrals for $\bfG(\overline{\rmB})$ and $\bfG$ at  $\infty q$. To see this, since $\pi$ is of general type, the transfer factors in the trace identity in Lemma \ref{first-trace-id} must vanish. Using Shelstad's character identity \cite[Theorem 6.3]{She1} and the same argument as in \cite[4.3.2]{Sor-level-raising}, we can further simplify the above identity to 
\begin{equation*}
\begin{aligned}
&\sum\limits_{\vec{\pi}_{q}} m^{\bfG(\overline{\rmB})}_{\mathrm{disc}}(\vec{\pi}_{\infty}\otimes\vec{\pi}_{q}\otimes\pi^{q\infty})\mathrm{tr}\phantom{.}\vec{\pi}_{q}(f^{\bfG(\overline{\rmB})}_{q})\\
=&\sum\limits_{\Pi_{q}} m^{\bfG}_{\mathrm{disc}}(\Pi_{\infty}\otimes\Pi_{q}\otimes\pi^{\infty q})\mathrm{tr}\phantom{.}{\Pi}_{q}(f^{\bfG}_{q}).\\
\end{aligned}
\end{equation*}
Let $f^{\bfG}_{q}=e_{\eta\Pa}$ and $f^{\bfG(\overline{\rmB})}_{q}=e_{\Pa^{\mathsf{D}}}$, they have matching orbital integrals by Theorem \ref{matching}.  Moreover $\vec{\pi}_{q}$ is para-spherical and hence the left-hand side of the above identity is non-zero. It follows that the Atkin--Lehner element $\eta$ has positive trace on $\Pi^{\Pa}_{q}$. Then $\Pi_{q}$ must be of type $\rmI\rmI\rma$ since an unramified type $\rmI$ representation is traceless for $\eta$ and  $\Pi_{q}$  is tempered. The first part then follows by the strong multiplicity one result in Theorem \ref{mult1}.

For the second part, the same argument as in the first part gives trace identity
\begin{equation*}
\begin{aligned}
&\sum\limits_{\vec{\pi}_{q}} m^{\bfG(\overline{\rmB})}_{\mathrm{disc}}(\vec{\pi}_{p}\otimes\vec{\pi}_{q}\otimes\pi^{pq})\mathrm{tr}\phantom{.}\vec{\pi}_{p}(f^{\bfG(\overline{\rmB})}_{p})\mathrm{tr}\phantom{.}\vec{\pi}_{q}(f^{\bfG(\overline{\rmB})}_{q})\\
=&\sum\limits_{\Pi_{q}} m^{\bfG}_{\mathrm{disc}}(\Pi_{p}\otimes\Pi_{q}\otimes\pi^{pq})\mathrm{tr}\phantom{.}{\Pi}_{p}(f^{\bfG}_{p})\mathrm{tr}\phantom{.}{\Pi}_{q}(f^{\bfG}_{q}).\\
\end{aligned}
\end{equation*}
Let $f^{\bfG}_{p}=e_{\eta\Pa}$ and $f^{\bfG(\overline{\rmB})}_{p}=e_{\Pa^{\rmD}}$ and similarly at $q$, we can again conclude that $\Pi_{p}$ and $\Pi_{q}$ must be of type $\rmI\rmI\rma$ by the same argument as in the first part.
\end{proof}

\section{Picard--Lefschetz formula and Monodromy filtration}
Let $\rmS=\Spec\rmR$ be the spectrum of a henselian DVR. We choose a uniformizer $\pi$ of $\rmR$. We denote by $s$ the closed point of $\rmS$ and by $\eta$ the generic point of $\rmS$. We assume the residue field $k(s)$ at $s$ is of characteristic $p$. Let $\overline{s}$ be a geometric point of $\rmS$ over $s$.  Let $\overline{\eta}$ be a seperable closure of $\eta$. 

For a morphsim $f: \rmX\rightarrow \rmS$, we obtain by base change the following natural inclusion maps
\begin{equation*}
\rmX_{\overline{s}}\xhookrightarrow{\overline{i}} \rmX_{\overline{\rmS}} \xhookleftarrow{\overline{j}} \rmX_{\overline{\eta}}.
\end{equation*}
Let $K\in \rmD^{+}(\rmX, \Lambda)$ with coefficient in $\Lambda=\mathbb{Z}/ \ell^{m}$ or $\ZZ_{\ell}$ with $\ell\neq p$, then we define the {nearby cycle complex} $\rmR\Psi(K)\in \rmD^{+}(\rmX_{\overline{s}}, \Lambda)$ by $\rmR\Psi(K)= \overline{i}^{*}\rmR\overline{j}_{*}(K|_{\rmX_{\overline{\eta}}})$.
For $K\in \rmD^{+}(\rmX, \Lambda)$, the adjunction map defines the specialization map $\mathrm{sp}$ and the following distinguished triangle
\begin{equation}\label{sp-triangle}
K\vert_{\rmX_{\overline{s}}}\xrightarrow{\mathrm{sp}} \rmR\Psi(K\vert_{\rmX_{\overline{\eta}}})\rightarrow \rmR\Phi(K)\rightarrow K\vert_{\rmX_{\overline{s}}}[1].
\end{equation}
The complex $\rmR\Phi(K)$ is known as the \emph{vanishing cycle complex} and is the mapping cone of the specialization map $\mathrm{sp}$. If $f$ is proper, then we have 
\begin{equation*}
\rmR\Gamma(\rmX_{\overline{\eta}}, K\vert_{\rmX_{\overline{\eta}}})\cong \rmR\Gamma(\rmX_{\overline{s}}, \rmR\Psi(K\vert_{\rmX_{\overline{\eta}}}))
\end{equation*}
and in general we always have the following long exact sequence 
\begin{equation*}
\cdots \rightarrow \mathrm{H}^{i}_{(\rmc)}(\rmX_{\overline{s}}, K\vert_{\rmX_{\overline{s}}})\xrightarrow{\rm{sp}}\mathrm{H}^{i}_{(\rmc)}(\rmX_{\overline{s}}, \rmR\Psi(K\vert_{X_{\overline{\eta}}})) \rightarrow  \mathrm{H}^{i}_{(\rmc)}(\rmX_{\overline{s}}, \rmR\Phi(K))\rightarrow \cdots.
\end{equation*}
The first map $\rm{sp}$ is called the \emph{specialization map} and the cohomology of vanishing cycles  $\mathrm{H}^{i}_{(\rmc)}(\rmX_{\overline{s}}, \rmR\Phi(K))$ measures the defect of $\rm{sp}$ from being an isomorphism. If we assume that $f$ is smooth, then  $\mathrm{H}^{i}_{(\rmc)}(\rmX_{\overline{s}}, \rmR\Phi(K))$ vanishes and the specialization map is an isomorphism.
\subsection{Picard--Lefschetz formula in odd dimension}
We recall the formalism of the Picard--Lefschetz formula. Suppose that $f: \rmX\rightarrow \rmS$ is a regular, flat, finite type morphism of relative dimension $n\geq 3$ which is smooth outside a finite collection $\Sigma$ of closed points in $\rmX_{s}$. We denote by $\rmX^{\square}_{s}=\rmX_{s}-\Sigma$ the regular locus of $\rmX_{s}$. In this subsection, we only consider the trivial coefficient $\Lambda$ for simplicity and all the results recalled here can be extended to more general constructible coefficients. We are mainly concerned with the (compactly supported) nearby-cycle cohomology $\mathrm{H}^{n}_{(\rmc)}(\rmX_{\overline{s}}, \rmR\Psi(\Lambda))$ and the action of the monodromy on it. 

Under the above assumptions, we have $\rmR\Phi(\Lambda)|\rmX_{\overline{s}}-\Sigma=0$ and moreover
\begin{equation*}
\rmR\Phi(\Lambda)=\bigoplus_{\sigma\in \Sigma} \rmR^{n}\Phi_{\sigma}(\Lambda)
\end{equation*}
is concentrated at degree $n$. Therefore the distinguished triangle \eqref{sp-triangle} gives the following long  exact sequence
\begin{equation}\label{sp-exact}
\begin{split}
0\rightarrow &\mathrm{H}^{n}_{(\rmc)}(\rmX_{\overline{s}}, \Lambda)\xrightarrow{\rm{sp}}  \mathrm{H}^{n}_{(\rmc)}(\rmX_{\overline{s}}, \rmR\Psi(\Lambda)) \rightarrow \bigoplus_{\sigma\in \Sigma} \rmR^{n}\Phi_{\sigma}(\Lambda)\rightarrow\\ 
&\mathrm{H}^{n+1}_{(\rmc)}(\rmX_{\overline{s}},\Lambda)\rightarrow \mathrm{H}^{n+1}_{(\rmc)}(\rmX_{\overline{s}}, \rmR\Psi(\Lambda))\rightarrow 0.
\end{split}
\end{equation}
Assume furthermore that every $\sigma\in \Sigma$ is an {ordinary quadratic singularity} of $\rmX_{s}$, this condition means $\rmX$ is \'{e}tale locally near each $\sigma$ isomorphic to 
\begin{itemize}
\item  $\rmV(\mathrm{Q})\subset \mathbb{A}^{2r}_{\rmS}$  for $\mathrm{Q}=\sum\limits_{1\leq i\leq r}\rmX_{i}\rmX_{i+r}+\pi$ if $n=2r-1$;
\item  $\rmV(\mathrm{Q})\subset \mathbb{A}^{2r+1}_{\rmS}$  for $\mathrm{Q}=\sum\limits_{1\leq i\leq r}\rmX_{i}\rmX_{i+r}+\rmX^{2}_{2r+1}+\pi$ if $n=2r$.
\end{itemize}

We introduce the following two spaces:
\begin{itemize}
\item the cohomology group $\bigoplus\limits_{\sigma\in\Sigma}\rmR^{n}\Phi_{\sigma}(\Lambda)$ will be referred to as the space of {\em vanishing cycles} on $\rmX$; 
\item the cohomology group $\bigoplus\limits_{\sigma\in \Sigma}\rmH^{n}_{\{\sigma\}}(\rmX_{\overline{s}}, \rmR\Psi(\Lambda))$ will be referred to as the
the space of {\em  co-vanishing cycles} on $\rmX$.
\end{itemize}
Note $\rmH^{n}_{\{\sigma\}}(\rmX_{\overline{s}}, \rmR\Psi(\Lambda)(r))$ is the dual of $\rmR^{n}\Phi_{\sigma}(\Lambda)(r+1)$ under the Poincar\'e duality
\begin{equation*}
\langle\cdot,\cdot\rangle: \rmR^{n}\Phi_{\sigma}(\Lambda)(r+1)\times \rmH^{n}_{\{\sigma\}}(\rmX_{\overline{s}}, \rmR\Psi(\Lambda)(r))\rightarrow \Lambda.
\end{equation*}

For each $\sigma\in\Sigma$, we have an isomorphism
\begin{equation*}
\rmR^{n}\Phi_{\sigma}(\Lambda)\cong\Lambda
\end{equation*} and It follows that $\rmH^{n}_{\{\sigma\}}(\rmX_{\overline{s}}, \rmR\Psi(\Lambda))$ is also free of rank $1$ over $\Lambda$.

\begin{lemma}\label{co-van}
We have the following isomorphisms.
\begin{enumerate}
\item When $n$ is even, then the natural map 
\begin{equation*}
\rmH^{n}_{\{\sigma\}}(\rmX_{\overline{s}}, \Lambda(r))\rightarrow\rmH^{n}_{\{\sigma\}}(\rmX_{\overline{s}}, \rmR\Psi(\Lambda)(r))
\end{equation*}
is zero.
\item When $n$ is odd, then the natural map 
\begin{equation*}
\rmH^{n}_{\{\sigma\}}(\rmX_{\overline{s}}, \Lambda(r))\rightarrow \rmH^{n}_{\{\sigma\}}(\rmX_{\overline{s}}, \rmR\Psi(\Lambda)(r))
\end{equation*}
is an isomorphism.
\end{enumerate}
\end{lemma}
\begin{proof}
Consider the natural short exact sequence 
\begin{equation*}
0\rightarrow \rmH^{n}_{\{\sigma\}}(\rmX_{\overline{s}}, \Lambda(r))\rightarrow \rmH^{n}_{\{\sigma\}}(\rmX_{\overline{s}}, \rmR\Psi(\Lambda)(r))\xrightarrow{\psi} \rmH^{n}_{\{\sigma\}}(\rmX_{\overline{s}}, \rmR\Phi(\Lambda)(r)).
\end{equation*}
The map $\psi$ induces an isomorphism 
\begin{equation*}
\rmH^{n}_{\{\sigma\}}(\rmX_{\overline{s}}, \rmR\Psi(\Lambda)(r))\xrightarrow{\sim} \rmH^{n}_{\{\sigma\}}(\rmX_{\overline{s}}, \rmR\Phi(\Lambda)(r)
\end{equation*}
when $n$ is even by \cite[Expose XV 2.2 D]{SGA7}. On the other hand, the map $\psi$ is zero and hence induces an isomorphism
\begin{equation}
\rmH^{n}_{\{\sigma\}}(\rmX_{\overline{s}}, \Lambda(r))\xrightarrow{\sim} \rmH^{n}_{\{\sigma\}}(\rmX_{\overline{s}}, \rmR\Psi(\Lambda)(r)) 
\end{equation}
when $n$ is odd by \cite[Expose XV 2.2 E]{SGA7}.
\end{proof}

From here on we assume that $n=2r-1\geq 3$ is {\em odd}. Let $\rmI_{\eta}\subset \Gal(\overline{\eta}/\eta)$ be  the inertia group and let $\xi\in \rmI_{\eta}$. Let $t_{\ell}: \rmI_{\eta}\rightarrow\Lambda(1)$ be the tame character. Then we have the local variation map
\begin{equation*}
\mathrm{Var}_{\sigma}(\xi): \rmR^{n}\Phi_{\sigma}(\Lambda)(r)\rightarrow \mathrm{H}^{n}_{\{\sigma\}}(\rmX_{\overline{s}}, \rmR\Psi(\Lambda)(r) )
\end{equation*}
as in \cite[1.4.2]{Illusie-van} and the action of $\xi-1$ on $\rmH^{n}_{(\rmc)}(\rmX_{\overline{\eta}}, \Lambda(r))$ can be factored as 
\begin{equation*}
\begin{tikzcd}
\mathrm{H}^{n}_{(\rmc)}(\rmX_{\overline{s}}, \rmR\Psi(\Lambda)(r)) \arrow[r] \arrow[d, "\xi-1"] & \bigoplus\limits_{\sigma\in\Sigma}\rmR^{n}\Phi_{\sigma}(\Lambda)(r) \arrow[d, "\mathrm{Var}_{\Sigma}(\xi)"]  \\
 \mathrm{H}^{n}_{(\rmc)}(\rmX_{\overline{s}}, \rmR\Psi(\Lambda)(r))    &\bigoplus\limits_{\sigma\in\Sigma}\rmH^{n}_{\{\sigma\}}(\rmX_{\overline{s}}, \rmR\Psi(\Lambda)(r)) \arrow[l]          
\end{tikzcd}
\end{equation*}
where $\mathrm{Var}_{\Sigma}(\xi)=\bigoplus\limits_{\sigma\in\Sigma}\mathrm{Var}_{\sigma}(\xi)$ is the sum of the local variation morphism described by the usual Picard--Lefschetz formula: 
\begin{equation}\label{pic-lef-form}
\mathrm{Var}_{\sigma}(\xi)(x)=(-1)^{r+1}t_{\ell}(\xi)\langle x, \delta\rangle \delta
\end{equation}
where $x\in\rmR^{n}\Phi_{\sigma}(\Lambda)(r)$ is any element and $\delta\in\rmH^{n}_{\{\sigma\}}(\rmX_{\overline{s}}, \rmR\Psi(\Lambda)(r))$ is a generator of this rank one module over $\Lambda$. One has also a Frobenius equivariant version
\begin{equation*}
\begin{tikzcd}
\mathrm{H}^{n}_{(\rmc)}(\rmX_{\overline{s}}, \rmR\Psi(\Lambda)(r+1)) \arrow[r] \arrow[d, "\mathrm{N}"] & \bigoplus\limits_{\sigma\in\Sigma}\rmR^{n}\Phi_{\sigma}(\Lambda(r+1)) \arrow[d, "\mathrm{N}_{\Sigma}"]  \\
\mathrm{H}^{n}_{(\rmc)}(\rmX_{\overline{s}}, \rmR\Psi(\Lambda)(r)) &\bigoplus\limits_{\sigma\in\Sigma}\rmH^{n}_{\{\sigma\}}(\rmX_{\overline{s}}, \rmR\Psi(\Lambda)(r))  \arrow[l]              
\end{tikzcd}          
\end{equation*}
where $\rmN$ is the monodromy operator and $\rmN_{\Sigma}=\bigoplus\limits_{\sigma\in\Sigma}\rmN_{\sigma}$ 
is the sum of the local monodromy operators.

We will need a refinement of the above diagram. Let $i_{\Sigma}: \Sigma_{\overline{s}}\hookrightarrow \rmX_{\overline{s}}$ and $j_{\Sigma}: \rmX^{\square}_{\overline{s}}\hookrightarrow \rmX_{\overline{s}}$ be the natural inclusions. Then we have an isomorphism $j^{\ast}_{\Sigma}\rmR\Psi(\Lambda)\cong \Lambda_{\rmX^{\square}_{\overline{s}}}$, since $\rmX_{\overline{s}}$ is smooth away from $\Sigma_{\overline{s}}$. Then we have a long exact sequence
\begin{equation*}
\cdots\rightarrow\rmH^{n}_{(\rmc)}(\rmX_{\overline{s}}, \rmR\Psi(\Lambda)(r))\rightarrow \rmH^{n}_{(\rmc)}(\Sigma_{\overline{s}}, \rmR\Psi(\Lambda)(r))\rightarrow \rmH^{n+1}_{(\rmc)}(\rmX_{\overline{s}}, j_{\Sigma!}j^{\ast}_{\Sigma}\rmR\Psi(\Lambda)(r))\rightarrow\cdots.
\end{equation*}
Note that we have an isomorphism $\rmH^{n}_{(\rmc)}(\Sigma_{\overline{s}}, \rmR\Psi(\Lambda)(r))\cong\bigoplus\limits_{\sigma\in\Sigma}\rmR^{n}\Phi_{\sigma}(\Lambda(r))$. 

Dualizing this long exact sequence, we obtain another long exact sequence
\begin{equation*}
\cdots\rightarrow\rmH^{n-1}_{(c)}(\rmX_{\overline{s}}, j_{\Sigma\ast}j^{\ast}_{\Sigma}\rmR\Psi(\Lambda)(r))\rightarrow\bigoplus\limits_{\sigma\in\Sigma}\rmH^{n}_{\{\sigma\}}(\rmX_{\overline{s}}, \rmR\Psi(\Lambda(r)))\rightarrow\rmH^{n}_{(\rmc)}(\rmX_{\overline{s}}, \rmR\Psi(\Lambda)(r))\rightarrow \cdots.
\end{equation*}
Here we remark that $ j_{\Sigma\ast}$ is understood as its derived version $\rmR j_{\Sigma\ast}$ and similar conventions will be used in the rest of this article.

Therefore we arrive at the following commutative diagram
\begin{equation}\label{Pic-Lef-no-c}
\begin{tikzcd}
\mathrm{H}^{n}_{(\rmc)}(\rmX_{\overline{s}}, \rmR\Psi(\Lambda)(r+1)) \arrow[r] \arrow[d, "\mathrm{N}"] & \bigoplus\limits_{\sigma\in\Sigma}\rmR^{n}\Phi_{\sigma}(\Lambda)(r+1) \arrow[d, "\mathrm{N}_{\Sigma}"]  \arrow[r, "\alpha_{(\rmc)}(r+1)"] & \mathrm{H}^{n+1}_{(\rmc)}(\rmX_{\overline{s}}, j_{\Sigma!}\Lambda(r+1))\\
\mathrm{H}^{n}_{(\rmc)}(\rmX_{\overline{s}}, \rmR\Psi(\Lambda)(r)) &\bigoplus\limits_{\sigma\in\Sigma}\rmH^{n}_{\{\sigma\}}(\rmX_{\overline{s}}, \rmR\Psi(\Lambda)(r))  \arrow[l] & \mathrm{H}^{n-1}_{(\rmc)}(\rmX_{\overline{s}}, j_{\Sigma\ast}\Lambda(r))\arrow[l, "\beta_{(\rmc)}(r)"]        
\end{tikzcd}          
\end{equation}
which we will refer to as the {\em balanced Picard--Lefschetz formula} for $\mathrm{H}^{n}_{(\rmc)}(\rmX_{\overline{s}}, \rmR\Psi(\Lambda)(r+1))$. 

On the other hand, we also obtain the following commutative diagram
\begin{equation}\label{Pic-Lef-c-no}
\begin{tikzcd}
\mathrm{H}^{n}_{\rmc}(\rmX_{\overline{s}}, \rmR\Psi(\Lambda)(r+1)) \arrow[r] \arrow[d, "\mathrm{N}"] & \bigoplus\limits_{\sigma\in\Sigma}\rmR^{n}\Phi_{\sigma}(\Lambda)(r+1) \arrow[d, "\mathrm{N}_{\Sigma}"]  \arrow[r, "\alpha_{\rmc}(r+1)"] & \mathrm{H}^{n+1}_{\rmc}(\rmX_{\overline{s}}, j_{\Sigma !}\Lambda(r+1))\cong \mathrm{H}^{n+1}_{\rmc}(\rmX^{\square}_{\overline{s}}, \Lambda(r+1))\\
\mathrm{H}^{n}(\rmX_{\overline{s}}, \rmR\Psi(\Lambda)(r)) &\bigoplus\limits_{\sigma\in\Sigma}\rmH^{n}_{\{\sigma\}}(\rmX_{\overline{s}}, \rmR\Psi(\Lambda)(r))  \arrow[l] & \mathrm{H}^{n-1}(\rmX_{\overline{s}}, j_{\Sigma\ast}\Lambda(r))\cong\mathrm{H}^{n-1}(\rmX^{\square}_{\overline{s}}, \Lambda(r))\arrow[l, "\beta(r)"]            
\end{tikzcd}          
\end{equation}
along with its dual commutative diagram
\begin{equation}\label{Pic-Lef-no-c}
\begin{tikzcd}
\mathrm{H}^{n}(\rmX_{\overline{s}}, \rmR\Psi(\Lambda)(r+1)) \arrow[r] \arrow[d, "\mathrm{N}"] & \bigoplus\limits_{\sigma\in\Sigma}\rmR^{n}\Phi_{\sigma}(\Lambda)(r+1) \arrow[d, "\mathrm{N}_{\Sigma}"]  \arrow[r, "\alpha(r+1)"] & \mathrm{H}^{n+1}(\rmX_{\overline{s}}, j_{\Sigma!}\Lambda(r+1))=\mathrm{H}^{n+1}_{\rmc-\partial}(\rmX^{\square}_{\overline{s}}, \Lambda(r+1))\\
\mathrm{H}^{n}_{\rmc}(\rmX_{\overline{s}}, \rmR\Psi(\Lambda)(r)) &\bigoplus\limits_{\sigma\in\Sigma}\rmH^{n}_{\{\sigma\}}(\rmX_{\overline{s}}, \rmR\Psi(\Lambda)(r))  \arrow[l] & \mathrm{H}^{n-1}_{\rmc}(\rmX_{\overline{s}}, j_{\Sigma\ast}\Lambda(r))=\mathrm{H}^{n-1}_{\partial-\rmc}(\rmX^{\square}_{\overline{s}}, \Lambda(r))\arrow[l, "\beta_{\rmc}(r)"].            
\end{tikzcd}          
\end{equation}
Here the cohomology groups
\begin{equation*}
\begin{aligned}
&\mathrm{H}^{n+1}_{\rmc-\partial}(\rmX^{\square}_{\overline{s}}, \Lambda(r+1))=\mathrm{H}^{n+1}(\rmX_{\overline{s}}, j_{\Sigma!}\Lambda(r+1))\\
&\mathrm{H}^{n-1}_{\partial-\rmc}(\rmX^{\square}_{\overline{s}}, \Lambda(r))=\mathrm{H}^{n-1}_{\rmc}(\rmX_{\overline{s}}, j_{\Sigma\ast}\Lambda(r))\\
\end{aligned}
\end{equation*}
are often called the partially compactified cohomology of $\rmX^{\square}_{\overline{s}}$ in the literature.

In the following, we will loosely refer to these two diagrams as the {\em unbalanced Picard--Lefschetz formula} for $\mathrm{H}^{n}_{(\rmc)}(\rmX_{\overline{s}}, \rmR\Psi(\Lambda)(r+1))$. We will only the first unbalanced Picard--Lefschetz formula will be used in this article.

\subsection{Monodromy filtration and Picard--Lefschetz}
We will define the monodromy filtration 
\begin{equation*}
0\subset_{\mathrm{Gr}_{-1, (\rmc)}(r)}\mathrm{F}_{-1}\mathrm{H}^{n}_{(\rmc)}(\rmX_{\overline{s}}, \rmR\Psi(\Lambda)(r)) \subset_{\mathrm{Gr}_{0, (\rmc)}(r)} \mathrm{F}_{0}\mathrm{H}^{n}_{(\rmc)}(\rmX_{\overline{s}}, \rmR\Psi(\Lambda)(r)) \subset_{\mathrm{Gr}_{1, (\rmc)}(r)} \mathrm{F}_{1}\mathrm{H}^{n}_{(\rmc)}(\rmX_{\overline{s}}, \rmR\Psi(\Lambda)(r)).
\end{equation*}
for  $\mathrm{H}^{n}_{(\rmc)}(\rmX_{\overline{s}}, \rmR\Psi(\Lambda)(r))$ using the balanced Picard--Lefschetz formula. Note that the monodromy operator $\rmN:\rmR\Psi(\Lambda)(1)\rightarrow \rmR\Psi(\Lambda)$ satisfies $\rmN^{2}=0$. Therefore we can define the monodromy filtration on $\rmR\Psi(\Lambda)$ by
\begin{equation*}
0\subset \mathrm{F}_{-1}\rmR\Psi(\Lambda)=\mathrm{im}(\rmN)\subset \rmF_{0}\rmR\Psi(\Lambda)=\ker(\rmN)\subset \rmF_{1}\rmR\Psi(\Lambda)=\rmR\Psi(\Lambda).
\end{equation*}
By \cite[Proposition 3.8]{Illusie-per}, we have 
\begin{equation}\label{mono-fil-sheaf}
\begin{aligned}
&\mathrm{F}_{-1}\rmR\Psi(\Lambda)(r)=\bigoplus\limits_{\sigma\in\Sigma}\rmH^{n}_{\{\sigma\}}(\rmX_{\overline{s}}, \rmR\Psi(\Lambda)(r))[-n]\\
&\mathrm{F}_{0}\rmR\Psi(\Lambda)(r)=\Lambda_{\rmX_{\overline{s}}}\\
\end{aligned}
\end{equation}
and also
\begin{equation*}
\mathrm{Gr}_{1}\rmR\Psi(\Lambda)(r)=\bigoplus\limits_{\sigma\in\Sigma}\rmR^{n}\Phi_{\sigma}(\Lambda)(r)[-n].
\end{equation*}
We define the monodromy filtration $\rmF_{i}\rmH^{n}_{(\rmc)}(\rmX_{\overline{s}}, \rmR\Psi(\Lambda)(r))$ of $\rmH^{n}_{(\rmc)}(\rmX_{\overline{s}}, \rmR\Psi(\Lambda)(r))$ as the image of the canonical map $\rmH^{n}_{(\rmc)}(\rmX_{\overline{s}}, \rmF_{i}\rmR\Psi(\Lambda)(r))\rightarrow \rmH^{n}_{(\rmc)}(\rmX_{\overline{s}}, \rmR\Psi(\Lambda)(r))$. 

The successive quotient of the monodromy filtration is therefore given by
\begin{equation}\label{mono-fil}
\begin{aligned}
&\mathrm{Gr}_{-1, (\rmc)}(r)=\mathrm{Gr}_{-1}\mathrm{H}^{n}_{(\rmc)}(\rmX_{\overline{s}}, \rmR\Psi(\Lambda)(r))= \Coker(\beta_{(\rmc)}(r))\\
&\mathrm{Gr}_{0, (\rmc)}(r)=\mathrm{Gr}_{0}\mathrm{H}^{n}_{(\rmc)}(\rmX_{\overline{s}}, \rmR\Psi(\Lambda)(r))= \mathrm{H}^{n}_{(\rmc)}(\rmX_{\overline{s}}, \Lambda(r))/\bigoplus\limits_{\sigma\in\Sigma}\rmH^{n}_{\{\sigma\}}(\rmX_{\overline{s}}, \rmR\Psi(\Lambda)(r))\\
&\mathrm{Gr}_{1, (\rmc)}(r)=\mathrm{Gr}_{1}\mathrm{H}^{n}_{(\rmc)}(\rmX_{\overline{s}}, \rmR\Psi(\Lambda)(r))=\Ker(\alpha_{(\rmc)}(r))\\
\end{aligned}
\end{equation}
using the balanced Picard--Lefschetz formula. 

The filtration induces a filtration $\mathrm{F}_{i}\mathrm{H}^{1}(\mathrm{I}_{\QQ_{p^{2}}}, \mathrm{H}^{n}_{(\rmc)}(\rmX_{\overline{s}}, \rmR\Psi(\Lambda)(r)))$ on the quotient module
\begin{equation*}
\mathrm{H}^{1}(\mathrm{I}_{\QQ_{p^{2}}}, \mathrm{H}^{n}_{(\rmc)}(\rmX_{\overline{s}}, \rmR\Psi(\Lambda)(r))). 
\end{equation*}
The monodromy operator $\mathrm{N}$ factors through 
\begin{equation*}
\mathrm{H}^{n}_{(\rmc)}(\rmX_{\overline{s}}, \rmR\Psi(\Lambda)(r+1))\twoheadrightarrow\mathrm{Gr}_{1, (\rmc)}(r+1)\xrightarrow{\mathrm{N}} \mathrm{Gr}_{-1, (\rmc)}(r) \hookrightarrow \mathrm{H}^{n}_{(\rmc)}(\rmX_{\overline{s}}, \rmR\Psi(\Lambda)(r)).
\end{equation*}
The above discussion leads to the following proposition.

\begin{proposition}\label{F-1}
Suppose $\alpha_{(\rmc)}(r+1))$ is  surjective. We have an isomorphism
\begin{equation*}
\mathrm{F}_{-1}\mathrm{H}^{1}(\mathrm{I}_{\QQ_{p^{2}}}, \mathrm{H}^{n}_{(\rmc)}(\rmX_{\overline{s}}, \rmR\Psi(\Lambda)(r)))\cong \frac{\Coker(\beta_{(\rmc)}(r))}{\mathrm{N}_{\Sigma}\Ker(\alpha_{(\rmc)}(r+1))}.
\end{equation*}
The map $\alpha_{(\rmc)}$ induces an isomorphims
\begin{equation*}
\mathrm{F}_{-1}\mathrm{H}^{1}(\mathrm{I}_{\QQ_{p^{2}}}, \mathrm{H}^{n}_{(\rmc)}(\rmX_{\overline{s}}, \rmR\Psi(\Lambda)(r)))\cong \Coker(\alpha_{(\rmc)}(r+1)\circ \rmN^{-1}_{\Sigma}\circ \beta_{(\rmc)}(r)).
\end{equation*}
\end{proposition}
\begin{proof}
The first isomorphism follows directly from the above discussion. The second isomorphism follows from applying $\alpha_{(\rmc)}$ to the first isomorphism and the fact that $\rmN_{\Sigma}$ is an isomorphism in our case. 
\end{proof}

\subsection{Monodromy filtration and intersection cohomology}
The following Lemma will explain the relation between the cohomology of the regular locus $\rmX^{\square}_{\overline{s}}$ and the intersection cohomology of $\rmX_{\overline{s}}$. Let $\mathrm{IH}^{\ast}_{(\rmc)}(\rmX_{\overline{s}}, \Lambda(\ast))$ be the intersection cohomology (with compact support) of $\rmX_{\overline{s}}$. Recall we denote by $j_{\Sigma}: \rmX^{\square}_{\overline{s}}\hookrightarrow \rmX_{\overline{s}}$ the open immersion of the regular locus. Let $\mathrm{IC}_{\rmX}(\Lambda)=j_{\Sigma!\ast}\Lambda[n]$ be Deligne's intersection complex of $\rmX_{\overline{s}}$, then by definition $\mathrm{IH}^{\ast}_{(\rmc)}(\rmX_{\overline{s}}, \Lambda(\ast))$ is defined by $\mathrm{H}^{\ast}_{(\rmc)}(\rmX_{\overline{s}}, \mathrm{IC}_{\rmX}(\Lambda(\ast))[-n])$.

\begin{lemma}
We have the following isomorphisims
\begin{enumerate}
\item $\mathrm{IH}^{k}_{(\rmc)}(\rmX_{\overline{s}}, \Lambda)\cong \mathrm{H}^{k}_{(\rmc)}(\rmX_{\overline{s}}, j_{\Sigma\ast}\Lambda)$ for $k\leq n-1$;
\item $\mathrm{IH}^{k}_{(\rmc)}(\rmX_{\overline{s}}, \Lambda)\cong \mathrm{H}^{k}_{(\rmc)}(\rmX_{\overline{s}}, \Lambda)$ for $k\geq n+1$;
\item $\mathrm{IH}^{n}_{(\rmc)}(\rmX_{\overline{s}}, \Lambda)\cong\mathrm{Im}(\rmH^{n}_{(\rmc)}(\rmX_{\overline{s}}, \Lambda)\rightarrow \rmH^{n}_{(\rmc)}(\rmX_{\overline{s}}, j_{\Sigma\ast} \Lambda))$ for $k=n$.
\item $\mathrm{H}^{k}_{\rmc}(\rmX^{\square}_{\overline{s}}, \Lambda)\cong\mathrm{H}^{k}_{\rmc}(\rmX_{\overline{s}}, \Lambda)$ for $k\geq 2$;
\end{enumerate}
In particular $\mathrm{H}^{k}(\rmX^{\square}_{\overline{s}}, \Lambda)\cong\mathrm{IH}^{k}(\rmX_{\overline{s}}, \Lambda)$ for $k\leq n-1$ and $\mathrm{H}^{k}_{\rmc}(\rmX^{\square}_{\overline{s}}, \Lambda)\cong\mathrm{IH}^{k}_{\rmc}(\rmX_{\overline{s}}, \Lambda)$ for $k\geq n+1$.
\end{lemma}
\begin{proof}
We have $\mathrm{IC}_{\rmX}=\tau_{\leq n-1}(\rmR j_{\Sigma\ast}\Lambda)[n]$ which is isomorphic to the intermediate extension $j_{\Sigma!\ast}\Lambda$ cf. \cite[2.2.3.1]{BBD} and fits into the following distinguished triangle
\begin{equation*}
\mathrm{IC}_{\rmX}[-n]\rightarrow\mathrm{R}j_{\Sigma\ast}\Lambda\rightarrow\tau_{\geq n}(\rmR j_{\Sigma\ast}\Lambda)\rightarrow\mathrm{IC}_{\rmX}[1-n].
\end{equation*}
 Note that $\rmH^{k}(\tau_{\geq n}(\rmR j_{\Sigma\ast}\Lambda))=0$ for $k\leq n-1$. Then the statement $(1)$ follows from considering the hypercohomology spectral sequence and the long exact sequence attached to this distinguished triangle.  We also obtain an injection $\mathrm{IH}^{n}_{(\rmc)}(\rmX_{\overline{s}}, \Lambda)\hookrightarrow \rmH^{n}_{(\rmc)}(\rmX_{\overline{s}}, j_{\Sigma\ast}\Lambda)$ from this discussion.

For the second statement $(2)$, we consider the distinguished triangle
\begin{equation*}
\Lambda\rightarrow \mathrm{IC}_{\rmX}[-n]\rightarrow \calF\rightarrow\Lambda[1]
\end{equation*}
where $\calF$ is a complex supported on the singular locus $\Sigma$. Note we have
\begin{equation*}
\rmH^{k}_{(\rmc)}(\rmX_{\overline{s}}, \calF)\cong \bigoplus\limits_{\sigma\in\Sigma}\rmH^{k}(\calF)_{\sigma}.
\end{equation*}
The distinguished triangle implies that $\rmH^{k}(\calF)_{\sigma}\cong\rmH^{k-n}(\mathrm{IC}_{\rmX})_{\sigma}$ for $k\geq 1$. Moreover the axioms \cite[Proposition 2.1.17]{BBD} characterizing the perverse sheaf $\mathrm{IC}_{\rmX}$ immediately implies that $\rmH^{i}(\mathrm{IC}_{\rmX})_{\sigma}=0$ for $i\geq 0$. Therefore we have
$\rmH^{k}_{(\rmc)}(\rmX_{\overline{s}}, \calF)=0$ for $k\geq n$. Thus $(2)$ follows from this by considering the long exact sequence attached to the above distinguished triangle. We also obtain a surjection $\mathrm{H}^{n}_{(\rmc)}(\rmX_{\overline{s}}, \Lambda)\twoheadrightarrow \mathrm{IH}^{n}_{(\rmc)}(\rmX_{\overline{s}}, \Lambda)$ from the above discussion. Now part $(3)$ follows from the discussions in part $(1)$ and part $(2)$ where we have proved that we have the following factorization
\begin{equation*}
\mathrm{H}^{n}_{(\rmc)}(\rmX_{\overline{s}}, \Lambda)\twoheadrightarrow \mathrm{IH}^{n}_{(\rmc)}(\rmX_{\overline{s}}, \Lambda)\hookrightarrow \mathrm{H}^{n}_{(\rmc)}(\rmX_{\overline{s}}, j_{\Sigma\ast}\Lambda)
\end{equation*}
of the canonical map $\mathrm{H}^{n}_{(\rmc)}(\rmX_{\overline{s}}, \Lambda)\rightarrow \mathrm{H}^{n}_{(\rmc)}(\rmX_{\overline{s}}, j_{\Sigma\ast}\Lambda)$.

For statement $(4)$, we need to consider the distinguished triangle
\begin{equation*}
j_{\Sigma!}j^{\ast}_{\Sigma}\Lambda\rightarrow \Lambda \rightarrow i_{\Sigma\ast}i^{\ast}_{\Sigma}\Lambda\rightarrow j_{\Sigma!}j^{\ast}_{\Sigma}\Lambda[1]
\end{equation*}
where $i_{\Sigma}:\Sigma\hookrightarrow \rmX_{\overline{s}}$ is the closed immersion of the singular locus. The induced long exact sequence immediately gives the desired statement.
\end{proof}

\begin{corollary}\label{intersection}
We have the following isomorphism
\begin{equation*}
\begin{aligned}
\mathrm{IH}^{n}_{(\rmc)}(\rmX_{\overline{s}}, \Lambda)&\cong\mathrm{Im}(\mathrm{H}^{n}_{(\rmc)}(\rmX_{\overline{s}}, \Lambda)\rightarrow \mathrm{H}^{n}_{(\rmc)}(\rmX_{\overline{s}}, j_{\Sigma\ast}\Lambda))\\
&\cong\mathrm{H}^{n}_{(\rmc)}(\rmX_{\overline{s}}, \Lambda)/\bigoplus\limits_{\sigma\in\Sigma}\rmH^{n}_{\{\sigma\}}(\rmX_{\overline{s}}, \Lambda).\\
\end{aligned}
\end{equation*}
\end{corollary}
\begin{proof}
The first isomorphism is already proved in the above lemma. The second isomorphism follows from the long exact sequence induced by the distinguished triangle
\begin{equation*}
i_{\Sigma\ast}i^{!}_{\Sigma}\Lambda\rightarrow \Lambda\rightarrow \rmR j_{\Sigma\ast}j^{\ast}_{\Sigma}\Lambda\rightarrow i_{\Sigma\ast}i^{!}_{\Sigma}\Lambda[1].
\end{equation*}
\end{proof}

From the above discussions, we conclude that graded pieces of the monodromy filtration 
\begin{equation*}
0\subset_{\mathrm{Gr}_{-1, (\rmc)}(r)}\mathrm{F}_{-1}\mathrm{H}^{n}_{(\rmc)}(\rmX_{\overline{s}}, \rmR\Psi(\Lambda)(r)) \subset_{\mathrm{Gr}_{0, (\rmc)}(r)} \mathrm{F}_{0}\mathrm{H}^{n}_{(\rmc)}(\rmX_{\overline{s}}, \rmR\Psi(\Lambda)(r)) \subset_{\mathrm{Gr}_{1, (\rmc)}(r)} \mathrm{F}_{1}\mathrm{H}^{n}_{(\rmc)}(\rmX_{\overline{s}}, \rmR\Psi(\Lambda)(r))
\end{equation*}
of $\mathrm{H}^{n}_{(\rmc)}(\rmX_{\overline{s}}, \rmR\Psi(\Lambda)(r))$ are given by
\begin{equation}\label{mono-fil-noncompact}
\begin{aligned}
&\mathrm{Gr}_{-1}\mathrm{H}^{n}_{(\rmc)}(\rmX_{\overline{s}}, \rmR\Psi(\Lambda)(r))= \Coker(\beta_{(\rmc)}(r):\mathrm{IH}^{n-1}_{(\rmc)}(\rmX_{\overline{s}}, \Lambda(r))\rightarrow  \bigoplus\limits_{\sigma\in\Sigma}\rmH^{n}_{\{\sigma\}}(\rmX_{\overline{s}}, \rmR\Psi(\Lambda)(r)));\\
&\mathrm{Gr}_{0}\mathrm{H}^{n}_{(\rmc)}(\rmX_{\overline{s}}, \rmR\Psi(\Lambda)(r))= \mathrm{IH}^{n}_{(\rmc)}(\rmX_{\overline{s}}, \Lambda(r));\\
&\mathrm{Gr}_{1}\mathrm{H}^{n}_{(\rmc)}(\rmX_{\overline{s}}, \rmR\Psi(\Lambda)(r))=\Ker(\alpha_{(\rmc)}(r):\bigoplus\limits_{\sigma\in\Sigma}\rmR^{n}\Phi_{\sigma}(\Lambda)(r) \rightarrow \mathrm{IH}^{n+1}_{(\rmc)}(\rmX_{\overline{s}}, \Lambda(r))).\\
\end{aligned}
\end{equation}
Note that we have used the fact that 
\begin{equation*}
\begin{aligned}
&\mathrm{Gr}_{0}\mathrm{H}^{n}_{(\rmc)}(\rmX_{\overline{s}}, \rmR\Psi(\Lambda)(r))\\
&\cong\mathrm{H}^{n}_{(\rmc)}(\rmX_{\overline{s}}, \Lambda(r))/\bigoplus\limits_{\sigma\in\Sigma}\rmH^{n}_{\{\sigma\}}(\rmX_{\overline{s}}, \rmR\Psi(\Lambda(r)))\\
&\cong\mathrm{H}^{n}_{(\rmc)}(\rmX_{\overline{s}}, \Lambda(r))/\bigoplus\limits_{\sigma\in\Sigma}\rmH^{n}_{\{\sigma\}}(\rmX_{\overline{s}}, \Lambda(r)))\\
&\cong \mathrm{IH}^{n}_{(\rmc)}(\rmX_{\overline{s}}, \Lambda).
\end{aligned}
\end{equation*}
\begin{remark}
We suspect that the graded pieces of the monodromy filtration of the perverse sheaf $\rmR\Psi(\Lambda)[n]$ can already be described by the intersection cohomology complex on $\rmX_{\overline{s}}$ in a similar fashion as above but we cannot prove this at the moment. This is suggested by the referee. On the other hand, one can also calculate the monodromy filtration of the perverse sheaf $\rmR\Psi(\Lambda)[n]$ for the semistable model described in the next subsection and in this case it is easy to prove the graded pieces of the monodromy filtration are indeed described by the intersection cohomology complex.
\end{remark}

\subsection{Descriptions of $\alpha_{\rmc}$ and $\beta$}\label{semi} We need to realize the maps $\alpha_{\rmc}$ and $\beta$ as certain Gysin and restriction maps. For this purpose, we need to introduce a semi-stable model of $\rmX$ over ${\rmS}_{1}=\Spec{\rmR_{1}}$. Here we let ${\rmF}_{1}$ be a totally ramified quadratic extension over the fraction field $\rmF$ of $\rmR$ and let ${\rmR}_{1}$ be the integral closure of $\rmR$ in ${\rmF}_{1}$. We will consider the base change of $\rmX$ to ${\rmR}_{1}$ but we will use the same notation. The following descriptions can be easily deduced from \cite[Proposition 2.4]{Illusie-van}:
\begin{itemize}
\item Let $\rmX^{\mathrm{Bl}}$ be the blow-up of  $\rmX$ at the singular locus $\Sigma$. The scheme  $\rmX^{\mathrm{Bl}}$ is semi-stable and its special fiber has the form  $\rmX^{\mathrm{Bl}}_{s}={\rmX}^{\bullet}_{s}+ \sum\limits_{\sigma\in\Sigma}\rmD_{\sigma}$ as a divisor in $\rmX^{\mathrm{Bl}}$. 
\item The scheme ${\rmX}^{\bullet}_{s}$ is the strict transform of $\rmX_{s}$ under this blow-up and $\rmD=\sum\limits_{\sigma\in\Sigma}\rmD_{\sigma}$ is the exceptional divisor in $\rmX^{\mathrm{Bl}}$ for this blow-up. Note the smooth scheme ${\rmX}^{\bullet}_{s}$ is also the blow-up of $\rmX_{s}$ along its singular locus $\Sigma$.
\item Each $\rmD_{\sigma}$ is the hypersurface of $\PP^{n+1}_{s}=\mathrm{Proj}\phantom{.}k(s)[\rmX_{1},\cdots, \rmX_{n}, \rmT]$ defined by $\overline{\rmQ}_{\sigma}-\overline{u}_{\sigma}\rmT^{2}=0$ for an ordinary quadratic form $\rmQ_{\sigma}$ over $\rmR$ and some unit ${u}_{\sigma}$. Here $\overline{(\cdot)}$ means the reductions of $(\cdot)$ modulo the maximal ideal of $\rmR$ and each $\rmD_{\sigma}$ is a smooth quadric.
\item Let $\rmC_{\sigma}$ be the hyperplane of $\rmD_{\sigma}$ defined by $\rmT=0$ which is also a  smooth quadric. We have ${\rmX}^{\bullet}_{s}\cap \rmD_{\sigma}=\rmC_{\sigma}$. Hence $\sum\limits_{\sigma\in\Sigma}\rmC_{\sigma}$ is the exceptional divisor in ${\rmX}^{\bullet}_{s}$ by blowing up $\rmX_{s}$ along its singular locus $\Sigma$.
\end{itemize} 

\begin{lemma}\label{quadric-class}
The cohomology group of the quadric $\rmC_{\sigma}$ is given by the formulas below.
\begin{enumerate}
\item If $n-1=2r-1$ is odd, then we have 
\begin{equation*}
\rmH^{i}(\rmC_{\sigma}, \Lambda(i/2))=\begin{cases}\Lambda\eta^{i/2} & \text{if $i$ even, $0\leq i\leq 2(n-1)$;}\\ 0 & \text{otherwise.}\\ \end{cases}
\end{equation*}
\item If $n-1=2r-2$ is even, then we have 
\begin{equation*}
\rmH^{i}(\rmC_{\sigma}, \Lambda(i/2))=\begin{cases}\Lambda\eta^{i/2} & \text{if $i$ even, $0\leq i\leq 2(n-1), i\neq n-1$;}\\   \Lambda\theta_{1}\oplus\Lambda\theta_{2} & \text{if $i=n-1$;}\\0 & \text{otherwise.}\\ \end{cases}
\end{equation*}
\end{enumerate}
Here $\eta\in \rmH^{2}(\rmC_{\sigma},\Lambda(1))$ is the class of the hyperplane section of $\rmC_{\sigma}$ and $\theta_{1}$ and $\theta_{2}$ are the cycle classes of the two generatrices of $\rmC_{\sigma}$ such that $\theta_{1}+\theta_{2}=\eta^{(n-1)/2}$ holds.
\end{lemma}
\begin{proof}
The proof of this result is contained in \cite[Expose XII Theorem 3.3]{SGA7}. The last claim is given in \cite[Theorem 3.3 (iii)(c)]{SGA7}.
\end{proof}

From here on, we assume that $n$ is odd. We are particularly interested in the middle degree cohomology $\rmH^{n-1}(\rmC_{\sigma}, \Lambda(r-1))$ and we make the following definitions: 
\begin{itemize}
\item we define the {\em primitive part} $\rmH^{n-1}(\rmC_{\sigma}, \Lambda(r-1))^{\bullet}$ of $\rmH^{n-1}(\rmC_{\sigma}, \Lambda(r-1))$ by
\begin{equation*}
\rmH^{n-1}(\rmC_{\sigma}, \Lambda(r-1))^{\bullet}:=\langle\eta^{(n-1)/2}\rangle^{\perp} 
\end{equation*}
where the orthogonal complement is taken with respect to the Poincare duality on $\rmC_{\sigma}$;
\item  we define the {\em primitive quotient} $\rmH^{n-1}(\rmC_{\sigma}, \Lambda(r-1))_{\bullet}$ of $\rmH^{n-1}(\rmC_{\sigma}, \Lambda(r-1))$ by
\begin{equation*}
\rmH^{n-1}(\rmC_{\sigma}, \Lambda(r-1))_{\bullet}:=\rmH^{n-1}(\rmC_{\sigma}, \Lambda(r-1))/\langle\eta^{(n-1)/2}\rangle.
\end{equation*}
\end{itemize}

\begin{lemma} \label{intersection}
We have the following descriptions of the space of vanishing cycles and the space of co-vanishing cycles.
\begin{enumerate}
\item The space of vanishing cycles $\rmR^{n}\Phi_{\sigma}(\Lambda)(r)$ at $\sigma$ can be identified with the primitive part  
\begin{equation*}
\rmH^{n-1}(\rmC_{\sigma}, \Lambda(r-1))^{\bullet} 
\end{equation*}
of $\rmH^{n-1}(\rmC_{\sigma}, \Lambda(r-1))$.
\item The space of co-vanishing cycles $\rmH^{n}_{\{\sigma\}}(\rmX_{\overline{s}}, \rmR\Psi(\Lambda)(r-1))$ at $\sigma$ can be identified with the primitive quotient 
\begin{equation*}
\rmH^{n-1}(\rmC_{\sigma}, \Lambda(r-1))_{\bullet} 
\end{equation*}
of $\rmH^{n-1}(\rmC_{\sigma}, \Lambda(r-1))$.
\item We can choose a generator $\eta^{\bullet}_{\sigma}$ of the space $\rmH^{n-1}(\rmC_{\sigma}, \Lambda(r-1))^{\bullet}$  such that 
\begin{equation*}
\mathrm{N}_{\sigma}\eta^{\bullet}_{\sigma}=(-1)^{r+1}\eta^{\sigma}_{\bullet} 
\end{equation*}
where $\eta^{\sigma}_{\bullet}$ is the dual base of   $\eta^{\bullet}_{\sigma}$ in the primitive quotient $\rmH^{n-1}(\rmC_{\sigma}, \Lambda(r-1))_{\bullet} $.
\end{enumerate}
\end{lemma}
\begin{proof}
For $(1)$ and $(2)$, see \cite[2.2.5 F Expose XV]{SGA7} or \cite[Remarks 2.7]{Illusie-van}.  For part $(3)$, we choose $\eta^{\bullet}_{\sigma}$ to be $\theta_{1}-\theta_{2}$ where $\theta_{1}$ and $\theta_{2}$ are the cycle classes of the generatrices of distinct types as in \cite[ Expose XII 3.5]{SGA7} and \cite[3.2]{Illusie-van}. Then the claim in part $(3)$ follows from the usual Picard--Lefschetz formula in \ref{pic-lef-form}. 
\end{proof}

The above discussions allows us to describe the maps $\alpha_{\rmc}$ and $\beta$ as certain Gysin or restriction maps as in the proposition below.

\begin{proposition}\label{alpha-beta}
We have the following descriptions of the maps $\alpha_{\rmc}$ and $\beta$.
\begin{enumerate}
\item The space of vanishing cycles  $\bigoplus\limits_{\sigma\in\Sigma}\rmR^{n}\Phi_{\sigma}(\Lambda)(r)$ and the map $\alpha_{\rmc}$ inserts in the following commutative diagram 
\begin{equation*}
\begin{tikzcd}
\bigoplus\limits_{\sigma\in\Sigma}\rmH^{n+1}(\rmC_{\sigma}, \Lambda(r)) \arrow[r, "\sim"]                 &       \bigoplus\limits_{\sigma\in\Sigma}\rmH^{n+1}(\rmC_{\sigma}, \Lambda(r))  \\
\bigoplus\limits_{\sigma\in\Sigma}\rmH^{n-1}(\rmC_{\sigma}, \Lambda(r-1)) \arrow[r] \arrow[u, "\cup \eta"'] & \rmH^{n+1}_{\rmc}({\rmX}^{\bullet}_{\overline{s}}, \Lambda(r)) \arrow[u] \\
\bigoplus\limits_{\sigma\in\Sigma}\rmR^{n}\Phi_{\sigma}(\Lambda)(r) \arrow[r, "\alpha_{\rmc}(r)"] \arrow[u]       & \rmH^{n+1}_{\rmc}(\rmX^{\square}_{\overline{s}}, \Lambda(r)) \arrow[u]
\end{tikzcd}
\end{equation*}
where the left and right columns are exact.
\item The space of co-vanishing cycles   $\bigoplus\limits_{\sigma\in\Sigma}\rmH^{n}_{\{\sigma\}}(\rmX_{\overline{s}}, \rmR\Psi(\Lambda)(r-1))$ and the map $\beta$ inserts in the following commutative diagram 
\begin{equation*}
\begin{tikzcd}
\bigoplus\limits_{\sigma\in\Sigma}\rmH^{n}_{\{\sigma\}}(\rmX_{\overline{s}}, \rmR\Psi(\Lambda)(r-1))              &\mathrm{H}^{n-1}(\rmX^{\square}_{\overline{s}}, \Lambda(r-1))\arrow[l, "\beta(r-1)"]           \\
\bigoplus\limits_{\sigma\in\Sigma}\rmH^{n-1}(\rmC_{\sigma}, \Lambda(r-1)) \arrow[u] &   \mathrm{H}^{n-1}({\rmX}^{\bullet}_{\overline{s}}, \Lambda(r-1)) \arrow[l] \arrow[u] \\
\bigoplus\limits_{\sigma\in\Sigma}\rmH^{n-3}(\rmC_{\sigma}, \Lambda(r-2)) \arrow[u, "\cup\eta"]       &  \bigoplus\limits_{\sigma\in\Sigma}\rmH^{n-3}(\rmC_{\sigma}, \Lambda(r-2))\arrow[l,"\sim"] \arrow[u]
\end{tikzcd}
\end{equation*}
where the left and right columns are exact.
\end{enumerate}
\end{proposition}
\begin{proof}
The second part is the dual statement of the first part. Therefore we will only prove the first part.  For the first part,  we have $\rmR^{n}\Phi_{\sigma}(\Lambda)(r) =\rmH^{n-1}(\rmC_{\sigma}, \Lambda(r-1))^{\bullet}$ and 
\begin{equation*}
\rmH^{n-1}(\rmC_{\sigma}, \Lambda(r-1))^{\bullet}=\Ker(\rmH^{n-1}(\rmC_{\sigma}, \Lambda(r-1))\xrightarrow{\cup \eta} \rmH^{n+1}(\rmC_{\sigma}, \Lambda(r)))
\end{equation*}
by Lemma \ref{intersection-number}.
The left column is therefore exact. The right column is the excision exact sequence for the closed immersion of the exceptional divisor $\rmC_{\sigma}\hookrightarrow {\rmX}^{\bullet}_{\overline{s}}$. The commutativity of the diagram follows by definition.
\end{proof}

\subsection{Weight-monodromy conjecture for interior cohomology}
We define the interior cohomology group $\rmH^{n}_{!}(\rmX_{\overline{s}}, \rmR\Psi(\Lambda)(r))$ by 
\begin{equation*}
\rmH^{n}_{!}(\rmX_{\overline{s}}, \rmR\Psi(\Lambda)(r))=\mathrm{Im}(\rmH^{n}_{\rmc}(\rmX_{\overline{s}}, \rmR\Psi(\Lambda)(r))\rightarrow \rmH^{n}(\rmX_{\overline{s}}, \rmR\Psi(\Lambda)(r))).
\end{equation*}
Note that the monodromy filtration 
\begin{equation*}
0\subset_{\mathrm{Gr}_{-1, !}(r)}\mathrm{F}_{-1}\mathrm{H}^{n}_{!}(\rmX_{\overline{s}}, \rmR\Psi(\Lambda)(r)) \subset_{\mathrm{Gr}_{0, !}(r)} \mathrm{F}_{0}\mathrm{H}^{n}_{!}(\rmX_{\overline{s}}, \rmR\Psi(\Lambda)(r)) \subset_{\mathrm{Gr}_{1, !}(r)} \mathrm{F}_{1}\mathrm{H}^{n}_{!}(\rmX_{\overline{s}}, \rmR\Psi(\Lambda)(r))
\end{equation*}
of  $\mathrm{H}^{n}_{!}(\rmX_{\overline{s}}, \rmR\Psi(\Lambda)(r))$ is then determined by the Picard-Lefschetz formula for $\mathrm{H}^{n}_{(\rmc)}(\rmX_{\overline{s}}, \rmR\Psi(\Lambda)(r))$ and the successive quotient of this filtration is given by
\begin{equation*}
\begin{aligned}
&\mathrm{Gr}_{-1, !}(r)=\mathrm{Gr}_{-1}\mathrm{H}^{n}_{!}(\rmX_{\overline{s}}, \rmR\Psi(\Lambda)(r))= \Coker(\beta_{!}(r): \mathrm{IH}^{n-1}_{!}(\rmX_{\overline{s}}, \Lambda(r))\rightarrow  \bigoplus\limits_{\sigma\in\Sigma}\rmH^{n}_{\{\sigma\}}(\rmX_{\overline{s}}, \rmR\Psi(\Lambda)(r)));\\
&\mathrm{Gr}_{0, !}(r)=\mathrm{Gr}_{0}\mathrm{H}^{n}_{!}(\rmX_{\overline{s}}, \rmR\Psi(\Lambda)(r))= \mathrm{IH}^{n}_{!}(\rmX_{\overline{s}}, \Lambda(r));\\
&\mathrm{Gr}_{1, !}(r)=\mathrm{Gr}_{1}\mathrm{H}^{n}_{!}(\rmX_{\overline{s}}, \rmR\Psi(\Lambda)(r))=\Ker(\alpha_{!}(r): \bigoplus\limits_{\sigma\in\Sigma}\rmR^{n}\Phi_{\sigma}(\Lambda)(r) \rightarrow \mathrm{IH}^{n+1}_{!}(\rmX_{\overline{s}}, \Lambda(r)));\\
\end{aligned}
\end{equation*}
where $\mathrm{IH}^{\ast}_{!}(\rmX_{\overline{s}}, \Lambda(r))=\mathrm{Im}(\mathrm{IH}^{\ast}_{\rmc}(\rmX_{\overline{s}}, \Lambda(r))\rightarrow \mathrm{IH}^{n}(\rmX_{\overline{s}}, \Lambda(r))$ and the maps $\alpha_{!}(r)$ and $\beta_{!}(r)$ are the composite maps given by
\begin{equation*}
\begin{aligned}
&\bigoplus\limits_{\sigma\in\Sigma}\rmR^{n}\Phi_{\sigma}(\Lambda)(r) \xrightarrow{\alpha_{\rmc}(r)} \mathrm{IH}^{n+1}_{\rmc}(\rmX_{\overline{s}}, \Lambda(r))\rightarrow \mathrm{IH}^{n+1}_{!}(\rmX_{\overline{s}}, \Lambda(r))\\
&\mathrm{IH}^{n-1}_{!}(\rmX_{\overline{s}}, \Lambda(r))\rightarrow \mathrm{IH}^{n-1}(\rmX_{\overline{s}}, \Lambda(r)) \xrightarrow{\beta(r)}\bigoplus\limits_{\sigma\in\Sigma}\rmH^{n}_{\{\sigma\}}(\rmX_{\overline{s}}, \rmR\Psi(\Lambda)(r)).\\
\end{aligned}
\end{equation*}

Let now $\Lambda=\overline{\QQ}_{l}$. The following corollary can be seen as an analogue of the weight-monodromy conjecture for the interior cohomology $\rmH^{n}_{!}(\rmX_{\overline{s}}, \rmR\Psi(\Lambda)(r))$. A special case of this is obtained by Van Hoften in \cite[Theorem 1]{VH19}.

\begin{corollary}\label{WM}
The graded pieces of the monodromy filtration of $\mathrm{H}^{n}_{!}(\rmX_{\overline{s}}, \rmR\Psi(\Lambda)(r))$ are pure of weight $-2, -1$ and $0$. In particular the weight-monodromy conjecture holds for $\mathrm{H}^{n}_{!}(\rmX_{\overline{s}}, \rmR\Psi(\Lambda)(r))$.
\end{corollary} 
\begin{proof}
By Weil's conjecture and Proposition \ref{alpha-beta}, it follows immediately that $\mathrm{Gr}_{-1,!}$ and $\mathrm{Gr}_{1,!}$ are pure of weights $-2$ and $0$. It remains to study the graded piece 
\begin{equation*}
\mathrm{Gr}_{0,!}(r)=\mathrm{IH}^{n}_{!}(\rmX_{\overline{s}}, \Lambda(r)).
\end{equation*}
By Gabber's purity theorem \cite{Gab}, \cite[Corollaire 5.3.2]{BBD} and \cite[Stabilites 5.1.14]{BBD}, the weights of $\mathrm{IH}^{n}(\rmX_{\overline{s}}, \Lambda(r))$ resp.  $\mathrm{IH}^{n}_{\rmc}(\rmX_{\overline{s}}, \Lambda(r))$ are $\geq -1$ resp. $\leq -1$. It then follows that the weight of $\mathrm{Gr}_{0,!}(r)$ has to be $-1$.
\end{proof}

\subsection{Nearby cycles  on certain Shimura varieties} Suppose $f:\rmX\rightarrow \rmS$ is a proper morphism.  Then the proper base change theorem implies that we have an isomorphism 
\begin{equation*}
\rmH^{i}(\rmX_{\overline{\eta}}, \Lambda)\cong \rmH^{i}(\rmX_{\overline{s}}, \rmR\Psi(\Lambda)).
\end{equation*}
Note that if $f: \rmX\rightarrow \rmS$ is not proper, then we don't always have the above isomorphism. However, in the setting of Shimura varieties with good compactification in sense of \cite{LS18a} and \cite{LS18b}, we do have such isomorphism. 

In this article we are only concerned with the Shimura variety $\rmX_{\Pa}(\rmB)$ and this Shimura variety belongs to the case 
\begin{center}
(Nm) {A flat integral model defined by taking normalization of a characteristic $0$ PEL type moduli problem over a product of  good reduction integral models of smooth PEL type moduli problem} 
\end{center}
in the classification of \cite{LS18a} and \cite{LS18b}. The following theorem summarizes the results we need. 
\begin{theorem}[{\cite[Corollary 4.6]{LS18b}}]
Suppose $\rmX$ is a Shimura variety that is in the case of (Nm) and let $\mathbb{V}$ be an automorphic \'etale sheaf defined as in \cite[\S 3]{LS18b}. Then the canonical adjunction morphisms
\begin{equation*}
\rmH^{i}(\rmX_{\overline{\eta}}, \mathbb{V})\rightarrow \rmH^{i}(\rmX_{\overline{s}}, \rmR\Psi(\mathbb{V})).
\end{equation*}
and
\begin{equation*}
\rmH^{i}_{\rmc}(\rmX_{\overline{s}}, \rmR\Psi(\mathbb{V}))\rightarrow \rmH^{i}_{\rmc}(\rmX_{\overline{\eta}}, \mathbb{V})
\end{equation*}
are isomorphisms for all $i$. 
\end{theorem}
Using this theorem, we will freely identify the cohomology of the generic fiber of the quaternionic unitary Shimura variety $\rmX_{\Pa}(\rmB)$ and the nearby cycle cohomology of the special fiber $\overline{\rmX}_{\Pa}(\rmB)$ for the rest of this article.

\section{The level raising matrix}
Let $\pi$ be a cuspidal automorphic representation of general type with weight $(3, 3)$ and trivial central character.
Let $\Sigma$ be the finite set of non-archemedean places of $\QQ$ containing the prime $q$ but not $p$ where $pq$ is the discriminant of the quaternion algebra $\rmB$.  We let 
\begin{equation*}
\TT^{\Sigma}=\bigotimes_{v\not\in \Sigma}\ZZ[\bfG(\ZZ_{v})\backslash\bfG(\QQ_{v})/\bfG(\ZZ_{v})]
\end{equation*}
be the abstract Hecke algebra unramified away from $\Sigma$. Suppose that the datum $(\pi, \Sigma)$ satisfy the following additional properties:
\begin{itemize}
\item the representation $\pi$ is unramified away from $\Sigma$. The Hecke parameter of $\pi$ at $p$ is given by $[\alpha_{p}, \beta_{p}, p^{3}\beta^{-1}_{p}, p^{3}\alpha^{-1}_{p}]$;
\item we have the homomorphism ${\phi}_{\pi, \lambda}: \TT^{\Sigma}\rightarrow \calO_{\lambda}$ as in Construction \ref{Hecke-Algebra} (2) which gives the maximal ideal 
\begin{equation*}
\fracm=\TT^{\Sigma\cup\{p\}}\cap \ker(\TT^{\Sigma}\xrightarrow{\phi_{\pi,\lambda}}\calO_{\lambda}\rightarrow\calO_{\lambda}/\lambda) 
\end{equation*}
as in Construction \ref{Hecke-Algebra} (3).
\end{itemize}
We fix the datum $(\pi, \Sigma, \fracm)$ in this section. We are eventually interested in understanding the singular quotient 
\begin{equation*}
\rmH^{1}_{\mathrm{sing}}(\QQ_{p^{2}},  \rmH^{3}_{\mathrm{c}}(\overline{\rmX}_{\Pa}(\rmB), \rmR\Psi(\calO_{\lambda}(2)))_{\fracm}).
\end{equation*}
As a first step, we have the following proposition.

\begin{proposition}\label{mondromy-exact-seq}
Suppose that $p(p^{2}-1)$ is invertible in $\calO_{\lambda}/\lambda$. We define 
\begin{equation*}
\mathsf{H}_{\fracm}=\rmH^{3}_{\mathrm{c}}(\overline{\rmX}_{\Pa}(\rmB), \rmR\Psi(\calO_{\lambda}(2)))_{\fracm}. 
\end{equation*}
There is then an exact sequence 
\begin{equation*}
\begin{aligned}
0&\rightarrow \mathrm{F}_{-1}\mathrm{H}^{1}(\mathrm{I}_{\QQ_{p^{2}}}, \mathsf{H}_{\fracm}) \rightarrow \rmH^{1}_{\mathrm{sing}}(\QQ_{p^{2}}, \mathsf{H}_{\fracm}) \rightarrow (\frac{\rmH^{3}_{\mathrm{c}}(\overline{\rmX}_{\Pa}(\rmB), \calO_{\lambda}(1))_{\fracm}}{\bigoplus\limits_{\sigma\in \rmZ_{\{1\}}(\overline{\rmB})}\rmH^{3}_{\{\sigma\}}(\overline{\rmX}_{\Pa}(\rmB), \rmR\Psi(\calO_{\lambda})(1))_{\fracm}})^{\rmG_{\FF_{p^{2}}}}.\\
\end{aligned}
\end{equation*}
\end{proposition}
\begin{proof}
We consider the two tautological exact sequences
 \begin{equation}\label{two-taut}
\begin{aligned}
0&\rightarrow \mathrm{F}_{-1}\mathrm{H}^{1}(\mathrm{I}_{\QQ_{p^{2}}}, \mathsf{H}_{\fracm})\rightarrow \rmH^{1}(\rmI_{\QQ_{p^{2}}}, \mathsf{H}_{\fracm}) \rightarrow \frac{\rmH^{1}(\rmI_{\QQ_{p^{2}}}, \sfH_{\fracm})}{\mathrm{F}_{-1}\mathrm{H}^{1}(\mathrm{I}_{\QQ_{p^{2}}}, \sfH_{\fracm})}\rightarrow 0\\
&0\rightarrow \rmF_{0}\sfH_{\fracm}(-1)/\rmF_{-1}\sfH_{\fracm}(-1)\rightarrow \frac{\rmH^{1}(\rmI_{\QQ_{p^{2}}}, \mathsf{H}_{\fracm})}{\mathrm{F}_{-1}\mathrm{H}^{1}(\mathrm{I}_{\QQ_{p^{2}}}, \mathsf{H}_{\fracm})}\rightarrow \mathrm{Gr}_{1}\sfH_{\fracm}(-1)\rightarrow 0.
\end{aligned}
\end{equation}
The latter exact sequence rewrites to
\begin{equation}\label{F-1andH}
0\rightarrow\frac{\rmH^{3}_{\mathrm{c}}(\overline{\rmX}_{\Pa}(\rmB), \calO_{\lambda}(1))_{\fracm}}{\bigoplus\limits_{\sigma\in \rmZ_{\{1\}}(\overline{\rmB})}\rmH^{3}_{\{\sigma\}}(\overline{\rmX}_{\Pa}(\rmB), \rmR\Psi(\calO_{\lambda})(1))_{\fracm}}\rightarrow\frac{\rmH^{1}(\rmI_{\QQ_{p^{2}}}, \mathsf{H}_{\fracm})}{\mathrm{F}_{-1}\mathrm{H}^{1}(\mathrm{I}_{\QQ_{p^{2}}}, \mathsf{H}_{\fracm})}\rightarrow \Ker(\alpha_{\rmc}(1))_{\fracm}\rightarrow 0
\end{equation}
by \eqref{mono-fil} and where the last term is the kernel of 
\begin{equation*}
\bigoplus\limits_{\sigma\in\rmZ_{\{1\}}(\overline{\rmB})}\rmR^{3}\Phi_{\sigma}(\calO_{\lambda})_{\fracm}(1)\xrightarrow{\alpha_{\rmc}(1)} \rmH^{4}_{\mathrm{c}}(\overline{\rmX}_{\Pa}, \calO_{\lambda}(1))_{\fracm}.
\end{equation*}
Taking invariant under $\rmG_{\FF_{p^{2}}}$ for the short exact sequence \ref{F-1andH}, we have an isomorphism  
\begin{equation*}
(\frac{\rmH^{3}_{\mathrm{c}}(\overline{\rmX}_{\Pa}(\rmB), \calO_{\lambda}(1))_{\fracm}}{\bigoplus\limits_{\sigma\in \rmZ_{\{1\}}(\overline{\rmB})}\rmH^{3}_{\{\sigma\}}(\overline{\rmX}_{\Pa}(\rmB), \rmR\Psi(\calO_{\lambda})(1))_{\fracm}})^{\rmG_{\FF_{p^{2}}}}\cong (\frac{\rmH^{1}(\rmI_{\QQ_{p^{2}}}, \sfH_{\fracm})}{\mathrm{F}_{-1}\mathrm{H}^{1}(\mathrm{I}_{\QQ_{p^{2}}}, \sfH_{\fracm})})^{\rmG_{\FF_{p^{2}}}}.
\end{equation*}
Taking invariant under $\rmG_{\FF_{p^{2}}}$ for the first short exact sequence of \eqref{two-taut}, we immediately obtain the desired exact sequence
\begin{equation*}
\begin{aligned}
0&\rightarrow \mathrm{F}_{-1}\mathrm{H}^{1}(\mathrm{I}_{\QQ_{p^{2}}}, \mathsf{H}_{\fracm})\rightarrow \rmH^{1}_{\mathrm{sing}}(\QQ_{p^{2}}, \mathsf{H}_{\fracm}) \rightarrow (\frac{\rmH^{3}_{\mathrm{c}}(\overline{\rmX}_{\Pa}(\rmB), \calO_{\lambda}(1))_{\fracm}}{\bigoplus\limits_{\sigma\in \rmZ_{\{1\}}(\overline{\rmB})}\rmH^{3}_{\{\sigma\}}(\overline{\rmX}_{\Pa}(\rmB), \rmR\Psi(\calO_{\lambda})(1))_{\fracm}})^{\rmG_{\FF_{p^{2}}}}.\\
\end{aligned}
\end{equation*}
\end{proof}
\subsection{The level raising matrix}
For constructing the level raising matrix, we introduce some notations. Recall that $\mathrm{DL}^{\square}(\Lambda_{i})$ is the complement in $\mathrm{DL}(\Lambda_{i})$ of its set of superspecial points. Note that each irreducible component of  $\overline{\rmX}^{\square\mathrm{ss}}_{\Pa}(\rmB)$ is isomorphic to $\mathrm{DL}^{\square}(\Lambda_{i})$ for some $i\in\{0,2\}$ and the irreducible components are parametrized by the Shimura sets $\rmZ_{\{0\}}(\overline{\rmB})$ and $\rmZ_{\{2\}}(\overline{\rmB})$. Let $\overline{\rmX}^{\square\mathrm{ss.n}}_{\Pa, \{0,2\}}(\rmB)$ be the normalization of the supersingular locus $\overline{\rmX}^{\square\mathrm{ss}}_{\Pa}(\rmB)$ and hence we have 
 \begin{equation*}
 \overline{\rmX}^{\square\mathrm{ss.n}}_{\Pa, \{0,2\}}(\rmB)=\overline{\rmX}^{\square\mathrm{ss.n}}_{\Pa, \{0\}}(\rmB)\sqcup \overline{\rmX}^{\square\mathrm{ss.n}}_{\Pa, \{2\}}(\rmB)  
 \end{equation*}
where $ \overline{\rmX}^{\square\mathrm{ss.n}}_{\Pa, \{i\}}(\rmB)$ is the disjoint union of $\mathrm{DL}^{\square}(\Lambda_{i})$ parametrized by $\rmZ_{\{i\}}(\overline{\rmB})$ for $i\in\{0,2\}$.  

\begin{construction}\label{ss-cycle-class}
We have two natural Gysin morphisms
\begin{equation}\label{Tate-inj}
\begin{aligned}
&\mathrm{inc}^{\{0\}}_{!}: \calO_{\lambda}[\rmZ_{\{0\}}(\overline{\rmB})]\cong\rmH^{0}(\overline{\rmX}^{\square\mathrm{ss.n}}_{\Pa, \{0\}}(\rmB), \calO_{\lambda})\rightarrow  \rmH^{2}(\overline{\rmX}^{\square}_{\Pa}(\rmB), \calO_{\lambda}(1))\\
&\mathrm{inc}^{\{2\}}_{!}: \calO_{\lambda}[\rmZ_{\{2\}}(\overline{\rmB})]\cong \rmH^{0}(\overline{\rmX}^{\square\mathrm{ss.n}}_{\Pa, \{2\}}(\rmB), \calO_{\lambda}) \rightarrow  \rmH^{2}(\overline{\rmX}^{\square}_{\Pa}(\rmB), \calO_{\lambda}(1))\\
\end{aligned}
\end{equation}
and the two natural restriction morphisms
\begin{equation}\label{Tate-surj}
\begin{aligned}
&\mathrm{inc}^{\ast}_{\{0\}}: \rmH^{4}_{\rmc}(\overline{\rmX}^{\square}_{\Pa}(\rmB), \calO_{\lambda}(2))\rightarrow  \rmH^{4}_{\rmc}(\overline{\rmX}^{\square\mathrm{ss.n}}_{\Pa, \{0\}}(\rmB), \calO_{\lambda}(2)) \cong \calO_{\lambda}[\rmZ_{\{0\}}(\overline{\rmB})] \\
&\mathrm{inc}^{\ast}_{\{2\}}: \rmH^{4}_{\rmc}(\overline{\rmX}^{\square}_{\Pa}(\rmB), \calO_{\lambda}(2))\rightarrow  \rmH^{4}_{\rmc}(\overline{\rmX}^{\square\mathrm{ss.n}}_{\Pa, \{2\}}(\rmB), \calO_{\lambda}(2))\cong \calO_{\lambda}[\rmZ_{\{2\}}(\overline{\rmB})].\\
\end{aligned}
\end{equation}
\end{construction}
The above construction fits into the following diagram which sets up the level raising matrix.

\begin{equation}\label{level-raise-diagram}
\begin{tikzcd}
 \calO_{\lambda}[\rmZ_{\{0\}}(\overline{\rmB})] \arrow[rd, "\mathrm{inc}^{\{0\}}_{!}"'] &                                    &  \calO_{\lambda}[\rmZ_{\{2\}}(\overline{\rmB})]\arrow[ld, "\mathrm{inc}^{\{2\}}_{!}"] \\
                   & \rmH^{2}(\overline{\rmX}^{\square}_{\Pa}(\rmB), \calO_{\lambda}(1)) \arrow[d, "\beta(1)"]                   &                   \\
                   & \bigoplus\limits_{\sigma\in \rmZ_{\{1\}}(\overline{\rmB})}\rmH^{3}_{\{\sigma\}}(\overline{\rmX}_{\Pa}(\rmB), \rmR\Psi(\calO_{\lambda})(1)) \arrow[d, "\mathrm{N}^{-1}_{\Sigma}"]                   &                   \\
                   & \bigoplus\limits_{\sigma\in \rmZ_{\{1\}}(\overline{\rmB})}\rmR^{3}\Phi_{\sigma}(\calO_{\lambda})(2)) \arrow[d, "\alpha_{\rmc}(2)"]                   &                   \\
                   &  \rmH^{4}_{\rmc}(\overline{\rmX}^{\square}_{\Pa}(\rmB), \calO_{\lambda}(2)) \arrow[ld, "\mathrm{inc}^{\ast}_{\{0\}}"'] \arrow[rd, "\mathrm{inc}^{\ast}_{\{2\}}"] &                   \\
 \calO_{\lambda}[\rmZ_{\{0\}}(\overline{\rmB})]            &                                    &   \calO_{\lambda}[\rmZ_{\{2\}}(\overline{\rmB})]            
\end{tikzcd}
\end{equation}

\begin{construction}
We use the above diagram to definie the following four maps
\begin{equation*}
\begin{aligned}
&\calT_{\mathrm{lr},\{00\}}= \mathrm{inc}^{\ast}_{\{0\}} \circ\alpha_{\rmc}(2)\circ\rmN^{-1}_{\Sigma}\circ\beta(1)\circ\mathrm{inc}^{\{0\}}_{!}\\
&\calT_{\mathrm{lr},\{02\}}= \mathrm{inc}^{\ast}_{\{2\}} \circ\alpha_{\rmc}(2)\circ\rmN^{-1}_{\Sigma}\circ\beta(1)\circ\mathrm{inc}^{\{0\}}_{!}\\
&\calT_{\mathrm{lr},\{20\}}= \mathrm{inc}^{\ast}_{\{0\}} \circ\alpha_{\rmc}(2)\circ\rmN^{-1}_{\Sigma}\circ\beta(1)\circ\mathrm{inc}^{\{2\}}_{!}\\
&\calT_{\mathrm{lr},\{22\}}= \mathrm{inc}^{\ast}_{\{2\}} \circ\alpha_{\rmc}(2)\circ\rmN^{-1}_{\Sigma}\circ\beta(1)\circ\mathrm{inc}^{\{2\}}_{!}.\\
\end{aligned}
\end{equation*}
We will call the resulting matrix
\begin{equation*}
\calT_{\mathrm{lr}}=\begin{pmatrix} &\calT_{\mathrm{lr},\{00\}} & \calT_{\mathrm{lr}, \{02\}}\\ &\calT_{\mathrm{lr}, \{20\}}& \calT_{\mathrm{lr}, \{22\}}\\ \end{pmatrix}
\end{equation*}
\end{construction}
the {\em level raising matrix}. We can localize the above diagram at the maximal ideal $\fracm$ and write the resulting matrix as $\calT_{\mathrm{lr},\fracm}$, however all the matrix entries will still be denoted by same symbols without referring to $\fracm$. If we need to consider this matrix modulo $\fracm$, then we will denote it by $\calT_{\mathrm{lr}/\fracm}$. Again the entries of this matrix will be denoted the same notation without referring to $\fracm$ again.

\subsection{A Hecke operator identity} We will identify the entries of the level raising matrix with certain Hecke operators, we first prove some auxiliary results about the spherical Hecke algebra of $\bfG=\GSp_{4}$ over $\QQ_{p}$. The computation in this subsection is completely done in the local set up and hence in this subsection $\bfG$ will be understood as a group over $\QQ_{p}$. Let $\bfG(\ZZ_{p})=\GSp_{4}(\ZZ_{p})$ be the hyperspecial subgroup of $\bfG(\QQ_{p})$. Consider the spherical Hecke algebra $\TT_{p}=\ZZ[\bfG(\ZZ_{p})\backslash\bfG(\QQ_{p})/\bfG(\ZZ_{p})]$.  We have the classical Satake transform
\begin{equation*}
\calS:\TT_{p}\rightarrow\ZZ[p^{-1}][\mathrm{X}^{\ast}(\hat{\rmT})]^{\rmW_{\bfG}}
\end{equation*}
where $\rmW_{\bfG}$ is the absolute Weyl group of $\bfG$.
We fix a datum $\bfT\subset \bfB \subset \bfG$ of the diagonal torus contained in the upper triangular Borel subgroup as in the global case. This determines a positive Weyl chamber $\rmP^{+}$ in $\mathrm{X}_{\ast}(\rmT)\cong\mathrm{X}^{\ast}(\hat{\rmT})$. An element $\nu\in \mathrm{X}^{\ast}(\hat{\rmT})$ determines an irreducible representation $V_{\nu}$ of the dual group $\hat{\rmG}=\GSpin(5)$ whose character is denoted by $\chi_{\nu}\in \mathcal{R}(\hat{\rmG})=\ZZ[\rmX^{\ast}(\hat{\rmT})]^{\rmW_{\bfG}}$. 

It is well-known that $\TT_{p}$ is isomorphic to $\mathbb{Z}[\rmT^{\pm}_{p,0}, \rmT_{p,1}, \rmT_{p,2}]$ where 
\begin{equation*}
\begin{split}
&\rmT_{p,0}=c_{\nu_{0}}=\mathbf{char}(\bfG(\ZZ_{p})\begin{pmatrix}p&&&\\&p&&\\&&p&\\&&&p\\\end{pmatrix}\bfG(\ZZ_{p})) \text{ for }\nu_{0}=(1, 1, 1, 1);\\
&\rmT_{p,2}=c_{\nu_{2}}=\mathbf{char}(\bfG(\ZZ_{p})\begin{pmatrix}p&&&\\& p&&\\&&1&\\&&&1\\\end{pmatrix}\bfG(\ZZ_{p})) \text{ for }\nu_{2}=(1, 1, 0, 0);\\
&\rmT_{p,1}=c_{\nu_{1}}=\mathbf{char}(\bfG(\ZZ_{p})\begin{pmatrix}p^{2} &&&\\&p&&\\&&p&\\&&& 1\\\end{pmatrix}\bfG(\ZZ_{p})) \text{ for } \nu_{1}=(2, 1, 1, 0).\\
\end{split}
\end{equation*}

The following spherical Hecke operator will appear naturally in later computations
\begin{equation*}
\rmT_{p^{2},2}=c_{2\nu_{2}}=\mathbf{char}(\bfG(\ZZ_{p})\begin{pmatrix}p^{2}&&&\\&p^{2}&&\\&&1&\\&&&1\\\end{pmatrix}\bfG(\ZZ_{p})) \text{ for } 2\nu_{2}=(2, 2, 0, 0).
\end{equation*}
We will also need the following two Hecke operators in $\ZZ[\rmK_{\{2\}}\backslash \bfG(\QQ_{p})/\rmK_{\{0\}}]$
\begin{equation*}
\begin{aligned}
&\rmT_{02}= \mathbf{char}(\rmK_{\{2\}}\begin{pmatrix}p&&&\\&p&&\\&&1&\\&&&1\\\end{pmatrix}\rmK_{\{0\}})\\
&\rmT_{20}= \mathbf{char}(\rmK_{\{0\}}\begin{pmatrix}p^{-1}&&&\\&p^{-1}&&\\&&1&\\&&&1\\\end{pmatrix}\rmK_{\{2\}}).\\
\end{aligned}
\end{equation*}

In the computations below, we will follow mostly the notations and conventions in \cite{Gross-Satake} and use results from there freely. We have the following useful formula \cite[3.12]{Gross-Satake}
\begin{equation*}
p^{\langle\nu, \rho\rangle}\chi_{\nu}=\calS(c_{\nu})+\sum_{\xi<\nu}d_{\nu}(\xi)\calS(c_{\xi}).
\end{equation*}
The integers $d_{\nu}(\xi)$ are given by $\rmP_{\xi, \nu}(p)$ for the Kazhdan-Lusztig polynomial $\rmP_{\xi, \nu}$ evaluated at $p$ and $2\rho$ is the sum of the positive roots. Since the weight $\nu_{2}$ is minuscule, we immediately have
\begin{equation*}
p^{\frac{3}{2}}\chi_{\nu_{2}}=\calS(c_{\nu_{2}})=\calS(\mathrm{T}_{p, 2}).
\end{equation*}
On the other hand for the non-minuscule weight $\nu_{1}$, we have
\begin{equation*}
p^{2}\chi_{\nu_{1}}=\calS(c_{\nu_{1}})+\calS(c_{0})=\calS(\mathrm{T}_{p, 1})+1.
\end{equation*}
The following identity is the main result of this subsection. It can perhaps be proved by a massive double coset computation whereas our proof uses crucially the Satake isomorphism to simplify the computation. This identity seems not to exist already in the literature and therefore we give full details which we think are of independent interest.
\begin{proposition}\label{Hecke-id}
The following identity holds in $\TT_{p}$
\begin{equation*}
\rmT_{p^{2}, 2}=\rmT^{2}_{p,2}-(p+1)\rmT_{p,1}-(p+1)(p^{2}+1).
\end{equation*}
\end{proposition}
\begin{proof}
The representation $V_{\nu_{1}}$ is the $5$-dimensional orthogonal representation and $V_{\nu_{2}}$ is the standard representation. Therefore $\wedge^{2} V_{\nu_{2}}=V_{\nu_{1}}\oplus \CC$. We also have $\wedge^{2}V_{\nu_{1}}=V_{2\nu_{2}}$ is the adjoint representation of dimension $10$. Thus we have
\begin{equation*}
\begin{aligned}
V^{\otimes 2}_{\nu_{2}}&=V_{2\nu_{2}}+\wedge^{2}V_{\nu_{2}}\\
&=V_{2\nu_{2}}+ V_{\nu_{1}}+\CC\\
\end{aligned}
\end{equation*}
Taking characters, we have $\chi^{2}_{\nu_{2}}=\chi_{2\nu_{2}}+\chi_{\nu_{1}}+1$ and hence $p^{3}\chi^{2}_{\nu_{2}}=p^{3}\chi_{2\nu_{2}}+p^{3}\chi_{\nu_{1}}+p^{3}$. Therefore
\begin{equation*}
\begin{aligned}
p^{3}\chi_{2\nu_{2}}&=p^{3}\chi^{2}_{\nu_{2}}-p^{3}\chi_{\nu_{1}}-p^{3}\\
&=\calS(\rmT^{2}_{p,2})-p\calS(\rmT_{p, 1})-(p+p^{3}).\\
\end{aligned}
\end{equation*}
On the other hand, we also have 
\begin{equation*}
\begin{aligned}
p^{3}\chi_{2\nu_{2}}&=\calS(c_{2\nu_{2}})+d_{2\nu_{2}}(\nu_{1})\calS(c_{\nu_{1}})+d_{2\nu_{2}}(0)\\
&=\calS(\rmT_{p^{2},2})+ d_{2\nu_{2}}(\nu_{1})\calS(\rmT_{p, 1})+d_{2\nu_{2}}(0).\\
\end{aligned}
\end{equation*} 
\begin{itemize}
\item $d_{2\nu_{2}}(\nu_{1})=1$; This follows from the fact that $\rmP_{\nu_{1}, 2\nu_{2}}$ has degree strictly less than $\langle2\nu_{2}-\nu_{1},\rho\rangle=1$ and the constant coefficient is $1$.
\item $d_{2\nu_{2}}(0)=1+p^{2}$; This follows from the fact that $2\nu_{2}$ is the highest weight of the adjoint representation, which implies that $\rmP_{0, 2\nu_{2}}(p)=\sum_{i=1}p^{m_{i}-1}$ where $m_{i}$ are the exponents of the Weyl group of type $\rmC_{2}$ which are $1$ and $3$, see \cite[4.6]{Gross-Satake}.
\end{itemize}
Equating the two expressions of $p^{3}\chi_{2\nu_{2}}$ gives the desired result.
\end{proof}

\subsection{Computing the level raising matrix} We now calculate the entries of the level raising  matrix $\calT_{\mathrm{lr}}$ term by term. The main result is the following proposition.

\begin{proposition}\label{lr-matrix}
The level raising matrix $\calT_{\mathrm{lr}}$ is given by
\begin{equation*}
\calT_{\mathrm{lr}}=\begin{pmatrix} &-(\rmT^{-1}_{p,0}\rmT_{p,1}+(p+1)(p^{2}+1)) & (p+1)\rmT_{02}\\ &(p+1)\rmT_{20}& -(\rmT^{-1}_{p,0}\rmT_{p,1}+(p+1)(p^{2}+1))\\ \end{pmatrix}.
\end{equation*}
\end{proposition}

\begin{proof}
To begin the proof, note that we can identify both the space of vanishing cycles
\begin{equation*}
\bigoplus\limits_{\sigma\in \rmZ_{\{1\}}(\overline{\rmB})}\rmH^{3}_{\{\sigma\}}(\overline{\rmX}_{\Pa}(\rmB), \rmR\Psi(\calO_{\lambda}(1))) 
\end{equation*}
and the space of co-vanishing cycles
\begin{equation*}
\bigoplus\limits_{\sigma\in \rmZ_{\{1\}}(\overline{\rmB})}\rmR^{3}\Phi_{\sigma}(\calO_{\lambda}(2)) 
\end{equation*}
with $\calO_{\lambda}[\rmZ_{\Pa}(\overline{\rmB})]$ ignoring the Galois actions using the generators $\eta^{\sigma}_{\bullet}\in\rmH^{3}_{\{\sigma\}}(\overline{\rmX}_{\Pa}(\rmB), \rmR\Psi(\calO_{\lambda}(1)))$ and $\eta^{\bullet}_{\sigma}\in \rmR^{3}\Phi_{\sigma}(\calO_{\lambda}(2))$ in Lemma \ref{intersection}. The level raising diagram \eqref{level-raise-diagram} reduces to 
\begin{equation}
\begin{tikzcd}
 \calO_{\lambda}[\rmZ_{\{0\}}(\overline{\rmB})] \arrow[rd, "\mathrm{inc}^{\{0\}}_{!}"'] &                                    &  \calO_{\lambda}[\rmZ_{\{2\}}(\overline{\rmB})]\arrow[ld, "\mathrm{inc}^{\{2\}}_{!}"] \\
                   & \rmH^{2}(\overline{\rmX}^{\square}_{\Pa}(\rmB), \calO_{\lambda}(1)) \arrow[d, "\beta(1)"]                   &                   \\
                   & \calO_{\lambda}[\rmZ_{\Pa}(\overline{\rmB})] \arrow[d, "\rmN^{-1}_{\Sigma}"]   &\\
                   & \calO_{\lambda}[\rmZ_{\Pa}(\overline{\rmB})] \arrow[d, "\alpha_{\rmc}(2)"]                   &                   \\
                   &  \rmH^{4}_{\rmc}(\overline{\rmX}^{\square}_{\Pa}(\rmB), \calO_{\lambda}(2)) \arrow[ld, "\mathrm{inc}^{\ast}_{\{0\}}"'] \arrow[rd, "\mathrm{inc}^{\ast}_{\{2\}}"] &                   \\
 \calO_{\lambda}[\rmZ_{\{0\}}(\overline{\rmB})]            &                                    &   \calO_{\lambda}[\rmZ_{\{2\}}(\overline{\rmB})].            
\end{tikzcd}
\end{equation}
This justifies the appearance of the Hecke operators appearing in the level raising matrix.
By Proposition \ref{alpha-beta}, the map $\alpha_{\rmc}(2)$ can be realized as a Gysin map and $\beta(1)$ can be realized as a restriction map. In fact, we have the following  commutative diagram where all the rows are exact:
\begin{equation*}
\begin{tikzcd}
\rmH^{2}(\overline{\rmX}^{\square}_{\Pa}(\rmB), \calO_{\lambda}(1)) \arrow[d] & \rmH^{2}(\overline{\rmX}^{\bullet}_{\Pa}(\rmB), \calO_{\lambda}(1)) \arrow[l] \arrow[d] & \bigoplus\limits_{\sigma\in \rmZ_{\Pa}(\overline{\rmB})}\rmH^{0}(\rmC_{\sigma}, \calO_{\lambda}) \arrow[l] \arrow[d, equal] \\
\bigoplus\limits_{\sigma\in \rmZ_{\Pa}(\overline{\rmB})}\rmH^{3}_{\{\sigma\}}(\overline{\rmX}_{\Pa}(\rmB), \calO_{\lambda}(1))  \arrow[d, "\rmN^{-1}_{\Sigma}"]           &  \bigoplus\limits_{\sigma\in \rmZ_{\Pa}(\overline{\rmB})}\rmH^{2}(\rmC_{\sigma}, \calO_{\lambda}(1))  \arrow[l]           &  \bigoplus\limits_{\sigma\in \rmZ_{\Pa}(\overline{\rmB})}\rmH^{0}(\rmC_{\sigma}, \calO_{\lambda}) \arrow[l]                                \\
\bigoplus\limits_{\sigma\in \rmZ_{\Pa}(\overline{\rmB})}\rmR^{3}\Phi_{\sigma}(\calO_{\lambda}(2))  \arrow[r] \arrow[d] & \bigoplus\limits_{\sigma\in \rmZ_{\Pa}(\overline{\rmB})}\rmH^{2}(\rmC_{\sigma}, \calO_{\lambda}(1))  \arrow[r] \arrow[d] & \bigoplus\limits_{\sigma\in \rmZ_{\Pa}(\overline{\rmB})}\rmH^{4}(\rmC_{\sigma}, \calO_{\lambda}(2))  \arrow[d, equal]           \\
\rmH^{4}_{\rmc}(\overline{\rmX}^{\square}_{\Pa}(\rmB), \calO_{\lambda}(2)) \arrow[r]           & \rmH^{4}_{\rmc}(\overline{\rmX}^{\bullet}_{\Pa}(\rmB), \calO_{\lambda}(2))\arrow[r]           &\bigoplus\limits_{\sigma\in \rmZ_{\Pa}(\overline{\rmB})}\rmH^{4}(\rmC_{\sigma}, \calO_{\lambda}(2)).                                   
\end{tikzcd}
\end{equation*}

Let $\overline{\rmX}^{\bullet\mathrm{ss}}_{\Pa}(\rmB)$ be the supersingular locus of $\overline{\rmX}^{\bullet}_{\Pa}(\rmB)$. We know that it has three types of irreducible components parametrized by $\rmZ_{\{i\}}(\overline{\rmB})$ for $i\in\{0,2\}$ and $\rmZ_{\Pa}(\overline{\rmB})$. We denote by $\overline{\rmX}^{\bullet\mathrm{ss}.n}_{\Pa, \{i\}}(\rmB)$ the disjoint union of the irreducible components of $\overline{\rmX}^{\bullet\mathrm{ss}}_{\Pa}(\rmB)$ parametrized by $\rmZ_{\{i\}}(\overline{\rmB})$ for $i\in\{0,2\}$. 

Therefore, for $i\in\{0,2\}$, we have the natural Gysin morphism
\begin{equation}\label{gys-sm}
\mathrm{inc}^{\{i\}}_{!}: \calO_{\lambda}[\rmZ_{\{i\}}(\overline{\rmB})]\cong\rmH^{0}(\overline{\rmX}^{\bullet\mathrm{ss}.n}_{\Pa, \{i\}}(\rmB), \calO_{\lambda})\rightarrow  \rmH^{2}(\overline{\rmX}^{\bullet}_{\Pa}(\rmB), \calO_{\lambda}(1))
\end{equation}
and the natural restriction morphism
\begin{equation}\label{res-sm}
\mathrm{inc}^{\ast}_{\{i\}}: \rmH^{4}_{\rmc}(\overline{\rmX}^{\bullet}_{\Pa}(\rmB), \calO_{\lambda}(2))\rightarrow  \rmH^{4}_{\rmc}(\overline{\rmX}^{\bullet\mathrm{ss}.n}_{\Pa, \{i\}}(\rmB), \calO_{\lambda}(2)) \cong \calO_{\lambda}[\rmZ_{\{i\}}(\overline{\rmB})].
\end{equation}
By the commutativity of the above diagram, to compute the entry $\calT_{\mathrm{lr},\{ij\}}$ for $i, j\in\{0, 2\}$, it suffices to understand the compotite map given by
\begin{equation*}
\begin{aligned}
&\calO_{\lambda}[\rmZ_{\{i\}}(\overline{\rmB})]\xrightarrow{\mathrm{inc}^{\{i\}}_{!}}  \rmH^{2}({\overline{\rmX}}^{\bullet}_{\Pa}(\rmB), \calO_{\lambda}(1))\rightarrow \bigoplus\limits_{\sigma\in \rmZ_{\Pa}(\overline{\rmB})}\rmH^{2}(\rmC_{\sigma}, \calO_{\lambda}(1))\rightarrow \bigoplus\limits_{\sigma\in \rmZ_{\Pa}(\overline{\rmB})}\rmH^{3}_{\{\sigma\}}(\overline{\rmX}_{\Pa}(\rmB), \calO_{\lambda}(1))\\
&\xrightarrow{\rmN^{-1}_{\Sigma}}\bigoplus\limits_{\sigma\in \rmZ_{\Pa}(\overline{\rmB})}\rmR^{3}\Phi_{\sigma}(\calO_{\lambda}(2)) \rightarrow \bigoplus\limits_{\sigma\in \rmZ_{\Pa}(\overline{\rmB})}\rmH^{2}(\rmC_{\sigma}, \calO_{\lambda}(1))\rightarrow \rmH^{4}_{\rmc}({\overline{\rmX}}^{\bullet}_{\Pa}(\rmB), \calO_{\lambda}(2))\xrightarrow{\mathrm{inc}^{\ast}_{\{j\}}} \calO_{\lambda}[\rmZ_{\{j\}}(\overline{\rmB})].
\end{aligned}
\end{equation*}
Using Lemma \ref{quadric-class}, \ref{intersection} and Corollary \ref{ss-bullet}, we arrive at the following term by term computation.

\subsubsection{$\calT_{\mathrm{lr},\{00\}}$} Note that $\rmZ_{\{i\}}(\overline{\rmB})$ corresponds to the set of vertex lattices of type $i$ for $i\in\{0,  2\}$ while $\rmZ_{\Pa}(\overline{\rmB})$ corresponds to the set of vertex lattices of type $1$.  Therefore to understand the map $\calT_{\lr, \{00\}}=\mathrm{inc}^{\ast}_{\{0\}} \circ\alpha_{\rmc}(2)\circ\rmN^{-1}_{\Sigma}\circ\beta(1)\circ\mathrm{inc}^{\{0\}}_{!}$, we need to understand the possible relative position of vertex lattices $\Lambda_{0}$, $\Lambda_{\Pa}$ and $\Lambda^{\prime}_{0}$ of type $0$, $1$ and $0$. It not hard to see that the only possible situation is that $\Lambda_{\Pa}\subset^{1} \Lambda_{0}$ and  $\Lambda_{\Pa}\subset^{1} \Lambda^{\prime}_{0}$. There are two possible cases.

Case $(1)$: $\Lambda_{0}=\Lambda^{\prime}_{0}$. Then we have $p\Lambda^{\vee}_{0}\subset^{1}p\Lambda^{\vee}_{\Pa}\subset^{2} \Lambda_{\Pa}\subset ^{1}\Lambda_{0}$. Giving such a $\Lambda_{\Pa}$ is equivalent to giving a complete flag in $\Lambda_{0}/p\Lambda_{0}$. The number of such flags is given by 
\begin{equation*}
[\bfG(\ZZ_{p}): \Kl]=(p+1)(p^{2}+1) 
\end{equation*}
where $\Kl$ is the Klingen parahoric.
 
Case $(2)$: $\Lambda_{0}\neq\Lambda^{\prime}_{0}$. Then we have $\Lambda_{0}\cap\Lambda^{\prime}_{0}=\Lambda_{\Pa}$ and the relative position between $\Lambda_{0}$ and $\Lambda^{\prime}_{0}$ is given by
\begin{equation*}
\rmK_{\{0\}}\begin{pmatrix}p&&&\\&1&&\\&&1&\\&&&p^{-1}\\\end{pmatrix}\rmK_{\{0\}}.
\end{equation*}
This operator is the same as $\rmT^{-1}_{p,0}\rmT_{p,1}$ by identifying $\rmK_{\{0\}}$ with the hyperspecial group $\bfG(\ZZ_{p})$. Therefore the entry $\calT_{\mathrm{lr}, \{00\}}$ is given by  
\begin{equation*}
\calT_{\mathrm{lr}, \{00\}}=-(\rmT^{-1}_{p,0}\rmT_{p,1}+(p+1)(p^{2}+1)). 
\end{equation*}
Here the sign $-1$ comes from the formula in Lemma \ref{intersection} $(3)$.

\subsubsection{$\calT_{\mathrm{lr}, \{22\}}$} To understand the map $\calT_{\lr, \{22\}}=\mathrm{inc}^{\ast}_{\{2\}} \circ\alpha_{\rmc}(2)\circ\rmN^{-1}_{\Sigma}\circ\beta(1)\circ\mathrm{inc}^{\{2\}}_{!}$, we need to understand the possible relative position of vertex lattices $\Lambda_{2}$, $\Lambda_{\Pa}$ and $\Lambda^{\prime}_{2}$ of type $2$, $1$ and $2$ respectively. It not hard to see that the only possible way is that $\Lambda_{2}\subset^{1}\Lambda_{\Pa}$ and  $\Lambda^{\prime}_{2}\subset^{1}\Lambda_{\Pa}$. The rest of the analysis is completely the same as in the case of $\calT_{\lr,\{00\}}$. Once we identify $\rmK_{\{2\}}$ with the hyperspecial group $\bfG(\ZZ_{p})$, the entry $\calT_{\lr, \{22\}}$ is given by 
\begin{equation*}
\calT_{\lr, \{22\}}=-(\rmT^{-1}_{p,0}\rmT_{p,1}+(p+1)(p^{2}+1)). 
\end{equation*}
Here the sign $-1$ comes from the formula in Lemma \ref{intersection} $(3)$.

\subsubsection{$\calT_{\mathrm{lr}, \{02\}}$} To understand the map $\calT_{\lr, \{02\}}=\mathrm{inc}^{\ast}_{\{2\}} \circ\alpha_{\rmc}(2)\circ\rmN^{-1}_{\Sigma}\circ\beta(1)\circ\mathrm{inc}^{\{0\}}_{!}$, we need to understand the possible relative position of vertex lattices $\Lambda_{0}$, $\Lambda_{\Pa}$ and $\Lambda^{\prime}_{2}$ of type $0$, $1$ and $2$. We find that the only possible situation is that
\begin{equation*}
\Lambda_{2}\subset^{1}\Lambda_{\Pa}\subset^{1}\Lambda_{0}=\Lambda^{\vee}_{0}\subset^{1}\Lambda^{\vee}_{\Pa}\subset^{1}\Lambda^{\vee}_{2}.
\end{equation*}
It follows each such $\Lambda_{\Pa}$ corresponds to a line in the two dimensional space $\Lambda_{0}/\Lambda_{2}$. 

Therefore the entry $\calT_{\lr, \{02\}}$ is given by
\begin{equation*}
\calT_{\lr, \{02\}}=(p+1)\mathrm{char}(\rmK_{\{0\}}\begin{pmatrix}p&&&\\&p&&\\&&1&\\&&&1 \end{pmatrix}\rmK_{\{2\}})=(p+1)\rmT_{02}.\
\end{equation*}

\subsubsection{$\calT_{\mathrm{lr}, \{20\}}$} To understand the map $\calT_{\mathrm{lr}, \{20\}}=\mathrm{inc}^{\ast}_{\{0\}} \circ\alpha_{\rmc}(2)\circ\rmN^{-1}_{\Sigma}\circ\beta(1)\circ\mathrm{inc}^{\{2\}}_{!}$, we need to understand the possible relative positions of the vertex lattices $\Lambda_{0}$, $\Lambda_{\Pa}$ and $\Lambda^{\prime}_{2}$ of type $0$, $1$ and $2$. We find that the only possible way is that
\begin{equation*}
\Lambda_{2}\subset^{1}\Lambda_{\Pa}\subset^{1}\Lambda_{0}=\Lambda^{\vee}_{0}\subset^{1}\Lambda^{\vee}_{\Pa}\subset^{1}\Lambda^{\vee}_{2}.
\end{equation*}
Each $\Lambda_{\Pa}$ corresponds to a line in the two dimensional space $\Lambda_{0}/\Lambda_{2}$. 

Therefore the entry $\calT_{\mathrm{lr},\{20\}}$ is given by
\begin{equation*}
\calT_{\mathrm{lr},\{20\}}=(p+1)\mathrm{char}(\rmK_{\{0\}}\begin{pmatrix}p^{-1}&&&\\&p^{-1}&&\\&&1&\\&&&1 \end{pmatrix}\rmK_{\{2\}})=(p+1)\rmT_{02}.\
\end{equation*}
This finishes the proof of Proposition \ref{lr-matrix}.
\end{proof}

We will compute the determinant of the level raising matrix and the supersingular matrix that we will define in the next section. For this reason, we will need to compute the composite $\rmT_{20}\circ\rmT_{02}$ and $\rmT_{02}\circ\rmT_{20}$ in terms of elements in the spherical Hecke algebra.
\begin{lemma}\label{T02T20}
The composite $\rmT_{20}\circ\rmT_{02}$  and $\rmT_{02}\circ\rmT_{20}$ are both given by 
\begin{equation*}
\mathrm{T}^{-1}_{p, 0}\rmT_{p^{2}, 2}+(p+1)\mathrm{T}^{-1}_{p, 0}\rmT_{p, 1}+(p^{2}+1)(p+1).
\end{equation*}
As a result, if $\rmT_{p, 0}$ acts trivially, then $\rmT_{20}\circ\rmT_{02}=\rmT_{02}\circ\rmT_{20}=\rmT^{2}_{p, 2}$
\end{lemma}

\begin{proof} The proof of this lemma is similar to that of the above proposition. We will compute $\rmT_{20}\circ\rmT_{02}$ and $\rmT_{02}\circ\rmT_{20}$.

\subsubsection{$\rmT_{20}\circ\rmT_{02}$}  Let $\Lambda_{0}$ be a vertex lattice of type $0$. Then $\rmT_{20}\circ\rmT_{02}(\Lambda_{0})$ classifies a pair of vertex lattices $(\Lambda_{2}, \Lambda^{\prime}_{0})$ of type $2$ and $0$. We need to understand all the possible relative position of the vertex lattices $\Lambda_{0}$, $\Lambda_{2}$ and $\Lambda^{\prime}_{0}$. The only possible situation is given below 
\begin{equation*}
\begin{aligned}
&p\Lambda^{\vee}_{0}\subset^{2}\Lambda_{2}\subset^{2} \Lambda_{0}=\Lambda^{\vee}_{0} \subset^{2} \Lambda^{\vee}_{2}\\
&p\Lambda^{\prime\vee}_{0}\subset^{2}\Lambda_{2}\subset^{2} \Lambda^{\prime}_{0}=\Lambda^{\prime\vee}_{0} \subset^{2} \Lambda^{\vee}_{2}.\\
\end{aligned}
\end{equation*}
Note that $\Lambda_{2}$ determines an isotropic subspace of the symplectic space $\Lambda_{0}\cap\Lambda^{\prime}_{0}/p\Lambda^{\vee}_{0}+p\Lambda^{\prime\vee}_{0}$ of possible dimensions in the set $\{0, 2, 4\}$.

When $\dim_{\FF}\Lambda_{0}\cap\Lambda^{\prime}_{0}/p\Lambda^{\vee}_{0}+p\Lambda^{\prime\vee}_{0}=0$, then 
\begin{equation*}
\begin{aligned}
&\Lambda_{0}\cap\Lambda^{\prime}_{0}\subset^{2} \Lambda_{0}\\ 
&\Lambda_{0}\cap\Lambda^{\prime}_{0}\subset^{2} \Lambda^{\prime}_{0}.\\ 
\end{aligned}
\end{equation*}
Hence this case contributes to the $\rmT_{20}\circ\rmT_{02}$ by the double coset operator
\begin{equation*}
\rmK_{\{0\}}\begin{pmatrix}p&&&\\&p&&\\&&p^{-1}&\\&&&p^{-1} \end{pmatrix}\rmK_{\{0\}}.
\end{equation*}
Once we identify $\rmK_{\{0\}}$ with $\bfG(\ZZ_{p})$, we see this is the same as $\mathrm{T}^{-1}_{0,p}\rmT_{p^{2}, 2}$.

When $\dim_{\FF}\Lambda_{0}\cap\Lambda^{\prime}_{0}/p\Lambda^{\vee}_{0}+p\Lambda^{\prime\vee}_{0}=2$, there are $p+1$ choices of $\Lambda_{2}$. And we have 
\begin{equation*}
\begin{aligned}
&\Lambda_{0}\cap\Lambda^{\prime}_{0}\subset^{1} \Lambda_{0}\\ 
&\Lambda_{0}\cap\Lambda^{\prime}_{1}\subset^{1} \Lambda^{\prime}_{0}.\\
\end{aligned}
\end{equation*}
Hence this case contributes to the $\rmT_{20}\circ\rmT_{02}$ by the double coset operator
\begin{equation*}
(p+1)\rmK_{\{0\}}\begin{pmatrix}p&&&\\&1&&\\&&1&\\&&&p^{-1} \end{pmatrix}\rmK_{\{0\}}.
\end{equation*}
Once we identify $\rmK_{\{0\}}$ with $\bfG(\ZZ_{p})$, we see this is the same as $(p+1)\mathrm{T}^{-1}_{0,p}\rmT_{p, 1}$.

When $\dim_{\FF}\Lambda_{0}\cap\Lambda^{\prime}_{0}/p\Lambda^{\vee}_{0}+p\Lambda^{\prime\vee}_{0}=4$, there are $[\bfG(\ZZ_{p}):\mathrm{Sie})]=(p+1)(p^{2}+1)$  choices of $\Lambda_{2}$. And we have 
\begin{equation*}
\begin{aligned}
&\Lambda_{0}\cap\Lambda^{\prime}_{0}\subset^{0} \Lambda_{0}\\ 
&\Lambda_{0}\cap\Lambda^{\prime}_{0}\subset^{0} \Lambda^{\prime}_{0}.\\ 
\end{aligned}
\end{equation*}
Hence this case contributes to the $\rmT_{20}\circ\rmT_{02}$ by the constant $(p^{2}+1)(p+1)$. All in all, we obtain the following formula
\begin{equation}\label{T02T20}
\rmT_{20}\circ\rmT_{02}= \mathrm{T}^{-1}_{p, 0}\rmT_{p^{2}, 2}+(p+1)\mathrm{T}^{-1}_{p, 0}\rmT_{p, 1}+(p^{2}+1)(p+1).
\end{equation} 

\subsubsection{$\rmT_{02}\circ\rmT_{20}$} The computation is almost the same to the above. Let $\Lambda_{2}$ be a vertex lattice of type $2$. Then $\rmT_{02}\circ\rmT_{20}(\Lambda_{2})$ classifies pair of vertex lattices $(\Lambda_{0}, \Lambda^{\prime}_{2})$ of type $0$ and $2$. We need to understand the relative position of vertex lattices $\Lambda_{0}$, $\Lambda_{2}$ and $\Lambda^{\prime}_{2}$. The only possible situation is given below 
\begin{equation*}
\begin{aligned}
&p\Lambda^{\vee}_{0}\subset^{2}\Lambda_{2}\subset^{2} \Lambda_{0}=\Lambda^{\vee}_{0} \subset^{2} \Lambda^{\vee}_{2};\\
&p\Lambda^{\vee}_{0}\subset^{2}\Lambda^{\prime}_{2}\subset^{2} \Lambda_{0}=\Lambda^{\vee}_{0} \subset^{2} \Lambda^{\prime\vee}_{2}.\\
\end{aligned}
\end{equation*}
Note that $\Lambda_{0}$ is determined by an isotropic subspace of the symplectic space $\Lambda_{2}\cap\Lambda^{\prime}_{2}/\Lambda^{\vee}_{2}+\Lambda^{\prime\vee}_{2}$ of possible dimensions $\{0, 2, 4\}$. 

When $\dim_{\FF} \Lambda_{2}\cap\Lambda^{\prime}_{2}/\Lambda^{\vee}_{2}+\Lambda^{\prime\vee}_{2}=0$, then in this case we have
\begin{equation*}
\begin{aligned}
&\Lambda_{2}\subset^{2} \Lambda_{2}+\Lambda^{\prime}_{2}\\ 
&\Lambda^{\prime}_{2}\subset^{2} \Lambda_{2}+\Lambda^{\prime}_{2}.\\ 
\end{aligned}
\end{equation*}
Hence this case contributes to the $\rmT_{02}\circ\rmT_{20}$ by the double coset operator
\begin{equation*}
\rmK_{\{2\}}\begin{pmatrix}p^{-1}&&&\\&p^{-1}&&\\&&p&\\&&&p\end{pmatrix}\rmK_{\{2\}}.
\end{equation*}
Once we identify $\rmK_{\{2\}}$ with $\bfG(\ZZ_{p})$, we see this is the same as $\mathrm{T}^{-1}_{0,p}\rmT_{p^{2}, 2}$.

When $\dim_{\FF} \Lambda^{\vee}_{2}\cap\Lambda^{\prime\vee}_{2}/\Lambda_{2}+\Lambda^{\prime}_{2}=2$, there are $p+1$ choices of $\Lambda_{0}$. And in this case we have 
\begin{equation*}
\begin{aligned}
&\Lambda_{2}\subset^{1} \Lambda_{2}+\Lambda^{\prime}_{2}\\ 
&\Lambda^{\prime}_{2}\subset^{1}\Lambda_{2}+\Lambda^{\prime}_{2}.\\ 
\end{aligned}
\end{equation*}
Hence this case contributes to the $\rmT_{20}\circ\rmT_{02}$ by the Hecke operator
\begin{equation*}
(p+1)\rmK_{\{2\}}\begin{pmatrix}p^{-1}&&&\\&1&&\\&&1&\\&&&p \end{pmatrix}\rmK_{\{2\}}.
\end{equation*}
Once we identify $\rmK_{\{2\}}$ with the hyperspecial subgroup $\bfG(\ZZ_{p})$, we see this is the same as $(p+1)\rmT^{-1}_{0,p}\rmT_{p, 1}$.

When $\dim_{\FF} \Lambda^{\vee}_{2}\cap\Lambda^{\prime\vee}_{2}/\Lambda_{2}+\Lambda^{\prime}_{2}=4$, there are $[\bfG(\ZZ_{p}):\mathrm{Sie})]=(p+1)(p^{2}+1)$  choices of $\Lambda_{0}$. And in this case we have
\begin{equation*}
\begin{aligned}
&\Lambda_{2}\subset^{0} \Lambda_{2}+\Lambda^{\prime}_{2}\\ 
&\Lambda^{\prime}_{2}\subset^{0} \Lambda_{2}+\Lambda^{\prime}_{2}.\\ 
\end{aligned}
\end{equation*}
Hence this case contributes to the $\rmT_{20}\circ\rmT_{02}$ by the constant $(p^{2}+1)(p+1)$. All in all, we obtain the following formula
\begin{equation}\label{T20T02}
\rmT_{20}\circ\rmT_{02}= \mathrm{T}^{-1}_{0,p}\rmT_{p^{2}, 2}+(p+1)\mathrm{T}^{-1}_{0,p}\rmT_{p, 1}+(p^{2}+1)(p+1).
\end{equation} 
The last claim in the lemma follows immediately from Proposition \ref{Hecke-id}. This finishes the proof of the lemma.
\end{proof}

The final result of this section is the computation of the determinant of the level raising matrix modulo $\fracm$.

\begin{proposition}\label{det-lr}
Let $(\pi, \Sigma, \fracm)$ be the datum considered as in the beginning of the section. Suppose the maximal ideal $\fracm$ appears in the support of $\calO_{\lambda}[\rmZ_{\rmH}(\overline{\rmB})]$. Let $[\alpha_{p}, \beta_{p}, p^{3}\beta^{-1}, p^{3}\alpha^{-1}_{p}]$ be the Hecke parameter of $\pi$ at $p$. Then the determinant of the level raising matrix modulo $\fracm$ is given by
\begin{equation*}
\det \calT_{\mathrm{lr}/\fracm}= p^{-2}\prod_{u=\pm1}(\alpha_{p}+p^{3}\alpha^{-1}_{p}-up(p+1))(\beta_{p}+p^{3}\beta^{-1}_{p}-up(p+1)).
\end{equation*}
\end{proposition}
\begin{proof}
Since $\pi$ has trivial central character, $\rmT_{p,0}$ acts trivially on $\calO_{\lambda}[\rmZ_{\rmH}(\overline{\rmB})]/\fracm$. We therefore have 
\begin{equation*}
\rmT_{20}\circ\rmT_{02}= \rmT_{02}\circ\rmT_{20}=\rmT_{p^{2}, 2}+(p+1)\rmT_{p, 1}+(p^{2}+1)(p+1)=\rmT^{2}_{p,2} 
\end{equation*}
by Lemma \ref{T02T20}. Then
\begin{equation*}
\begin{aligned}
\det \calT_{\mathrm{lr}/\fracm}&= ((\rmT_{p,1}+(p+1)(p^{2}+1))^{2}-(p+1)^{2}\rmT_{20}\rmT_{02})\\
&=(\rmT_{p,1}+(p+1)(p^{2}+1)-(p+1)\rmT_{p,2})(\rmT_{p,1}+(p+1)(p^{2}+1)+(p+1)\rmT_{p,2})\\
&=p^{-2}[(\alpha_{p}+p^{3}\alpha^{-1}_{p})(\beta_{p}+p^{3}\beta^{-1}_{p})-p(p+1)(\alpha_{p}+\beta_{p}+p^{3}\alpha^{-1}_{p}+p^{3}\beta^{-1}_{p})+p^{2}(p+1)^{2}]\\
&\phantom{aaaaaa}[(\alpha_{p}+p^{3}\alpha^{-1}_{p})(\beta_{p}+p^{3}\beta^{-1}_{p})+p(p+1)(\alpha_{p}+\beta_{p}+p^{3}\alpha^{-1}_{p}+p^{3}\beta^{-1}_{p})+p^{2}(p+1)^{2}]\\
&=p^{-2}\prod_{u=\pm1}(\alpha_{p}+p^{3}\alpha^{-1}_{p}-u p(p+1))(\beta_{p}+p^{3}\beta^{-1}_{p}-u p(p+1)).\\
\end{aligned}
\end{equation*}
This finishes the proof.
\end{proof}

\section{Tate cycles on the quaternionic unitary Shimura variety}
Let $(\pi, \Sigma, \fracm)$ be the datum considered as in the beginning of the last section. In this section, our goal is to establish the following priniciple: the localized cohomology group 
\begin{equation*}
\rmH^{2}(\overline{\rmX}^{\square}_{\Pa}(\rmB), \calO_{\lambda}(1))_{\fracm}\cong \mathrm{IH}^{2}(\overline{\rmX}_{\Pa}(\rmB), \calO_{\lambda}(1))_{\fracm} 
\end{equation*}
is generated by cycles coming from the supersingular locus. On the dual side, similar result also holds for 
\begin{equation*}
\rmH^{4}_{\rmc}(\overline{\rmX}_{\Pa}(\rmB), \calO_{\lambda}(2))_{\fracm}= \mathrm{IH}^{4}_{\rmc}(\overline{\rmX}_{\Pa}(\rmB), \calO_{\lambda}(2))_{\fracm},
\end{equation*}
however we need to show this cohomology group contains no $\calO_{\lambda}$-torsion first. We will establish the full result simultaneously with our main theorem on arithmetic level raising in {Corollary} \ref{no-torsion}. 

\subsection{The supersingular matrix} Recall from the last section, $\overline{\rmX}^{\square\mathrm{ss.n}}_{\Pa, \{0,2\}}(\rmB)$ is the normalization of the supersingular locus $\overline{\rmX}^{\square\mathrm{ss}}_{\Pa}(\rmB)$ and hence $\overline{\rmX}^{\square\mathrm{ss.n}}_{\Pa, \{0,2\}}(\rmB)=\overline{\rmX}^{\square\mathrm{ss.n}}_{\Pa, \{0\}}(\rmB)\sqcup \overline{\rmX}^{\square\mathrm{ss.n}}_{\Pa, \{2\}}(\rmB)$ where $\overline{\rmX}^{\square\mathrm{ss.n}}_{\Pa, \{i\}}(\rmB)$ is the disjoint union of $\mathrm{DL}^{\square}(\Lambda_{i})$ parametrized by $\rmZ_{\{i\}}(\overline{\rmB})$ for $i\in\{0,2\}$. 
\begin{lemma}
We have the following isomorphism
\begin{equation*}
\rmH^{2}(\overline{\rmX}^{\square\mathrm{ss.n}}_{\Pa,\{0,2\}}(\rmB), \calO_{\lambda}(1))\cong\rmH^{2}_{\rmc}(\overline{\rmX}^{\square\mathrm{ss.n}}_{\Pa,\{0,2\}}(\rmB), \calO_{\lambda}(1)).
\end{equation*}
\end{lemma}
\begin{proof}
It suffices to show 
\begin{equation*}
\rmH^{2}(\mathrm{DL}^{\square}(\Lambda_{i}), \calO_{\lambda}(1))\cong\rmH^{2}_{\rmc}(\mathrm{DL}^{\square}(\Lambda_{i}), \calO_{\lambda}(1))
\end{equation*}
for each vertex lattice $\Lambda_{i}$ with $i\in\{0,2\}$.
Using the excision long exact sequence for the pair $(\mathrm{DL}(\Lambda_{i}), \mathrm{DL}^{\square}(\Lambda_{i}))$, it is not hard to see that
\begin{equation*}
\rmH^{2}(\mathrm{DL}^{\square}(\Lambda_{i}), \calO_{\lambda}(1))\cong \rmH^{2}(\mathrm{DL}(\Lambda_{i}), \calO_{\lambda}(1))\cong \rmH^{2}_{\rmc}(\mathrm{DL}^{\square}(\Lambda_{i}), \calO_{\lambda}(1)).
\end{equation*}
Then the lemma follows.
\end{proof}

\begin{construction}
We have the following natural restriction morphism
\begin{equation*}
\mathrm{inc}^{\ast}_{\{\mathrm{ss}\}}: \rmH^{2}(\overline{\rmX}^{\square}_{\Pa}(\rmB), \calO_{\lambda}(1))\rightarrow \rmH^{2}(\overline{\rmX}^{\square\mathrm{ss.n}}_{\Pa,\{0,2\}}(\rmB), \calO_{\lambda}(1)).
\end{equation*}
We also have the natural Gysin map denoted by
\begin{equation*}
\mathrm{inc}^{\{\mathrm{ss}\}}_{!}: \rmH^{2}_{\rmc}(\overline{\rmX}^{\square\mathrm{ss.n}}_{\Pa,\{0,2\}}(\rmB), \calO_{\lambda}(1))\rightarrow \rmH^{4}_{\rmc}(\overline{\rmX}^{\square}_{\Pa}(\rmB), \calO_{\lambda}(2)).
\end{equation*}
\end{construction}

We can fit the two maps $(\mathrm{inc}^{\ast}_{\{\mathrm{ss}\}}, \mathrm{inc}^{\{\mathrm{ss}\}}_{!})$ in the above construction in the diagram below.

\begin{equation}\label{ss-diagram}
\begin{tikzcd}
 \calO_{\lambda}[\rmZ_{\{0\}}(\overline{\rmB})] \arrow[rd, "\mathrm{inc}^{\{0\}}_{!}"'] &                                    &  \calO_{\lambda}[\rmZ_{\{2\}}(\overline{\rmB})]\arrow[ld, "\mathrm{inc}^{\{2\}}_{!}"] \\
                   & \rmH^{2}(\overline{\rmX}^{\square}_{\Pa}(\rmB), \calO_{\lambda}(1)) \arrow[d, "\mathrm{inc}^{\ast}_{\{\mathrm{ss}\}}"]                   &                   \\
                   & \rmH^{2}(\overline{\rmX}^{\square\mathrm{ss.n}}_{\Pa,\{0,2\}}(\rmB), \calO_{\lambda}(1)) \arrow[d, "\cong"]                   &                   \\
                    & \rmH^{2}_{\rmc}(\overline{\rmX}^{\square\mathrm{ss.n}}_{\Pa,\{0,2\}}(\rmB), \calO_{\lambda}(1)) \arrow[d, "\mathrm{inc}^{\{\mathrm{ss}\}}_{!}"]                   &                   \\
                   &  \rmH^{4}_{\rmc}(\overline{\rmX}^{\square}_{\Pa}(\rmB), \calO_{\lambda}(2)) \arrow[ld, "\mathrm{inc}^{\ast}_{\{0\}}"'] \arrow[rd, "\mathrm{inc}^{\ast}_{\{2\}}"] &                   \\
 \calO_{\lambda}[\rmZ_{\{0\}}(\overline{\rmB})]            &                                    &   \calO_{\lambda}[\rmZ_{\{2\}}(\overline{\rmB})]            
\end{tikzcd}
\end{equation}

The above diagram sets up the intersection matrix for the supersingular locus which we will elaborate below.
\begin{construction}
We obtain naturally the following four maps from the above diagram
\begin{equation*}
\begin{aligned}
&\calT_{\mathrm{ss},\{00\}}= \mathrm{inc}^{\ast}_{\{0\}}\circ \mathrm{inc}^{\{\mathrm{ss}\}}_{!}\circ \mathrm{inc}^{\ast}_{\{\mathrm{ss}\}}\circ\mathrm{inc}^{\{0\}}_{!}\\
&\calT_{\mathrm{ss},\{02\}}= \mathrm{inc}^{\ast}_{\{2\}}\circ  \mathrm{inc}^{\{\mathrm{ss}\}}_{!}\circ \mathrm{inc}^{\ast}_{\{\mathrm{ss}\}}\circ\circ\mathrm{inc}^{\{0\}}_{!}\\
&\calT_{\mathrm{ss},\{20\}}= \mathrm{inc}^{\ast}_{\{0\}}\circ \mathrm{inc}^{\{\mathrm{ss}\}}_{!}\circ \mathrm{inc}^{\ast}_{\{\mathrm{ss}\}}\circ\mathrm{inc}^{\{2\}}_{!}\\
&\calT_{\mathrm{ss},\{22\}}= \mathrm{inc}^{\ast}_{\{2\}} \circ \mathrm{inc}^{\{\mathrm{ss}\}}_{!}\circ \mathrm{inc}^{\ast}_{\{\mathrm{ss}\}}\circ\circ\mathrm{inc}^{\{2\}}_{!}.\\
\end{aligned}
\end{equation*}
The resulting matrix
\begin{equation*}
\calT_{\mathrm{ss}}=\begin{pmatrix} &\calT_{\mathrm{ss},\{00\}} & \calT_{\mathrm{ss}, \{02\}}\\ &\calT_{\mathrm{ss}, \{20\}}& \calT_{\mathrm{ss}, \{22\}}\\ \end{pmatrix}
\end{equation*}
will be referred to as the supersingular matrix. We can localize the above diagram at the maximal ideal $\fracm$ and write the resulting matrix as $\calT_{\mathrm{ss},\fracm}$, however all the matrix entries will still be denoted by symbols without referring to $\fracm$. If we need to consider this matrix modulo $\fracm$, then we will denote it by $\calT_{\mathrm{ss}/\fracm}$. The entries of this matrix will be denoted without referring to $\fracm$.
\end{construction}

\subsection{Computing the supersingular matrix} We compute the entries of the supersingular matrix using the same procedure as we did for computing the  level raising matrix. 
\begin{proposition}\label{ss-matrix}
The supersingular matrix $\calT_{\mathrm{ss}}$ is given by
\begin{equation*}
\calT_{\mathrm{ss}}=\begin{pmatrix} &-2p(p+1)(p-1)^{2} & -4p\rmT_{02}\\ &-4p\rmT_{20}& -2p(p+1)(p-1)^{2}\\ \end{pmatrix}.
\end{equation*}
\end{proposition}
\begin{proof} As we did for computing the level raising matrix, we make use of the desingularization $\overline{\rmX}^{\bullet}_{\Pa}(\rmB)$ of $\overline{\rmX}_{\Pa}(\rmB)$. Recall that $\overline{\rmX}^{\bullet\mathrm{ss.n}}_{\Pa, \{0, 2\}}(\rmB)$ is given by
\begin{equation*}
\overline{\rmX}^{\bullet\mathrm{ss.n}}_{\Pa, \{0, 2\}}(\rmB)=\overline{\rmX}^{\bullet\mathrm{ss.n}}_{\Pa, \{0\}}(\rmB)\sqcup \overline{\rmX}^{\bullet\mathrm{ss.n}}_{\Pa, \{2\}}(\rmB) 
\end{equation*}
where $\overline{\rmX}^{\bullet\mathrm{ss.n}}_{\Pa, \{i\}}(\rmB)$ is the disjoint union of $\mathrm{DL}^{\bullet}(\Lambda_{i})$ parametrized by $\rmZ_{\{i\}}(\overline{\rmB})$ for $i\in\{0,2\}$. Then we can set up the supersingular matrix \eqref{ss-diagram} using the following diagram instead
\begin{equation}\label{ss-diagram-sm}
\begin{tikzcd}
\calO_{\lambda}[\rmZ_{\{0\}}(\overline{\rmB})] \arrow[rd, "\mathrm{inc}^{\{0\}}_{!}"'] &                                    &  \calO_{\lambda}[\rmZ_{\{2\}}(\overline{\rmB})]\arrow[ld, "\mathrm{inc}^{\{2\}}_{!}"] \\
                   & \rmH^{2}(\overline{\rmX}^{\bullet}_{\Pa}(\rmB), \calO_{\lambda}(1)) \arrow[d, "\mathrm{inc}^{\ast}_{\{\mathrm{ss}\}}"]                   &                   \\
                   & \rmH^{2}(\overline{\rmX}^{\bullet\mathrm{ss.n}}_{\Pa, \{0,2\}}(\rmB), \calO_{\lambda}(1)) \arrow[d, "\mathrm{inc}^{\{\mathrm{ss}\}}_{!}"]                   &                   \\
                   &  \rmH^{4}_{\rmc}(\overline{\rmX}^{\bullet}_{\Pa}(\rmB), \calO_{\lambda}(2)) \arrow[ld, "\mathrm{inc}^{\ast}_{\{0\}}"'] \arrow[rd, "\mathrm{inc}^{\ast}_{\{2\}}"] &                   \\
 \calO_{\lambda}[\rmZ_{\{0\}}(\overline{\rmB})]            &                                    &   \calO_{\lambda}[\rmZ_{\{2\}}(\overline{\rmB})].            
\end{tikzcd}
\end{equation}
Here the map
\begin{equation*}
\mathrm{inc}^{\ast}_{\{\mathrm{ss}\}}: \rmH^{2}(\overline{\rmX}^{\bullet}_{\Pa}(\rmB), \calO_{\lambda}(1))\rightarrow \rmH^{2}(\overline{\rmX}^{\bullet\mathrm{ss}.n}_{\Pa}(\rmB), \calO_{\lambda}(1)) 
\end{equation*}
is the natural restriction map and 
\begin{equation*}
\mathrm{inc}^{\{\mathrm{ss}\}}_{!}: \rmH^{2}(\overline{\rmX}^{\bullet\mathrm{ss}.n}_{\Pa}(\rmB), \calO_{\lambda}(1))\rightarrow \rmH^{4}_{\rmc}(\overline{\rmX}^{\bullet}_{\Pa}(\rmB), \calO_{\lambda}(2)) 
\end{equation*}
is the natural Gysin map. We clearly have a commutative diagram
\begin{equation*}
\begin{tikzcd}
 \rmH^{2}(\overline{\rmX}^{\square}_{\Pa}(\rmB), \calO_{\lambda}(1))  \arrow[d, "\mathrm{inc}^{\ast}_{\{\mathrm{ss}\}}"] & \rmH^{2}(\overline{\rmX}^{\bullet}_{\Pa}(\rmB), \calO_{\lambda}(1))  \arrow[d,  "\mathrm{inc}^{\ast}_{\{\mathrm{ss}\}}"]\arrow[l] \\
\rmH^{2}_{(\rmc)}(\overline{\rmX}^{\square\mathrm{ss.n}}_{\Pa, \{0,2\}}(\rmB), \calO_{\lambda}(1)) \arrow[d, "\mathrm{inc}^{\{\mathrm{ss}\}}_{!}"] &  \rmH^{2}(\overline{\rmX}^{\bullet\mathrm{ss.n}}_{\Pa, \{0,2\}}(\rmB), \calO_{\lambda}(1)) \arrow[d, "\mathrm{inc}^{\{\mathrm{ss}\}}_{!}"] \arrow[l]\\
\rmH^{4}_{\rmc}(\overline{\rmX}^{\square}_{\Pa}(\rmB), \calO_{\lambda}(2)) \arrow[r]  &\rmH^{4}_{\rmc}(\overline{\rmX}^{\bullet}_{\Pa}(\rmB), \calO_{\lambda}(2))  \\
\end{tikzcd}
\end{equation*}
which allows us to use \eqref{ss-diagram-sm} to compute the entries of the supersingular matrix.

\subsubsection{$\calT_{\mathrm{ss},\{00\}}$} Recall $\calT_{\mathrm{ss},\{00\}}= \mathrm{inc}^{\ast}_{\{0\}}\circ \mathrm{inc}^{\{\mathrm{ss}\}}_{!}\circ \mathrm{inc}^{\ast}_{\{\mathrm{ss}\}}\circ\mathrm{inc}^{\{0\}}_{!}$. 
The contribution of the composite maps given by
\begin{equation*}
\begin{aligned}
\calO_{\lambda}[\rmZ_{\{0\}}(\overline{\rmB})]\xrightarrow{\mathrm{inc}^{\{0\}}_{!}}\rmH^{2}(\overline{\rmX}^{\bullet}_{\Pa}(\rmB), \calO_{\lambda}(1))\xrightarrow{\mathrm{inc}^{\ast}_{\{\mathrm{ss}\}}}\rmH^{2}(\overline{\rmX}^{\bullet\mathrm{ss.n}}_{\Pa, \{0\}}(\rmB), \calO_{\lambda}(1))&\xrightarrow{\mathrm{inc}^{\{\mathrm{ss}\}}_{!}} \rmH^{4}_{\rmc}(\overline{\rmX}^{\bullet}_{\Pa}(\rmB), \calO_{\lambda}(2))\\
&\xrightarrow{\mathrm{inc}^{\ast}_{\{0\}}}\calO_{\lambda}[\rmZ_{\{0\}}(\overline{\rmB})]\\
\end{aligned}
\end{equation*}
to $\calT_{\mathrm{ss},\{00\}}$ is  $4p^{2}(p+1)$. This follows from the fact that for each $\Lambda_{0}\in\calL_{\{0\}}$ the degree of the variety $\mathrm{DL}^{\bullet}(\Lambda_{0})$ is $(p+1)$ and the normal bundle has degree $-2p$ by {Lemma \ref{intersection-number}}. To compute the contribution of the composite maps given by
\begin{equation*}
\begin{aligned}
\calO_{\lambda}[\rmZ_{\{0\}}(\overline{\rmB})]\xrightarrow{\mathrm{inc}^{\{0\}}_{!}}\rmH^{2}(\overline{\rmX}^{\bullet}_{\Pa}(\rmB), \calO_{\lambda}(1))\xrightarrow{\mathrm{inc}^{\ast}_{\{\mathrm{ss}\}}}\rmH^{2}(\overline{\rmX}^{\bullet\mathrm{ss.n}}_{\Pa, \{2\}}(\rmB), \calO_{\lambda}(1))&\xrightarrow{\mathrm{inc}^{\{\mathrm{ss}\}}_{!}} \rmH^{4}_{\rmc}(\overline{\rmX}^{\bullet}_{\Pa}(\rmB), \calO_{\lambda}(2))\\
&\xrightarrow{\mathrm{inc}^{\ast}_{\{0\}}}\calO_{\lambda}[\rmZ_{\{0\}}(\overline{\rmB})].\\
\end{aligned}
\end{equation*}
Let $\Lambda_{0}$ be a vertex lattice of type $0$ and $e_{\Lambda_{0}}\in \calO_{\lambda}[\rmZ_{\{0\}}(\overline{\rmB})]$ be the class represents $\mathrm{DL}^{\bullet}(\Lambda_{0})$. Then we have 
\begin{equation*}
\mathrm{inc}^{\{0\}}_{!}\circ \mathrm{inc}^{\ast}_{\{\mathrm{ss}\}}(e_{\Lambda_{0}})=\sum_{\Lambda_{2}}e_{\Lambda_{0}, \Lambda_{2}}
\end{equation*}
where $\Lambda_{2}$ runs through all the lattices of type $2$ such that $p\Lambda^{\vee}_{0}\subset^{2}\Lambda_{2}\subset \Lambda_{0}$ and where $e_{\Lambda_{0},\Lambda_{2}}$ represents the class $\mathrm{DL}^{\bullet}(\Lambda_{0})\cap \mathrm{DL}^{\bullet}(\Lambda_{2})\simeq\PP^{1}$ in $\rmH^{2}(\overline{\rmX}^{\bullet\mathrm{ss.n}}_{\Pa, \{2\}}(\rmB), \calO_{\lambda}(1))$. Then we would like to compute
\begin{equation*}
\mathrm{inc}^{\{\mathrm{ss}\}}\circ\mathrm{inc}^{\ast}_{\{0\}}(e_{\Lambda_{0}, \Lambda_{2}})=\sum_{\Lambda^{\prime}_{0}} c_{\Lambda^{\prime}_{0}}(e_{\Lambda_{0},\Lambda_{2}})e_{\Lambda^{\prime}_{0}}
\end{equation*}
where $c_{\Lambda^{\prime}_{0}}(e_{\Lambda_{0},\Lambda_{2}})$ is given by the formula
\begin{equation*}
\int_{\overline{\rmX}^{\bullet}_{\Pa}(\rmB)}e_{\Lambda_{0},\Lambda_{2}}\cdot e_{\Lambda^{\prime}_{0}}=\int_{\overline{\rmX}^{\bullet}_{\Pa}(\rmB)}e_{\Lambda_{0}}\cdot e_{\Lambda_{2}}\cdot e_{\Lambda^{\prime}_{0}}.
\end{equation*}
If $\Lambda_{0}\neq \Lambda^{\prime}_{0}$, then $c_{\Lambda^{\prime}_{0}}(e_{\Lambda_{0},\Lambda_{2}})=0$. If $\Lambda_{0}=\Lambda^{\prime}_{0}$, then $c_{\Lambda^{\prime}_{0}}(e_{\Lambda_{0},\Lambda_{2}})=\mathrm{deg}(N_{\mathrm{DL}^{\bullet}(\Lambda_{0})}(\overline{X}^{\bullet}_{\Pa}(\rmB)))=-2p$. It follows then that $\calT_{\mathrm{ss},\{00\}}=4p^{2}(p+1)-2p(p+1)(p^{2}+1)=-2p(p+1)(p-1)^{2}$.

\subsubsection{$\calT_{\mathrm{ss},\{02\}}$} The contribution of the composite maps given by
\begin{equation*}
\begin{aligned}
\calO_{\lambda}[\rmZ_{\{0\}}(\overline{\rmB})]\xrightarrow{\mathrm{inc}^{\{0\}}_{!}}\rmH^{2}(\overline{\rmX}^{\bullet}_{\Pa}(\rmB), \calO_{\lambda}(1))\xrightarrow{\mathrm{inc}^{\ast}_{\{\mathrm{ss}\}}}\rmH^{2}(\overline{\rmX}^{\bullet\mathrm{ss.n}}_{\Pa, \{0\}}(\rmB), \calO_{\lambda}(1))&\xrightarrow{\mathrm{inc}^{\{\mathrm{ss}\}}_{!}} \rmH^{4}_{\rmc}(\overline{\rmX}^{\bullet}_{\Pa}(\rmB), \calO_{\lambda}(2))\\
&\xrightarrow{\mathrm{inc}^{\ast}_{\{2\}}}\calO_{\lambda}[\rmZ_{\{2\}}(\overline{\rmB})]\\
\end{aligned}
\end{equation*}
to the term $\calT_{\mathrm{ss},\{02\}}$ is given by $-2p\mathrm{T}_{02}$ by a similar reasoning as above. The contribution of the composite maps given by
\begin{equation*}
\begin{aligned}
\calO_{\lambda}[\rmZ_{\{0\}}(\overline{\rmB})]\xrightarrow{\mathrm{inc}^{\{0\}}_{!}}\rmH^{2}(\overline{\rmX}^{\bullet}_{\Pa}(\rmB), \calO_{\lambda}(1))\xrightarrow{\mathrm{inc}^{\ast}_{\{\mathrm{ss}\}}}\rmH^{2}(\overline{\rmX}^{\bullet\mathrm{ss.n}}_{\Pa, \{2\}}(\rmB), \calO_{\lambda}(1))&\xrightarrow{\mathrm{inc}^{\{\mathrm{ss}\}}_{!}} \rmH^{4}_{\rmc}(\overline{\rmX}^{\bullet}_{\Pa}(\rmB), \calO_{\lambda}(2))\\
&\xrightarrow{\mathrm{inc}^{\ast}_{\{2\}}}\calO_{\lambda}[\rmZ_{\{2\}}(\overline{\rmB})]\\
\end{aligned}
\end{equation*}
to the term $\calT_{\mathrm{ss},\{02\}}$ is again given by $-2p\mathrm{T}_{02}$. Therefore the total contribution gives the entry 
\begin{equation*}
\calT_{\mathrm{ss},\{02\}}=-4p\mathrm{T}_{02}.
\end{equation*}
\subsubsection{$\calT_{\mathrm{ss},\{20\}}$} The computation is the same as in the case for $\calT_{\mathrm{ss},\{02\}}$. Thus this entry is given by 
\begin{equation*}
\calT_{\mathrm{ss},\{20\}}=-4p\mathrm{T}_{20}.
\end{equation*}

\subsubsection{$\calT_{\mathrm{ss},\{22\}}$} The computation is exactly the same as in the case for $\calT_{\mathrm{ss},\{00\}}$. Thus this entry is given by 
\begin{equation*}
\calT_{\mathrm{ss},\{22\}}=-2p(p+1)(p-1)^{2}. 
\end{equation*}
This finishes the proof of Proposition \ref{ss-matrix}.
\end{proof}


\begin{proposition}
Let $(\pi, \Sigma, \fracm)$ be the datum considered as in the beginning of the section. Suppose the maximal ideal $\fracm$ appears in the support of $\calO_{\lambda}[\rmZ_{\rmH}(\overline{\rmB})]$.  Suppose that $[\alpha_{p}, \beta_{p}, p^{3}\beta^{-1}, p^{3}\alpha^{-1}_{p}]$ is the Hecke parameter of $\pi$ at $p$. Then determinant of the supersingular matrix modulo $\fracm$ is given by
\begin{equation*}
\det\phantom{.}\calT_{\mathrm{ss}/\fracm}= -4p^{2}\prod_{u\in \{\pm 1\}}[2(\alpha_{p}+p^{3}\alpha^{-1}_{p}+\beta_{p}+p^{3}\beta^{-1}_{p})-u(p^{3}-p^{2}-p+1)].
\end{equation*}
\end{proposition}
\begin{proof}
Since $\pi$ has trivial central character, then $\rmT_{p,0}$ acts trivially on $\calO_{\lambda}[\rmZ_{\rmH}(\overline{\rmB})]/\fracm$. We therefore have 
\begin{equation*}
\rmT_{20}\circ\rmT_{02}= \rmT_{02}\circ\rmT_{20}=\rmT_{p^{2}, 2}+(p+1)\rmT_{p, 1}+(p^{2}+1)(p+1)=\rmT^{2}_{p,2}
\end{equation*}
by Lemma \ref{T02T20}. 

Then we have
\begin{equation*}
\begin{aligned}
\det\phantom{.}\calT_{\mathrm{ss}/\fracm}&=4p^{2}[(p+1)^{2}(p-1)^{4}-4\mathrm{T}^{2}_{p,2}]\\
&=-4p^{2}\prod_{u\in \{\pm 1\}}[2(\alpha_{p}+p^{3}\alpha^{-1}_{p}+\beta_{p}+p^{3}\beta^{-1}_{p})-u(p^{3}-p^{2}-p+1)]\\
\end{aligned}
\end{equation*}
which finishes the proof of the proposition.
\end{proof}

\subsection{Tate classes in the intersection cohomology} We fix the datum $(\pi, \Sigma, \fracm)$ as in the beginning of this section. Let $p$ be a level raising special prime for $\pi$ of length $m$.  We make the following assumption regarding to the vanishing of the nearby cycle cohomology of the quaternionic unitary Shimura variety.
\begin{assumption}\label{vanishing}
Let $\rmH^{d}_{(\rmc)}(\overline{\rmX}_{\Pa}(\rmB), \rmR\Psi(\calO_{\lambda}))_{\fracm}$ be the usual or compact support nearby cycle cohomology of the special fiber $\overline{\rmX}_{\Pa}(\rmB)$. We assume the maximal ideal $\fracm$ satisfies the following properties.
\begin{enumerate}
\item The natural morphism $\rmH^{d}_{\rmc}(\overline{\rmX}_{\Pa}(\rmB), \rmR\Psi(\calO_{\lambda}))_{\fracm}\xrightarrow{\sim} \rmH^{d}(\overline{\rmX}_{\Pa}(\rmB), \rmR\Psi(\calO_{\lambda}))_{\fracm}$ is an isomorphism for all $d$.
\item The cohomology $\rmH^{d}_{(\rmc)}(\overline{\rmX}_{\Pa}(\rmB), \rmR\Psi(\calO_{\lambda}))_{\fracm}=0$ for $d\neq 3$, and that the middle degree cohomology 
\begin{equation*}
\rmH^{3}_{(\rmc)}(\overline{\rmX}_{\Pa}(\rmB), \rmR\Psi(\calO_{\lambda}))_{\fracm} 
\end{equation*}
is a finite free $\calO_{\lambda}$ module. 
\end{enumerate}
\end{assumption}

\begin{remark}The first assumption is usually known in literature as that of $\fracm$ being a non-Eisenstein ideal. It is expected to be satisfied if $\fracm$ is associated to a cuspidal automorphic representation of general type and a large image assumption on the residual Galois representation. The second assumption is more serious. This is reasonable assumption in light of the recent work of Caraiani--Scholze \cite{CS-compact, CS-non-compact} and recent advances by Koshikawa \cite{Kos-a, Kos-b}. Although they deal with certain unitary Shimura varieties, the methods are general enough to treat at least our cases: if one assumes that $\pi$ has an unramified type $\rmI$ component which is generic modulo $\ell$, then combining the proof of  Koshikawa \cite{Kos-b} with a semi-perversity result of the pushforward along the Hodge-Tate period map, one can potentially prove that $\fracm$ satisfies $(2)$. Such extensions of their results are recently obtained by Hamann--Lee in \cite[Corollary 5.6]{HL-vanishing}. In another direction, if we make the assumption that $\rmX_{\Pa}(\rmB)$ has good reduction at $\ell$ and that $\pi$ is ordinary at $\ell$, then the method of \cite{MT-van} could be used to show that the Assumption \ref{vanishing} is fulfilled.

\end{remark}

\begin{definition}\label{generic}
If $\fracm$ satisfies either of the two conditions below, we call $\fracm$ a generic maximal ideal.
\begin{enumerate}
\item We say $\fracm$ is generic level raising at $p$ if it is level raising special at $p$ in the sense of Definition \ref{level-raising-prime-local} and Assumption \ref{vanishing} is satisfied. Moreover we require the following condition on the Hecke parameters $[\alpha_{p}, \beta_{p}, p^{3}\beta^{-1}_{p}, p^{3}\alpha^{-1}_{p}]$ for $\pi$ at $p$:
\begin{equation*}
2(\alpha_{p}+p^{3}\alpha^{-1}_{p}+\beta_{p}+p^{3}\beta^{-1}_{p})\not \equiv \pm(p^{3}-p^{2}-p+1)\mod \lambda.
\end{equation*}

\item We say $\fracm$ is generic non-level raising at $p$ if the Hecke parameters $[\alpha_{p}, \beta_{p}, p^{3}\beta^{-1}_{p}, p^{3}\alpha^{-1}_{p}]$ for $\pi$ satisfies 
\begin{equation*}
\begin{aligned}
&\alpha_{p}+p^{3}\alpha^{-1}_{p}\not\equiv \pm (p+ p^{2})\mod \lambda\\ 
&\beta_{p}+p^{3}\beta^{-1}_{p}\not\equiv \pm (p+ p^{2})\mod \lambda\\
\end{aligned}
\end{equation*}
and Assumption \ref{vanishing} is satisfied.
\end{enumerate}
\end{definition}

\begin{theorem}\label{Tate}
Let $\fracm$ be a generic maximal ideal as in Definition \ref{generic}, we have the following statements.
\begin{enumerate}
\item The map from Construction \ref{ss-cycle-class}
\begin{equation*}
(\mathrm{inc}^{\{0\}}_{!, \fracm}+\mathrm{inc}^{\{2\}}_{!, \fracm}):\calO_{\lambda}[\rmZ_{\{0\}}(\overline{\rmB})]_{\fracm}\oplus\calO_{\lambda}[\rmZ_{\{2\}}(\overline{\rmB})]_{\fracm} \rightarrow  \rmH^{2}(\overline{\rmX}^{\square}_{\Pa}(\rmB), \calO_{\lambda}(1))_{\fracm}
\end{equation*}
is an isomorphism.
\item The map from Construction \ref{ss-cycle-class}
\begin{equation*}
(\mathrm{inc}^{\ast}_{\{0\}, \fracm}, \mathrm{inc}^{\ast}_{\{2\}, \fracm}): \rmH^{4}_{\rmc}(\overline{\rmX}^{\square}_{\Pa}(\rmB), \calO_{\lambda}(2))_{\fracm}\rightarrow \calO_{\lambda}[\rmZ_{\{0\}}(\overline{\rmB})]_{\fracm}\oplus\calO_{\lambda}[\rmZ_{\{2\}}(\overline{\rmB})]_{\fracm} 
\end{equation*}
is surjective and whose kernel is the torsion part of $\rmH^{4}_{\rmc}(\overline{\rmX}^{\square}_{\Pa}(\rmB), \calO_{\lambda}(2))_{\fracm}$.
\end{enumerate}
\end{theorem}
\begin{proof}
Suppose for the moment that $\fracm$ is generic level raising at $p$. Then the determinant of the matrix $\calT_{\mathrm{ss}}/\fracm$ is given by 
\begin{equation*}
-4p^{2}\prod_{u\in \{\pm 1\}}[2(\alpha_{p}+p^{3}\alpha^{-1}_{p}+\beta_{p}+p^{3}\beta^{-1}_{p})-u(p^{3}-p^{2}-p+1)]
\end{equation*}
This is non-zero by our definition of the generic level raising condition. By construction, $\calT_{\mathrm{ss}/\fracm}$ factors through $(\mathrm{inc}^{\{0\}}_{!, \fracm}+\mathrm{inc}^{\{2\}}_{!, \fracm})$. This immediately implies that we have an injection
\begin{equation*}
\calO_{\lambda}[\rmZ_{\{0\}}(\overline{\rmB})]_{\fracm}\oplus\calO_{\lambda}[\rmZ_{\{2\}}(\overline{\rmB})]_{\fracm}\hookrightarrow \rmH^{2}(\overline{\rmX}^{\square}_{\Pa}(\rmB), \calO_{\lambda}(1))_{\fracm}
\end{equation*}
by the Nakayama's lemma. Note that $\rmH^{2}(\overline{\rmX}^{\square}_{\Pa}(\rmB), \calO_{\lambda}(1))_{\fracm}$ is a free $\calO_{\lambda}$-module as it injects into 
\begin{equation*}
\bigoplus_{\sigma\in\rmZ_{\{1\}}(\overline{\rmB})}\rmH^{3}_{\{\sigma\}}(\overline{\rmX}_{\Pa}(\rmB), \rmR\Psi(\calO_{\lambda})(1))_{\fracm}
\end{equation*}
which follows from the bottom exact sequence in \eqref{Pic-Lef-no-c} and the generic assumption on $\fracm$ which implies the vanishing of $\rmH^{2}(\overline{\rmX}_{\Pa}(\rmB), \rmR\Psi(\calO_{\lambda}))_{\fracm}$. Therefore to finish the proof in this case, we need to show that 
\begin{equation*}
2\rank_{\calO_{\lambda}} \calO_{\lambda}[\rmZ_{\rmH}(\overline{\rmB})]_{\fracm}\geq \mathrm{rank}_{\calO_{\lambda}}\rmH^{2}(\overline{\rmX}^{\square}_{\Pa}(\rmB), \calO_{\lambda}(1))_{\fracm}.
\end{equation*}
For this purpose, let $\vec{\pi}$ be an automorphic representation of $\bfG(\mathbb{A}_{\QQ})$ of general type. It suffices to show that
\begin{equation*}
2\dim_{\overline{\QQ}_{\ell}} \overline{\QQ}_{\ell}[\rmZ_{\rmH}(\overline{\rmB})][\iota_{\ell}\vec{\pi}^{\infty pq}]\geq \dim_{\overline{\QQ}_{\ell}}\rmH^{2}(\overline{\rmX}^{\square}_{\Pa}(\rmB), \overline{\QQ}_{\ell}(1))[\iota_{\ell}\vec{\pi}^{\infty pq}].
\end{equation*}
Consider the graded piece 
\begin{equation*}
\mathrm{Gr}_{-1}(1)=\bigoplus_{\sigma\in \rmZ_{\{1\}}(\overline{\rmB})}\rmH^{3}_{\{\sigma\}}(\overline{\rmX}_{\Pa}(\rmB), \rmR\Psi(\overline{\QQ}_{\ell})(1))/\rmH^{2}(\overline{\rmX}^{\square}_{\Pa}(\rmB), \overline{\QQ}_{\ell}(1))
\end{equation*}
of the monodromy filtration of $\rmH^{3}(\overline{\rmX}_{\Pa}(\rmB), \rmR\Psi(\overline{\QQ}_{\ell})(1))$.
We can canonically identify the space 
\begin{equation*}
\bigoplus\limits_{\sigma\in \rmZ_{\{1\}}(\overline{\rmB})}\rmH^{3}_{\{\sigma\}}(\overline{\rmX}_{\Pa}(\rmB), \rmR\Psi(\overline{\QQ}_{\ell})(1))[\iota_{\ell}\vec{\pi}^{\infty p}] 
\end{equation*}
with the space $\overline{\QQ}_{\ell}[\rmZ_{\Pa}(\overline{\rmB})][\iota_{\ell}\vec{\pi}^{\infty p}]$. Note that to complete a representation $\vec{\pi}^{\infty p}$ that occurs in $\rmH^{3}(\overline{\rmX}_{\Pa}(\rmB), \rmR\Psi(\overline{\QQ}_{\ell})(1))$ to a representation $\vec{\pi}$ of $\bfG(\mathbb{A})$, $\vec{\pi}_{p}$ must be  a type $\mathrm{II}\mathrm{a}$ representation of  $\bfG(\QQ_{p})$ by Theorem \ref{JL}.  It follows from this that any representation $\vec{\pi}_{p}$ completing  $\vec{\pi}^{\infty p}$ such that $\vec{\pi}^{\infty p}$ occurs in
\begin{equation*}
\bigoplus_{\sigma\in \rmZ_{\{1\}}(\overline{\rmB})}\rmH^{3}_{\{\sigma\}}(\overline{\rmX}_{\Pa}(\rmB), \rmR\Psi(\overline{\QQ}_{\ell})(1))/\rmH^{2}(\overline{\rmX}^{\square}_{\Pa}(\rmB), \overline{\QQ}_{\ell}(1))
\end{equation*}
must also be a type $\mathrm{IIa}$ representation. On the other hand,  to complete a representation $\vec{\pi}^{\infty p}$ that occurs in $\overline{\QQ}_{\ell}[\rmZ_{\Pa}(\overline{\rmB})]\cong\bigoplus_{\sigma\in \rmZ_{\{1\}}(\overline{\rmB})}\rmH^{3}_{\{\sigma\}}(\overline{\rmX}_{\Pa}(\rmB), \rmR\Psi(\overline{\QQ}_{\ell})(1))$ to a representation $\vec{\pi}$ of $\bfG(\mathbb{A})$, $\vec{\pi}_{p}$  can be either an unramified type $\mathrm{I}$ representation or a type $\mathrm{II}\mathrm{a}$ representation. It follows then by the uniqueness of local newforms \cite[Theorem 5.6.1]{BS-newform} or \cite[Theorem 2.3.1]{Sch-Iwahori} and strong multiplicity one results in Theorem \ref{mult1}, Theorem \ref{mult1-definite} and Theorem \ref{mult1-indefinite}, those representation $\vec{\pi}_{p}$ of $\bfG(\QQ_{p})$ which completes $\vec{\pi}^{\infty p}$ such that $\vec{\pi}^{\infty p}$ occurs in $\rmH^{2}(\overline{\rmX}^{\square}_{\Pa}(\rmB), \overline{\QQ}_{\ell}(1))$ can only be an unramified type $\rmI$ representation. By the oldform principle \cite[Theorem 5.6.1]{BS-newform} or \cite[Theorem 2.3.1]{Sch-Iwahori},  we have the desired inequality
\begin{equation*}
2\dim_{\overline{\QQ}_{\ell}} \overline{\QQ}_{\ell}[\rmZ_{\rmH}(\overline{\rmB})][\iota_{\ell}\vec{\pi}^{\infty p}]\geq \dim_{\overline{\QQ}_{\ell}}\rmH^{2}(\overline{\rmX}^{\square}_{\Pa}(\rmB), \overline{\QQ}_{\ell}(1))[\iota_{\ell}\vec{\pi}^{\infty p}].
\end{equation*}

Next we show the same result under the assumption that $\fracm$ is generic non-level raising. Suppose that  $\fracm$ is not level raising special. Then the determinant of the level raising matrix $\det \calT_{\mathrm{lr}/\fracm}$ is given by 
\begin{equation*}
\det \calT_{\mathrm{lr}/\fracm}= (p)^{-2}\prod_{u=\pm1}(\alpha_{p}+p^{3}\alpha^{-1}_{p}-u p(p+1))(\beta_{p}+p^{3}\beta^{-1}_{p}-u p(p+1)).  
\end{equation*}
This is non-zero by our assumption. This immediately implies that we have an injection
\begin{equation*}
\calO_{\lambda}[\rmZ_{\{0\}}(\overline{\rmB})]_{\fracm}\oplus\calO_{\lambda}[\rmZ_{\{2\}}(\overline{\rmB})]_{\fracm}\hookrightarrow \rmH^{2}(\overline{\rmX}^{\square}_{\Pa}(\rmB), \calO_{\lambda}(1))_{\fracm}
\end{equation*}
by the Nakayama's lemma. The rest of the argument is the same as in the previous case. Finally, the second part of the proposition is dual to the first part. In particular, the map
\begin{equation*}
 (\mathrm{inc}^{\ast}_{\{0\}, \fracm}, \mathrm{inc}^{\ast}_{\{2\}, \fracm}): \rmH^{4}_{\rmc}(\overline{\rmX}^{\square}_{\Pa}(\rmB), \calO_{\lambda}(2))_{\fracm}\rightarrow \calO_{\lambda}[\rmZ_{\{0\}}(\overline{\rmB})]_{\fracm}\oplus\calO_{\lambda}[\rmZ_{\{2\}}(\overline{\rmB})]_{\fracm} 
\end{equation*}
is surjective. Using the same argument as in the first part, we see that $(\mathrm{inc}^{\ast}_{\{0\}, \fracm}, \mathrm{inc}^{\ast}_{\{2\}, \fracm})$
is an isomorphism up to the torsion in $\rmH^{4}_{\rmc}(\overline{\rmX}^{\square}_{\Pa}(\rmB), \calO_{\lambda}(2))_{\fracm}$. The second part now follows.
\end{proof}

We finish this section with several remarks about this theorem. First of all, this result can be understood as the statement that the intersection cohomology of the special fiber $\mathrm{IH}^{2}(\overline{\rmX}_{\Pa}(\rmB), \calO_{\lambda}(1))_{\fracm}$ is generated by Tate cycles coming from the supersingular locus. This result along with \cite[Proposition 6.3.1]{LTXZZ} can be seen as an analogue of the main result of \cite{XZ} in the bad reduction case for low degree cohomology groups localized at generic maximal ideals. It seems that these cohomology groups always correspond to oldforms. It would be of great interest to formulate some predictions in a similar manner of \cite{XZ}. Secondly, we have established the following analogue of the Ihara's lemma in the setting of definite quaternionic unitary groups of degree $2$. 
\begin{corollary}
There is an injection 
\begin{equation*}
\mathbf{Ih}_{\fracm}: \calO_{\lambda}[\rmZ_{\{0\}}(\overline{\rmB})]_{\fracm}\oplus\calO_{\lambda}[\rmZ_{\{2\}}(\overline{\rmB})]_{\fracm}\xhookrightarrow{\phantom{aaaa}} \calO_{\lambda}[\rmZ_{\{1\}}(\overline{\rmB})]_{\fracm}.
\end{equation*}
\end{corollary}
\begin{proof}
We have shown in the above proof that we have an isomorphism
\begin{equation*}
\calO_{\lambda}[\rmZ_{\{0\}}(\overline{\rmB})]_{\fracm}\oplus\calO_{\lambda}[\rmZ_{\{2\}}(\overline{\rmB})]_{\fracm}\xrightarrow{(\mathrm{inc}^{\{0\}}_{!, \fracm}+\mathrm{inc}^{\{2\}}_{!, \fracm})} \rmH^{2}(\overline{\rmX}^{\square}_{\Pa}(\rmB), \calO_{\lambda}(1))_{\fracm}
\end{equation*}
and an injection
\begin{equation*}
\rmH^{2}(\overline{\rmX}^{\square}_{\Pa}(\rmB), \calO_{\lambda}(1))_{\fracm} \xhookrightarrow{\phantom{aa}\beta\phantom{aa}}  \bigoplus_{\sigma\in Z_{\{1\}}(\overline{\rmB})}\rmH^{3}_{\{\sigma\}}(\overline{\rmX}_{\Pa}(\rmB), \rmR\Psi(\calO_{\lambda})(1))
=\calO_{\lambda}[\rmZ_{\{1\}}(\overline{\rmB})]_{\fracm}.
\end{equation*}
The composite of these two maps give rise to an injection which is the desired
\begin{equation*}
\mathbf{Ih}_{\fracm}: \calO_{\lambda}[\rmZ_{\{0\}}(\overline{\rmB})]_{\fracm}\oplus\calO_{\lambda}[\rmZ_{\{2\}}(\overline{\rmB})]_{\fracm}\xhookrightarrow{\phantom{aaaa}} \calO_{\lambda}[\rmZ_{\{1\}}(\overline{\rmB})]_{\fracm}.
\end{equation*}
\end{proof}
Note that there is no natural degeneracy map from neither $\calO_{\lambda}[\rmZ_{\{0\}}(\overline{\rmB})]_{\fracm}$ nor $\calO_{\lambda}[\rmZ_{\{2\}}(\overline{\rmB})]_{\fracm}$ to 
$\calO_{\lambda}[\rmZ_{\{1\}}(\overline{\rmB})]_{\fracm}$. Instead, these maps should be the analogues of the level raising operators in the sense of \cite{BS-newform}. Finally suppose that $\fracm$ is generic level raising at $p$, we will show that $\rmH^{4}(\overline{\rmX}_{\Pa}(\rmB), \calO_{\lambda}(2))_{\fracm}$ is torsion free following the method of \cite{LTXZZ} using Galois deformation argument.

\section{Galois deformation rings and cohomology of Shimura varieties}
In this section, we recall the main results in \cite{Wangd} where we use the method of \cite{LTXZZa} to show that the cohomology of the quaternionic unitary Shimura varieties are free over certain universal deformation rings. These results are necessary as we lack of torsion vanishing results on the cohomology of the special fibers of these quaternionic unitary Shimura varieties.

\subsection{Galois representation for $\GSp_{4}$}
First we recall some results on associating Galois representations to cuspidal automorphic representations of $\GSp_{4}(\mathbb{A})$ and the local-global compatibility of Langlands correspondence. Let $\pi$ be a cuspidal automorphic representation of $\GSp_{4}(\mathbb{A})$ with trivial central character.  We assume that $\pi$ has weight $(k_{1}, k_{2})$ and trivial central character.  In this article, we will only consider discrete automorphic representation of $\GSp_{4}(\mathbb{A})$ that is of general type. This means there is a cuspidal automorphic representation $\Pi$ of $\GL_{4}(\mathbb{A})$ of symplectic type such that for each place $v$ of $\QQ$, the $L$-parameter obtained from $\mathrm{rec}_{\mathrm{GT}}(\pi_{v})$ by composing with the embedding $\GSp_{4}\hookrightarrow\GL_{4}$ is precisely $\mathrm{rec}_{\GL_{4}}(\Pi_{v})$ which is the Langlands parameter attached to $\Pi_{v}$. Here $\Pi$ is of symplectic type if the partial $L$-function $L^{\rmS}(s, \Pi, \wedge^{2})$ has a pole at $s=1$ for any finite set $\rmS$ of places of $\QQ$. In this case, we say $\Pi$ is the transfer of $\pi$.  

We now recall some general results of Mok and Sorensen on the existence of Galois representations attached to $\pi$. This theorem is very general and we temporarily drop the assumption that $\pi$ is of general type.

\begin{theorem}\label{Galois}
Suppose $\pi$ is a cuspidal automorphic representation of $\GSp_{4}(\mathbb{A})$ of weight $(k_{1}, k_{2})$ where $k_{1}\geq k_{2}\geq 3$ and $k_{1} \equiv k_{2}\mod 2$. Suppose that $\pi$ has trivial central character. Let $\ell\geq 5$ be a fixed prime.Then there is a continuous semisimple representation 
\begin{equation*}
\rho_{\pi, \iota_{\ell}}: \rmG_{\QQ}\rightarrow \GSp_{4}(\overline{\QQ}_{\ell})
\end{equation*}
satisfying the following properties.
\begin{enumerate}
\item $\rmc\circ\rho_{\pi, \iota_{\ell}}=\epsilon^{-3}_{\ell}$;
\item For each finite place $v$, we have
\begin{equation*}
\mathrm{WD}(\rho_{\pi,\iota_{\ell}}\vert_{\rmG_{\QQ_{v}}})^{\rmF-\mathrm{ss}}\cong \mathrm{rec}_{\mathrm{GT}}(\pi_{v}\otimes\vert\rmc\vert^{-3/2})^{\mathrm{ss}};
\end{equation*}
\item The local representation $\rho_{\pi,\iota_{\ell}}\vert_{\rmG_{\QQ_{\ell}}}$ is de Rham with Hodge--Tate weights 
\begin{equation*}
(2-\frac{k_{1}+k_{2}}{2}, -\frac{k_{1}-k_{2}}{2}, 1+\frac{k_{1}-k_{2}}{2}, -1+\frac{k_{1}+k_{2}}{2}).
\end{equation*}
\item If $\pi$ is unramified at $\ell$, then $\rho_{\pi,\iota_{\ell}}\vert_{\rmG_{\QQ_{v}}}$ is moreover crystalline at $\ell$.
\item If $\rho_{\pi, \iota_{\ell}}$ is irreducible, then for all finite place $v$, $\rho_{\pi, \iota_{\ell}}\vert_{\rmG_{\QQ_{v}}}$ is pure.  
\end{enumerate}
\end{theorem}

\begin{remark}
Several remarks about this theorem are in order. The construction of the Galois representation in this case is given in \cite{La-Siegel, La-SiegelII, Wei-Galois, Tay-Siegel} using the Langlands--Kottwitz method. The statements about local-global compatibilities are mostly proved in \cite{Sor-local-global} and completed in \cite{Mok-GSp}. In fact the authors use the strong transfer of cuspidal automorphic representation from $\GSp_{4}$ to $\GL_{4}$ to construct the desired Galois representation in the totally real field case. The Harish-Chandra parameter $(\mu_{1}, \mu_{2})$ of $\pi$ used more often has the following relation with the weight of $\pi$ used in this theorem: $k_{1}=\mu_{1}+1$ and $k_{2}=\mu_{2}+2$. In this article, we will be interested in the case $\pi$ is of general type and has trivial central character and weights $(3, 3)$.  In our normalization, this is the case when these representation appear in the cohomology of quaternionic unitary Shimura varieties with trivial coeffficient.
\end{remark}

\begin{definition}\label{strong-field}
Let $\pi$ be a cuspidal automorphic representation of $\GSp_{4}(\mathbb{A})$ as in the previous remark. We say a number field $E\subset \CC$ is a strong coefficient field of $\pi$ if for every prime $\lambda$ of $E$ there exits a continuous homomorphism 
\begin{equation*}
\rho_{\pi,\lambda}: \rmG_{\QQ}\rightarrow \GSp_{4}(E_{\lambda})
\end{equation*}
up to conjugation, such that for every isomorphism $\iota_{\ell}:\CC\xrightarrow{\sim}\overline{\QQ}_{\ell}$ inducing the prime $\lambda$, the representations $\rho_{\pi,\lambda}\otimes_{E_{\lambda}}\overline{\QQ}_{\ell}$ and $\rho_{\pi,\iota_{\ell}}$ are conjugate.
\end{definition}

Let $p$ be a place at  which $\pi$ is an unramified type $\rmI$ representation. Then the Frobenius eigenvalues of $\rho_{\pi, \lambda}(\mathrm{Frob}_{p})$ agree with the Hecke parameter $[\alpha_{p}, \beta_{p}, \gamma_{p}, \delta_{p}]$ of $\pi_{p}$ by Theorem \ref{Galois} $(2)$. We will always make the following assumption on $\rho_{\pi, \lambda}$.
\begin{assumption}\label{irrd}
The Galois representation $\rho_{\pi, \lambda}: \rmG_{\QQ}\rightarrow \GSp_{4}(E_{\lambda})$ is residually absolutely irreducible. 
\end{assumption}

The above assumption allows us to define the residual Galois representation
\begin{equation}\label{residual}
\overline{\rho}_{\pi,\lambda}: \rmG_{\QQ}\rightarrow \GSp_{4}(k)
\end{equation}
which is unique up to conjugation. Note it also follows that $\pi$ is necessarily of general type if assumption \ref{irrd} is effective.

\subsection{Galois deformation rings} We summarize the main results in \cite{Wangd}. Suppose we are given a Galois representation $\overline{\rho}: \rmG_{\QQ}\rightarrow\GSp_{4}(k)$. We will denote the restriction $\overline{\rho}\vert_{\rmG_{\QQ_{v}}}: \rmG_{\QQ_{v}}\rightarrow\GSp_{4}(k)$ by $\overline{\rho}_{v}$.  Let $\Sigma_{\min}$ and $\Sigma_{\lr}$ be two sets of  non-archimedean places of $\QQ$ away from $\ell$ and such that
\begin{itemize}
\item $q$ is contained in $\Sigma_{\lr}$;
\item $\Sigma_{\min}$, $\Sigma_{\lr}$ and $\{p\}$ are mutually disjoint;
\item for every $v\in\Sigma_{\lr}\cup\{p\}$, the prime $v$ satisfies $\ell\nmid v(v^{2}-1)$.
\end{itemize}
\begin{definition}\label{rigid}
For a Galois representation $\overline{\rho}: \rmG_{\QQ}\rightarrow\GSp_{4}(k)$ and $\psi=\epsilon^{-3}_{\ell}$, we say the pair $\overline{\rho}$ is rigid for $(\Sigma_{\mathrm{min}}, \Sigma_{\mathrm{lr}})$ if the followings are satisfied:
\begin{itemize}
\item For every $v\in\Sigma_{\min}$, every lifting of $\overline{\rho}_{v}$ is minimally ramified in the sense of \cite[Definition 3.4]{Wangd};
\item For every $v\in\Sigma_{\lr}$, the generalized eigenvalues of $\overline{\rho}_{v}(\phi_{v})$ contain the pair $\{v^{-1}, v^{-2}\}$ exactly once;
\item At $v=\ell$, $\overline{\rho}_{v}$ is regular Fontaine-Laffaille crystalline as in \cite[Definition 3.10]{Wangd};
\item For $v\not\in\Sigma_{\min}\cup\Sigma_{\lr}\cup\{l\}$, the representation $\overline{\rho}_{v}$ is unramified.
\end{itemize}
\end{definition}
Suppose that $\overline{\rho}$ is rigid for $(\Sigma_{\min}, \Sigma_{\lr})$. We consider a global deformation problem in the sense of \cite[Definition 2.2]{Wangd} which has the form
\begin{equation*}\
\calS^{\ast}=(\overline{\rho}, \psi, \Sigma_{\min}\cup\Sigma_{\lr}\cup\{p\}\cup\{\ell\}, \{\calD_{v}\}_{v\in\Sigma_{\min}\cup\Sigma_{\lr}\cup\{p\}\cup\{\ell\}})
\end{equation*}
where  $\ast=\{\ram, \unr, \mix\}$. Here local deformation problems are defined by following rules:
\begin{itemize}
\item For $v\in\Sigma_{\min}$, $\calD_{v}$ is the local deformation problem $\calD^{\min}_{v}$ classifying all the minimal ramified liftings defined in \cite[Defintion 3.4]{Wangd};
\item For $v\in\Sigma_{\lr}\cup\{q\}$, $\calD_{v}$ is the local deformation problem $\calD^{\mathrm{ram}}_{v}$ defined as in \cite[Definition 3.6 (2)]{Wangd} classifying certain ramified liftings;
\item For $v=\ell$, $\calD_{v}$ is the local deformation problem $\calD^{\mathrm{FL}}_{v}$ defined as in \cite[Definition 3.1.1]{Wangd} classifying regular Fontaine--Laffaille crystalline liftings;
\item For $v=p$, depending on $\ast\in\{\ram, \unr, \mix\}$, we have
\begin{itemize}
\item $\calD_{v}$ is the local deformation problem $\calD^{\unr}_{v}$ which is given by \cite[Definition 3.6 (2)]{Wangd} when $\ast=\unr$;
\item $\calD_{v}$ is the local deformation problem $\calD^{\ram}_{v}$ which is given by \cite[Definition 3.6 (3)]{Wangd} when $\ast=\ram$;
\item $\calD_{v}$ is the local deformation problem $\calD^{\mix}_{v}$ defined in \cite[Definition 3.6 (1)]{Wangd} when $\ast=\mix$;
\end{itemize}
\end{itemize}
For each global deformation problem $\calS^{\ast}$ with $\ast=\{\ram, \unr, \mix\}$, we have its corresponding universal deformation ring denoted by 
$\rmR^{\ast}=\rmR^{\mathrm{univ}}_{\calS^{\ast}}$ for $\ast=\{\ram, \unr, \mix\}$.
\begin{construction}\label{indef-level}
When $\rmB$ is indefinite, we choose the level $\rmK=\rmK_{\Pa}$ for the Shimura variety $\Sh(\rmB, \rmK_{\Pa})$ in the following way:
\begin{itemize}
\item For $v\not\in\Sigma_{\lr}\cup\Sigma_{\min}$ or $v=\ell$, then $\rmK_{v}$ is hyperspecial;
\item For $v\in\Sigma_{\lr}-\{p, q\}$, then $\rmK_{v}$ is the paramodular subgroup of $\GSp_{4}(\ZZ_{v})$;
\item For $v\in\Sigma_{\min}$, then $\rmK_{v}$ is contained in the pro-$v$ Iwahori subgroup $\Iw_{1}(v)$;
\item For $v=p, q$, then  $\rmK_{v}$ is the paramodular subgroup $\Pa^{\rmD}$ resp. $\Pa^{\rmD^{\prime }}$ of $\GU_{2}(\rmD)$ resp. of $\GU_{2}(\mathbf{D})$.
\end{itemize}

Let $\Sigma_{\Pa}=\Sigma_{\min}\cup\Sigma_{\lr}\cup\{p, q\}$. We denote by $\bfT^{\ram}$ the image of $\TT^{\Sigma_{\Pa}}$ in 
\begin{equation*}
\End_{\calO_{\lambda}}(\rmH^{3}(\overline{\rmX}_{\Pa}(\rmB), \rmR\Psi(\calO_{\lambda}))).
\end{equation*}
\end{construction}
\begin{construction}\label{def-level}
We choose the level $\overline{\rmK}=\rmK_{\rmH}$ for the Shimura variety $\Sh(\overline{\rmB}, \rmK_{\rmH})$ when $\overline{\rmB}$ is definite in the following way:
\begin{itemize}
\item For $v\not\in\Sigma_{\lr}\cup\Sigma_{\min}$ or $v=\ell$ or $v=p$, then $\overline{\rmK}_{v}$ is hyperspecial;
\item For $v\in\Sigma_{\lr}$, then $\overline{\rmK}_{v}$ is the paramodular subgroup of $\GSp(\ZZ_{v})$;
\item  For $v\in\Sigma_{\min}$, then $\overline{\rmK}_{v}$ is contained in the pro-$v$ Iwahori subgroup $\Iw_{1}(v)$;
\item For $v=q$, then  $\overline{\rmK}_{v}$ is the paramodular subgroup $\Pa^{\mathbf{D}}$ of $\GU_{2}(\mathbf{D})$.
\end{itemize}
Let $\Sigma_{\rmH}=\Sigma_{\min}\cup\Sigma_{\lr}\cup\{q\}$. We denote by $\bfT^{\unr}$ the image of $\TT^{\Sigma_{\rmH}}$ in 
$\End_{\calO_{\lambda}}(\calO_{\lambda}[\rmZ_{\rmH}(\overline{\rmB})])$. 
\end{construction}

Let $\overline{\rho}$ be a Galois representation 
\begin{equation*}
\overline{\rho}: \rmG_{\QQ}\rightarrow \GSp_{4}(k)
\end{equation*}
and $\fracm$ be the maximal ideal corresponding to $\overline{\rho}$ in $\TT^{\Sigma}$ for $\Sigma\in \{\Sigma_{\Pa}, \Sigma_{\rmH}\}$, in the sense that the characteristic polynomial $\det(\rmX-\overline{\rho}(\Frob_{v}))$ is congruent to the Hecke polynomial $\mathcal{Q}_{v}(\rmX)$ given by
\begin{equation*}
\rmX^{4}-\rmT_{v,2}\rmX^{3}+(v\rmT_{v,1}+(v^{3}+v)\rmT_{v, 0})\rmX^{2}-v^{3}\rmT_{v, 0}\rmT_{v, 2}\rmX+v^{6}\rmT^{2}_{v, 0}
\end{equation*}
modulo $\fracm$ for each $v$ not in $\Sigma$. We will also assume that the image of $\overline{\rho}$ contains $\GSp_{4}(\FF_{\ell})$ and that
\begin{equation*}
\rmH^{3}_{\rmc}(\overline{\rmX}_{\Pa}(\rmB), \rmR\Psi(\calO_{\lambda}))_{\fracm}=\rmH^{3}(\overline{\rmX}_{\Pa}(\rmB), \rmR\Psi(\calO_{\lambda}))_{\fracm}.
\end{equation*}

The following theorem is the main result of \cite{Wangd} proved using the Taylor--Wiles method.
\begin{theorem}\label{Free}
Let $\overline{\rho}$ and $\fracm$ be given as above. Suppose the following assumptions hold:
\begin{itemize}
\item[(D1)] the image of $\overline{\rho}(\rmG_{\QQ})$ contains $\GSp_{4}(\FF_{\ell})$;
\item[(D2)] $\overline{\rho}$ is rigid for $(\Sigma_{\min}, \Sigma_{\lr})$;
\item[(D3)]  For every finite set $\Sigma^{\prime}$ of nonarchimedean places containing $\Sigma$ and every open compact subgroup $\rmK^{\prime}$ of $\rmK$ satisfying $\rmK^{\prime}_{v}=\rmK_{v}$ for all $v\not\in\Sigma^{\prime}$, we have 
\begin{equation*}
\rmH^{d}(\overline{\mathrm{X}}_{\Pa}(\rmB),\rmR\Psi (k))_{\fracm^{\prime}}=0
\end{equation*}
for $d\neq 3$ where $\fracm^{\prime}=\fracm\cap \TT^{\Sigma^{\prime}}$.
\end{itemize}
Then the following holds true.
\begin{enumerate}
\item If $\bfT^{\unr}_{\fracm}$ is non-zero, then we have an isomorphim of complete intersection rings: 
\begin{equation*}
\rmR^{\unr}=\bfT^{\unr}_{\fracm}
\end{equation*}
and $\calO_{\lambda}[\rmZ_{\rmH}(\overline{\rmB})]_{\fracm}$ is a finite free $\bfT^{\unr}_{\fracm}$-module.
\item  If $\bfT^{\ram}_{\fracm}$ is non-zero, then we have an isomorphim of complete intersection rings: 
\begin{equation*}
\rmR^{\ram}=\bfT^{\ram}_{\fracm}
\end{equation*}
 and $\rmH^{3}(\overline{\rmX}_{\Pa}(\rmB), \rmR\Psi(\calO_{\lambda}))_{\fracm}$ is a finite free $\bfT^{\ram}_{\fracm}$-module.
\end{enumerate}
\end{theorem}

\begin{remark}
Let $\pi$ be a cuspidal automorphic representation of general type. Assume that $\overline{\rho}_{\pi, \lambda}(\rmG_{\QQ})$ contains $\GSp_{4}(\FF_{\ell})$ for sufficiently large $\ell$, then we have shown that $\overline{\rho}_{\pi, \lambda}$ is indeed rigid, see \cite[Theorem 5.5]{Wangd}.  
\end{remark}

\section{Arithmetic level raising and applications}
Now we are ready to prove our main result on arithmetic level raising on the quaternionic Shimura varieties studied in the previous sections. First we recall the setting we are in. Let $\pi$ be a cuspidal automorphic representation of $\GSp_{4}(\mathbb{A})$ with weight $(3,3)$ and trivial central character. Suppose that $\pi$ is of general type and $E$ be the coefficient field of $\pi$. Let $\Sigma_{\pi}$ be the minimal set of finite places outside of which $\pi$ is unramified. We fix an isomorphism $\iota_{\ell}: \CC\xrightarrow{\sim}\overline{\QQ}_{\ell}$ which induces a place $\lambda$ in $E$ over $\ell$. We fix the following datum.
\begin{itemize}
\item A finite set $\Sigma_{\min}$ of places of $\QQ$ contained in $\Sigma_{\pi}$;
\item A finite set $\Sigma_{\lr}$ of places of $\QQ$ which is disjoint from $\Sigma_{\min}$ and contained in $\Sigma_{\lr}$;
\item A finite set of places $\Sigma$ containing $\Sigma_{\min}\cup\Sigma_{\lr}$ which is away from $\ell$; 
\item We have, by Construction \ref{Hecke-Algebra} $(2)$, a homomorphism
\begin{equation*}
\phi_{\pi}:\TT^{\Sigma}\rightarrow \calO_{E}
\end{equation*}
and its $\lambda$-adic avatar $\phi_{\pi, \lambda}:\TT^{\Sigma}\rightarrow \calO_{\lambda}$ for the valuation ring $\calO_{\lambda}\subset E_{\lambda}$. 
\item We can associate to it a Galois representation
\begin{equation*}
\rho_{\pi, \lambda}: \mathrm{G}_{\QQ}\rightarrow \GSp_{4}(E_{\lambda})
\end{equation*}
which we assume is residually absolutely irreducible as in Assumption \ref{irrd}.
\item Let $p$ be a prime which is level raising special of length $m$ for $\pi$ in the sense of Definition \ref{level-raise-prime}. 
\item Let $q$ be a prime such that $\ell\nmid q(q^{2}-1)$. Suppose that $\pi$ is of type $\mathrm{II}\rma$ and $q$ is contained in $\Sigma_{\lr}$. In this case $\pi_{q}$ admits Jacquet--Langlands transfer to a representation of the group $\GU_{2}(\mathbf{D})$ where $\mathbf{D}$ which is the quaternion division algebra over $\QQ_{q}$.
\item We choose the level $\rmK_{\Pa}$ and $\Sigma_{\Pa}$ as in Construction \ref{indef-level} for the Shimura varieity $\Sh(\rmB, \rmK_{\Pa})$ whose integral model is $\rmX_{\Pa}(\rmB)$ and whose special fiber is $\overline{\rmX}_{\Pa}(\rmB)$; and choose the level  $\rmK_{\rmH}$ and $\Sigma_{\rmH}$ as in Construction \ref{def-level}  for the Shimura set $\Sh(\overline{\rmB}, \rmK_{\rmH})$ which is also denoted by $\rmZ_{\rmH}(\overline{\rmB})$.
\item We introduced two ideals of $\TT^{\Sigma\cup\{p\}}$ by
\begin{equation}
\begin{aligned}
&\fracm=\TT^{\Sigma\cup\{p\}}\cap \ker(\TT^{\Sigma}\xrightarrow{\phi_{\pi,\lambda}}\calO_{\lambda}\rightarrow\calO_{\lambda}/\lambda)\\
&\fracn_{m}=\TT^{\Sigma\cup\{p\}}\cap \ker(\TT^{\Sigma}\xrightarrow{\phi_{\pi,\lambda}}\calO_{\lambda}\rightarrow\calO_{\lambda}/\lambda^{m})\\
\end{aligned}
\end{equation}
as in Construction \ref{Hecke-Algebra} $(3)$ and we assume that $\fracm$ is generic level raising at $p$ in the sense of \ref{generic}.
\end{itemize}
\begin{construction}
We define the map 
\begin{equation*}
\nabla: \calO_{\lambda}[\rmZ_{\rmH}(\overline{\rmB})]\oplus \calO_{\lambda}[\rmZ_{\rmH}(\overline{\rmB})]\rightarrow \calO_{\lambda}[\rmZ_{\rmH}(\overline{\rmB})]
\end{equation*}
by the formula
\begin{equation*}
\begin{aligned}
&\nabla(x, y)=[\rmT^{-1}_{p, 0}\rmT_{p,1}+(p+1)(p^{2}+1)]x-(p+1)\rmT_{20}y \\
\end{aligned}
\end{equation*}
for any $(x, y)\in \calO_{\lambda}[\rmZ_{\rmH}(\overline{\rmB})]\oplus\calO_{\lambda}[\rmZ_{\rmH}(\overline{\rmB})]$.  

We also write 
\begin{equation*}
\nabla_{\fracm}:  \calO_{\lambda}[\rmZ_{\rmH}(\overline{\rmB})]_{\fracm}\oplus \calO_{\lambda}[\rmZ_{\rmH}(\overline{\rmB})]_{\fracm}\rightarrow \calO_{\lambda}[\rmZ_{\rmH}(\overline{\rmB})]_{\fracm} 
\end{equation*}
for the localization of $\nabla$ at $\fracm$
\end{construction}

Note that the composition of the level raising matrix
\begin{equation*}
\calT_{\lr, \fracm}: \calO_{\lambda}[\rmZ_{\rmH}(\overline{\rmB})]_{\fracm}\oplus \calO_{\lambda}[\rmZ_{\rmH}(\overline{\rmB})]_{\fracm}\rightarrow \calO_{\lambda}[\rmZ_{\rmH}(\overline{\rmB})]_{\fracm}\oplus \calO_{\lambda}[\rmZ_{\rmH}(\overline{\rmB})]_{\fracm}
\end{equation*}
with the map $\nabla_{\fracm}$ is given by sending  $(x, y)\in \calO_{\lambda}[\rmZ_{\rmH}(\overline{\rmB})]\oplus \calO_{\lambda}[\rmZ_{\rmH}(\overline{\rmB})]$ to $\det \calT_{\lr,\fracm}\cdot x$ where we recall that
\begin{equation*}
\det\calT_{\lr, \fracm}= (\rmT^{-1}_{p,0}\rmT_{p,1}+(p+1)(p^{2}+1))^{2}-(p+1)^{2}\rmT_{20}\circ\rmT_{02} 
\end{equation*}
by Proposition \ref{lr-matrix}.

\begin{proposition}\label{surj}
Let $p$ be a level raising special prime for $\pi$ and $\fracm$ be the maximal ideal as above. We have a surjective morphism
\begin{equation*}
\mathrm{F}_{-1}\mathrm{H}^{1}(\mathrm{I}_{\QQ_{p^{2}}}, \mathrm{H}^{3}_{\rmc}(\overline{\rmX}_{\Pa}(\rmB), \calO_{\lambda}(2))_{\fracm})\twoheadrightarrow 
\calO_{\lambda}[\rmZ_{\rmH}(\rmB)]_{\fracm}/\det\calT_{\lr, \fracm}
\end{equation*}
whose kernel can be identified with the torsion part of $\rmH^{4}_{\rmc}(\overline{\rmX}_{\Pa}(\rmB), \calO_{\lambda}(2))_{\fracm}$.
\end{proposition}

\begin{proof}
By Proposition \ref{F-1}, we have an isomorphism
\begin{equation*}
\mathrm{F}_{-1}\mathrm{H}^{1}(\mathrm{I}_{\QQ_{p^{2}}}, \rmH^{3}_{\rmc}(\overline{\rmX}_{\Pa}(\rmB), \rmR\Psi(\calO_{\lambda})(2))_{\fracm})\xrightarrow{\sim} \Coker(\alpha_{\rmc}(2)\circ \rmN^{-1}_{\Sigma}\circ \beta_{\rmc}(1))_{\fracm}
\end{equation*}
By Assumption \ref{vanishing}, $\rmH^{3}_{\rmc}(\overline{\rmX}_{\Pa}(\rmB), \rmR\Psi(\calO_{\lambda})(2))_{\fracm}\cong \rmH^{3}(\overline{\rmX}_{\Pa}(\rmB), \rmR\Psi(\calO_{\lambda})(2))_{\fracm}$ and hence we can use the unbalanced Picard--Lefschetz formula to calculate $\mathrm{F}_{-1}\mathrm{H}^{1}(\mathrm{I}_{\QQ_{p^{2}}}, \rmH^{3}_{\rmc}(\overline{\rmX}_{\Pa}(\rmB), \rmR\Psi(\calO_{\lambda})(2))_{\fracm})$ which means we have
\begin{equation*}
\mathrm{F}_{-1}\mathrm{H}^{1}(\mathrm{I}_{\QQ_{p^{2}}}, \rmH^{3}_{\rmc}(\overline{\rmX}_{\Pa}(\rmB), \rmR\Psi(\calO_{\lambda})(2))_{\fracm})\xrightarrow{\sim} \Coker(\alpha_{\rmc}(2)\circ \rmN^{-1}_{\Sigma}\circ \beta(1))_{\fracm}.\\
\end{equation*}

On the other hand, by Proposition \ref{Tate} we have an isomorphism
\begin{equation*}
(\mathrm{inc}^{\{0\}}_{!, \fracm}+\mathrm{inc}^{\{2\}}_{!, \fracm}):\calO_{\lambda}[\rmZ_{\{0\}}(\overline{\rmB})]_{\fracm}\oplus\calO_{\lambda}[\rmZ_{\{2\}}(\overline{\rmB})]_{\fracm} \rightarrow  \rmH^{2}(\overline{\rmX}^{\square}_{\Pa}(\rmB), \calO_{\lambda}(1))_{\fracm}
\end{equation*}
and a surjection 
\begin{equation*}
(\mathrm{inc}^{\ast}_{\{0\}, \fracm}, \mathrm{inc}^{\ast}_{\{2\}, \fracm}): \rmH^{4}_{\rmc}(\overline{\rmX}^{\square}_{\Pa}(\rmB), \calO_{\lambda}(2))_{\fracm}\rightarrow \calO_{\lambda}[\rmZ_{\{0\}}(\overline{\rmB})]_{\fracm}\oplus\calO_{\lambda}[\rmZ_{\{2\}}(\overline{\rmB})]_{\fracm} 
\end{equation*}
whose kernel is given by the torsion part of $\rmH^{4}_{\rmc}(\overline{\rmX}^{\square}_{\Pa}(\rmB), \calO_{\lambda}(2))_{\fracm}$.

It follows that we have a surjection
\begin{equation*}
\mathrm{F}_{-1}\mathrm{H}^{1}(\mathrm{I}_{\QQ_{p^{2}}}, \rmH^{3}_{\rmc}(\overline{\rmX}_{\Pa}(\rmB), \rmR\Psi(\calO_{\lambda})(2))_{\fracm})\rightarrow \frac{\calO_{\lambda}[\rmZ_{\{0\}}(\overline{\rmB})]_{\fracm}\oplus\calO_{\lambda}[\rmZ_{\{2\}}(\overline{\rmB})]_{\fracm}}{\calT_{\lr,\fracm}(\calO_{\lambda}[\rmZ_{\{0\}}(\overline{\rmB})]_{\fracm}\oplus\calO_{\lambda}[\rmZ_{\{2\}}(\overline{\rmB})]_{\fracm})}.
\end{equation*}
The kernel of this map is given by
\begin{equation*}
\frac{\rmH^{4}_{\rmc}(\overline{\rmX}^{\square}_{\Pa}(\rmB), \calO_{\lambda}(2))^{\mathrm{tor}}_{\fracm}}{\mathrm{Im}(\alpha_{\rmc}(2)\circ \rmN^{-1}_{\Sigma}\circ \beta(1))_{\fracm}\cap \rmH^{4}_{\rmc}(\overline{\rmX}^{\square}_{\Pa}(\rmB), \calO_{\lambda}(2))^{\mathrm{tor}}_{\fracm}}.
\end{equation*}
Composing this with the map $\nabla_{\fracm}:\calO_{\lambda}[\rmZ_{\rmH}(\overline{\rmB})]_{\fracm}\oplus\calO_{\lambda}[\rmZ_{\rmH}(\overline{\rmB})]_{\fracm}\rightarrow \calO_{\lambda}[\rmZ_{\rmH}(\overline{\rmB})]_{\fracm}$, we therefore obtain a map
\begin{equation*}
\mathrm{F}_{-1}\mathrm{H}^{1}(\mathrm{I}_{\QQ_{p^{2}}}, \rmH^{3}_{\rmc}(\overline{\rmX}_{\Pa}(\rmB), \rmR\Psi(\calO_{\lambda})(2))_{\fracm})\rightarrow \calO_{\lambda}[\rmZ_{\rmH}(\overline{\rmB})]_{\fracm}/\det\calT_{\lr,\fracm}.
\end{equation*}
which is surjective by Nakayama's lemma.

It remains to show that 
\begin{equation*}
\mathrm{Im}(\alpha_{\rmc}(2)\circ \rmN^{-1}_{\Sigma}\circ \beta(1))_{\fracm} \cap \rmH^{4}(\overline{\rmX}^{\square}_{\Pa}(\rmB), \calO_{\lambda}(2))^{\mathrm{tor}}_{\fracm} 
\end{equation*}
is trivial. By \cite[Proposition 4.2.2(1)]{Nekovar}, the group $\mathrm{F}_{-1}\mathrm{H}^{1}(\mathrm{I}_{\QQ_{p^{2}}}, \rmH^{3}_{\rmc}(\overline{\rmX}_{\Pa}(\rmB), \rmR\Psi(E_{\lambda})(2))_{\fracm})$ vanishes. It follows that the $\calO_{\lambda}$-rank of $\mathrm{Im}(\alpha_{\rmc}(2)\circ \rmN^{-1}_{\Sigma}\circ \beta(1))_{\fracm}$ is the same as the $\calO_{\lambda}$-rank of $\rmH^{4}_{c}(\overline{\rmX}^{\square}_{\Pa}(\rmB), \calO_{\lambda}(2))_{\fracm}$. Poincare duality for $\overline{\rmX}^{\square}_{\Pa}(\rmB)$ implies that we have 
\begin{equation*}
\dim_{E_{\lambda}} \rmH^{4}_{\rmc}(\overline{\rmX}^{\square}_{\Pa}(\rmB), E_{\lambda}(2))_{\fracm}=\dim_{E_{\lambda}}  \rmH^{2}(\overline{\rmX}^{\square}_{\Pa}(\rmB), E_{\lambda}(1))_{\fracm}.
\end{equation*} 
Since the source $\rmH^{2}(\overline{\rmX}^{\square}_{\Pa}(\rmB), \calO_{\lambda}(1))_{\fracm}$ of $\mathrm{Im}(\alpha_{\rmc}(2)\circ \rmN^{-1}_{\Sigma}\circ \beta(1))_{\fracm}$  is free over $\calO_{\lambda}$, the intersection  
\begin{equation*}
\mathrm{Im}(\alpha_{\rmc}(2)\circ \rmN^{-1}_{\Sigma}\circ \beta(1))_{\fracm} \cap \rmH^{4}(\overline{\rmX}^{\square}_{\Pa}(\rmB), \calO_{\lambda}(2))^{\mathrm{tor}}_{\fracm}
\end{equation*}
has to be trivial. The proposition is proved. 
\end{proof}
Now we are ready to state and prove our main result on arithmetic level raising for the quaternionic Shimura variety studied in this article. We first recall a list of running assumptions below.
\begin{itemize}
\item We assume that $\rmH^{i}_{(\rmc)}(\overline{\rmX}_{\Pa}(\rmB), \rmR\Psi(\calO_{\lambda}))_{\fracm}=0$ for $i\neq 3$, and that 
\begin{equation*}
\rmH^{3}_{\rmc}(\overline{\rmX}_{\Pa}(\rmB), \rmR\Psi(\calO_{\lambda}))_{\fracm}\xrightarrow{\sim} \rmH^{3}(\overline{\rmX}_{\Pa}(\rmB), \rmR\Psi(\calO_{\lambda}))_{\fracm}
\end{equation*}
is a finite free $\calO_{\lambda}$ module. This is Assumption \ref{vanishing}.
\item We assume that for $\pi$, the associated Galois representation 
\begin{equation*}
\rho_{\pi, \lambda}:\rmG_{\QQ}\rightarrow \GSp_{4}(E_{\lambda})
\end{equation*}
is residually absolutely irreducible. This is Assumption \ref{irrd}.
\end{itemize}

\begin{theorem}[Arithmetic level raising for $\GSp_{4}$]\label{arithmetic-level-raising}
Suppose that $\pi$ is a cuspidal automorphic representation of $\GSp_{4}(\mathbb{A})$ of general type. We assume that the assumptions in \ref{vanishing}, \ref{irrd} recalled above are effective. Let $p$ be a level raising special prime for $\pi$ of depth $m$. We assume further that
\begin{enumerate}
\item $\overline{\rho}_{\pi,\lambda}$ is rigid for $(\Sigma_{\min}, \Sigma_{\lr})$ as in Definition \ref{rigid};
\item The image of $\overline{\rho}_{\pi,\lambda}(\rmG_{\QQ})$ contains $\GSp_{4}(\FF_{\ell})$;
\item $\fracm$ is in the support of $\calO_{\lambda}[\rmZ_{\rmH}(\overline{\rmB})]$ and is generic level raising.
\end{enumerate}

Then the following holds.
\begin{enumerate}
\item There is an isomorphism 
\begin{equation*}
\rmH^{1}_{\sing}(\QQ_{p^{2}}, \rmH^{3}_{\rmc}(\overline{\rmX}_{\Pa}(\rmB), \rmR\Psi(\calO_{\lambda}(2)))_{\fracm})\xrightarrow{\sim}\calO_{\lambda}[\rmZ_{\rmH}(\overline{\rmB})]_{\fracm}/\det\calT_{\lr}.
\end{equation*}

\item The above isomorphism induces an isomorphism
\begin{equation*}
\rmH^{1}_{\sing}(\QQ_{p^{2}}, \rmH^{3}_{\rmc}(\overline{\rmX}_{\Pa}(\rmB), \rmR\Psi(\calO_{\lambda}/\lambda^{m}(2)))_{\fracm})\xrightarrow{\sim}\calO_{\lambda}/\lambda^{m}[\rmZ_{\rmH}(\overline{\rmB})]_{\fracm}.
\end{equation*}
\end{enumerate}
\end{theorem}

\subsection{Proof of Theorem \ref{arithmetic-level-raising}}
Consider the universal deformation problem 
\begin{equation*}
\calS^{\ast}=(\overline{\rho}_{\pi,\lambda}, \psi, \Sigma_{\min}\cup\Sigma_{\lr}\cup\{p\}\cup\{l\}, \{\calD_{v}\}_{v\in\Sigma_{\min}\cup\Sigma_{\lr}\cup\{p\}\cup\{l\}})
\end{equation*}
where  $\ast=\{\ram, \unr, \mix\}$. Recall that $\Spf\phantom{.}\rmR^{\ram}$ and $\Spf\phantom{.}\rmR^{\unr}$ are subspaces of $\Spf\phantom{.}\rmR^{\mix}$. Thus we have surjections
\begin{equation*}
\begin{aligned}
&\rmR^{\mix}\twoheadrightarrow \rmR^{\ram}\\
&\rmR^{\mix}\twoheadrightarrow \rmR^{\unr}.\\
\end{aligned}
\end{equation*}
we define $\rmR^{\mathrm{cong}}=\rmR^{\unr}\otimes_{\rmR^{\mix}}\rmR^{\ram}$. We have a universal lifting 
\begin{equation*}
\rho^{\mix}: \rmG_{\QQ}\rightarrow \GSp_{4}(\rmR^{\mix}).
\end{equation*}
It also induces two representations
\begin{equation*}
\begin{aligned}
&\rho^{\ram}: \rmG_{\QQ}\rightarrow \GSp_{4}(\rmR^{\ram})\\
&\rho^{\unr}: \rmG_{\QQ}\rightarrow \GSp_{4}(\rmR^{\unr})\\
\end{aligned}
\end{equation*}
by the above surjections.

Denote by $\rmP_{\QQ_{p}}$ be the maximal closed subgroup of the inertia subgroup $\rmI_{\QQ_{p}}\subset \rmG_{\QQ_{p}}$ of pro-order coprime to $\ell$. Then $\rmG_{\QQ_{p}}/\rmP_{\QQ_{p}}\cong t^{\ZZ_{\ell}}\rtimes \phi^{\widehat{\ZZ}}_{p}$ is a $p$-tame group. By definition, $\rho^{\mix}$ is trivial on $\rmP_{\QQ_{p}}$. Let $\overline{\rmv}_{0}$ and $\overline{\rmv}_{1}$ be eigenvectors in $k^{4}=(\calO_{\lambda}/\lambda)^{4}$ for $\rho^{\mix}(\phi^{2}_{p})$ with eigenvalues $p^{-4}$ and $p^{-2}$, respectively. By Hensel's lemma, they lift to $\rmv_{0}$ and $\rmv_{1}$ in $(\rmR^{\mix})^{\oplus 4}$ for $\rho^{\mix}(\phi^{2}_{p})$ with eigenvalues $\rms_{0}$ and $\rms_{1}$ lifting $p^{-4}$ and $p^{-2}$. Let $\mathbf{x}\in\rmR^{\mix}$ be the unique element such that $\rho^{\mix}(t)\rmv_{1}=\mathbf{x}\rmv_{0}+\rmv_{1}$. It follows from the relation
\begin{equation*}
\rho^{\mix}(\phi^{2}_{p})\rho^{\mix}(t)\rho^{\mix}(\phi^{-2}_{p})=\rho^{\mix}(t)^{p^{2}}
\end{equation*}
that $\mathbf{x}(\rms_{0}-p^{-4})=0$. By the definitions of $\rmR^{\ast}$ for $\ast=\{\ram, \unr, \mix\}$, we have
\begin{equation*}
\begin{aligned}
&\rmR^{\unr}=\rmR^{\mix}/(\bfx)\\
&\rmR^{\ram}=\rmR^{\mix}/(\rms_{0}-p^{-4})\\
&\rmR^{\mathrm{cong}}=\rmR^{\mix}/(\bfx, \rms_{0}-p^{-4}).\\
\end{aligned}
\end{equation*}

By Theorem \ref{Free},  $\calO_{\lambda}[\rmZ_{\rmH}(\overline{\rmB})]_{\fracm}$ is a free over $\bfT^{\unr}_{\fracm}$ and we have an isomorphism $\rmR^{\unr}\cong \bfT_{\fracm}$. Thus $\calO_{\lambda}[\rmZ_{\rmH}(\overline{\rmB})]_{\fracm}$ is a free $\rmR^{\unr}$-module of rank $d_{\unr}$. Note that the characteristic polynomial of $\rho^{\mix}(\phi^{2}_{p})$ can written as $(\rmX-\rms_{0})(\rmX-p^{-6}\rms^{-1}_{0})\rmQ(\rmX)$ where $\rmQ(\rmX)$ is a polynomial in $\rmR^{\unr}[\rmX]$ whose reduction in $\calO_{\lambda}/\lambda[\rmX]=k[\rmX]$ does not contain $p^{-2}$ and $p^{-4}$ as its roots. By Proposition \ref{det-lr} and the equations
\begin{equation*}
\begin{aligned}
&(\alpha_{p}+p^{3}\alpha^{-1}_{p}- p(p+1))(\alpha_{p}+p^{3}\alpha^{-1}_{p}+p(p+1))=\alpha^{2}_{p}+p^{6}\alpha^{-2}_{p}- p^{2}(p^{2}+1),\\
&(\beta_{p}+p^{3}\beta^{-1}_{p}-p(p+1))(\beta_{p}+p^{3}\beta^{-1}_{p}+p(p+1))=\beta^{2}_{p}+p^{6}\beta^{-2}_{p}- p^{2}(p^{2}+1),\\
\end{aligned}
\end{equation*}
we have 
\begin{equation*}
(\det\calT_{\lr, \fracm})\cdot\calO_{\lambda}[\rmZ_{\rmH}(\overline{\rmB})]_{\fracm}=(\rms_{0}-p^{-4})\cdot\calO_{\lambda}[\rmZ_{\rmH}(\overline{\rmB})]_{\fracm}.
\end{equation*}
Thus we have
\begin{equation*}
\calO_{\lambda}[\rmZ_{\rmH}(\overline{\rmB})]_{\fracm}/\det\phantom{.}\calT_{\lr,\fracm}\cong \calO_{\lambda}[\rmZ_{\rmH}(\overline{\rmB})]_{\fracm}\otimes_{\rmR^{\unr}}\rmR^{\mathrm{cong}}
\end{equation*}
and it follows that $\calO_{\lambda}[\rmZ_{\rmH}(\overline{\rmB})]_{\fracm}/\det\phantom{.}\calT_{\lr,\fracm}$ is a free $\rmR^{\mathrm{cong}}$-module of rank $d_{\unr}$.

Since we have a surjection 
\begin{equation*}
\mathrm{F}_{-1}\mathrm{H}^{1}(\mathrm{I}_{\QQ_{p^{2}}}, \mathrm{H}^{3}_{\rmc}(\overline{\rmX}_{\Pa}(\rmB), \rmR\Psi(\calO_{\lambda}(2)))_{\fracm})\twoheadrightarrow 
\calO_{\lambda}[\rmZ_{\rmH}(\rmB)]_{\fracm}/\det\calT_{\lr, \fracm}
\end{equation*}
by Proposition \ref{surj}, we have $\bfT^{\ram}_{\fracm}$ is non-zero.  Moreover $\mathrm{H}^{3}_{\rmc}(\overline{\rmX}_{\Pa}(\rmB), \rmR\Psi(\calO_{\lambda}))_{\fracm}$ is a free $\rmR^{\ram}$-module by Theorem \ref{Free} $(2)$, say of rank $d_{\ram}$. Consider the $\rmR^{\ram}$-module
\begin{equation*}
\rmH:=\Hom_{\rmG_{\QQ}}((\rmR^{\ram})^{\oplus4},\rmH^{3}_{\rmc}(\overline{\rmX}_{\Pa}(\rmB), \rmR\Psi(\calO_{\lambda}))_{\fracm})
\end{equation*}
where $\rmG_{\QQ}$ acts on $(\rmR^{\ram})^{\oplus4}$ via $\rho^{\ram}$ which is free of rank $d_{\ram}$. It follows this that we have an isomorphism
\begin{equation*}
\mathrm{H}^{3}_{\rmc}(\overline{\rmX}_{\Pa}(\rmB), \rmR\Psi(\calO_{\lambda}))_{\fracm}\xrightarrow{\sim}\rmH\otimes_{\rmR^{\ram}}(\rmR^{\ram})^{\oplus4}.
\end{equation*}
It follows from that
\begin{equation*}
\rmH^{1}_{\sing}(\QQ_{p^{2}}, \rmH^{3}_{\rmc}(\overline{\rmX}_{\Pa}(\rmB), \rmR\Psi(\calO_{\lambda}(2)))_{\fracm})\xrightarrow{\sim}\rmH\otimes_{\rmR^{\ram}}\rmH^{1}_{\sing}(\QQ_{p^{2}},(\rmR^{\ram})^{\oplus4}(2))
\end{equation*}
We still denote by $\rmv_{0}$ and $\rmv_{1}$ the projections of $\rmv_{0}$ and $\rmv_{1}$ in $(\rmR^{\ram})^{\oplus4}$ of $\rmG_{\QQ}$-modules. Then a straightforward computation shows that
\begin{equation*}
\rmH^{1}_{\sing}(\QQ_{p^{2}}, (\rmR^{\ram})^{\oplus4}(2))\xrightarrow{\sim} \rmR^{\ram}\rmv_{0}/\rmx\rmv_{0}\cong \rmR^{\mathrm{cong}}. 
\end{equation*}
Hence we have 
\begin{equation*}
\rmH^{1}_{\sing}(\QQ_{p^{2}}, \rmH^{3}_{\rmc}(\overline{\rmX}_{\Pa}(\rmB), \rmR\Psi(\calO_{\lambda}(2)))_{\fracm})
\end{equation*}
is a free $\rmR^{\mathrm{cong}}$-module of rank $d_{\ram}$.
\begin{lemma}
We have an equality 
\begin{equation*}
d_{\unr}=d_{\ram}.
\end{equation*}
\end{lemma}
\begin{proof}
Let $\eta^{\unr}\in\Spec\phantom{.}\rmR^{\unr}[1/\ell](\overline{\QQ}_{\ell})$ be in the support of $\calO_{\lambda}[\rmZ_{\rmH}(\overline{\rmB})]_{\fracm}$ which gives rise to an automorphic representation $\pi^{\unr}$ of $\GSp_{4}(\mathbb{A})$ such that $\overline{\rho}_{\pi, \lambda}\otimes_{k}\overline{\FF}_{\ell}$ and $\rho_{\pi^{\unr},\iota_{\ell}}$ are residually isomorphic. Then we have
\begin{equation*}
d_{\unr}=\dim_{\overline{\QQ}_{\ell}}\overline{\QQ}_{\ell}[\rmZ_{\rmH}(\overline{\rmB})][\iota_{\ell}\phi_{\pi^{\unr}}].
\end{equation*}
Similarly, let $\eta^{\ram}\in\Spec\phantom{.}\rmR^{\ram}[1/\ell](\overline{\QQ}_{\ell})$ be in the support of $\rmH^{3}_{\rmc}(\overline{\rmX}_{\Pa}(\rmB), \rmR\Psi(\calO_{\lambda}(1)))_{\fracm}$ which gives rise to an automorphic representation $\pi^{\ram}$ of $\GSp_{4}(\mathbb{A}_{\QQ})$ such that $\rho_{\pi^{\ram},\iota_{\ell}}$ is residually isomorphic to $\overline{\rho}_{\pi, \lambda}\otimes_{k}\overline{\FF}_{\ell}$. Then we have
\begin{equation*}
4d_{\ram}=\dim_{\overline{\QQ}_{\ell}} \rmH^{3}_{\rmc}(\overline{\rmX}_{\Pa}(\rmB), \rmR\Psi(\overline{\QQ}_{\ell}))[\iota_{\ell}\phi_{\pi^{\ram}}].
\end{equation*}
Note that $\pi^{\unr}$ and $\pi^{\ram}$ are necessarily of general type. By considering the transfer of $\pi^{\unr}$ and $\pi^{\ram}$ to $\GL_{4}(\mathbb{A})$, we have $d_{\ram}=d_{\unr}$ by \cite[Lemma 6.4.2]{LTXZZ}.
\end{proof}
\begin{myproof}{Theorem}{\ref{arithmetic-level-raising}} Now we can prove our main theorem on arithmetic level raising. We have proved that
we have a surjection
\begin{equation*}
\mathrm{F}_{-1}\mathrm{H}^{1}(\mathrm{I}_{\QQ_{p^{2}}}, \mathrm{H}^{3}_{\rmc}(\overline{\rmX}_{\Pa}(\rmB), \rmR\Psi(\calO_{\lambda}(2)))_{\fracm})\twoheadrightarrow 
\calO_{\lambda}[\rmZ_{\rmH}(\rmB)]_{\fracm}/\det\phantom{.}\calT_{\lr, \fracm}
\end{equation*}
where the target is free of rank $d_{\unr}$ over $\rmR^{\mathrm{cong}}$. On the other hand, we have an injection
\begin{equation*}
\mathrm{F}_{-1}\mathrm{H}^{1}(\mathrm{I}_{\QQ_{p^{2}}}, \mathrm{H}^{3}_{\rmc}(\overline{\rmX}_{\Pa}(\rmB), \rmR\Psi(\calO_{\lambda}(2)))_{\fracm})\hookrightarrow \rmH^{1}_{\sing}(\QQ_{p^{2}}, \rmH^{3}_{\rmc}(\overline{\rmX}_{\Pa}(\rmB), \rmR\Psi(\calO_{\lambda}(2))_{\fracm})
\end{equation*}
and the target is free of rank $d_{\ram}$ over $\rmR^{\mathrm{cong}}$. Since $d_{\unr}=d_{\ram}$, we have the desired isomorphism
\begin{equation}\label{first-isom}
\rmH^{1}_{\sing}(\QQ_{p^{2}}, \rmH^{3}_{\rmc}(\overline{\rmX}_{\Pa}(\rmB), \rmR\Psi(\calO_{\lambda}(2)))_{\fracm})\xrightarrow{\sim}\calO_{\lambda}[\rmZ_{\rmH}(\overline{\rmB})]_{\fracm}/\det\phantom{.}\calT_{\lr}.
\end{equation}
This finishes the proof of the first part of the theorem.

For the second part, by the freeness of $\rmH^{3}_{\rmc}(\overline{\rmX}_{\Pa}(\rmB), \rmR\Psi(\calO_{\lambda}(2)))_{\fracm}$ as a $\bfT^{\ram}_{\fracm}$-module. It follows that the natural map
\begin{equation*}
\mathrm{H}^{1}(\mathrm{I}_{\QQ_{p^{2}}}, \mathrm{H}^{3}_{\rmc}(\overline{\rmX}_{\Pa}(\rmB), \rmR\Psi(\calO_{\lambda}(2)))_{\fracm})/\fracn\xrightarrow{\sim} \mathrm{H}^{1}(\mathrm{I}_{\QQ_{p^{2}}}, \mathrm{H}^{3}_{\rmc}(\overline{\rmX}_{\Pa}(\rmB), \rmR\Psi(\calO_{\lambda}(2)))_{\fracm}/\fracn)
\end{equation*}
is an isomorphism. We also have the short exact sequence as in Proposition \ref{mondromy-exact-seq}
\begin{equation*}
\begin{aligned}
0&\rightarrow \mathrm{F}_{-1}\mathrm{H}^{1}(\mathrm{I}_{\QQ_{p^{2}}}, \rmH^{3}_{\mathrm{c}}(\overline{\rmX}_{\Pa}(\rmB), \rmR\Psi(\calO_{\lambda}(2)))_{\fracm}) \rightarrow \rmH^{1}(\rmI_{\QQ_{p^{2}}}, \rmH^{3}_{\mathrm{c}}(\overline{\rmX}_{\Pa}(\rmB), \rmR\Psi(\calO_{\lambda})(2))_{\fracm}) \\
&\rightarrow \rmH^{3}_{\mathrm{c}}(\overline{\rmX}_{\Pa}(\rmB), \rmR\Psi(\calO_{\lambda}(2)))_{\fracm}/ \rmF_{-1}\rmH^{3}_{\mathrm{c}}(\overline{\rmX}_{\Pa}(\rmB), \rmR\Psi(\calO_{\lambda}(2)))_{\fracm}\rightarrow 0\\
\end{aligned}
\end{equation*}
which is split by the $\rmG_{\FF_{p^{2}}}$-action. And it follows that
\begin{equation*}
\mathrm{F}_{-1}\mathrm{H}^{1}(\mathrm{I}_{\QQ_{p^{2}}}, \rmH^{3}_{\mathrm{c}}(\overline{\rmX}_{\Pa}(\rmB), \rmR\Psi(\calO_{\lambda}(2)))_{\fracm})/\fracn\xrightarrow{\sim} \mathrm{F}_{-1}\mathrm{H}^{1}(\mathrm{I}_{\QQ_{p^{2}}}, \rmH^{3}_{\mathrm{c}}(\overline{\rmX}_{\Pa}(\rmB), \rmR\Psi(\calO_{\lambda}(2)))_{\fracm}/\fracn)
\end{equation*}
is an isomorphism.  Since we have shown above that there is an isomorphism
\begin{equation*}
\mathrm{F}_{-1}\mathrm{H}^{1}(\mathrm{I}_{\QQ_{p^{2}}}, \mathrm{H}^{3}_{\rmc}(\overline{\rmX}_{\Pa}(\rmB), \rmR\Psi(\calO_{\lambda}(2)))_{\fracm})\xrightarrow{\sim} \rmH^{1}_{\sing}(\QQ_{p^{2}}, \rmH^{3}_{\rmc}(\overline{\rmX}_{\Pa}(\rmB), \rmR\Psi(\calO_{\lambda}(2))_{\fracm}),
\end{equation*}
it follows that 
\begin{equation*}
\mathrm{H}^{1}_{\sing}({\QQ_{p^{2}}}, \rmH^{3}_{\mathrm{c}}(\overline{\rmX}_{\Pa}(\rmB), \rmR\Psi(\calO_{\lambda}(2)))_{\fracm})/\fracn\xrightarrow{\sim} \mathrm{H}^{1}_{\sing}({\QQ_{p^{2}}}, \rmH^{3}_{\mathrm{c}}(\overline{\rmX}_{\Pa}(\rmB), \rmR\Psi(\calO_{\lambda}(2)))_{\fracm}/\fracn).
\end{equation*}
On the other hand, we have $\bfT^{\ram}/\fracn\cong \calO_{\lambda}/\lambda^{m}$ and thus 
\begin{equation*}
\mathrm{H}^{1}_{\sing}({\QQ_{p^{2}}}, \rmH^{3}_{\mathrm{c}}(\overline{\rmX}_{\Pa}(\rmB), \rmR\Psi(\calO_{\lambda}(2)))_{\fracm}/\fracn)\cong \mathrm{H}^{1}_{\sing}({\QQ_{p^{2}}}, \rmH^{3}_{\mathrm{c}}(\overline{\rmX}_{\Pa}(\rmB), \rmR\Psi(\calO_{\lambda}/\lambda^{m}(2)))_{\fracm})
\end{equation*}
by the freeness of $ \rmH^{3}_{\mathrm{c}}(\overline{\rmX}_{\Pa}(\rmB), \rmR\Psi(\calO_{\lambda}(2)))_{\fracm}$ as $\bfT^{\ram}_{\fracm}$-module. 
Finally the second claim follows from the observation that 
\begin{equation*}
\begin{aligned}
\calO_{\lambda}[\rmZ_{\rmH}(\overline{\rmB})]_{\fracm}/\det\phantom{.}\calT_{\lr}&\cong  \calO_{\lambda}[\rmZ_{\rmH}(\overline{\rmB})]_{\fracm}\otimes_{\rmR^{\unr}}\rmR^{\mathrm{cong}}\\
&\cong  \calO_{\lambda}[\rmZ_{\rmH}(\overline{\rmB})]_{\fracm}\otimes_{\rmR^{\unr}}\bfT^{\ram}_{\fracm}/\fracn\\
&\cong  \calO_{\lambda}[\rmZ_{\rmH}(\overline{\rmB})]_{\fracm}\otimes \calO_{\lambda}/\lambda^{m}\\
\end{aligned}
\end{equation*}
and the isomorphism \ref{first-isom}.
\end{myproof}

\begin{corollary}\label{no-torsion}
We maintain the assumptions in Theorem \ref{arithmetic-level-raising}. 
\begin{enumerate}
\item Then we have that
\begin{equation*}
\rmH^{2}(\overline{\rmX}^{\square}_{\Pa}(\rmB), \calO_{\lambda})_{\fracm}\cong\mathrm{IH}^{2}(\overline{\rmX}_{\Pa}(\rmB), \calO_{\lambda})_{\fracm}
\end{equation*}
is a torsion free $\calO_{\lambda}$-module.
\item Then we have that
\begin{equation*}
\rmH^{4}_{\rmc}(\overline{\rmX}^{\square}_{\Pa}(\rmB), \calO_{\lambda})_{\fracm}\cong\mathrm{IH}^{4}_{\rmc}(\overline{\rmX}_{\Pa}(\rmB), \calO_{\lambda})_{\fracm}
\end{equation*}
is a torsion free $\calO_{\lambda}$-module. In particular, the map from Construction \ref{ss-cycle-class}
\begin{equation*}
(\mathrm{inc}^{\ast}_{\{0\}, \fracm}, \mathrm{inc}^{\ast}_{\{2\}, \fracm}): \rmH^{4}_{\rmc}(\overline{\rmX}^{\square}_{\Pa}(\rmB), \calO_{\lambda}(2))_{\fracm}\rightarrow \calO_{\lambda}[\rmZ_{\{0\}}(\overline{\rmB})]_{\fracm}\oplus\calO_{\lambda}[\rmZ_{\{2\}}(\overline{\rmB})]_{\fracm} 
\end{equation*}
is an isomorphism.
\end{enumerate}
\end{corollary}
\begin{proof}
Recall we have proved that the kernel of the natural map
\begin{equation*}
\mathrm{F}_{-1}\mathrm{H}^{1}(\mathrm{I}_{\QQ_{p^{2}}}, \mathrm{H}^{n}_{\rmc}(\overline{\rmX}_{\Pa}(\rmB), \calO_{\lambda}(2))_{\fracm})\twoheadrightarrow 
\calO_{\lambda}[\rmZ_{\rmH}(\rmB)]_{\fracm}/\det\phantom{.}\calT_{\lr, \fracm}
\end{equation*}
can be identified with the torsion part of $\rmH^{4}(\overline{\rmX}_{\Pa}(\rmB), \calO_{\lambda}(2))_{\fracm}$. We have shown in the above proof that this surjection is an isomorphism and the result follows.
\end{proof}
\begin{corollary}
We maintain the assumptions in Theorem \ref{arithmetic-level-raising}. Then we have that
\begin{equation*}
\mathrm{H}^{3}_{(\rmc)}(\overline{\rmX}_{\Pa}(\rmB), \calO_{\lambda})_{\fracm}
\end{equation*}
is a torsion free $\calO_{\lambda}$-module.
Moreover the intersection cohomology 
\begin{equation*}
\mathrm{IH}^{3}(\overline{\rmX}_{\Pa}(\rmB), \calO_{\lambda})_{\fracm}
\end{equation*}
is a torsion free $\calO_{\lambda}$-module.
\end{corollary}
\begin{proof}
For the first part,  we have an injection 
\begin{equation*}
\rmH^{3}_{(\rmc)}(\overline{\rmX}_{\Pa}(\rmB), \calO_{\lambda})_{\fracm}\hookrightarrow \rmH^{3}_{(\rmc)}(\overline{\rmX}_{\Pa}(\rmB), \rmR\Psi(\calO_{\lambda}))_{\fracm}
\end{equation*}
and $\rmH^{3}_{(\rmc)}(\overline{\rmX}_{\Pa}(\rmB), \rmR\Psi(\calO_{\lambda}))_{\fracm}$ is a torsion free $\calO_{\lambda}$-module by our assumption. Thus we have that $\rmH^{3}_{(\rmc)}(\overline{\rmX}_{\Pa}(\rmB), \calO_{\lambda})_{\fracm}$ is also torsion-free. 

Next, we show the cohomology group $\mathrm{IH}^{3}(\overline{\rmX}_{\Pa}(\rmB), \calO_{\lambda})_{\fracm}$ is torsion free. We consider the long exact sequence 
\begin{equation*}
\cdots\rightarrow \mathrm{IH}^{i}(\overline{\rmX}_{\Pa}(\rmB), \calO_{\lambda})_{\fracm}\rightarrow \mathrm{IH}^{i}(\overline{\rmX}_{\Pa}(\rmB), \calO_{\lambda})_{\fracm}\rightarrow \mathrm{IH}^{i}(\overline{\rmX}_{\Pa}(\rmB), \calO_{\lambda}/\lambda)_{\fracm}\rightarrow\cdots
\end{equation*}
induced by the short exact sequence 
\begin{equation*}
0\rightarrow \calO_{\lambda}\rightarrow\calO_{\lambda}\rightarrow \calO_{\lambda}/\lambda\rightarrow0.
\end{equation*}
We need to show the injection 
\begin{equation*}
\IH^{2}(\overline{\rmX}_{\Pa}(\rmB), \calO_{\lambda})_{\fracm}/\lambda\hookrightarrow \IH^{2}(\overline{\rmX}_{\Pa}(\rmB), \calO_{\lambda}/\lambda)_{\fracm}
\end{equation*}
is in fact surjective. For this, note that we have established that 
\begin{equation*}
\rmH^{2}(\overline{\rmX}^{\square}_{\Pa}(\rmB), \calO_{\lambda})_{\fracm}\cong\IH^{2}(\overline{\rmX}_{\Pa}(\rmB), \calO_{\lambda})_{\fracm} 
\end{equation*}
is torsion free and thus $\dim_{k} \IH^{2}(\overline{\rmX}_{\Pa}(\rmB), \calO_{\lambda})_{\fracm}/\lambda=\dim_{E_{\lambda}} \IH^{2}(\overline{\rmX}_{\Pa}(\rmB), E_{\lambda})_{\fracm}$. But we have 
\begin{equation*}
\begin{aligned}
&\dim_{k} \IH^{2}(\overline{\rmX}_{\Pa}(\rmB), \calO_{\lambda}/\lambda(1))_{\fracm}=\dim_{k} \IH^{4}_{\rmc}(\overline{\rmX}_{\Pa}(\rmB), \calO_{\lambda}/\lambda(2))_{\fracm};\\
&\dim_{E_{\lambda}} \IH^{2}(\overline{\rmX}_{\Pa}(\rmB), E_{\lambda}(1))_{\fracm}=\dim_{E_{\lambda}} \IH^{4}_{\rmc}(\overline{\rmX}_{\Pa}(\rmB), E_{\lambda}(2))_{\fracm}.\\
\end{aligned}
\end{equation*}
by generalized Poincare duality \cite[Proposition 2.1.17]{BBD}.  By considering the same long exact sequence and reasoning as above but applied to $\IH^{4}_{\rmc}(\overline{\rmX}_{\Pa}(\rmB), \calO_{\lambda}(2))_{\fracm}$, we have
\begin{equation}\label{dim-k-E}
\dim_{k} \IH^{4}_{\rmc}(\overline{\rmX}_{\Pa}(\rmB), \calO_{\lambda})_{\fracm}/\lambda=\dim_{k} \IH^{4}_{\rmc}(\overline{\rmX}_{\Pa}(\rmB), \calO_{\lambda}/\lambda)_{\fracm}
\end{equation}
since $\IH^{5}_{\rmc}(\overline{\rmX}_{\Pa}(\rmB), \calO_{\lambda})_{\fracm}\cong \rmH^{5}_{\rmc}(\overline{\rmX}_{\Pa}(\rmB), \calO_{\lambda})_{\fracm}\cong \rmH^{5}_{\rmc}(\overline{\rmX}_{\Pa}(\rmB), \rmR\Psi(\calO_{\lambda}))_{\fracm}$
vanishes. By the torsion-freeness of $\IH^{4}_{\rmc}(\overline{\rmX}_{\Pa}(\rmB), \calO_{\lambda})_{\fracm}$ and \eqref{dim-k-E}, we have 
\begin{equation*}
\dim_{k} \IH^{4}_{\rmc}(\overline{\rmX}_{\Pa}(\rmB), \calO_{\lambda}/\lambda)_{\fracm}=\dim_{k} \IH^{4}_{\rmc}(\overline{\rmX}_{\Pa}(\rmB), \calO_{\lambda})_{\fracm}/\lambda=\dim_{E_{\lambda}} \IH^{4}_{\rmc}(\overline{\rmX}_{\Pa}(\rmB), E_{\lambda})_{\fracm}.
\end{equation*}
It follows that 
\begin{equation*}
\dim_{k} \IH^{2}(\overline{\rmX}_{\Pa}(\rmB), \calO_{\lambda}/\lambda))_{\fracm}=\dim_{E_{\lambda}} \IH^{2}(\overline{\rmX}_{\Pa}(\rmB), E_{\lambda})_{\fracm}.
\end{equation*}
and hence $\dim_{k} \IH^{2}(\overline{\rmX}_{\Pa}(\rmB), \calO_{\lambda})_{\fracm}/\lambda=\dim_{k} \IH^{2}(\overline{\rmX}_{\Pa}(\rmB), \calO_{\lambda}/\lambda)_{\fracm}$. Then we can conclude that
\begin{equation*}
\IH^{2}(\overline{\rmX}_{\Pa}(\rmB), \calO_{\lambda}(1))_{\fracm}/\lambda\hookrightarrow \IH^{2}(\overline{\rmX}_{\Pa}(\rmB), \calO_{\lambda}/\lambda)_{\fracm} 
\end{equation*}
is in fact surjective and we are done.
\end{proof}

\begin{corollary}
We maintain the assumptions in Theorem \ref{arithmetic-level-raising}. Then the compactly supported intersection cohomology
\begin{equation*}
\mathrm{IH}^{3}_{\rmc}(\overline{\rmX}_{\Pa}(\rmB), \calO_{\lambda})_{\fracm}
\end{equation*}
is a torsion free $\calO_{\lambda}$-module.
 \end{corollary}
\begin{proof}
Consider the specialization exact sequence for $\rmH^{3}_{(\rmc)}(\overline{\rmX}_{\Pa}(\rmB), \rmR\Psi(\calO_{\lambda}))_{\fracm}$:
\begin{equation*}
\begin{tikzcd}
0 \arrow[r] & \rmH^{3}_{\rmc}(\overline{\rmX}_{\Pa}(\rmB), \calO_{\lambda})_{\fracm} \arrow[r] \arrow[d] & \rmH^{3}_{\rmc}(\overline{\rmX}_{\Pa}(\rmB), \rmR\Psi(\calO_{\lambda}))_{\fracm} \arrow[r] \arrow[d] & \bigoplus\limits_{\sigma\in \rmZ_{\{1\}}\overline{\rmB})}\rmR^{3}\Phi_{\sigma}(\calO_{\lambda})_{\fracm} \arrow[d, equal] \\
0 \arrow[r] & \rmH^{3}(\overline{\rmX}_{\Pa}(\rmB), \calO_{\lambda})_{\fracm} \arrow[r]           & \rmH^{3}(\overline{\rmX}_{\Pa}(\rmB), \rmR\Psi(\calO_{\lambda}))_{\fracm}\arrow[r]           & \bigoplus\limits_{\sigma\in \rmZ_{\{1\}}\overline{\rmB})}\rmR^{3}\Phi_{\sigma}(\calO_{\lambda})_{\fracm}.       
\end{tikzcd}
\end{equation*}
Since the middle vertical map is an isomorphism by our assumption, the first vertical map is also an isomorphism. It follows from this that
\begin{equation*}
\frac{\rmH^{3}_{\mathrm{c}}(\overline{\rmX}_{\Pa}(\rmB), \calO_{\lambda})_{\fracm}}{\bigoplus\limits_{\sigma\in \rmZ_{\{1\}}(\overline{\rmB})}\rmH^{3}_{\{\sigma\}}(\overline{\rmX}_{\Pa}(\rmB), \rmR\Psi(\calO_{\lambda}))_{\fracm}}\cong
\frac{\rmH^{3}(\overline{\rmX}_{\Pa}(\rmB), \calO_{\lambda})_{\fracm}}{\bigoplus\limits_{\sigma\in \rmZ_{\{1\}}(\overline{\rmB})}\rmH^{3}_{\{\sigma\}}(\overline{\rmX}_{\Pa}(\rmB), \rmR\Psi(\calO_{\lambda}))_{\fracm}}
\end{equation*}
But the righthand side is isomorphic to $\mathrm{IH}^{3}(\overline{\rmX}_{\Pa}(\rmB), \calO_{\lambda})_{\fracm}$ and hence is torsion-free by the previous lemma and therefore the lefthand side which is isomorphic to $\mathrm{IH}^{3}_{\rmc}(\overline{\rmX}_{\Pa}(\rmB), \calO_{\lambda})_{\fracm}$ is also torsion free.
\end{proof}

\subsection{Level raising for automorphic representation of $\GSp_{4}$} Let $\pi$ be a cuspidal automorphic representation of $\GSp_{4}(\mathbb{A})$ of general type with weight $(3,3)$ and trivial central character. Let $p$ be a level raising special prime for $\pi$ of length $m$. From the arithmetic level raising theorem proved in the last subsection, we can deduce a level raising theorem for $\pi$ which can be applied to deduce level lowering theorems in \cite{Wange}. This is the content of this subsection. 

Suppose the component $\pi_{q}$ of $\pi$ at the prime $q$ is of type $\rmI\rmI\rma$. Then $\pi$ admits a global Jacquet--Langlands transfer to an automorphic representation of $\bfG(\overline{\rmB})(\mathbb{A})$ by Theorem \ref{mult1-definite}. In particular the maximal ideal $\fracm$ associated to $\pi$ in the Construction \ref{Hecke-Algebra} $(3)$ lies in the support of $\calO_{\lambda}[\rmZ_{\rmH}(\overline{\rmB})]$ as a $\TT^{\Sigma\cup\{p\}}$-module. 

\begin{theorem}
Let $\pi$ be an automorphic representation of $\GSp_{4}(\mathbb{A})$ as above. Suppose that $\pi$ satisfy all the assumptions in Theorem \ref{arithmetic-level-raising}. Let $p$ be a level raising special prime for $\pi$ of length $m$. 

Then there exists an automorphic representation $\mathbf{\Pi}$ of $\GSp_{4}(\mathbb{A})$ of general type with weight $(3,3)$ and trivial central character such that
\begin{enumerate}
\item the component $\mathbf{\Pi}_{p}$ at $p$ is of type $\rmI\rmI\rma$;
\item we have an isomorphism of the residual Galois representation
\begin{equation*}
\overline{\rho}_{\pi, \lambda}\cong \overline{\rho}_{\mathbf{\Pi}, \lambda}.
\end{equation*}
\end{enumerate}
In this case we will say $\mathbf{\Pi}$ is a level raising of $\pi$. 
\end{theorem}
\begin{proof}
As remark in the above, the localization $\calO_{\lambda}[\rmZ_{\rmH}(\overline{\rmB})]_{\fracm}$
is non-zero. By Theorem \ref{arithmetic-level-raising}, we have an isomorphism 
\begin{equation*}
\rmH^{1}_{\sing}(\QQ_{p^{2}}, \rmH^{3}_{\rmc}(\overline{\rmX}_{\Pa}(\rmB), \rmR\Psi(\calO_{\lambda}(2)))_{\fracm})\xrightarrow{\sim}\calO_{\lambda}[\rmZ_{\rmH}(\overline{\rmB})]_{\fracm}/\det\phantom{.}\calT_{\lr}.
\end{equation*}
In particular $\rmH^{3}_{\rmc}(\overline{\rmX}_{\Pa}(\rmB), \rmR\Psi(\calO_{\lambda}(2)))_{\fracm}$ is non-zero. Since  $\rmH^{3}_{\rmc}(\overline{\rmX}_{\Pa}(\rmB), \rmR\Psi(\calO_{\lambda}(2)))_{\fracm}$ is a free $\rmR^{\ram}$-module, we can find a point $\eta^{\ram}\in \Spec\phantom{.}\rmR^{\ram}[1/\ell]$ in the support of 
\begin{equation*}
\rmH^{3}_{\rmc}(\overline{\rmX}_{\Pa}(\rmB), \rmR\Psi(\calO_{\lambda}(2)))_{\fracm}.
\end{equation*}
Then we can find  an automorphic representation $\mathbf{\Pi}$ of $\GSp_{4}(\mathbb{A})$ which satisfy the isomorphism 
\begin{equation*}
\overline{\rho}_{\pi, \lambda}\cong \overline{\rho}_{\mathbf{\Pi}, \lambda}.
\end{equation*}
In particular $\mathbf{\Pi}$ is of general type as $\overline{\rho}_{\mathbf{\Pi}, \lambda}$ is absolutely irreducible. By the Picard-Lefschetz formula applied to the Shimura variety ${\rmX}_{\Pa}(\rmB)$, we see that the monodromy of the associated Weil-Deligne representation of  $\rho_{\mathbf{\Pi}, \lambda}$ is necessarily is given by 
\begin{equation*}
\rmN_{\Pa}=\begin{pmatrix}0&&&\\&0&1&\\&&0&\\&&&0\\ \end{pmatrix}.
\end{equation*}
up to conjugation. Since $\mathbf{\Pi}$ is of general type and hence tempered at $p$, it follows that $\mathbf{\Pi}$ is of type $\rmI\rmI\rma$ by the local Langlands correspondence for non-supercuspidal representations, see for example \cite[Table 2]{Sch-Iwahori}. The theorem is proved.
\end{proof}


\begin{thebibliography}{Wang}
\bibitem[Art01]{Art-1} 
J.~Arthur,
{A stable trace formula. I. General expansions}, J. Inst. Math. Jussieu {\bf 1} (2002), no.{\bf 2}, 175--277. 
 
\bibitem[Art02]{Art-2} 
J.~Arthur,
{A stable trace formula .II. Global descent}, Invent. Math. {\bf 143} (2001), no.{\bf 1}, 157--220. 

\bibitem[Art03]{Art-Main} 
J.~Arthur,
{A stable trace formula .III. Proof of the main theorem}, Ann. of Math. (2) {\bf 158} (2003), 769--837. 

\bibitem[Art04]{Arthur-GSp} 
J.~Arthur, {Automorphic representations of $\GSp(4)$}
{in Contributions to automorphic forms, geometry, and number theory}, {Johns Hopkins Univ. Press}, {Baltimore, MD}, (2004), 65--81.

\bibitem[BBD]{BBD}
A.~ Beilinson, J.~ Bernstein and P.~Deligne and O.~ Gabber 
{Faisceaux pervers}, Asterisque 100, Paris, Soc. Math. Fr. 1982.
 
\bibitem[BD05]{BD-Main} 
M.~Bertolini and H.~Darmon,
{Iwasawa's main conjecture for elliptic curves over anticyclotomic {$\Bbb Z_p$}-extensions}, Ann. of Math. (2) {\bf 162}, (2005), 1--64. 

\bibitem[BS07]{BS-newform}
R.~Brooks and R.~Schmidt, {Local newforms for GSp(4)}, {Lecture Notes in Mathematics}, {\bf 1918} (2007), {Springer, Berlin}, {viii+307}.

\bibitem[Clo00]{Clo}
L.~Clozel, {Ribet's theorem for U(3)}, {Amer. J. Math.}, {\bf 122} (2000), 1265--1287. 



\bibitem[CS17]{CS-compact}
A.~Caraiani and P.~Scholze, {On the generic part of the cohomology of compact unitary Shimura
   varieties}, {Ann. of Math. (2)}, {\bf 186} (2017), no.{\bf3}, 649--766.
   
\bibitem[CS19]{CS-non-compact}
A.~Caraiani and P.~Scholze, {On the generic part of the cohomology of non-compact unitary Shimura
   varieties}, {Ann. of Math. (2)}, {\bf 199} (2024) no.{\bf2}, 483--590. 

\bibitem[DL76]{DL}
P.~Deligne and G.~Lusztig, Representations of reductive groups over finite fields, 
{Ann. of Math. (2)}  {\bf 103} (1976), no. {\bf1}, 103--161. 

\bibitem[Gab81]{Gab}
O.~Gabber, Purete de la cohomologie de MacPherson-Goresky redige par P. Deligne, prepublication I.H.E.S., 1981.

\bibitem[Gee11]{Gee11}              
T.~ Gee, {Automorphic lifts of prescribed types},  {Math. Ann.} {\bf 350 (1)}, (2011), 107--144.

\bibitem[Gro98]{Gross-Satake}
B.~Gross, {On the Satake isomorphism}, {Galois representations in arithmetic algebraic geometry (Durham, 1996)}, {London Math. Soc. Lecture Note Ser.  Vol. 254}, {Cambridge Univ. Press, Cambridge}, {1998}.

\bibitem[GT11]{GT11}	
W-T.~Gan and S.~Takeda, The local Langlands conjecture for {${\rm GSp}(4)$},
{Ann. of Math.}{\bf 3}(2011), no. {\bf 3}, 1841--1882.

\bibitem[GT19]{GT19}    
T.~Gee and O.~Ta\"ibi, Arthur's multiplicity formula for GSp(4) and restriction to Sp(4).
{J. Ec. polytech. Math.}(2019), {\bf6}, 469--535. 

\bibitem[Hal97]{Hal97}              
T.~ Hales, {The fundamental lemma for Sp(4)},  {Proc. American. Math. Soc. } {\bf 125 (1)} (1997), 301--308.

\bibitem[HL23]{HL-vanishing}              
L.~ Hamann and S.~Lee, {Torsion Vanishing for Some Shimura Varieties}, {arXiv:2309.08705}, preprint, 2023.

\bibitem[II10]{Ichino-Ikeda}
A.~Ichino and T.~ Ikeda,  {On the Periods of Automorphic Forms on Special Orthogonal Groups and the Gross--Prasad Conjecture}, {Geom. Funct. Anal. }{\bf19} (2010), 1378--1425.

\bibitem[Ill02]{Illusie-van}              
L.~ Illusie, {Sur la Formule de Picard-Lefschetz},  {Algebraic Geometry 2000, Azumino}, { Adv. Stud. Pure Math.}{\bf 36}, {Math. Soc. Japan, Tokyo (2002)}, 249--268. 

\bibitem[Ill03]{Illusie-per}              
L.~ Illusie, {Perversité et variation},  {manuscripta math.}{\bf 112},  {(2003)}, 271--295. 

\bibitem[Kos19]{Kos-a}
T.~Koshikawa, {Vanishing theorems for the mod p cohomology of some simple Shimura varieties}, {arXiv:1910.03147}, preprint, (2019).

\bibitem[Kos21]{Kos-b}
T.~Koshikawa, {On the generic part of the cohomology of local and global Shimura varieties}, {arXiv:2106.10602}, preprint, (2021).

\bibitem[Ko83]{Ko83}  
R.~Kottwitz, {Sign changes in harmonic analysis on reductive groups},
               {Trans. Amer. Math. Soc.} {\bf 278} (1983), 289--297.

\bibitem[KR00]{KR-Siegel}
S.~Kudla and M.~Rapoport, {Cycles on Siegel threefolds and derivatives of {E}isenstein series}, {Ann. Sci. \'{E}cole Norm. Sup. (4)} {\bf 33} (2000), no. {\bf5}, 695--756. 

\bibitem[Lau97]{La-Siegel}
G.~Laumon, {Sur la cohomologie \`a supports compacts des vari\'{e}t\'{e}s de Shimura
   pour ${\rm GSp}(4)_{\bf Q}$}, {Compositio Math.} {\bf 105} (1997), no. {\bf 3}, 267--359.

\bibitem[Lau05]{La-SiegelII}
G.~Laumon, {Fonctions z\^{e}tas des vari\'{e}t\'{e}s de Siegel de dimension trois}, {Formes automorphes. II. Le cas du groupe $\rm GSp(4)$} {Ast\'{e}risque.} {\bf 302} (2005), 1--66.   

\bibitem[Liu17]{Liu-cubic}
Y.~ Liu, {\em Bounding cubic-triple product {S}elmer groups of elliptic curves}, 
J. Eur. Math. Soc. (JEMS). {\bf 21} (2017) no.5  1411--1508.

\bibitem[LS13]{LS13}  
K-W.~Lan and J.~Suh, {Vanishing theorems for torsion automorphic sheaves on general {PEL}-type {S}himura varieties},
{Adv. Math.} {\bf242} (2013), 228--286. 
                              
\bibitem[LS18b]{LS18b}  
    K-W.~Lan and B.~Stroh, {Nearby cycles of automorphic \'{e}tale sheaves},
               {Compos. Math.} {\bf154} (2018), no. {\bf1}, 80--119.    

\bibitem[LS18a]{LS18a}  
    K-W.~Lan and B.~Stroh, {Nearby cycles of automorphic \'{e}tale sheaves II},
               {Cohomology of arithmetic groups, Springer Proc. Math. Stat.} {\bf245} (2018), Springer, Cham, 83--106.    
   
\bibitem[LTXZZ]{LTXZZ}    
Y.~Liu, Y.~Tian, L.~Xiao, W.~Zhang and X.~Zhu, {On the Beilinson-Bloch-Kato conjecture for Rankin-Selberg motives}, {Invent. math.} {\bf 228} (2022), 107--375.

\bibitem[LTXZZa]{LTXZZa}   
Y.~Liu, Y.~Tian, L.~Xiao, W.~Zhang and X.~Zhu, {Deformation of rigid conjugate self-dual Galois representations}, {arXiv:2108.06998}, preprint,(2021). 

\bibitem[MT02]{MT-van}
A.~Mokrane and J.~Tilouine, {Cohomology of Siegel varieties with p-adic integral coefficients and applications}, {Cohomology of Siegel varieties},  {Asterisque} no. {\bf280} (2002), 1--95.

\bibitem[Mok05]{Mok-GSp}
C.~Mok, {Galois representations attached to automorphic forms on {${\rm GL}_2$} over {CM} fields}, {Compos. Math.} {\bf150} (2014), no. {\bf 4}, 523--567.

\bibitem[Nek07]{Nekovar}
J.~ Nekov{\'a}{\v{r}}, {On the parity of ranks of {S}elmer groups. {III}.}, {\em Doc. Math.}, 12:243--274, 2007.

\bibitem[Oki22]{Oki}
 Y.~Oki, {On supersingular loci of Shimura varieties for quaternionic
 unitary groups of degree 2}, {Manuscripta Math.}  {\bf 167} (2022), 263--343.
    
\bibitem[Rib84]{ribet}
K.~Ribet, {Congruence relations between modular forms}, {Proceedings of the International Congress of Mathematicians}, {Vol. 1, 2 (Warsaw, 1983)}, 503--514, {PWN, Warsaw}, 1984. 

\bibitem[She79]{She1}
D.~Shelstad, {Characters and inner forms of a quasi-split group over $\RR$}, {Compositio. Math.} {\bf 39} (1979), 11--45. 

\bibitem[RW21]{RW}   
M.~ Rosner and R.~ Weissauer, {Global liftings between inner forms GSp4},
{arXiv:2103.14715}, preprint, 2021.    

\bibitem[RZ96]{RZ-space}
M.~Rapoport and T.~Zink, {Period spaces for {$p$}-divisible groups}, {Annals of Mathematics Studies} {\bf 141}, { Princeton University Press, Princeton, NJ} (1996), xxii+324.     

\bibitem[SGA7 II]{SGA7}
P.~Deligne and N.~ Katz, {Groupes de monodromie en g\'eom\'etrie alg\'ebrique. {II}}, {S{\'e}minaire de G{\'e}om{\'e}trie Alg{\'e}brique du Bois-Marie 1967--1969 (SGA 7, {II})}, {Lect. Notes Math. Vol. 340},  {Springer-Verlag, Berlin}, 1973.

\bibitem[Sch05]{Sch-Iwahori}
R.~Schmidt, {Iwahori-spherical representations of GSp(4) and Siegel
   modular forms of degree 2 with square-free level}, {J. Math. Soc. Japan}, {\bf 57} (2005), {no.1}, 259--293.  

\bibitem[Sch18]{Sch-packet}
R.~Schmidt, {Packet structure and paramodular forms}, {Trans. Amer. Math. Soc.}, {\bf 370} (2018) {no.5}, 3085--3112.

\bibitem[She79]{She1}
D.~Shelstad, {Characters and inner forms of a quasi-split group over $\RR$}, {Compositio. Math.} {\bf 39} (1979), 11--45.

\bibitem[ST93]{ST-classification}
P.~Sally and M.~Tadi\'{c}, {Induced representations and classifications for {${\rm GSp}(2,F)$} and {${\rm Sp}(2,F)$}}, {M\'{e}m. Soc. Math. France (N.S.)} {\bf 52} (1993), 75--133.

\bibitem[Sor09]{Sor-level-raising}
 C.~Sorensen, {Level-raising for Saito--Kurokawa forms}, {Compositio Math.}  {\bf 145} (2010), {no. 4}, 915--953.
 
 \bibitem[Sor09a]{Sor09}
 C.~Sorensen, {Potential level-lowering for $\rm{GSp}(4)$ },
 {J. Inst. Math. Jussieu.}  {\bf 8} (2009), 595--622.   
 
\bibitem[Sor10]{Sor-local-global}
 C.~Sorensen, {Galois representations attached to Hilbert-Siegel modular forms}, {Doc. Math.}  {\bf 15} (2010), 623--670. 
 
 \bibitem[Swe25]{Sweeting}
N.~Sweeting, {On the Bloch-Kato conjecture for some four-dimensional symplectic Galois representations}, {arXiv:2503.19226}, preprint.
 
 \bibitem[Tay91]{Tay-low}
 R.~Taylor, {Galois representations associated to Siegel modular forms of low
   weight}, {Duke Math. J.}  {\bf 63} (1991), 281--332.
 
\bibitem[Tay93]{Tay-Siegel}
 R.~Taylor, {On the $\ell$-adic cohomology of Siegel threefolds}, {Invent. Math.}  {\bf 114} (1993), 289--310. 

\bibitem[VH21]{VH19}
P.~van Hoften, A geometric Jacquet-Langlands correspondence for paramodular Siegel threefolds, 
{Math. Z.} (2021), no. {\bf 299},  2029--2061.

\bibitem[Wal97]{Wal-trans}
 J.~Waldspurger, {Le lemme fondamental implique le transfert}, {Compositio. Math.}  {\bf 105} (1997), 153--236.
 
\bibitem[Wang20a]{Wanga}    
H.~Wang, {On the Bruhat--Tits stratification of a quaternionic unitary Shimura variety}, {Math. Ann.} {\bf 376} (2020), 1107--1144.
 
\bibitem[Wang20b]{Wangb}    
H.~Wang, {On quaternionic unitary Rapoport--Zink spaces with parahoric level structures}, {IMRN} {\bf 7}( 2022), 5108--5151.

\bibitem[Wang22]{Wangc}
H.~Wang, { Arithmetic level raising on triple product of Shimura curves and Gross--Schoen Diagonal cycles I: Ramified case}, {Algebra Number Theory} {\bf 16}(10) (2022),  2289--2338.

\bibitem[Wang22a]{Wangd}
H.~Wang, {Deformation of rigid Galois representations and cohomology of certain quaternionic unitary Shimura varieties}, {arXiv:2204.07807}, preprint.

\bibitem[Wang22b]{Wange}
H.~Wang, {Level lowering on Siegel modular threefold of paramodular level}, {arXiv:1910.07569}, preprint.

\bibitem[Wei05]{Wei-Galois}    
R.~Weissauer, {Four dimensional Galois representations}, {Formes automorphes. II. Le cas du groupe $\rm GSp(4)$}, {Ast\'{e}risque}, {\bf 302}, {67--150}.

\bibitem[XZ17]{XZ}
L.~Xiao and X.~Zhu, {Cycles on Shimura varieties via geometric Satake}, {arXiv:1707.05700}, preprint.                           
                      
\end{thebibliography}
\end{document}